\documentclass[a4paper,12pt]{article}
\usepackage{tikz}
\usetikzlibrary{arrows}
\usetikzlibrary{calc}
\usetikzlibrary{positioning}
\usetikzlibrary{cd}
\usepackage{graphicx,amsmath,amssymb,amsthm,amscd,mathrsfs,mathabx,stmaryrd}
\usepackage[margin=2.5cm]{geometry}

\newtheorem{thm}{Theorem}[section]
\newtheorem{df}[thm]{Definition}
\newtheorem{prop}[thm]{Proposition}
\newtheorem{lem}[thm]{Lemma}
\newtheorem{cor}[thm]{Corollary}
\newtheorem{rem}[thm]{Remark}

\newtheorem{conj}[thm]{Conjecture}
\newtheorem{prob}[thm]{Problem}

\title{Deformation of moduli spaces 
of meromorphic $G$-connections on $\mathbb{P}^{1}$ via unfolding of irregular singularities
}
\date{}

\author{Kazuki Hiroe\footnote{
	The author is supported by JSPS KAKENHI Grant Number 20K03648.}\\
Department of Mathematics and Informatics, Chiba University\\
1-33, Yayoi-cho, Inage-ku, Chiba-shi, Chiba, 263-8522 JAPAN\\
email: {\tt kazuki@math.s.chiba-u.ac.jp}
}

\usepackage{imakeidx}
\makeindex[name=cN, title=Notation]

\begin{document}
\maketitle
\begin{abstract}
Unfolding singular points in linear differential equations 
is a classical technique for studying the properties of irregular 
singularities by relating them to regular singularities. 
In this paper, we propose a general framework for unfolding 
unramified irregular singularities of meromorphic connections
on the trivial principal $G$-bundle  
over $\mathbb{P}^{1}$.

One of our main results is the description of the unfolding of singularities 
in terms of deformations of their moduli spaces.
We show that every moduli space of irreducible meromorphic $G$-connections 
with unramified irregular singularities on $\mathbb{P}^{1}$ can be deformed 
into a moduli space of irreducible Fuchsian $G$-connections on $\mathbb{P}^{1}$.

Furthermore, we study the unfolding of additive Deligne-Simpson problems, 
in which the unfolding of irregular singularities naturally generates a family 
of such problems.
As an application of our main result, 
we prove that a Deligne-Simpson problem for $G$-connections 
with unramified irregular singularities admits a solution 
if and only if every unfolded Deligne-Simpson problem 
in the family admits a simultaneous solution.

We also provide a combinatorial and diagrammatic framework 
of the unfolding process in terms of spectral types and unfolding diagrams. 
Finally, we address a conjecture proposed by Oshima 
concerning the existence of irreducible $G$-connections
that realize prescribed spectral types and their unfoldings. 
Our main result gives an affirmative answer to this conjecture.
\end{abstract}

\tableofcontents

\section{Introduction}
Unfolding singular points in linear differential equations 
is a classical technique used to study the properties of irregular 
singularities by deforming them into regular ones. 
It is therefore natural to expect that moduli spaces of meromorphic 
connections with irregular singularities admit deformations into moduli 
spaces of meromorphic Fuchsian connections, that is, connections with 
only regular singularities. 
In this paper, we consider moduli spaces of meromorphic connections 
on the trivial principal $G$-bundle over $\mathbb{P}^1$ with unramified irregular 
singularities, and we show that each such moduli space admits a deformation 
whose generic fibers are moduli spaces of Fuchsian connections on $\mathbb{P}^1$. 
We also apply this deformation to address several existence problems concerning 
irreducible linear differential equations on $\mathbb{P}^1$.

\subsection{Classical and recent examples}\label{sec:classnew}
There are many classical and recent examples of the confluence 
and unfolding of singular points in differential equations on 
the Riemann sphere $\mathbb{P}^1$. 
As a first example, let us recall the differential equation for 
the Gauss hypergeometric function, which has regular singularities 
at $0$, $1$, and $\infty$. 
It is classically known that the confluence of the singularities at 
$1$ and $\infty$ yields the differential equation for the Kummer confluent 
hypergeometric function, which has a regular singularity at $0$ 
and an irregular singularity at $\infty$. 
Furthermore, the confluence of the singularities at $0$ and $\infty$ 
leads to the differential equation for the Hermite-Weber function, 
which has an irregular singularity at $\infty$. 
This process is illustrated by the following diagram, 
where the arrows represent the confluence of singularities:
 \vspace{3mm}
\begin{center}
    \begin{tikzpicture}[auto]
        \node[shape=rectangle, draw] (a) at (-3, 0) {Gauss};
        \node[shape=rectangle, draw] (b) at (0, 0) {Kummer};
        \node[shape=rectangle, draw] (c) at (3.5, 0) {Hermite-Weber};
        \draw[->] (a) to (b);
        \draw[->] (b) to (c);
\end{tikzpicture}.
\end{center}\vspace{3mm}

\noindent 
Similarly, in the case of the Heun differential equation, 
it is known that several differential equations with irregular 
singularities arise through the confluence of singular points: 
the confluent Heun differential equation, the biconfluent Heun differential equation, 
the doubly-confluent Heun differential equation, 
and the triconfluent Heun differential equation:
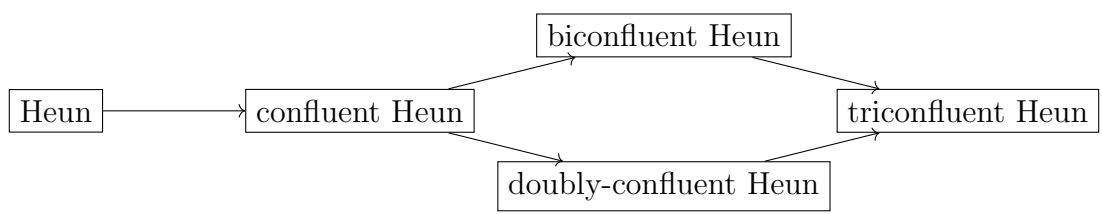
\vspace{3mm}
\begin{center}
    \begin{tikzpicture}[auto]
        \node[shape=rectangle, draw] (a) at (-10, 0) {Heun};
        \node[shape=rectangle, draw] (b) at (-7, 0) {confluent Heun};
        \node[shape=rectangle, draw] (c) at (-4, 1) {biconfluent Heun};
        \node[shape=rectangle, draw] (d) at (-4, -1) {doubly-confluent Heun};
        \node[shape=rectangle, draw] (e) at (-1, 0) {triconfluent Heun};
        \draw[->] (a) to (b);
        \draw[->] (b) to (c);
        \draw[->] (b) to (d);
        \draw[->] (c) to (e);
        \draw[->] (d) to (e);
\end{tikzpicture}.
\end{center}\vspace{3mm}

\noindent
In the more recent papers \cite{Oshi1} and \cite{Oshi2}, 
Oshima constructed versal unfolding families of unramified 
irregular singularities of differential equations without 
accessory parameters. In \cite{KNS}, Kawakami, Nakamura, 
and Sakai studied linear Fuchsian differential equations 
on $\mathbb{P}^1$ that are related to four-dimensional 
Painlev\'e-type equations, 
and compiled lists of linear differential equations obtained from 
these Fuchsian equations via confluence of singular points. 
Here is an example from the list in \cite{KNS}:
\begin{center}
    \begin{tikzpicture}[auto]
        \node[shape=rectangle, draw, font=\small, inner sep=1.2] (a) at (-10, 0) {22,22,22,211};
        \node[shape=rectangle, draw, font=\small, inner sep=1.2] (b) at (-7, -1) {(2)(11),22,22};
        \node[shape=rectangle, draw, font=\small, inner sep=1.2] (c) at (-7, 1) {(2)(2),22,211};
        \node[shape=rectangle, draw, font=\small, inner sep=1.2] (d) at (-4, -1.5) {(2)(2),(2)(11)};
        \node[shape=rectangle, draw, font=\small, inner sep=1.2] (e) at (-4, 0) {((2))((11)),22};
        \node[shape=rectangle, draw, font=\small, inner sep=1.2] (f) at (-4, 1.5) {((2))((2)),211};
        \node[shape=rectangle, draw, font=\small, inner sep=1.2] (g) at (-1, 0) {(((2)))(((11)))};
        \draw[->] (a) to (b);
        \draw[->] (a) to (c);
        \draw[->] (b) to (d);
        \draw[->] (b) to (e);
        \draw[->] (c) to (d);
        \draw[->] (c) to (e);
        \draw[->] (c) to (f);
        \draw[->] (e) to (g);
        \draw[->] (f) to (g);
        \draw[->] (d) to (g);
\end{tikzpicture}.
\end{center}\vspace{3mm}

\noindent
In this diagram, the collections of partitions of the integer $4$,
with parentheses, represent the spectral types at the singular points 
of differential equations of rank $4$. 
For further details, we refer to the original paper \cite{KNS}. 
We will also introduce a generalization of spectral types for meromorphic 
$G$-connections later in this paper (see Sections \ref{sec:spunf}, \ref{sec:spectral} and 
\ref{sec:exII}).

The purpose of this paper is to propose a general framework for the unfolding 
of singular points in linear differential equations on $\mathbb{P}^1$, 
encompassing both classical and recent examples such as those mentioned above. 
Furthermore, since in these examples the number of accessory parameters is 
preserved under unfolding, it is natural to expect that the unfolding induces 
deformations of the spaces of accessory parameters, that is, 
the moduli spaces of meromorphic connections. 
In this paper, we describe the unfolding of singular points in differential 
equations as deformations of their moduli spaces.

\subsection{Unfolding of  
canonical forms for meromorphic \(G\)-connections on the formal punctured disk}
Before discussing the deformation of moduli spaces of connections, 
we introduce the deformation of canonical forms of meromorphic connections 
on the formal punctured disk. 

Let $G$ be a connected linear reductive algebraic group defined over $\mathbb{C}$. 
As recalled in Section \ref{sec:HTLBV}, 
according to the formal classification theory developed by Hukuhara, Turrittin, 
Levelt, Babbitt, and Varadarajan, meromorphic $G$-connections 
on the formal punctured disk are classified by canonical forms, i.e., 
$\mathfrak{g}$-valued meromorphic $1$-forms
\[
    H\,dz=\left(\frac{H_{k}}{z^{k}}+\cdots+ 
    \frac{H_{1}}{z}+H_{\mathrm{res}}\right)\frac{dz}{z}
\]
defined on ramified coverings of the punctured disk, 
where the $H_i$ are mutually commuting semisimple elements and 
$H_{\mathrm{res}}$ commutes with all $H_i$. 
In other words, for any meromorphic $G$-connection $\nabla$ 
on the formal punctured disk, 
there exists a $\mathfrak{g}$-valued $1$-form 
$H\,dz$
of the above type such that 
$\nabla$ is gauge-equivalent to $H\,dz$ on a suitable ramified covering. 
In this paper, we focus on the unramified case.

We now explain the unfolding of irregular singularities of unramified canonical 
forms. Let $H$ be a canonical form as above. 
For $\mathbf{c}=(c_{0},c_{1},\ldots,c_{k})\in \mathbb{C}^{k+1}$, 
define
\[
    H(\mathbf{c})
    =\left(\frac{H_{k}}{(z-c_{1})(z-c_{2})
    \cdots (z-c_{k})}+\cdots+ 
    \frac{H_{1}}{(z-c_{1})}+H_{\mathrm{res}}\right)\frac{1}{(z-c_{0})}
\]
and call this function on $\mathbb{C}^{k+1}$ {\em unfolding} of $H$. 
Clearly, $H(\mathbf{0})=H$, so $H(\mathbf{c})$ defines a deformation of $H$. 
Moreover, we can write $H(\mathbf{c})$ as a partial fraction decomposition:
\[
    H(\mathbf{c})=
    \frac{A^{[0]}}{z-c_{0}}+\frac{A^{[1]}}{z-c_{1}}+\cdots 
    +\frac{A^{[k]}}{z-c_{k}}
\]
which represents a sum of regular singular canonical forms for generic $\mathbf{c}$, 
i.e., when $c_{i}\neq c_{j}$ for $i\neq j$. 
Thus, $H(\mathbf{c})$ provides a deformation of the unramified irregular 
canonical form $H$ into a sum of regular singular canonical forms.

To analyze this deformation in more detail, 
we introduce a stratification of $\mathbb{C}^{k+1}$
associated with partitions of the index set 
$\{0,1,\ldots,k\}$.
For a partition $\mathcal{I}=I_{0}\sqcup I_{2}\sqcup \cdots \sqcup I_{r}=\{0,1,\ldots,k\}$, 
define the subset
\begin{align*}
	C(\mathcal{I})&=
	\left\{
		(c_{0},c_{1},\ldots,c_{k})\in \mathbb{C}^{k+1}\,\middle|\,
		\begin{array}{ll}
			c_{i}=c_{j}&\text{if }i,j\in I_{l}\text{ for some }l,\\
			c_{i}\neq c_{j}&\text{otherwise}
		\end{array}
	\right\},\\
	&\cong \left\{
		(c_{I_{0}},c_{I_{1}},\ldots,c_{I_{r}})\in \mathbb{C}^{r+1}\mid 
		c_{I_{i}}\neq c_{I_{j}}\text{ for }i\neq j
		\right\}
\end{align*}
These subsets define a stratification of $\mathbb{C}^{k+1}$:
\[
\mathbb{C}^{k+1}=\bigsqcup_{\mathcal{I}\colon 
\text{partitions of }\{0,1,\ldots,k\}}\mathcal{C}(\mathcal{I}),
\] 
see Section \ref{sec:partstra}.

Then for $\mathbf{c}\in \mathcal{C}(\mathcal{I})$,
we have the partial fraction decomposition
\[
	H(\mathbf{c})=
	\sum_{j=0}^{r}
    \left(
        \frac{A^{[j]}_{k_{j}}}{(z-c_{I_{j}})^{k_{j}}}
        +\cdots+ 
        \frac{A^{[j]}_{1}}{z-c_{I_{j}}}
        +A^{[j]}_{0}
	\right)\frac{1}{z-c_{I_{j}}}
\] 
which is again a sum of unramified canonical forms, where  
the pole orders correspond to  the cardinalities $\sharp I_{i}$ of the components in the partition $\mathcal{I}$.

Thus  
for each fixed $\mathbf{c}\in \mathbb{C}^{k+1}$,
there exists a unique stratum $\mathcal{C}(\mathcal{I})$ 
containing $\mathbf{c}$,
and 
we can interpret the deformed canonical form $H(\mathbf{c})$
as the collection 
\[
    H(\mathbf{c})=\left(
        \left(
            \frac{A^{[j]}_{k_{j}}}{(z-c_{I_{j}})^{k_{j}}}
            +\cdots+ 
            \frac{A^{[j]}_{1}}{z-c_{I_{j}}}
            +A^{[j]}_{0}
        \right)\frac{1}{z-c_{I_{j}}}\right)_{j=0,\ldots,r}
\] 
of canonical forms at $z=c_{I_{j}}$ for $j=0,\ldots,r$.
Therefore, the family  
$(H(\mathbf{c}))_{\mathbf{c}\in \mathbb{C}^{k+1}}$
can be viewed as 
a family of collections of unramifed canonical forms, 
organized according to the stratification of $\mathbb{C}^{k+1}$.

On the unique closed stratum $\mathcal{C}(\{0,\ldots,k\})$,
corresponding to the trivial partition, 
$H(\mathbf{c})$ takes the form  
\[
H(\mathbf{c})=\left(\frac{H_{k}}{(z-c_{0})^{k}}+\cdots+ 
\frac{H_{1}}{(z-c_{0})}+H_{\mathrm{res}}\right)\frac{1}{(z-c_{0})}
\]
which is essentially same as the original canonical form $H$.
On
the 
unique open stratum
$\mathcal{C}(\{0\}\sqcup\cdots\sqcup\{k\})$,
corresponding to the finest partition, 
the $H(\mathbf{c})$ is a collection 
of regular singular canonical forms.  
On the other strata, the $H(\mathbf{c})$ represent collections of 
canonical forms with intermediate singularities.

\subsection{Statement of the main theorem}
Now we explain the main theorem of this paper.
Let $G$ be a connected reductive linear algebraic group defined over 
$\mathbb{C}$.
Let $|D|=\{a_{1},a_{2},\ldots,a_{d}\}\subset \mathbb{P}^{1}$
be a finite set, 
and let $D=\sum_{i=1}^{d}(k_{a_{i}}+1)a_{i}$
be an effective divisor with non negative integers $k_{a_i}$.
Consider a collection $\mathbf{H}=(H^{(a)})_{a\in |D|}$
of unramified canonical forms
\[
    H^{(a)}\,dz_{a}=\left(\frac{H^{(a)}_{k_{a}}}{z_{a}^{k_{a}}}+\cdots+ 
    \frac{H^{(a)}_{1}}{z_{a}}+H^{(a)}_{\mathrm{res}}\right)\frac{dz_{a}}{z_{a}}
\]
where $z_{a}$ are coordinate 
centered at $a\in D$.
Following \cite{Boa1} by Boalch,
we define 
the moduli space $\mathcal{M}^{\mathrm{ir}}_{\mathbf{H}}$ 
of irreducible meromorphic connections 
on the trivial $G$-bundle over $\mathbb{P}^{1}$
with local canonical form $H^{(a)}$ 
at each singular points $a\in |D|$, see Definition \ref{df:irredmod}.
If nonempty, $\mathcal{M}^{\mathrm{ir}}_{\mathbf{H}}$ is a holomorphic 
symplectic orbifold, see Proposition \ref{prop:orbifold}.

For each $a\in |D|$, consider the deformation 
$H^{(a)}(\mathbf{c}^{(a)})$ of the canonical form $H^{(a)}$,
with parameter space $\mathbb{C}^{k_{a}+1}$ stratified as 
in the previous section.
These stratifications induce a product stratification on  
$\prod_{a\in |D|}\mathbb{C}^{k_{a}+1}$.
For any 
fixed $\mathbf{c}=(\mathbf{c}^{(a)})_{a\in |D|}\in \prod_{a\in |D|}\mathbb{C}^{k_{a}+1}$,
there exists a unique stratum $\prod_{a\in |D|}\mathcal{C}(\mathcal{I}^{(a)})$
containing $\mathbf{c}$, and we regard 
$\mathbf{H}(\mathbf{c})=(H^{(a)}(\mathbf{c}^{(a)}))_{a\in |D|}$
as a collection of canonical forms. 

Thus 
we can define
the moduli space $\mathcal{M}^{\mathrm{ir}}_{\mathbf{H}(\mathbf{c})}$ 
of irreducible meromorphic $G$-connections associated to the 
collection  $\mathbf{H}(\mathbf{c})$. 
Then the following holds:
\begin{thm}[Theorem \ref{thm:main1}]\label{thm:ourthm}
	Suppose $\mathcal{M}_{\mathbf{H}}^{\mathrm{ir}}\neq \emptyset$.
    Then there exists an open neighborhood $\widetilde{\mathbb{B}}_{\mathbf{H}}$
    of $\mathbf{0}$ in $\prod_{a\in |D|}\mathbb{C}^{k_{a}+1}$
    and a complex orbifold $\mathcal{M}_{\mathbf{H},\widetilde{\mathbb{B}}_{\mathbf{H}}}^{\mathrm{ir}}$
    with a holomorphic map $\pi_{\widetilde{\mathbb{B}}_{\mathbf{H}}} \colon \mathcal{M}_{\mathbf{H},\widetilde{\mathbb{B}}_{\mathbf{H}}}^{\mathrm{ir}}
    \rightarrow \widetilde{\mathbb{B}}_{\mathbf{H}}$
    satisfying the following.
	\begin{enumerate}
		\item 
        The orbifold $\mathcal{M}_{\mathbf{H},\widetilde{\mathbb{B}}_{\mathbf{H}}}^{\mathrm{ir}}$
		is a deformation of $\mathcal{M}_{\mathbf{H}}^{\mathrm{ir}}$, i.e.,
        $\pi_{\widetilde{\mathbb{B}}_{\mathbf{H}}}\colon \mathcal{M}_{\mathbf{H},\widetilde{\mathbb{B}}_{\mathbf{H}}}^{\mathrm{ir}}
		\rightarrow \widetilde{\mathbb{B}}_{\mathbf{H}}$
		is a surjective submersion and
		we have 
		\[
			\mathcal{M}_{\mathbf{H}}^{\mathrm{ir}}\cong \pi_{\widetilde{\mathbb{B}}_{\mathbf{H}}}^{-1}(\mathbf{0}).
		\] 
        \item Every stratum $\prod_{a\in |D|}\mathcal{C}(\mathcal{I}^{(a)})$
        of $\prod_{a\in |D|}\mathbb{C}^{k_{a}+1}$ has the nonempty intersection with 
        the base space $\widetilde{\mathbb{B}}_{\mathbf{H}}$.
		\item For each $\mathbf{c}\in \widetilde{\mathbb{B}}_{\mathbf{H}}$
		the fiber $\pi_{\widetilde{\mathbb{B}}_{\mathbf{H}}}^{-1}(\mathbf{c})$ is isomorphic to an open dense 
		subspace of $\mathcal{M}_{\mathbf{H}(\mathbf{c})}^{\mathrm{ir}}$.
	\end{enumerate}
\end{thm}
In particular, if $\mathbf{c}\in \prod_{a\in |D|}\mathcal{C}(\{0\}\sqcup\cdots\sqcup \{k_{a}\})\cap \widetilde{\mathbb{B}}_{\mathbf{H}}$,
then $\mathcal{M}_{\mathbf{H}(\mathbf{c})}^{\mathrm{ir}}$ is a moduli space of Fuchsian $G$-connections.
Therefore, this theorem shows
that every (non-empty) moduli space $\mathcal{M}_{\mathbf{H}}^{\mathrm{ir}}$
of $G$-connections with unramified irregular singularities admits a deformation to a moduli space of Fuchsian
$G$-connections.
This raises the following natural question:
\vspace{3mm}

	\noindent\textit{What kind of moduli spaces $\mathcal{M}_{\mathbf{H}(\mathbf{c})}^{\mathrm{ir}}$ 
	of meromorphic connections 
	appear on the other strata of $\widetilde{\mathbb{B}}_{\mathbf{H}}$?}
\vspace{3mm}\\
As observed in the examples from Section \ref{sec:classnew},  
various types of differential equations 
emerge during the unfolding of irregular singular points. 
For example, unfolding the triconfluent Heun equation yields
the biconfluent Heun, doubly-confluent Heun,
confluent Heun, and finally the classical Heun equation.
Therefore, the question concerning the deformation of moduli spaces 
of connections can be reformulated as the following 
natural problem regarding the unfolding of irregular singularities in    
differential equations:\vspace{3mm}

	\noindent\textit{What kind of differential equations can arise from a given equation 
	through the unfolding of its irregular singularities?}
	\vspace{3mm}\\
In the next section, in response to the above questions,
we shall extract discrete data from the canonical 
forms $\mathbf{H}$, called the spectral types,
and introduce the unfolding diagram of these spectral types.
Together, the spectral types and their unfolding diagrams 
provide a combinatorial and diagrammatic framework for understanding
how
the singularities and local canonical forms 
of $G$-connections change and relate to one another
during the deformation of the moduli space $\mathcal{M}_{\mathbf{H}}^{\mathrm{ir}}$.

We now make some remarks on related previous work.
The relationship between the confluence or unfolding of 
singular points in linear differential equations and 
the deformation of their moduli spaces has  
already been studied by several researchers,
particularly in the case of $G=\mathrm{GL}_{n}$.
In \cite{Ina}, Inaba 
constructed a one-parameter deformation of 
moduli spaces of meromorphic connections with unramified irregular 
singularities and discussed applications
to their isomonodromic deformations.
The main theorem above generalizes 
Theorem 2.11 in \cite{Ina},
where Inaba considered only a special class of unramified irregular singularities,
referred to as  
{\em generic} unramified irregular singularities.
In \cite{GMR} Gaiur, Mazzocco, and Rubtsov studied
the confluence of singular points in linear differential equations 
and the deformation of phase spaces of their isomonodromic deformations
from a different perspective, and further explored 
their quantizations.

Analyzing Stokes structures 
is another important application of the unfolding, which is not 
addressed in this paper.
There is a substantial body of classical and recent work on this topic: 
see for example, \cite{Scha} by Sch\"afke, \cite{Ram} by Ramis,
\cite{Zhan} by Zhang, \cite{Gul} by Glutsuk,
\cite{LaRo} by Lambert-Rousseau, \cite{HuRo} by 
Hurtubise-Rousseau, \cite{Kl} by Klime\v{s},
and the references therein. 
\subsection{Spectral types and unfolding diagrams}\label{sec:spunf}
Let  $H\,dz=\left(\sum_{i=1}^{k}H_{i}/z^{i}+H_{\mathrm{res}}\right)dz/z$
be an unramified canonical form.
Let $H_{\mathrm{res}}=H_{0}+J_{0}$ be the Jordan decomposition 
of $H_{\mathrm{res}}$ with the semisimple $H_{0}$ and the nilpotent $J_{0}$.
Fix a set of simple roots $\Pi$ of $G$.
Then, under a suitable normalization via the action of the Weyl group of $G$,
we associate to $H$ a sequence of subsets of simple roots defined by
\[
    \Pi_{i}:=\{\alpha\in \Pi\mid \alpha(H_{k})=\cdots=\alpha(H_{i+1})=\alpha(H_{i})=0\}
\] 
for $i=0,1,\ldots,k$. 
Let $L_{i}$ be the Levi subgroup of $G$ associated with $\Pi_{i}$.
Then the pair 
\[
    \mathrm{sp}(H):=(\Pi_{k}\supset \cdots \supset \Pi_{1}\supset \Pi_{0};[J_{0}])
\]
consisting of the sequence of subsets of simple roots 
and the $L_{1}$-orbit of $J_{0}$, is called the {\em spectral type} 
of the canonical form $H$.
For a collection $\mathbf{H}=(H^{(a)})_{a\in |D|}$ of canonical forms, 
we define the collection of spectral types 
$\mathrm{sp}(\mathbf{H}):=(\mathrm{sp}(H^{(a)}))_{a\in |D|}$.
Alternatively,
without reference to a specific canonical form, we define 
an abstract spectral type as 
a pair $(\Pi_{k}\supset \cdots \supset \Pi_{1}\supset \Pi_{0};[J_{0}])$,
where $J_{0}$ is a nilpotent element in the Lie algebra $\mathfrak{l}_{0}$ of the Levi subgroup $L_{0}$.
For more details, see Section \ref{sec:spectral}.

We now describe how the unfolding $H(\mathbf{c})$
of $H$ is reflected in its spectral type $\mathrm{sp}(H)$.
First we  introduce the unfolding of spectral types.
Let $S=(\Pi_{k}\supset \cdots \supset \Pi_{1}\supset \Pi_{0};[J_{0}])$
be an abstract spectral type.
Consider the set $\mathcal{P}_{[k+1]}$
of all partitions of the index set $\{0,1,\ldots,k\}$,
partially ordered by the refinement of partitions.
For a partition $\mathcal{I}=I_{0}\sqcup \cdots \sqcup I_{r}=\{0,1,\ldots,k\}$,
each component $I_{j}$ defines a subsequence $\Pi_{I_{j}}$
of $\Pi_{k}\supset \cdots \supset \Pi_{1}\supset \Pi_{0}$
by reindexing.
Without loss of generality,
we may assume that $0\in I_{0}$.
To the partition $\mathcal{I}$,
we associate the collection of spectral types $S^\mathcal{I}=(S^{I_{j}})_{j=0,1,\ldots,r}$
defined by 
\[
    S^{I_{0}}:=(\Pi_{I_{0}}; [J_{0}])\quad \text{and}\quad 
    S^{I_{j}}:=(\Pi_{I_{j}}; [0])\text{ for }
    j=1,\ldots,r,
\]
which corresponds to the decomposition of 
the sequence $\Pi_{k}\supset \cdots \Pi_{1}\supset \Pi_{0}$
with regard to the partition $\mathcal{I}$.
The following shows that this decomposition actually describes 
the unfolding $\mathbf{H}(\mathbf{c})$ 
of the collection $\mathbf{H}=(H^{(a)})_{a\in |D|}$ .
\begin{prop}[Corollary \ref{cor:specdecomp}]\label{prop:introprp}
Let  $\widetilde{\mathbb{B}}_{\mathbf{H}}$ be as  
in Theorem \ref{thm:ourthm}. 
Then 
for each stratum $\prod_{a\in |D|}\mathcal{C}(\mathcal{I}^{(a)})$ of $\prod_{a\in |D|}\mathbb{C}^{k_{a}+1}$
and each 
$\mathbf{c}\in \widetilde{\mathbb{B}}_{\mathbf{H}}\cap \prod_{a\in |D|}\mathcal{C}(\mathcal{I}^{(a)})$,
we have 
\[
	\mathrm{sp}(\mathbf{H}(\mathbf{c}))=(\mathrm{sp}(H^{(a)})^{\mathcal{I}^{(a)}})_{a\in |D|}.
\]
\end{prop}
That is, 
on each stratum $\prod_{a\in |D|}\mathcal{C}(\mathcal{I}^{(a)})$,
the spectral types of $\mathbf{H}(\mathbf{c})$ 
are all constant,
and the spectral type
on the stratum
is obtained by decomposing the original spectral type 
$\mathrm{sp}(\mathbf{H})$
with respect to the partitions $\mathcal{I}^{(a)}$, for each $a\in |D|$.

We now introduce a diagrammatic description of
the unfolding 
of spectral types.
Let 
\[
	\mathbf{S}=(S_{i})_{i=1,\ldots,d}=(\Pi^{(i)}_{k_{i}}\supset \cdots\supset \Pi^{(i)}_{1}\supset \Pi^{(i)}_{0};[J^{(i)}_{0}])_{i=1,\ldots,d}
\]
be a collection of abstract spectral types.
The product
 $\prod_{i=1}^{d}\mathcal{P}_{[k_{i}+1]}$ naturally becomes a poset as well, i.e.,
for $(\mathcal{I}^{(i)})_{i=1,\ldots,d},(\mathcal{J}^{(i)})_{i=1,\ldots,d}\in \prod_{i=1}^{d}\mathcal{P}_{[k_{i}+1]}$,
we have $(\mathcal{I}^{(i)})_{i=1,\ldots,d}\le (\mathcal{J}^{(i)})_{i=1,\ldots,d}$ 
if and only if $\mathcal{I}^{(i)}\le \mathcal{J}^{(i)}$ for all $i=1,\ldots,d.$
Thus we obtain the Hasse diagram of the poset $\prod_{i=1}^{d}\mathcal{P}_{[k_{i}+1]}$.
Then we attach to each vertex $(\mathcal{I}^{(i)})_{i=1,\ldots,d}\in \prod_{i=1}^{d}\mathcal{P}_{[k_{i}+1]}$
the corresponding unfolding $\mathbf{S}^{(\mathcal{I}^{(i)})_{i=1,2,\ldots,d}}$
of $\mathbf{S}$.
Then we call the resulting diagram the {\em unfolding diagram} of $\mathbf{S}$.
We notice that there may appear same collections of abstract spectral types in several vertices,
namely,
it may happen $\mathbf{S}^{(\mathcal{I}^{(i)})_{i=1,\ldots,d}}=\mathbf{S}^{(\mathcal{J}^{(i)})_{i=1,\ldots,d}}$
for some $(\mathcal{I}^{(i)})_{i=1,\ldots,d}, (\mathcal{J}^{(i)})_{i=1,\ldots,d}
\in \prod_{i=1}^{d}\mathcal{P}_{[k_{i}+1]}$.
Therefore by identifying these same collections of abstract spectral types
we can obtain the smaller diagram, which we call the {\em reduced unfolding diagram}
of $\mathbf{S}$. See Section \ref{sec:unfspc} for the detail.

Let us see some examples.
Let $G=\mathrm{GL}_{2}$
and we consider
the canonical form $H$ with 
$\mathrm{sp}(H)=(\emptyset\supset \emptyset\supset\emptyset\supset\emptyset;[0])$
which is the canonical form for the triconfluent Heun equation,
see Section \ref{sec:exI}.
Then the deformation of the moduli space $\mathcal{M}_{(H)}^{\mathrm{ir}}$
given in Theorem \ref{thm:ourthm} is depicted by the 
following reduced unfolding diagram of the spectral type,\vspace{3mm}\\
\tiny
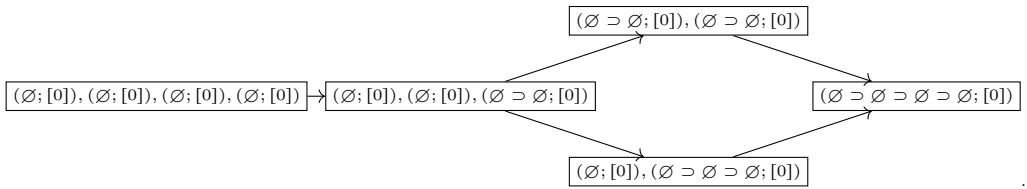
\begin{center}
	\begin{tikzpicture}[auto]
		\node[shape=rectangle, draw] (a) at (-8, 0) {$(\emptyset;[0]),(\emptyset;[0]),(\emptyset;[0]),(\emptyset;[0])$};
		\node[shape=rectangle, draw] (b) at (-4, 0) {$(\emptyset;[0]),(\emptyset;[0]),(\emptyset\supset\emptyset;[0])$};
		\node[shape=rectangle, draw] (c) at (-1, 1) {$(\emptyset\supset \emptyset ;[0]),(\emptyset\supset\emptyset;[0])$};
		\node[shape=rectangle, draw] (d) at (-1, -1) {$(\emptyset;[0]),(\emptyset\supset\emptyset\supset\emptyset;[0])$};
		\node[shape=rectangle, draw] (e) at (1.5, 0) {$(\emptyset\supset \emptyset\supset\emptyset\supset\emptyset;[0])$};
		\draw[->] (a) to (b);
		\draw[->] (b) to (c);
		\draw[->] (b) to (d);
		\draw[->] (c) to (e);
		\draw[->] (d) to (e);
	\end{tikzpicture}.
\end{center}
\normalsize\vspace{3mm}

\noindent
We notice that this diagram coincides with the diagram 
in Section \ref{sec:classnew}
which describes the confluence of singular points of Heun equations.
Therefore we can say that the deformation of $\mathcal{M}_{(H)}^{\mathrm{ir}}$
recovers the confluence procedure of the Heun equation.

Also let  $G=\mathrm{GL}_{4}$
and we consider the canonical form $H$ with 
\[
	\mathrm{sp}(H)=\left(\{e_{12},e_{34}\}\supset\{e_{12},e_{34}\}\supset\{e_{12},e_{34}\}\supset \{e_{12}\};[0]\right).
\]
Here $e_{ij}$ for $1\le i<j\le 4$ are standard simple roots of $\mathrm{GL}_{4}$.
Then the deformation of the moduli space $\mathcal{M}_{(H)}^{\mathrm{ir}}$
is depicted by the 
reduced unfolding diagram,\vspace{3mm}\\
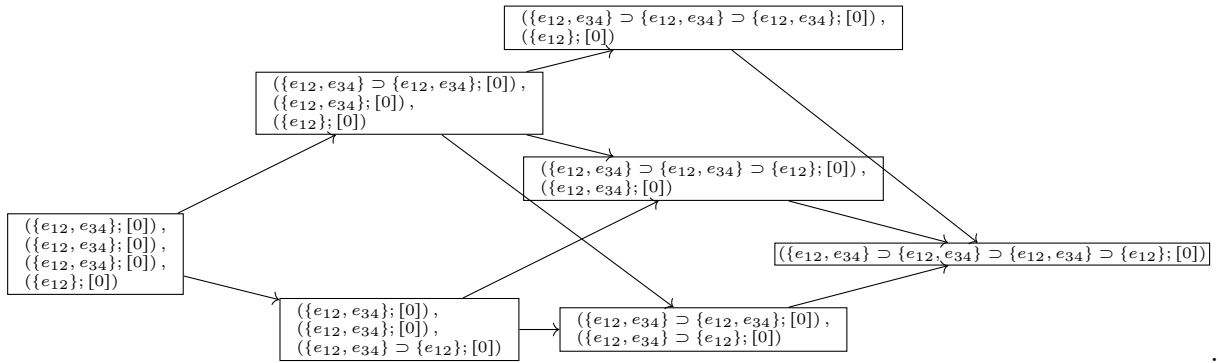
\begin{center}
    \begin{tikzpicture}[auto]
        \node[shape=rectangle, draw, font=\tiny, inner sep=1.2] (a) at (-8, 0.5) {$\begin{array}{l}\left(\{e_{12},e_{34}\};[0]\right),\\\left(\{e_{12},e_{34}\};[0]\right),\\ \left( \{e_{12},e_{34}\};[0]\right),\\\left(\{e_{12}\};[0]\right)\end{array}$};
        \node[shape=rectangle, draw, font=\tiny, inner sep=1.2] (b) at (-6.5, -1) {$\begin{array}{l}\left(\{e_{12},e_{34}\};[0]\right),\\ \left(\{e_{12},e_{34}\};[0]\right),\\ \left( \{e_{12},e_{34}\}\supset\{e_{12}\};[0]\right)\end{array}$};
        \node[shape=rectangle, draw, font=\tiny, inner sep=1.2] (c) at (-6.5, 2) {$\begin{array}{l}\left(\{e_{12},e_{34}\}\supset\{e_{12},e_{34}\};[0]\right),\\ \left(\{e_{12},e_{34}\};[0]\right),\\ \left(\{e_{12}\};[0]\right)\end{array}$};
        \node[shape=rectangle, draw, font=\tiny, inner sep=1.2] (d) at (-2.5, -1) {$\begin{array}{l}\left(\{e_{12},e_{34}\}\supset\{e_{12},e_{34}\};[0]\right),\\ \left( \{e_{12},e_{34}\}\supset\{e_{12}\};[0]\right)\end{array}$};
        \node[shape=rectangle, draw, font=\tiny, inner sep=1.2] (e) at (-2.3, 1) {$\begin{array}{l}\left(\{e_{12},e_{34}\}\supset\{e_{12},e_{34}\}\supset\{e_{12}\};[0]\right),\\ \left( \{e_{12},e_{34}\};[0]\right)\end{array}$};
        \node[shape=rectangle, draw, font=\tiny, inner sep=1.2] (f) at (-2.3, 3) {$\begin{array}{l}\left(\{e_{12},e_{34}\}\supset\{e_{12},e_{34}\}\supset\{e_{12},e_{34}\};[0]\right),\\ \left(\{e_{12}\};[0]\right)\end{array}$};
        \node[shape=rectangle, draw, font=\tiny, inner sep=1.2] (g) at (1.4, 0) {$\left(\{e_{12},e_{34}\}\supset\{e_{12},e_{34}\}\supset\{e_{12},e_{34}\}\supset \{e_{12}\};[0]\right)$};
        \draw[->] (a) to (b);
        \draw[->] (a) to (c);
        \draw[->] (b) to (d);
        \draw[->] (b) to (e);
        \draw[->] (c) to (d);
        \draw[->] (c) to (e);
        \draw[->] (c) to (f);
        \draw[->] (e) to (g);
        \draw[->] (f) to (g);
        \draw[->] (d) to (g);
\end{tikzpicture}.
\end{center}\vspace{3mm}

\noindent
This diagram also reproduces the one presented in Section \ref{sec:classnew}
which illustrates the confluence of singular points in the
differential equation studied in \cite{KNS}.
Furthermore, it  can be verified that the other diagrams describing 
the confluence of singular points, as lisked by  
Kawakami-Nakamura-Sakai in \cite{KNS}
can likewise be recovered from our reduced unfolding diagrams.

Moreover  moduli spaces $\mathcal{M}_{\mathbf{H}}^{\mathrm{ir}}$  
with $\mathrm{dim}\mathcal{M}_{\mathbf{H}}^{\mathrm{ir}}=0$
correspond to differential equations without accessory parameters.
In such cases, 
the deformations of  $\mathcal{M}_{\mathbf{H}}^{\mathrm{ir}}$ constructed in 
Theorem \ref{thm:ourthm} recover the versal unfolding families of differential equations
without accessory parameters proposed by \cite{Oshi2}.

Therefore one can conclude that 
the deformation constructed in Theorem \ref{thm:ourthm}
gives a natural generalization of  
many classical and recent known examples of confluent/unfolding families 
of linear differential equations on $\mathbb{P}^{1}$,
and also a natural translation of the 
theory for the unfolding of singular points of differential 
equations to the theory for deformation of their moduli spaces of meromorphic connections.

\subsection{Additive Deligne-Simpson problem}\label{sec:intads}
We now explain an application of Theorem \ref{thm:ourthm}
to an existence problem for irreducible $G$-connections,
known as  the additive Deligne-Simpson problem.

The additive Deligne-Simpson problem 
asks whether there exists an irreducible Fuchsian differential equation 
\[
    \frac{dY}{dz}=\sum_{i=1}^{r}\frac{A_{i}}{z-a_{i}}Y\quad (A_{i}\in M_{n}(\mathbb{C}))
\]
on $\mathbb{P}^{1}$,
such that 
the residue matrices $A_{i}$ at $z=a_{i}$ for  $i=1,\ldots,r$ and $A_{\infty}:=-\sum_{i=1}^{r}A_{i}$
at $z=\infty$, lie 
in prescribed conjugacy classes $C_{i}\subset  M_{n}(\mathbb{C})$ for $i\in \{1,\ldots,r,\infty\}$.
This problem has been studied by Deligne, Simpson, Kostov 
and many others,
and was finally solved by Crawley-Boevey in \cite{C} using the theory of 
quiver representations. A detailed historical account of the  
problem is provided in \cite{Kos} by Kostov. 
There are many generalizations of this problem. 
For 
instance, Boalch \cite{Boarx} , Hiroe-Yamakawa \cite{HY}, and Hiroe \cite{H2}
have studied similar problems for linear differential equations
with unramified irregular singularities. Related problems for 
equations with ramified irregular singularities have been discussed
in \cite{KLMNS} by Kulkarni-Livesay-Matherne-Nguyen-Sage
and \cite{LSN} by Livesay-Sage-Nguyen.
Furthermore, in \cite{JY} by Jakob-Yun a similar problem 
for $G$-connections with ramified irregular singularities is discussed.

We now propose the following generalization of 
the additive Deligne-Simpson problem to $G$-connections,
extending the problems considered in \cite{Boarx}, \cite{HY}, and \cite{H2}.
Assume,under the projective transformation, that the of 
singularities is  $|D|=\{a_{1},\ldots,a_{d}\}\subset \mathbb{C}$.
\begin{prob}[Additive Deligne-Simpson problem for $\mathbf{H}$]\normalfont
	Let us take a collection of unramified canonical forms
	$\mathbf{H}=(H^{(a)})_{a\in |D|}$
	satisfying that the sum of residues 
    $
		\sum_{a\in |D|}H^{(a)}_{\mathrm{res}}
	$ is contained in the semisimple part $\mathfrak{g}_{\mathrm{ss}}$ of $\mathfrak{g}$.	
	The problem is to  find an irreducible meromorphic connection over the trivial $G$-bundle on $\mathbb{P}^{1}$
	\[
	\nabla_{A}=\sum_{a\in |D|}\sum_{i=0}^{k_{a}}\frac{A^{(a)}_{i}}{(z-a)^{i}}\frac{dz}{z-a}\quad (A^{(a)}_{i}\in \mathfrak{g})
	\]
    with singularities only on $|D|$
	such that 
	\[
		\sum_{i=0}^{k_{a}}\frac{A^{(a)}_{i}}{(z-a)^{i+1}}\in \mathbb{O}_{H^{(a)}}
	\]
	for all $a\in |D|$.
\end{prob}
Here $\mathbb{O}_{H^{(a)}}$ stands for the $G(\mathbb{C}[\![z_{a}]\!])$-orbit through 
$H^{(a)}\in \mathfrak{g}(\mathbb{C}(\!(z_{a}^{-1})\!))$, see Section \ref{sec:HTLBV} for the detail.

As in the previous sections, 
we consider a family of canonical forms
 $\mathbf{H}(\mathbf{c})_{\mathbf{c}\in U}=((H^{(a)}(\mathbf{c}^{(a)}))_{a\in |D|})_{\mathbf{c}\in U}
$ parametrized by $\mathbf{c}\in U$, where $U\subset \prod_{a\in |D|}\mathbb{C}^{k_{a}+1}$
is  an open neighborhood
of the origin $\mathbf{0}$. 
Fibers $(H^{(a)}(\mathbf{c}^{(a)}))_{a\in |D|}$ 
on each stratum of $U\subset \prod_{a\in |D|}\mathbb{C}^{k_{a}+1}$ can be seen 
as collections of canonical forms 
and their spectral types are described
by the unfolding of the spectral types of $\mathbf{H}$
as we saw in Proposition \ref{prop:introprp}.
Therefore one can ask whether
there exists a family of irreducible $G$-connections
which realizes the deformation $\mathbf{H}(\mathbf{c})$
of $\mathbf{H}$.
\begin{prob}[Additive Deligne-Simpson problem for an unfolding 
	family of canonical forms]\normalfont
	Let $U$ be an open neighborhood of $\mathbf{0}\in \prod_{a\in |D|}\mathbb{C}^{k_{a}+1}$.
	Let $(\mathbf{H}(\mathbf{c}))_{\mathbf{c}\in U}$ be 
	the unfolding of a collection of canonical forms $\mathbf{H}=(H^{(a)})_{a\in |D|}$
	as above.
	Then find a family of meromorphic $G$-connections on $\mathbb{P}^{1}$,
	\[
	\nabla_{A(\mathbf{c})}=A(\mathbf{c})\,dz
	\]
	satisfying all the following conditions.
	\begin{enumerate}
		\item The $\mathfrak{g}$-valued $1$-form 
		$A(\mathbf{c})\,dz$ depends holomorphically  on $\mathbf{c}\in U$. 
		\item For each $\mathbf{c}\in U$, 
		 $\nabla_{A(\mathbf{c})}$ is a solution 
		of the additive Deligne-Simpson problem for 
		$\mathbf{H}(\mathbf{c})$.
	\end{enumerate}
\end{prob}
In other words, we seek 
a holomorphic family $\nabla_{A(\mathbf{c})}=A(\mathbf{c})\,dz$
of irreducible $G$-connections that  
simultaneously solve 
all the Deligne-Simpson problems associated with 
the deformed canonical forms $\mathbf{H}(\mathbf{c})$ for $\mathbf{c}\in U$.

As a direct consequence of Theorem \ref{thm:ourthm},
we obtain the following result:
if the additive Deligne-Simpson problem for $\mathbf{H}$ has a soluiton,
then the every additive Deligne-Simpson problem in the unfolding family
$(\mathbf{H}(\mathbf{c}))_{\mathbf{c}\in U}$ also has a simultaneous solution. 
\begin{thm}[Theorem \ref{thm:addDS}]\label{thm:addDS0}
	Let $\mathbf{H}$ be a collection of unramified canonical forms as above.
    Let $\nabla_{A}=A\,dz$ be a solution of the additive Deligne-Simpson problem 
    for $\mathbf{H}$.
    Then there exists an open neighborhood $U$ of $\mathbf{0}\in \prod_{a\in |D|}\mathbb{C}^{k_{a}+1}$
    and a holomorphic family of $G$-connections $\nabla_{A(\mathbf{c})}=A(\mathbf{c})\,dz$\ $(\mathbf{c}\in U)$ on $\mathbb{P}^{1}$
    such that 
    \begin{itemize}
    \item $\nabla_{A(\mathbf{0})}=\nabla_{A}$,
    \item $\nabla_{A(\mathbf{c})}=A(\mathbf{c})\,dz$ is a solution of 
    the additive Deligne-Simpson for the family $(\mathbf{H}(\mathbf{c}))_{\mathbf{c}\in U}$.
	\end{itemize}
    In particular, the following are equivalent.
	\begin{enumerate}
		\item The additive Deligne-Simpson problem for $\mathbf{H}$ has a solution.
		\item There exists an open neighborhood $U$ of $\mathbf{0}\in\prod_{a\in |D|}\mathbb{C}^{k_{a}+1}$ and 
		the additive Deligne-Simpson problem for the family $(\mathbf{H}(\mathbf{c}))_{\mathbf{c}\in U}$
		has a solution.
	\end{enumerate}
\end{thm}
\subsection{A conjecture by Oshima}\label{sec:intspc}
We now explain a similar 
existence problem proposed by Oshima, which is closely related 
to the additive Deligne-Simpson problem.

Let us consider a collection $\mathbf{S}=(S_{i})_{i=1,\ldots,d}$
of abstract spectral types.
A collection $\mathbf{S}$ is said to be {\em irreducibly realizable}
if there exists an irreducible meromorphic $G$-connection 
whose spectral type is $\mathbf{S}$.
More precisely, $\mathbf{S}$ is irreducibly realizable 
if and only if there exists a collection $\mathbf{H}=(H^{(a_{i})})_{i=1,\ldots,d}$
of unramified canonical forms such that 
\begin{enumerate}
\item $\mathbf{S}=\mathrm{sp}(\mathbf{H})$,
\item the additive Deligne-Simpson problem for $\mathbf{H}$ has a solution.
\end{enumerate}
In this case, a solution $\nabla_{A}=A\,dz$ of this additive Deligne-Simpson 
problem is called a {\em realization} of $\mathbf{S}$.

Furthermore, 
the collection $\mathbf{S}=(S_{i})_{i=1,2,\ldots,d}$
is 
said to be {\em versally realizable}
if there exists a collection $\mathbf{H}=(H^{(a_{i})})_{i=1,2,\ldots,d}$
of unramified canonical forms 
such that:
\begin{enumerate}
	\item $\mathbf{S}=\mathrm{sp}(\mathbf{H})$,
	\item there exists an open neighborhood $U$ of $\mathbf{0}\in \prod_{i=1}^{d}\mathbb{C}^{k_{i}+1}$ 
	such that the unfolding family $(\mathbf{H}(\mathbf{c}))_{\mathbf{c}\in U}$
	of $\mathbf{H}$ realizes the unfolding diagram of $\mathbf{S}$, 
	namely, 
	for each $(\mathcal{I}^{(i)})_{i=1,\ldots,d}\in \prod_{i=1}^{d}\mathcal{P}_{[k_{i}+1]}$
	and $\mathbf{c}\in U\cap \prod_{i=1}^{d}\mathcal{C}(\mathcal{I}^{(i)})$,
	we have
	\[
		\mathrm{sp}(\mathbf{H}(\mathbf{c}))=\mathbf{S}^{(\mathcal{I}^{(i)})_{i=1,\ldots,d}},
	\]
	\item 
	 for every $\mathbf{c}\in U$,
	 the additive Deligne-Simpson problem for $\mathbf{H}(\mathbf{c})$ admits a solution.
\end{enumerate}

Oshima proposed the following conjecture
in \cite{Oshi1} and \cite{Oshi2},
which asserts that 
the realizability of 
$\mathbf{S}$
is equivalent to the versal realizability of 
it, and moreover equivalent to the 
realizability of the {\em finest unfolding} of $\mathbf{S}$,
\begin{align*}
	&\mathbf{S}^{\mathrm{reg}}:=\mathbf{S}^{(\{0\}\sqcup\cdots \sqcup\{k_{i}\})_{i=1,\ldots,d}}\\
	&=\Big((\Pi^{(1)}_{0};[J^{(1)}_{0}]),(\Pi^{(1)}_{1};[0]),\ldots &,(\Pi^{(1)}_{k_{1}};[0]),\ldots,
	 (\Pi^{(d)}_{0};[J^{(d)}_{0}]),(\Pi^{(d)}_{1};[0]),\ldots,(\Pi^{(d)}_{k_{d}};[0])\Big).
\end{align*}
Here we note that 
the finest unfolding $\mathbf{S}^{\mathrm{reg}}$ is the collection of spectral types for regular singular canonical forms
associated to the 
finest partitions $\{0\}\sqcup\cdots \sqcup\{k_{i}\}=\{0,\ldots,k_{i}\}$, $i=1,\ldots,d$. 
\begin{conj}[Conjecture \ref{conj:osh}, Oshima \cite{Oshi1}, \cite{Oshi2}]\label{conj:osh0}\normalfont
	Let $\mathbf{S}=(S_{i})_{i=1,\ldots,d}$
	 be a 
	 collection of abstract spectral types.
	Then the following are equivalent.
	\begin{enumerate}
		\item The collection $\mathbf{S}$ is irreducibly realizable.
		\item The collection $\mathbf{S}$ is versally realizable.
		\item The finest unfolding $\mathbf{S}^{\mathrm{reg}}$ of $\mathbf{S}$ is irreducibly realizable.
	\end{enumerate}
\end{conj}
In \cite{Oshi1} and \cite{Oshi2}, Oshima proved  
this conjecture when $G=\mathrm{GL}_{n}$ and 
$
-2\le \mathrm{rig}(\mathbf{S})\le 2$.
Here $\mathrm{rig}(\mathbf{S})$ is the index of rigidity of 
$\mathbf{S}$ which will be explained in 
Section \ref{sec:indrig}.

With regard to this conjecture,
we show the following as 
a corollary of 
Theorems \ref{thm:ourthm} and \ref{thm:addDS0}.
\begin{thm}[Theorem \ref{thm:specdiag}]
Let $\mathbf{S}=(S_{i})_{i=1,\ldots,d}$ be a collection of abstract spectral types.
Suppose that $\mathbf{S}$ is irreducibly realizable 
and let $\nabla_{A}=A\,dz$ be a realization of $\mathbf{S}$.
Then there exists an open neighborhood $U$ of $\mathbf{0}\in \prod_{i=1}^{d}\mathbb{C}^{k_{a_{i}}+1}$
and a holomorphic family of $G$-connections $\nabla_{A(\mathbf{c})}=A(\mathbf{c})\,dz$\ $(\mathbf{c}\in U)$ on $\mathbb{P}^{1}$
such that 
\begin{itemize}
\item $\nabla_{A(\mathbf{0})}=\nabla_{A}$,
\item for each $(\mathcal{I}_{i})_{i=1,\ldots,d}\in \prod_{i=1}^{d}\mathcal{P}_{[k_{i}+1]}$
and $\mathbf{c}\in U\cap \prod_{i=1}^{d}\mathcal{C}(\mathcal{I}^{(i)})$,
$\nabla_{A(\mathbf{c})}=A(\mathbf{c})\,dz$ is a realization of 
$\mathbf{S}^{(\mathcal{I}_{i})_{i=1,\ldots,d}}$.
\end{itemize}
In particular, the following are equivalent.
\begin{enumerate}
\item The collection $\mathbf{S}$ of abstract spectral types is irreducibly 
realizable.
\item For every $(\mathcal{I}^{(i)})_{i=1,\ldots,d}\in \prod_{i=1}^{d}\mathcal{P}_{[k_{i}+1]}$,
the collection $\mathbf{S}^{(\mathcal{I}^{(i)})_{i=1,\ldots,d}}$
is irreducibly realizable.
\end{enumerate}
\end{thm}
As a result we obtain the following.
\begin{cor}
	In Conjecture \ref{conj:osh0}, $1$ and $2$ are equivalent and 
	the implication $1 \Rightarrow 3$ holds true.
\end{cor}
In the case of $G=\mathrm{GL}_{n}$,
the implication $3\Rightarrow 1$ in Conjecture \ref{conj:osh0}
is shown
in \cite{H3}.
Therefore combining these results,
we obtain the complete affirmative solution of the conjecture by Oshima 
in the case of $G=\mathrm{GL}_{n}$. 
\subsection{A further discussion--Unfolding is symplectic and birational}
In this article, we have constructed unfolding families of moduli spaces of meromorphic 
$G$-connections. 
As a counterpart to this construction via the Riemann–Hilbert correspondence,
 an important problem is to construct unfolding families of wild character 
 varieties. 
 In fact, in \cite{Kl2} by Klimes and \cite{PR} by Paul and Ramis, 
 the unfolding of Stokes data of the wild character variety for Painlev\'e V 
 was studied. They showed that there exists an unfolding map from the wild character 
 variety for Painlev\'e V to the character variety for Painlev\'e VI,
and moreover, that this map is symplectic and birational.
Also they stated as a conjecture that 
similar unfolding maps exist for other wild character varieties of Painlev\'e equations,
and that these maps are symplectic and birational as well.

Building on these preceding works, 
we discussed the unfolding of wild character varieties in another paper 
with Yamakawa \cite{HYp}. 
There, we presented a general construction of unfolding maps from 
wild character varieties with unramified irregular singularities 
to character varieties with only regular singularities, 
not only for $\mathrm{SL}_2$ but also for any reductive linear algebraic group,
and moreover, we proved that these unfolding maps are symplectic and birational,
as a natural generalization of the results in \cite{Kl2} and \cite{PR}, and an 
affirmative answer of their above conjecture.
\subsection*{Acknowledgements}
This project began around 2018 
when the conference in honor of Toshio Oshima 
was held at Josai University.
Since then, the author has had many opportunities 
to give talks in conferences and seminars, 
discuss with many researchers, and 
receive many valuable comments.
The author thanks all of 
them, Philip Boalch, Masahiro Futaki,
Sampei Hirose, Michiaki Inaba, Kohei Iwaki, Shingo Kamimoto, Tatsuki Kuwagaki, 
Frank Loray,
Takuro Mochizuki, Hiraku Nakajima, Vladimir
Rubtsov, Masa-Hiko Saito, Shinji Sasaki, Yoshitsugu Takei, and Daisuke Yamakawa.
Especially, the author 
expresses his gratitude to Hiroshi Kawakami, Akane Nakamura, and 
Hidetaka Sakai for their work breaking a new ground for the unfolding of singularities 
of differential equations. 
The author thanks Shunya Adachi for his careful reading of the draft of this article.
And again, the author thanks Daisuke Yamakawa for many valuable discussions and
comments on a draft of this article.
Finally, the author expresses his deepest gratitude to Toshio Oshima for his inspiring works and
many valuable discussions.

\section{Meromorphic \(G\)-connections on the formal punctured disk}
In this section, we 
 recall the local and formal classification theory
of meromorphic $\mathfrak{g}$-valued connections, 
following the work of Babbitt and Varadarajan \cite{BV},
which generalizes the classical theory developed by 
Hukuhara \cite{Huk}, Turrittin \cite{Tur}, and also by Levelt \cite{Lev}.
We also introduce certain numerical invariants 
associated with the canonical forms of  meromorphic connections on the formal punctured 
disk, namely $\delta$-invariants and spectral types.  

Let us begin by fixing some notation.
For an indeterminate $z$, $\mathbb{C}[z]$, $\mathbb{C}[\![z]\!]$, and 
$\mathbb{C}(\!(z)\!)$ denote the ring of polynomials, formal power 
series, and field of formal Laurent series in $z$ over 
$\mathbb{C}$, respectively. 
The {\em order} of $f=\sum_{i=m}^{\infty}a_{i}z^{i}\in \mathbb{C}(\!(z)\!)$
is defined as the smallest integer $n\in \mathbb{Z}$
such that $a_{n}\neq 0$ and is denoted by
\(
	\index[cN]{$\underset{z=0}{\mathrm{ord\,}}(f)$}\underset{z=0}{\mathrm{ord\,}}(f):=\mathrm{min\,}\{n\in \mathbb{Z}\mid a_{n}\neq 0\}.	
\)
We denote the ideal of a commutative ring $\mathcal{R}$ with the generators
$r_{1},r_{2},\ldots,r_{l}\in \mathcal{R}$ by $\langle r_{1},r_{2},\ldots,r_{l}\rangle_{\mathcal{R}}$
or simply $\langle r_{1},r_{2},\ldots,r_{l}\rangle$.

Throughout this paper, let
$G$ denote 
a fixed connected reductive affine algebraic group defined over $\mathbb{C}$, and 
let $\mathfrak{g}$ denotes its Lie algebra.
Let $H$ be a closed subgroup of $G$ and $\mathfrak{h}$ its Lie algebra.
Let $\mathcal{R}$ be a unital $\mathbb{C}$-algebra. Then 
$H(\mathcal{R})$ denotes the group of $\mathcal{R}$-valued points of $H$.
That is, for the coordinate ring 
\index[cN]{$\mathcal{O}(H)$}$\mathcal{O}(H)$ of $H$,
we define $H(\mathcal{R}):=\mathrm{Hom}_{\mathbb{C}\text{-}\mathrm{alg}}(\mathcal{O}(H),\mathcal{R})$
whose group multiplication comes from the comultiplication of the Hopf algebra 
$\mathcal{O}(H)$.
We also define the corresponding Lie algebra by $\mathfrak{h}(\mathcal{R}):=\mathfrak{h}\otimes_{\mathbb{C}}\mathcal{R}$.
The Lie bracket on $\mathfrak{h}(\mathcal{R})$
is defined by $[X_{1}\otimes r_{1},X_{2}\otimes r_{2}]:=[X_{1},X_{2}]\otimes r_{1}r_{2}$
for $X_{1},X_{2}\in \mathfrak{h}$ and $r_{1},r_{2}\in \mathcal{R}$.
Here we notice that 
in the definition of $\mathfrak{h}(\mathcal{R})$,
the 
$\mathbb{C}$-algebra $\mathcal{R}$ need not to be unital.
Thus we can consider the 
Lie algebra $\mathfrak{h}(\mathcal{R})$ 
even for non-unital $\mathbb{C}$-algebra $\mathcal{R}$.

We fix an embedding $G\hookrightarrow \mathrm{GL}_{N}$ throughout  this paper.
We frequently write $G(\mathbb{C})=G$ if there is no ambiguity.

For a group $H$ acting on a set $M$, the stabilizer subgroup for $m\in M$
is denoted by \index[cN]{$\mathrm{Stab}_{H}(m)$}
$\mathrm{Stab}_{H}(m):=\{h\in H\mid h\cdot m=m\}$
or \index[cN]{$H_{m}$}$H_{m}$ for short.

For a Lie algebra $\mathfrak{h}$ and $X\in \mathfrak{h}$, \index[cN]{$\mathfrak{h}_{X}$}
$\mathfrak{h}_{X}$ denotes the 
kernel of the linear map $\mathrm{ad}(X)\in \mathrm{End}_{\mathbb{C}}(\mathfrak{h})$.

\subsection{Hukuhara-Turrittin theory for $\mathfrak{g}$-valued connections}\label{sec:HTLBV}
Let us introduce $G$-connections on
the formal punctured disk. 
\begin{df}[Formal meromorphic $G$-connection]\normalfont
A $\mathfrak{g}$-valued $1$-form $A\,dz$ with $A\in \mathfrak{g}(\mathbb{C}(\!(z)\!))$ 
is called 
a {\em formal meromorphic $G$-connection on
the formal punctured disk}.

In particular when $\underset{z=0}{\mathrm{ord\,}}A=-1$, we say that the formal connection $A\,dz$
has a {\em regular singularity of the first kind} at $z=0$.
Sometimes we omit the word, first kind, and just call regular singularity.
\end{df}
Let us recall equivalence relations for $G$-connections.
Let $\mathcal{R}$ be either $\mathbb{C}[\![z]\!]$ or $\mathbb{C}(\!(z)\!)$.
Under the following realizations
\begin{align*}
	\mathrm{GL}_{N}(\mathbb{C}[\![z]\!])&=\left\{g=\sum_{i=0}^{\infty}g_{i}z^{i}\,\middle|\, g_{i}\in M_{N}(\mathbb{C}),\ \mathrm{det}(g_{0})\neq 0\right\},\\
	\mathrm{GL}_{N}(\mathbb{C}(\!(z)\!))&=\left\{g=\sum_{i=m}^{\infty}g_{i}z^{i}\,\middle|\, g_{i}\in M_{N}(\mathbb{C}),\ m\in \mathbb{Z},\ \mathrm{det}(g_{m})\neq 0\right\},
\end{align*}
we can consider 
the derivation $\frac{d}{dz}g$ for $g\in \mathrm{GL}_{N}(\mathcal{R})$
defined by the term-wise differentiation. 
Recall that we have fixed an embedding $\rho\colon G\hookrightarrow \mathrm{GL}_{N}$.
It is known that there exists a map \index[cN]{$\delta_{G}$}
$\delta_{G}\colon G(\mathcal{R})\rightarrow \mathfrak{g}(\mathcal{R})$
such that 
\[
	d\rho(\delta_{G}(g))=\left(\frac{d}{dz}\rho(g)\right)\rho(g)^{-1}\quad (g\in G(\mathcal{R})).
\]
Let us note that the map $\delta_{G}$ is independent of the choice of the embedding $\rho$.
For more detail, we refer the section 1.6 in \cite{BV}.

\begin{df}[Gauge transformation]\normalfont
Let $\mathcal{R}$ be either $\mathbb{C}[\![z]\!]$ or $\mathbb{C}(\!(z)\!)$. 
For $g\in G(\mathcal{R})$, we can define the 
transformation for $G$-connections by 
\[
	A\,dz\longmapsto \mathrm{Ad}(g)(A)\,dz+\delta_{G}(g)\,dz,
\]
which is called the {\em gauge transformation} over $\mathcal{R}$.
\end{df}

We now recall the classification theory of $G$-connections under the gauge transformation.
\begin{df}[Canonical form]\label{df:canform}\normalfont
	An element $H\in \mathfrak{g}(\mathbb{C}(\!(z)\!))$
	is called a {\em canonical form} if 
	$H\,dz$ is of the  form 
	\[
		H\,dz=\left(\sum_{i=1}^{k}H_{i}z^{-i}+H_{\mathrm{res}}\right)\frac{dz}{z}	
	\]
	with semisimple elements $H_{i}$
	satisfying $[H_{i},H_{j}]=0$ and $[H_{i},H_{\mathrm{res}}]=0$
	for $1\le i,j\le k$. \index[cN]{$H_{\mathrm{res}}$}
	
	In particular \index[cN]{$H_{\mathrm{irr}}$}
	$\displaystyle H_{\mathrm{irr}}:=\sum_{i=1}^{k}H_{i}z^{-(i+1)}$
	is called the {\em irregular part} of the canonical form $H$.
	If the irregular part is zero, then $H$ is called a {\em regular singular} canonical form.
\end{df}

Babbit-Varadarajan showed that any meromorphic $G$-connections 
are reduced to canonical forms under the gauge transformation,
as a natural generalization of the classical results established by Hukuhara, Turrittin, and Levelt.
\begin{thm}[Babbitt-Varadarajan \cite{BV}, cf. Hukukara \cite{Huk}, Turrittin \cite{Tur}, and Levelt \cite{Lev}]\label{thm:htbv}
	Let us
	consider a meromorphic $G$-connection $A\,dz$
	with $A\in \mathfrak{g}(\mathbb{C}(\!(z)\!))$.
	There exists a field extension $\mathbb{C}(\!(t)\!)$ of $\mathbb{C}(\!(z)\!)$
	satisfying $t^{q}=z$ with $q\in \mathbb{Z}_{\ge 1}$,
	and $A\,dz$ is transformed to $H\,dt$ with a canonical form 
	$H\in \mathfrak{g}(\mathbb{C}(\!(t)\!))$
	via a gauge transformation over $\mathbb{C}(\!(t)\!)$.

	In this case $H\,dt$ is called a {\em canonical form} of $A\,dz$.
	The minimal $q\in \mathbb{Z}_{\ge 1}$  
	to transform $A\,dz$ to a canonical form is called the {\em ramification index of $A\,dz$}.
	If in particular $q=1$, $A\,dz$ is called an {\em unramified connection}.
\end{thm}

Now we take an unramified canonical form
\[
	H\,dz=\left(\sum_{i=1}^{k}H_{i}z^{-i}+H_{\mathrm{res}}\right)\frac{dz}{z}.
\]
Then the isomorphism class of $H\,dz$ under the gauge  transformation 
of $G(\mathbb{C}[\![z]\!])$ is parametrized by the orbit \index[cN]{$O_{H}$}
\[
	O_{H}:=\left\{
	\mathrm{Ad}(g)(H)+\delta_{G}(g)\in \mathfrak{g}(\mathbb{C}(\!(z)\!))\,\middle|\, g\in G(\mathbb{C}[\![z]\!])
	\right\}.
\]

As an analogue of $O_{H}$, Boalch introduced the truncated orbit of $H$
in \cite{Boa1} which can be regarded as a coadjoint orbit of a finite-dimensional complex Lie group as it will be explained later.

\begin{df}[Truncated orbit of canonical form]\normalfont
Let us take an unramified canonical form $H$ and regard it as an element in $\mathfrak{g}(\mathbb{C}(\!(z)\!)/\mathbb{C}[\![z]\!])$.
 Then the orbit of $H$ under the adjoint action of $G(\mathbb{C}[\![z]\!])$, \index[cN]{$\mathbb{O}_{H}$}
 \[
	\mathbb{O}_{H}:=\left\{\mathrm{Ad}(g)H\in \mathfrak{g}(\mathbb{C}(\!(z)\!)/\mathbb{C}[\![z]\!])\,\middle|\, g\in G(\mathbb{C}[\![z]\!])\right\}
 \]
 is called the {\em truncated orbit} of $H$.
\end{df}

We can compare these orbits $O_{H}$ and $\mathbb{O}_{H}$ as follows.
\begin{prop}\label{prop:generic}
	Let $\pi\colon \mathbb{C}(\!(z)\!)\rightarrow \mathbb{C}(\!(z)\!)/\mathbb{C}[\![z]\!]$
	be the natural projection.
	Let us take an unramified canonical form $H\in \mathfrak{g}(\mathbb{C}(\!(z)\!))$ as above.
	Let us consider the Lie subalgebra 
	$\mathfrak{g}_{H_{\text{irr}}}=\{X\in \mathfrak{g}\mid [X,H_{i}]=0\text{ for all }i=1,2,\ldots,k\}$ of $\mathfrak{g}$.
	
	Suppose that $\mathrm{ad}_{\mathfrak{g}_{H_{\text{irr}}}}(H_{0})\in \mathrm{End}_{\mathbb{C}}(\mathfrak{g}_{H_{\text{irr}}})$ has 
	no nonzero integers as its eigenvalues.
	Then 
	 $X\in O_{H}$ if and only if $\pi(X)\in \mathbb{O}_{H}$.
\end{prop}
\begin{proof}
	Lemma 6.2.1 and Proposition 9.3.2 in \cite{BV}
	show that 
	if $\pi(X)\in \mathbb{O}_{H}$, then our assumption implies 
	that there exists $g\in
	G(\mathbb{C}[\![z]\!])$ such that
	$\mathrm{Ad}(g)X+\delta_{G}(g)=H$, i.e., $X\in O_{H}$.

	Conversely, we assume that $X\in O_{H}$. Then there exists $g\in G(\mathbb{C}[\![z]\!])$
	such that $X=\mathrm{Ad}(g)H+\delta_{G}(g)$. On the other hand, we have 
	$\mathrm{Ad}(g)H+\delta_{G}(g)
	\equiv \mathrm{Ad}(g)H\quad (\mathrm{mod }\ \mathbb{C}[\![z]\!])$
	since $\delta_{G}(g)\in \mathfrak{g}\otimes \mathbb{C}[\![z]\!]$.
	Thus $\pi(X)\in \mathbb{O}_{H}$.
\end{proof}
\begin{df}[Truncated canonical form]\normalfont
	Let $H\,dz$ be an unramified canonical form.
	Let $A\,dz$ with $A\in \mathfrak{g}(\mathbb{C}(\!(z)\!))$
	be a meromorphic $G$-connection.
	If $\pi(A)\in \mathbb{O}_{H}$,
	then we call $H$ the 
	{\em truncated canonical form} of $A\,dz$.
	In this case, we say that $A\,dz$ has the unramified 
	truncated canonical form $H$.

	If $H$ satisfies the condition in Proposition \ref{prop:generic},
	the truncated canonical form of $A\,dz$ coincides with the 
	usual canonical form appears in Theorem \ref{thm:htbv}.
\end{df}

\subsection{Spectral types of canonical forms}\label{sec:spectral}
To an unramified canonical form 
\[
		H\,dz=\left(\sum_{i=1}^{k}H_{i}z^{-i}+H_{\mathrm{res}}\right)\frac{dz}{z},
\]
we associate the following discrete data, called a spectral type.

Let $H_{\mathrm{res}}=H_{0}+J_{0}$ be the Jordan decomposition 
with the semisimple $H_{0}$ and nilpotent $J_{0}$.
Since $H_{0},H_{1},\ldots,H_{k}$ are all
semisimple and mutually commuting, the direct sum
$\bigoplus_{i=0}^{k}\mathbb{C}H_{i}$ forms a toral 
Lie subalgebra of $\mathfrak{g}$, 
i.e., one consisting entirely of semisimple elements.
Thus we can choose a maximal toral subalgebra,
namely, a Cartan subalgebra \index[cN]{$\mathfrak{t}$}$\mathfrak{t}$
containing this toral subalgebra.
Then we obtain the root space decomposition of $\mathfrak{g}$,
\[
	\mathfrak{g}=\mathfrak{t}\oplus \bigoplus_{\alpha\in \Phi}\mathfrak{g}_{\alpha},	
\] 
where 
$\mathfrak{g}_{\alpha}:=\left\{
		X\in \mathfrak{g}\mid \mathrm{Ad}(t)X=\alpha(t)X,\,t\in \mathfrak{t}
	\right\}
$
are root spaces with respect to $\alpha\in \mathfrak{t}^{*}$, 
and $\Phi$ is the set of roots, i.e., $\Phi:=\{\alpha\in \mathfrak{t}^{*}\mid \mathfrak{g}_{\alpha}\neq 0\}$.
We also fix a simple system \index[cN]{$\Pi$}$\Pi$
of roots $\Phi$ and denote the 
associated positive and negative systems 
by $\Phi^{+}$ and $\Phi^{-}$ respectively. 

Recall that $\mathfrak{t}$
has a fundamental domain $D$ under the action of the Weyl group $W$ of $G$.
It is defined as 
\[
	\mathcal{D}:=
	\{X\in \mathfrak{t}\,\mid
	\mathfrak{Re\,}\alpha(X)\ge 0 \text{ and if }
	\mathfrak{Re\,}\alpha(X)=0\text{ then }\mathfrak{Im\,}\alpha(X)\ge 0
	\text{ for any }\alpha\in \Pi
	\}.	
\]
The $W$-action allows us to assume 
that $H_{k}\in \mathcal{D}$.
We define the subset 
\[
	\Pi_{k}:=\{\alpha\in \Phi\mid \alpha(H_{k})=0\}
\]
of $\Pi$.
Similarly under the action of $W_{k}:=\mathrm{Stab}_{W}(H_{k})$, we may assume that $H_{k-1}$ belongs to 
\[
	\mathcal{D}_{k}:=
	\{X\in \mathfrak{t}\mid 
	\mathfrak{Re\,}\alpha(X)\ge 0 \text{ and if }
	\mathfrak{Re\,}\alpha(X)=0\text{ then }\mathfrak{Im\,}\alpha(X)\ge 0
	\text{ for any }\alpha\in \Pi_{k}
	\}.	
\]
We then define the subset of $\Pi_{k}$ by
\[
	\Pi_{k-1}:=\{\alpha\in \Pi_{k}\mid \alpha(H_{k-1})=0\}
	=\{\alpha\in \Pi\mid \alpha(H_{k})=\alpha(H_{k-1})=0\}.
\]
Inductively for $i=0,1,\ldots,k-2$,
we may assume that  
$H_{i}$ belongs to
\[
	\mathcal{D}_{i+1}:=\{X\in \mathfrak{t}\mid 
	\mathfrak{Re\,}\alpha(X)\ge 0 \text{ and if }
	\mathfrak{Re\,}\alpha(X)=0\text{ then }\mathfrak{Im\,}\alpha(X)\ge 0
	\text{ for any }\alpha\in \Pi_{i+1}
	\}	
\]
under the action of $W_{i+1}:=\mathrm{Stab}_{W_{i+2}}(H_{i+1})$.
We then define the subset of $\Pi_{i+1}$ by 
\begin{align*}
	\Pi_{i}&:=\{\alpha\in \Pi_{i+1}\mid \alpha(H_{i})=0\}\\
	&=\{\alpha\in \Pi_{i+2}\mid \alpha(H_{i+1})=\alpha(H_{i})=0\}\\
	&=\cdots\\
	&=\{\alpha\in \Pi\mid \alpha(H_{k})=\cdots=\alpha(H_{i+1})=\alpha(H_{i})=0\}.
\end{align*}
Therefore we obtain the sequence 
\[
	\Pi=\Pi_{k+1} \supset \Pi_{k}\supset \cdots \supset \Pi_{0}
\]
of the set of simple roots associated with the canonical form $H$.

For each \index[cN]{$\Pi_{i}$}$\Pi_{i}$, we can consider the subset of roots 
\[
\Phi_{i}:=\left(\bigoplus_{\alpha\in \Pi_{i}}\mathbb{Z}\alpha\right)\cap \Phi
\]
generated by $\Pi_{i}$ and then define the Lie subalgebra\index[cN]{$\mathfrak{l}_{i}$}
\[
	\mathfrak{l}_{i}:=\mathfrak{t}\oplus \bigoplus_{\alpha\in \Phi_{i}}\mathfrak{g}_{\alpha}
	=\{X\in \mathfrak{l}_{i+1}\mid [X,H_{i}]=0\}.
\]

Thus we also obtain the sequence of Lie subalgebras 
\[
	\mathfrak{g}=\mathfrak{l}_{k+1}\supset \mathfrak{l}_{k}\supset \cdots 
	\supset \mathfrak{l}_{0}	
\]
of $\mathfrak{g}$.
For each pair $l,m$ of positive integers with $0\le l< m\le k$,
$\mathfrak{l}_{l}$ can be seen as the Levi subalgebra of 
the parabolic subalgebra of $\mathfrak{l}_{m}$, \index[cN]{$\mathfrak{p}_{l,m}$}
\[
	\mathfrak{p}_{l,m}:=\mathfrak{l}_{l}\oplus \bigoplus_{\Phi_{m}^{+}\backslash \Phi_{l}}\mathfrak{g}_{\alpha}.
\]
Let us denote the nilpotent radical of $\mathfrak{p}_{l,m}$ and its opposite by  \index[cN]{$\mathfrak{u}_{l,m}$} \index[cN]{$\mathfrak{n}_{l,m}$}
\begin{align*}
	\mathfrak{u}_{l,m}&:=\bigoplus_{\Phi_{m}^{+}\backslash \Phi_{l}}\mathfrak{g}_{\alpha},&
	\mathfrak{n}_{l,m}&:=\bigoplus_{\Phi_{m}^{-}\backslash \Phi_{l}}\mathfrak{g}_{\alpha}
\end{align*}
respectively.
Here we put $\Phi_{m}^{\pm}:=\Phi_{m}\cap \Phi^{\pm}$.
Then we have the decomposition 
\[
	\mathfrak{l}_{m}=\mathfrak{n}_{l,m}\oplus \mathfrak{l}_{l}\oplus \mathfrak{u}_{l,m}.	
\]
Especially when $m=k$
we simply write \index[cN]{$\mathfrak{p}_{l}$}$\mathfrak{p}_{l}:=\mathfrak{p}_{l,k}$, \index[cN]{$\mathfrak{u}_{l}$}$\mathfrak{u}_{l}:=\mathfrak{u}_{l,k}$, 
\index[cN]{$\mathfrak{n}_{l}$}$\mathfrak{n}_{l}:=\mathfrak{n}_{l,k}$. 
Also we denote the analytic subgroups of $G=G(\mathbb{C})$ corresponding to $\mathfrak{l}_{l}$, $\mathfrak{p}_{l,m}$,
$\mathfrak{u}_{l,m}$, $\mathfrak{n}_{l,m}$, $\mathfrak{u}_{l}$, and $\mathfrak{n}_{l}$ by \index[cN]{$L_{i}$}$L_{i}$, 
\index[cN]{$P_{l,m}$}$P_{l,m}$, \index[cN]{$U_{l,m}$} 
$U_{l,m}$, \index[cN]{$N_{l,m}$}$N_{l,m}$, \index[cN]{$U_{l}$}$U_{l}$, and  \index[cN]{$N_{l}$}$N_{l}$ respectively.

\begin{df}[Spectral type]\normalfont
	For the unramified canonical form $H\,dz$, the {\em spectral type} of $H$
	is the 
	pair \index[cN]{$\mathrm{sp}(H)$}$\mathrm{sp}(H):=(\Pi_{H};[J_{0}])$ of the above sequence
	\[
		\Pi_{H}\colon \Pi_{k} \supset \Pi_{k-1}\supset \cdots \supset \Pi_{0}	
	\]
	of subsets of $\Pi$, and the orbit \index[cN]{$[J_{0}]$}$[J_{0}]$ of $J_{0}$ under the adjoint action of $L_{1}$.
	Here we note that $J_{0}\in \mathfrak{l}_{0}$ since it commute with all $H_{i}$ for $i=0,1,\ldots,k$.

	If a meromorphic $G$-connection $A\,dz$
	has the above $H$ as the unramified truncated canonical form,
	then we also call $(\Pi_{H},[J_{0}])$ the {\em spectral type of $A\,dz$}.

	Without referring a canonical form $H$, we can consider 
	the pair $(\Pi_{k}\supset \Pi_{k-1}\supset\cdots\supset \Pi_{0};[J_{0}])$
	of a sequence of subsets of $\Pi$ and $L_{1}$-orbit of 
	a nilpotent element $J_{0}\in \mathfrak{l}_{0}$.
	We call this pair an {\em abstract spectral type}.
\end{df}

\subsection{$\delta$-invariants of canonical forms}\label{sec:deltainv}
We introduce some other invariants of a canonical form
\[
		H\,dz=\left(\sum_{i=1}^{k}H_{i}z^{-i}+H_{\mathrm{res}}\right)\frac{dz}{z}.
\]
\begin{df}[Irregularities of canonical forms]\normalfont
For the above canonical form $H\,dz$, the quantity\index[cN]{$\mathrm{Irr\,}(H)$}
\[
	\mathrm{Irr\,}(H):=\sum_{i=1}^{k}(\mathrm{dim\,}G-\mathrm{dim\,}L_{i})	
\]
is called the {\em irregularity} of $H\,dz.$
\end{df}
This quantity is analogous to the Komatsu-Malgrange irregularity 
for connections on the punctured disk.
\begin{rem}\normalfont
	Let us consider the case $G=\mathrm{GL}_{n}$
	and take a 
	$\mathrm{GL}_{n}$-connection $A\,dz$
	with the above unramified canonical form $H\,dz.$
	Then $A\,dz$
	induces the covariant exterior derivative 
	\[
	d_{A}=d+A\,dz\colon \mathfrak{gl}_{n}(\mathbb{C}(\!(z)\!))
	\rightarrow \mathfrak{gl}_{n}(\mathbb{C}(\!(z)\!))\otimes_{\mathbb{C}(\!(z)\!)}
	\mathbb{C}(\!(z)\!)\,dz
	\]
	on $\mathfrak{gl}_{n}(\mathbb{C}(\!(z)\!))$
	by the adjoint action.
	We may rewrite
	\[
		H\,dz=\left(\begin{pmatrix}
			h_{1}(z^{-1})&&&\\
			&h_{2}(z^{-1})&&\\
			&&\ddots&\\
			&&&h_{n}(z^{-1})
		\end{pmatrix}	
		+H_{\mathrm{res}}\right)\frac{dz}{z}
	\]
	with polynomials $h_{i}(t)\in t\,\mathbb{C}[t]$ for $i=1,2,\ldots,n.$
	Then it is known that the Komatsu-Malgrange irregularity of $d_{A}$
	equals the sum of degrees of the polynomials,
	\[
		\sum_{i=1}^{n}\sum_{j=1}^{n}\mathrm{deg\,}(h_{i}(t)-h_{j}(t)).	
	\]
	A straightforward computation yields the following equations
	\[
		\sharp\{(i,j)\in \{1,2,\ldots,n\}^{\times 2}\mid 
		\mathrm{deg\,}(h_{i}(t)-h_{j}(t))=l
		\}
		=\dim L_{l+1}-\dim L_{l}	
	\]
	for $l=1,2,\ldots,k$. Here, $\sharp A$ denotes the cardinality of a set $A$.
	Thus we obtain 
	\[
		\sum_{i=1}^{n}\sum_{j=1}^{n}\mathrm{deg\,}(h_{i}(t)-h_{j}(t))
		=\sum_{l=1}^{k}l\cdot (\mathrm{dim\,}L_{l+1}-\mathrm{dim\,}L_{l})
		=\sum_{i=1}^{k}(\mathrm{dim\,}G-\mathrm{dim\,}L_{i})=\mathrm{Irr}(H).
	\]
\end{rem}

\begin{df}[$\delta$-invariants of canonical forms]\normalfont
	For the above canonical form $H\,dz$,
	the quantity \index[cN]{$\delta(H)$}
	\[
		\delta(H):=\mathrm{dim\,}G+ \mathrm{Irr\,}(H)-\mathrm{dim\,}\mathrm{Stab}_{G}(H)	
	\]
	is called the {\em $\delta$-invariant} of $H\,dz$. 
\end{df}
\begin{df}[Truncated $\delta$-invariant]\normalfont
	Let $\nabla_{A}=A\,dz$ 
	be a meromorphic $G$-connection with the unramified 
	truncated canonical form $H$, namely,
	$A\in \mathbb{O}_{H}$.
	Then define the {\em truncated $\delta$-invariant} of $\nabla_{A}$ by \index[cN]{$\delta^{\mathrm{tr}}(\nabla_{A})$}
	\(
		\delta^{\mathrm{tr}}(\nabla_{A}):=\delta(H).	
	\)
\end{df}
\begin{rem}\normalfont
	In the case $G=\mathrm{GL}_{n}$,
	the quantity 
	\[
		\delta(d_{A}):=\mathrm{rk\,}(d_{A})+\mathrm{Irr\,}(d_{A})-
	\mathrm{dim}_{\mathbb{C}}\mathrm{Ker\,}(d_{A})
	\]
	appears as an important invariant 
	for a connection 
	\[
	d_{A}\colon \mathfrak{gl}_{n}(\mathbb{C}(\!(z)\!))
	\rightarrow \mathfrak{gl}_{n}(\mathbb{C}(\!(z)\!))\otimes_{\mathbb{C}(\!(z)\!)}
	\mathbb{C}(\!(z)\!)\,dz.
	\]
	Since we can check that this invariant coincides with 
	our $\delta$-invariant 
	under the assumption in Proposition \ref{prop:generic} in the case of $G=\mathrm{GL}_{n}$,
	we may regard our $\delta$ is 
	an analogue of this classical quantity.

	Indeed, this classical quantity appears
	as a local part of Euler characteristic 
	of a meromorphic connection on a compact Riemann surface,
	see Deligne \cite{Del}, Arinkin \cite{Ari}, and  Malgrange \cite{Mal} for the detail.
	Furthermore, a relevance of this quantity
	and the $\delta$-invariant of singularities of a plane curve germ 
	is pointed out in \cite{H4}.
\end{rem}
Let us notice that the invariants $\mathrm{Irr\,}(H)$
and $\delta(H)$  depend only on the spectral type 
$(\Pi_{H},[J_{0}])$ of $H\,dz$.
Therefore we can also define the 
$\delta$-invariant $\delta(S)$ for an abstract spectral type $S$ as well.

\section{Truncated orbits of unramified canonical forms}
The truncated orbit  $\mathbb{O}_{H}$ 
of an unramified canonical form 
\[
	H\,dz=\left(\frac{H_{k}}{z^{k}}+\cdots+\frac{H_{1}}{z}+H_{\mathrm{res}}\right)\frac{dz}{z}
\]
was introduced in the paper \cite{Boa1} by Boalch and its many fundamental properties 
were studied there, see also 
a subsequent paper \cite{HY} by 
Hiroe and Yamakawa, and  
for the case of $G$-connections, we also refer Yamakawa's paper \cite{Yam}.
In this section, among these properties, we shall recall that $\mathbb{O}_{H}$ can been seen as a coadjoint orbit 
of a finite-dimensional complex Lie group.
We shall also give a formula for the dimension of $\mathbb{O}_{H}$ in terms of the $\delta$-invariant 
of $H$, which was introduced in the previous section.

\subsection{Affine algebraic groups over the ring of formal power series}
Following Babbitt and Varadarajan \cite{BV}, we recall some fundamental structures of $G(\mathbb{C}[\![z]\!])$.
Recall that $\mathbb{C}[\![z]\!]$ is a local ring with the maximal ideal \index[cN]{$\mathfrak{m}_{z}$}
$\mathfrak{m}_{z}:=\langle z\rangle$.
The projection map 
$\pi\colon \mathbb{C}[\![z]\!]\rightarrow \mathbb{C}[\![z]\!]/\mathfrak{m}_{z}\cong \mathbb{C}$
admits a section given by the inclusion $\iota\colon \mathbb{C}\hookrightarrow \mathbb{C}[\![z]\!]$. 
Thus  
the induced group homomorphisms
$\pi_{*}\colon G(\mathbb{C}[\![z]\!])\rightarrow G$ and 
$\iota_{*}\colon G\rightarrow G(\mathbb{C}[\![z]\!])$
satisfy 
$\pi_{*}\circ \iota_{*}=\mathrm{id}_{G}$, and 
we have 
the short exact sequence 
\[
	1\rightarrow \mathrm{Ker\,}\pi_{*}\rightarrow G(\mathbb{C}[\![z]\!])\xrightarrow[]{\pi_{*}}  G\rightarrow 1
\]
with the right splitting $\iota_{*}$.
Then by putting  \index[cN]{$G(\mathbb{C} \llbracket z\rrbracket)_{1}$}
$G(\mathbb{C}[\![z]\!])_{1}:=\mathrm{Ker\,}\pi_{*}$,
we obtain the semidirect product decomposition
\[
	G(\mathbb{C}[\![z]\!])\cong 	G\ltimes G(\mathbb{C}[\![z]\!])_{1}.
\] 
Since $\mathrm{GL}_{N}(\mathbb{C}[\![z]\!])$ is written as
\[
	\mathrm{GL}_{N}(\mathbb{C}[\![z]\!])=
	\left\{
		\sum_{i=0}^{\infty}g_{i}z^{i}\,\middle|\, g_{0}\in \mathrm{GL}_N(\mathbb{C}), g_{i}\in M_{N}(\mathbb{C}), i=1,2,\ldots
	\right\},
\]
the normal subgroup $\mathrm{GL}_{N}(\mathbb{C}[\![z]\!])_{1}$ has the explicit description
\[
	\mathrm{GL}_{N}(\mathbb{C}[\![z]\!])_{1}=
	\left\{
		\sum_{i=0}^{\infty}g_{i}z^{i}\in \mathrm{GL}(\mathbb{C}[\![z]\!])
		\,\middle|\, g_{0}=I_{N}
	\right\}.
\]
Here $I_{N}$ denotes the identity matrix of size $N$.
Therefore under the embedding $G\hookrightarrow \mathrm{GL}_{N}$,
we can regard 
\begin{equation*}\label{eq:trunc}
	G(\mathbb{C}[\![z]\!])_{1}=\mathrm{GL}_{N}(\mathbb{C}[\![z]\!])_{1}\cap G(\mathbb{C}[\![z]\!]).	
\end{equation*}

Let us consider the Lie algebra \index[cN]{$\mathfrak{g}(\mathbb{C}[\llbracket z\rrbracket)_{1}$}$\mathfrak{g}(\mathbb{C}[\![z]\!])_{1}$ of $G(\mathbb{C}[\![z]\!])_{1}$. 
Then we have 
the semidirect product decomposition 
\[
	\mathfrak{g}(\mathbb{C}[\![z]\!])\cong \mathfrak{g}\oplus \mathfrak{g}(\mathbb{C}[\![z]\!])_{1}.	
\]
Since $\mathfrak{g}(\mathbb{C}[\![z]\!])_{1}$ is the kernel of the projection
\[
\mathfrak{g}(\mathbb{C}[\![z]\!])\ni \sum_{i=0}^{\infty}X_{i}z^{i}
\longmapsto
X_{0}\in \mathfrak{g},
\]
we have 
\[
	\mathfrak{g}(\mathbb{C}[\![z]\!])_{1}=\left\{
		\sum_{i=0}^{\infty}X_{i}z^{i}\in \mathfrak{g}(\mathbb{C}[\![z]\!])\,\middle|\,
		X_{0}=0
	\right\},
\]
and can identify 
\[
	\mathfrak{g}(\mathbb{C}[\![z]\!])_{1}\cong \mathfrak{g}\otimes_{\mathbb{C}}\mathfrak{m}_{z}.
\]
Here we regard $\mathfrak{m}_{z}$ as a (non-unital) $\mathbb{C}$-algebra.

Let us give an explicit description of elements in 
$G(\mathbb{C}[\![z]\!])_{1}$.
For a non-negative integer $l\in \mathbb{Z}_{\ge 0}$,
we consider the quotient ring \index[cN]{$\mathbb{C}[z]_{l}$}
\[
\mathbb{C}[z]_{l}:=\mathbb{C}[\![z]\!]/\langle z^{l+1}\rangle
\cong \mathbb{C}[z]/\langle z^{l+1}\rangle,
\]
which 
is a finite-dimensional $\mathbb{C}$-algebra with 
the unique maximal ideal \index[cN]{$\mathfrak{m}_{z}^{(l)}$}$\mathfrak{m}_{z}^{(l)}:=\langle z\rangle_{\mathbb{C}[z]_{l}}$.
We identify 
$\mathbb{C}[z]_{l}$
with the 
space of polynomials $\{\sum_{i=0}^{l}a_{i}z^{i}\mid a_{i}\in \mathbb{C}\}\cong\mathbb{C}^{l+1}$
of degree at most $l$ as $\mathbb{C}$-vector spaces.
Then the projection map 
$\pi_{l}\colon \mathbb{C}[z]_{l}\rightarrow \mathbb{C}[z]_{l}/\mathfrak{m}_{z}^{(l)}\cong \mathbb{C}$
gives the split short exact sequence 
\[
	1\rightarrow \mathrm{Ker\,}(\pi_{l})_{*}\rightarrow G(\mathbb{C}[z]_{l})\xrightarrow[]{(\pi_{l})_{*}}G\rightarrow 1	
\]
and the semidirect product decomposition
\[
	G(\mathbb{C}[z]_{l})=G\ltimes G(\mathbb{C}[z]_{l})_{1}	
\]
by putting \index[cN]{$G(\mathbb{C}[z]_{l})_{1}$}$G(\mathbb{C}[z]_{l})_{1}:=\mathrm{Ker\,}(\pi_{l})_{*}$.
Similarly as above,
the Lie algebra \index[cN]{$\mathfrak{g}(\mathbb{C}[z]_{l})_{1}$}$\mathfrak{g}(\mathbb{C}[z]_{l})_{1}$
of $G(\mathbb{C}[z]_{l})_{1}$
is isomorphic to $\mathfrak{g}\otimes_{\mathbb{C}}\mathfrak{m}_{z}^{(l)}$.

Since the ring of formal power series can be obtained as the projective limit of $\mathbb{C}[z]_{l}$, i.e.,
$\mathbb{C}[\![z]\!]=\displaystyle \varprojlim_{l}\mathbb{C}[z]_{l}$, we have 
\[
	G(\mathbb{C}[\![z]\!])
		=\mathrm{Hom}_{\mathbb{C}\text{-}\mathrm{alg}}(\mathcal{O}(G),\varprojlim_{l}\mathbb{C}[z]_{l})
		\cong \varprojlim_{l}\mathrm{Hom}_{\mathbb{C}\text{-}\mathrm{alg}}(\mathcal{O}(G),\mathbb{C}[z]_{l})
		=\varprojlim_{l}G(\mathbb{C}[z]_{l}).
\]
On the other hand, the induced sequence  
\[
	1\longrightarrow \varprojlim_{l}\mathrm{Ker\,}(\pi_{l})_{*}
	\longrightarrow \varprojlim_{l}G(\mathbb{C}[z]_{l})
	\longrightarrow \varprojlim_{l}G=G
\]
is also exact since the projective limit functor is left exact.
Also since the right most arrow coincides with $\pi_{*}$, this is moreover a short exact sequence.
Thus we have 
$
G(\mathbb{C}[\![z]\!])_{1}=\mathrm{Ker\,}\pi_{*}\cong \displaystyle\varprojlim_{l}(\pi_{l})_{*}=\displaystyle\varprojlim_{l}G(\mathbb{C}[z]_{l})_{1}.
$
\begin{lem}
	The finite-dimensional complex Lie group $G(\mathbb{C}[z]_{l})_{1}$
	is a connected and simply connected nilpotent Lie group.
\end{lem}
\begin{proof}
Since $\mathfrak{m}_{z}^{(l)}$ is a nilpotent ideal, 
the Lie algebra $\mathfrak{g}(\mathbb{C}[z]_{l})_{1}\cong \mathfrak{g}\otimes_{\mathbb{C}}\mathfrak{m}_{z}^{(l)}$
is also nilpotent.
Thus $G(\mathbb{C}[z]_{l})_{1}$ is a nilpotent Lie group.
Also 
since 
\[
	\mathrm{GL}_{N}(\mathbb{C}[z]_{l})_{1}\cong\left\{I_{N}+X_{1}z+\cdots X_{l}z^{l}\,\middle|\,X_{i}\in M_{N}(\mathbb{C})\right\}\cong M_{N}(\mathbb{C})^{\oplus l},
\]
is connected and simply connected nilpotent Lie group,
it follows from Corollary 1.2.2 in \cite{CorGre} that
the Lie subgroup $G(\mathbb{C}[z]_{l})_{1}$
is connected and simply connected nilpotent Lie group as well.
\end{proof}
The nilpotent Lie algebra $\mathfrak{g}(\mathbb{C}[z]_{l})_{1}$ is 
equipped with the gradation with respect to the degree of $z$ in $\mathbb{C}[z]_{l}$.
Namely if we set 
\[\mathfrak{g}^{[i]}:=\{Xz^{i}\in \mathfrak{g}(\mathbb{C}[z]_{l})_{1}\mid X\in \mathfrak{g}\}
= 
\mathfrak{g}\otimes_{\mathbb{C}}
\left((\mathfrak{m}_{z}^{(l)})^{i}/(\mathfrak{m}_{z}^{(l)})^{i-1}\right)
\]
for $i=1,2,\ldots,l$, then 
we have the natural decomposition
\[
	\mathfrak{g}(\mathbb{C}[z]_{l})_{1}=\bigoplus_{i=1}^{l}\mathfrak{g}^{[i]}
\]
satisfying 
\[
	[\mathfrak{g}^{[i]},\mathfrak{g}^{[j]}]\subset \mathfrak{g}^{[i+j]}
\]
for $i,j=1,2,\ldots,l.$

\begin{prop}\label{prop:explicitdesc}
	The exponential map $\mathrm{exp}\colon \mathfrak{g}(\mathbb{C}[z]_{l})_{1}\rightarrow G(\mathbb{C}[z]_{l})_{1}$
	is biholomorphic. 
	Furthermore, the map 
	\[
		\mathfrak{g}(\mathbb{C}[z]_{l})_{1}=\bigoplus_{i=1}^{l}\mathfrak{g}^{[i]}\ni (Z_{i})_{i=1,2,\ldots,l}
		\longmapsto 	e^{Z_{l}}\cdots e^{Z_{2}}e^{Z_{1}}
		\in G(\mathbb{C}[z]_{l})_{1}
	\]
	is biholomorphic.
\end{prop}
\begin{proof}
The first assertion follows from the well-known fact for nilpotent Lie groups, see Theorem 1.2.1 in \cite{CorGre}.

The second assertion follows from the graded structure of $\mathfrak{g}(\mathbb{C}[z]_{l})_{1}$ as follows.
Let us take $g\in G(\mathbb{C}[z]_{l})_{1}$. Then the first assertion tells us that 
there uniquely exists $Z\in \mathfrak{g}(\mathbb{C}[z]_{l})_{1}$ such that 
$g=e^{Z}$.
Then according to the grading, we can uniquely write 
$Z=Z_{1}+Z^{(1)}$ by $Z_{1}\in \mathfrak{g}^{[1]}$ and $Z^{(1)}\in \bigoplus_{i=2}^{l}\mathfrak{g}^{[i]}$.
Thus the Campbell-Baker-Hausdorff formula shows that there uniquely 
exists $Z^{(2)}\in \bigoplus_{i=2}^{l}\mathfrak{g}^{[i]}$ such that 
\[
	e^{-Z_{1}}e^{Z}=e^{Z^{(2)}}.	
\]
Similarly for $Z^{(2)}$ we have the unique decomposition 
$Z^{(2)}=Z_{2}+Z^{(3)}$ by $Z_{2}\in \mathfrak{g}^{[2]}$ and $X^{(3)}\in \bigoplus_{i=3}^{l}\mathfrak{g}^{[i]}$.
Then moreover 
there uniquely 
exists $Z^{(3)}\in \bigoplus_{i=3}^{l}\mathfrak{g}^{[i]}$ such that 
\[
	e^{-Z_{2}}e^{-Z_{1}}e^{Z}=e^{Z^{(3)}}.	
\]
This procedure allows us to find 
$(Z_{i})_{i=1,2,\ldots,l}\in \bigoplus_{i=1}^{l}\mathfrak{g}^{[i]}$
uniquely from $g\in G(\mathbb{C}[z]_{l})_{1}$, and thus we obtain the well-defined map 
\[
	G(\mathbb{C}[z]_{l})_{1}\ni g\longmapsto \sum_{i=1}^{l}Z_{i}\in \mathfrak{g}(\mathbb{C}[z]_{l})_{1}	
\]
which is holomorphic by the construction. 
And obviously this map gives the inverse map of the above map.
\end{proof}

From this proposition, we obtain the following explicit description of $G(\mathbb{C}[\![z]\!])_{1}$.
\begin{prop}[Proposition 9.3.1 in \cite{BV}]
	Each $g(z)\in G(\mathbb{C}[\![z]\!])_{1}$
has the unique infinite product expansion,
\[
	g(z)=\prod_{l=1}^{\infty}e^{X_{l}z^{l}}:=\lim_{l\to \infty}e^{X_{l}z^{l}}\cdots e^{X_{2}z^{2}}e^{X_{1}z},
	\quad (X_{l}\in \mathfrak{g}).
\]
\end{prop}
\begin{proof}
	This follows from 
	Proposition \ref{prop:explicitdesc} and the isomorphism $G(\mathbb{C}[\![z]\!])_{1}\cong \displaystyle\varprojlim_{l}G(\mathbb{C}[z]_{l})_{1}$.
\end{proof}

Let us notice that the natural projection $\pi_{l}\colon G(\mathbb{C}[\![z]\!])_{1}
\rightarrow G(\mathbb{C}[z]_{l})_{1}$
is surjective.
Indeed, Proposition \ref{prop:explicitdesc} says that every element in
$G(\mathbb{C}[z]_{l})_{1}$ is of the form $e^{X_{l}z^{l}}\cdots e^{X_{2}z^{2}}e^{X_{1}z}$
for $X_{i}\in \mathfrak{g}$ 
which is the projection image of  
$e^{X_{l}z^{l}}\cdots e^{X_{2}z^{2}}e^{X_{1}z}\in G(\mathbb{C}[\![z]\!])_{1}.$ 
Thus 
since $G(\mathbb{C}[\![z]\!])=G\ltimes G(\mathbb{C}[\![z]\!])_{1}$
and $G(\mathbb{C}[z]_{l})=G\ltimes G(\mathbb{C}[z]_{l})_{1}$,
the projection
$\pi_{l}\colon G(\mathbb{C}[\![z]\!])
\rightarrow G(\mathbb{C}[z]_{l})$
is surjective as well.

\subsection{Truncated orbit as coadjoint orbit}\label{sec:trucoad}
The truncated orbit $\mathbb{O}_{H}$ of an unramified canonical form $H\in \mathfrak{g}(\mathbb{C}(\!(z)\!)/\mathbb{C}[\![z]\!])$ is defined as 
the orbit under the action of $G(\mathbb{C}[\![z]\!])$.
We shall explain that $\mathbb{O}_{H}$ can be seen as a coadjoint orbit 
of a finite-dimensional complex Lie group.
Before explaining this fact, we prepare some notation.
For a positive integer $l$, we consider a $\mathbb{C}[\![z]\!]$-module \index[cN]{$\mathbb{C}[z^{-1}]_{l}$}
\[
	\mathbb{C}[z^{-1}]_{l}:=z^{-(l+1)}\mathbb{C}[\![z]\!]/\mathbb{C}[\![z]\!]
\]
which can be also seen 
 as a module over the quotient ring $\mathbb{C}[z]_{l}$. 
Then for a Lie subalgebra $\mathfrak{h}\subset \mathfrak{g}$,
we can define the $\mathbb{C}$-vector space \index[cN]{$\mathfrak{h}(\mathbb{C}[z^{-1}]_{l})$}
\[
	\mathfrak{h}(\mathbb{C}[z^{-1}]_{l}):=\mathfrak{h}\otimes_{\mathbb{C}}\mathbb{C}[z^{-1}]_{l}
\]
on which the action of the Lie algebra $\mathfrak{h}(\mathbb{C}[z]_{l})$ is naturally defined via 
the Lie algebra structure of $\mathfrak{h}$ and the $\mathbb{C}[z]_{l}$-module structure of $\mathbb{C}[z^{-1}]_{l}$.

We shall frequently use the following basis and identifications as $\mathbb{C}$-vector spaces,
\begin{align*}
	\mathbb{C}[z^{-1}]_{l}&\cong \left\{\sum_{i=0}^{l}\frac{a_{i}}{z^{i+1}}\,\middle|\, a_{i}\in \mathbb{C}\right\},&
	\mathbb{C}[z]_{l}&\cong \left\{\sum_{i=0}^{l}a_{i}z^{i}\,\middle|\, a_{i}\in \mathbb{C}\right\}.
\end{align*}
According to these identifications, we set 
$f(0):=a_{0}$ for $f(z)=\sum_{i=0}^{l}a_{i}z^{i}\in \mathbb{C}[z]_{l}$ and 
$\underset{z=0}{\mathrm{res\,}}g(z):=b_{0}$ for $g(z)=\sum_{i=0}^{l}\frac{b_{i}}{z^{i+1}}\in \mathbb{C}[z^{-1}]_{l}$.

Let us introduce a non-degenerate bilinear pairing 
on $\mathfrak{g}(\mathbb{C}[z]_{l})\times \mathfrak{g}(\mathbb{C}[z^{-1}]_{l})$ as follows.
Recall that the trace form 
\begin{equation}\label{eq:traceform}
	M_{N}(\mathbb{C}[z]_{l})\times M_{N}(\mathbb{C}[z^{-1}]_{l})\ni
(A,B)\mapsto \underset{z=0}{\mathrm{res\,}}(\mathrm{tr}(AB))\in \mathbb{C}
\end{equation}
is non-degenerate.
Since we have the embeddings $\mathfrak{g}(\mathbb{C}[z]_{l})\hookrightarrow 
\mathfrak{gl}_{N}(\mathbb{C}[z]_{l})$
and $\mathfrak{g}(\mathbb{C}[z^{-1}]_{l})\hookrightarrow 
\mathfrak{gl}_{N}(\mathbb{C}[z^{-1}]_{l})$
induced by $G\hookrightarrow \mathrm{GL}_{N}$,
this trace pairing defines the non-degenerate pairing 
on $\mathfrak{g}(\mathbb{C}[z]_{l})\times \mathfrak{g}(\mathbb{C}[z^{-1}]_{l})$.
This non-degenerate pairing moreover gives us the identification 
\[
\mathfrak{g}(\mathbb{C}[z]_{l})^{*}\cong \mathfrak{g}(\mathbb{C}[z^{-1}]_{l}).
\]

We can consider the canonical form $H$ as an element in $\mathfrak{g}(\mathbb{C}[z^{-1}]_{l})$
by regarding $\mathfrak{g}(\mathbb{C}[z^{-1}]_{l})$ as a subspace of 
$\mathfrak{g}(\mathbb{C}(\!(z)\!)/\mathbb{C}[\![z]\!])$. 
Then we compare the orbits of $H$ under the actions 
of $G(\mathbb{C}[\![z]\!])$ and $G(\mathbb{C}[z]_{l})$ as follows.
The natural projection map 
$p_{l}\colon \mathbb{C}[\![z]\!]\rightarrow \mathbb{C}[z]_{l}$
induces the short exact sequence 
\[
	1\rightarrow \mathrm{Ker\,}(p_{l})_{*}\rightarrow G(\mathbb{C}[\![z]\!])
	\xrightarrow[]{(p_{l})_{*}}	G(\mathbb{C}[z]_{l})\rightarrow 1.
\]
As we previously remarked, the projection map  
$(p_{l})_{*}\colon G(\mathbb{C}[\![z]\!])
\rightarrow G(\mathbb{C}[z]_{l})$ is surjective.
Since 
$\mathrm{Ker\,}(p_{k})_{*}$ is contained in the stabilizer group $\mathrm{Stab}_{G(\mathbb{C}[\![z]\!])}(H)$
of $H$, we also have the exact sequence
\[
	1\rightarrow \mathrm{Ker\,}(p_{k})_{*}\rightarrow \mathrm{Stab}_{G(\mathbb{C}[\![z]\!])}(H)
	\xrightarrow[]{(p_{k})_{*}|_{\mathrm{Stab}_{G(\mathbb{C}[\![z]\!])}(H)}}	\mathrm{Stab}_{G(\mathbb{C}[z]_{k})}(H)
	\rightarrow 1.
\]
Therefore we obtain the isomorphism 
\[
	G(\mathbb{C}[\![z]\!])/\mathrm{Stab}_{G(\mathbb{C}[\![z]\!])}(H)\cong G(\mathbb{C}[z]_{k})/\mathrm{Stab}_{G(\mathbb{C}[z]_{k})}(H),
\]
which allows us to regard the truncated orbit $\mathbb{O}_{H}$
as the $G(\mathbb{C}[z]_{k})$-orbit through $H\in \mathfrak{g}(\mathbb{C}[z^{-1}]_{k})$.
Namely, $\mathbb{O}_{H}$ can be regarded as the   
$G(\mathbb{C}[z]_{k})$-coadjoint orbit through $H\in \mathfrak{g}(\mathbb{C}[z]_{k})^{*}$
under the identification 
$\mathfrak{g}(\mathbb{C}[z]_{k})^{*}\cong \mathfrak{g}(\mathbb{C}[z^{-1}]_{k}).$

\subsection{$\delta$-invariant and dimension of truncated orbit}
We shall show that $\delta$-invariants of unramified canonical forms
represent dimensions of corresponding truncated orbits.

Let us consider 
the stabilizer of $H$.
\begin{prop}[Lemma 3.3 in \cite{Yam}]\label{prop:stabh}
	The stabilizer 
	of $H\in \mathfrak{g}(\mathbb{C}[z]_{k})^{*}$
	in $G(\mathbb{C}[z]_{k})_{1}$
	is of the following form,
	\[
		\mathrm{Stab}_{(G(\mathbb{C}[z]_{k})_{1})}(H)=
		\left\{
			e^{X_{k}z^{k}}\cdots e^{X_{2}z^{2}}e^{X_{1}z}\in G(\mathbb{C}[z]_{k})_{1}\,\middle|\,
			X_{i}\in \mathfrak{l}_{i},\ i=1,\ldots,k
		\right\}
		.
	\]
	We moreover have
	\[
		\mathrm{Stab}_{G(\mathbb{C}[z]_{k})}(H)=\mathrm{Stab}_{G}(H)\ltimes \mathrm{Stab}_{G(\mathbb{C}[z]_{k})_{1}}(H).	
	\]
\end{prop}
\begin{proof}
	Let us recall that 
	for 
	\[
	g=e^{z^{k}X_{k}}\cdots e^{zX_{1}}\in G(\mathbb{C}[z]_{k})_{1}
	\text{ and }
	B=\sum_{i=0}^{k}B_{i}z^{-i-1}\in \mathfrak{g}(\mathbb{C}[z^{-1}]_{k}),
	\]
	we have $\mathrm{Ad}^{*}(g)(B)=\sum_{i=0}^{k}C_{i}z^{-i-1}$ with 
	\[
	C_{i}
	=\sum_{l_{1},l_{2},\ldots,l_{k}\ge 0}
	\frac{\mathrm{ad}(X_{k})^{l_{k}}\circ \cdots\circ \mathrm{ad}(X_{2})^{l_{2}}\circ \mathrm{ad}(X_{1})^{l_{1}}
	(B_{i+l_{1}+2l_{2}+\cdots +kl_{k}})}
	{l_{1}!l_{2}!\cdots l_{k}!}.
	\]
	Let us  apply this formula to the case $g\in \mathrm{Stab}_{G(\mathbb{C}[z]_{k})_{1}}(H)$
	and $B=H$.
	Then the equation $C_{k-1}=H_{k-1}$ implies
	\[
		H_{k-1}=H_{k-1}+\mathrm{ad}(X_{1})(H_{k})
	\] 
	which gives us $X_{1}\in \mathfrak{l}_{k}$.
	Secondly, the equation $C_{k-2}=H_{k-2}$ implies 
	\[
		H_{k-2}=H_{k-2}+\mathrm{ad}(X_{1})(H_{k-1})+\mathrm{ad}(X_{2})(H_{k}).	
	\]
	Let us recall the decomposition 
	$\mathfrak{g}=\mathfrak{g}_{H_{k}}\oplus \mathrm{ad}(H_{k})(\mathfrak{g})$
	and the equation $\mathfrak{g}_{H_{k}}=\mathfrak{l}_{k}$.
	Here we recall that $\mathfrak{g}_{H_{k}}$ denotes the 
	kernel of $\mathrm{ad}(H_{k})\in \mathrm{End}_{\mathbb{C}}(\mathfrak{g})$.
	Then since $\mathrm{ad}(X_{1})(H_{k-1})\in \mathfrak{l}_{k}$ and $\mathrm{ad}(X_{2})(H_{k})
	\in \mathrm{ad}(H_{k})(\mathfrak{g})$,
	$\mathrm{ad}(X_{1})(H_{k-1})=\mathrm{ad}(X_{2})(H_{k})=0$.
	This implies $X_{1}\in \mathfrak{l}_{k-1}$ and $X_{2}\in \mathfrak{l}_{k}$.
	
	Similarly, for $i>2$ under the assumption $X_{j}\in \mathfrak{l}_{k-i+j+1}$ for $j=1,2,\ldots,i-1$
	the equation $C_{k-i}=H_{k-i}$
	implies 
	\begin{multline*}
		\sum_{l_{1},l_{2},\ldots,l_{i-1}> 0}
		\frac{\mathrm{ad}(X_{i-1})^{l_{i-1}}\circ \cdots\circ \mathrm{ad}(X_{2})^{l_{2}}\circ \mathrm{ad}(X_{1})^{l_{1}}
		(H_{k-i+l_{1}+2l_{2}+\cdots +(i-1)l_{i-1}})}
		{l_{1}!l_{2}!\cdots l_{i-1}!}\\
		+\mathrm{ad}(X_{i})(H_{k})=0.
	\end{multline*}
	From the assumption, the first term belongs to $\mathfrak{l}_{k}$.
	The above argument shows $\mathrm{ad}(X_{i})(H_{k})=0$, i.e.,
	$X_{i}\in \mathfrak{l}_{k}$.
	Further this implies that 
	\begin{multline*}
		\sum_{l_{1},l_{2},\ldots,l_{i-2}> 0}
		\frac{\mathrm{ad}(X_{i-1})^{l_{i-1}}\circ \cdots\circ \mathrm{ad}(X_{2})^{l_{2}}\circ \mathrm{ad}(X_{1})^{l_{1}}
		(H_{k-i+l_{1}+2l_{2}+\cdots +(i-2)l_{i-2}})}
		{l_{1}!l_{2}!\cdots l_{i-2}!}\\
		+\mathrm{ad}(X_{i-1})(H_{k-1})=0.
	\end{multline*}
	Here we used the equation $\mathrm{ad}(X_{i-1})\circ \mathrm{ad}(X_{1})(H_{k})=\mathrm{ad}(X_{i-1})(H_{k})=0$
	which follows from the assumption $X_{1}\in \mathfrak{l}_{k-i+2}$
	and $X_{i-1}\in \mathfrak{l}_{k}$.
	Then since the first term belongs to $\mathfrak{l}_{k-1}$,
	the decomposition $\mathfrak{l}_{k}=(\mathfrak{l}_{k})_{H_{k-1}}\oplus \mathrm{ad}(H_{k-1})(\mathfrak{l}_{k})$
	and the equation $(\mathfrak{l}_{k})_{H_{k-1}}=\mathfrak{l}_{k-1}$
	imply that $\mathrm{ad}(X_{i-1})(H_{k-1})$, i.e., $X_{i-1}\in \mathfrak{l}_{k-1}$.
	Iterating this procedure, we obtain 
	$X_{j}\in \mathfrak{l}_{k-i+j}$ for $j=1,2,\ldots,i$.

	Thus inductively, we obtain $X_{j}\in \mathfrak{l}_{j}$ for $j=1,2,\ldots,k$ which shows 
	the desired equation for $\mathrm{Stab}_{G(\mathbb{C}[z]_{k})_{1}}(H)$.
	
	Let us see the second assertion. 
	Take $g\in \mathrm{Stab}_{G(\mathbb{C}[z]_{k})}(H)$
	and decompose $g=g_{0}\cdot g_{1}$ by
	$g_{0}\in G$ and $g_{1}\in G(\mathbb{C}[z]_{k})_{1}$.
	Then 
	since the constant term of $g_{1}$ is the identity, we have 
	\[
		\mathrm{Ad}^{*}(g_{1})(H)-H\equiv 0\quad (\mathrm{mod\,}z^{-k}),
	\]
	and moreover $g_{0}\in\mathrm{Stab}_{G}(H_{k})=L_{k}$.
	This implies that 
	\[
		\mathrm{Ad}^{*}(g_{1})(H)=\mathrm{Ad}^{*}(g_{0}^{-1})(H)\in \mathfrak{l}_{k}(\mathbb{C}[z^{-1}]_{k})
	\]
	which further shows 
	\[
		g_{1}\in \left\{
			e^{X_{k}z^{k}}\cdots e^{X_{2}z^{2}}e^{X_{1}z}\in G(\mathbb{C}[z]_{k})_{1}\,\middle|\,
			X_{i}\in \mathfrak{l}_{k},\ i=1,\ldots,k
		\right\}.	
	\]
	Then since $H_{k}$ is in the center of $\mathfrak{l}_{k}$, we have  
	\[
	\mathrm{Ad}^{*}(g_{1})(H)-H\equiv 0\quad (\mathrm{mod\,}z^{-(k-1)}),
	\]
	and moreover $g_{0}\in \mathrm{Stab}_{G}(H_{k})\cap \mathrm{Stab}_{G}(H_{k-1})=L_{k-1}$.
	Therefore we have 
	\[
		\mathrm{Ad}^{*}(g_{1})(H)=\mathrm{Ad}^{*}(g_{0}^{-1})(H)\in \mathfrak{l}_{k-1}(\mathbb{C}[z^{-1}]_{k})
	\]
	which further shows 
	\[
		g_{1}\in \left\{
			e^{X_{k}z^{k}}\cdots e^{X_{2}z^{2}}e^{X_{1}z}\in G(\mathbb{C}[z]_{k})_{1}\,\middle|\,
			X_{k}\in \mathfrak{l}_{k},\,X_{i}\in \mathfrak{l}_{k-1},\ i=1,\ldots,k-1
		\right\}.	
	\]
	Here we recall that $e^{X_{k}z^{k}}$ with $X_{k}\in \mathfrak{l}_{k}$
	is in $\mathrm{Stab}_{G(\mathbb{C}[z]_{k})_{1}}(H)$ by the first assertion.
	By iterating this procedure, we obtain the second assertion.
\end{proof}
\begin{cor}\label{cor:specdimorb}
Let $H\,dz\in \mathfrak{g}(\mathbb{C}[z^{-1}]_{k})\,dz$ be an unramified canonical form.
Then we have 
\[
	\mathrm{dim\,}\mathbb{O}_{H}=\delta(H).	
\]
\end{cor}
\begin{proof}
	Recall that 
	\[
		\mathrm{dim\,}G(\mathbb{C}[z]_{k})=\mathrm{dim\,}\mathfrak{g}(\mathbb{C}[z]_{k})
		=\mathrm{dim\,}\mathfrak{g}\cdot \mathrm{dim\,}\mathbb{C}[z]_{k}
		=(k+1)\cdot\mathrm{dim\,}G.
	\]
	Therefore from Proposition \ref{prop:stabh} we have 
	\begin{align*}
		\mathrm{dim\,}\mathbb{O}_{H}&=\mathrm{dim\,}G(\mathbb{C}[z]_{k})-\mathrm{dim\,}\mathrm{Stab}_{G(\mathbb{C}[z]_{k})}(H)\\
		&=(k+1)\cdot \mathrm{dim\,}G-\left(\sum_{i=1}^{k}\mathrm{dim\,}L_{i}+\mathrm{dim\,}\mathrm{Stab}_{G}(H)\right)\\
		&=\mathrm{dim\,}G+\sum_{i=1}^{k}(\mathrm{dim\,}G-\mathrm{dim\,}L_{i})-\mathrm{dim\,}\mathrm{Stab}_{G}(H)\\
		&=\mathrm{dim\,}G+\mathrm{Irr\,}(H)-\mathrm{dim\,}\mathrm{Stab}_{G}(H)
		=\delta(H).
	\end{align*}
\end{proof}

\section{Meromorphic $G$-connections on $\mathbb{P}^{1}$ with unramified canonical forms}\label{sec:rigidindex}
Let us consider meromorphic connections on the trivial $G$-bundle 
over $\mathbb{P}^{1}$ and 
define the invariant of them, called the index of rigidity.
\subsection{Meromorphic $G$-connections on $\mathbb{P}^{1}$ and index of rigidity}\label{sec:indrig}
Let us fix a finite set \index[cN]{$\left\lvert D \right\rvert$}$|D|:=\{a_{1},a_{2},\ldots,a_{d}\}$ 
of points in $\mathbb{P}^{1}$, and  
define an effective divisor \index[cN]{$D$}$D:=\sum_{a\in \{a_{1},\ldots,a_{d}\}}(k_{a}+1)\cdot a$ 
with \index[cN]{$k_{a}$}$k_{a}\in\mathbb{Z}_{\ge 0}$.
Let \index[cN]{$z_{a}$}$z_{a}$ for $a\in |D|$ be coordinate functions on $\mathbb{P}^{1}$
centered at $a$.
Let \index[cN]{$\mathcal{O}_{\mathbb{P}^{1}}$}$\mathcal{O}_{\mathbb{P}^{1}}$ be the sheaf of regular functions on $\mathbb{P}^{1}$.
Let \index[cN]{$\mathcal{O}_{\mathbb{P}^{1},D}$}$\mathcal{O}_{\mathbb{P}^{1},D}$ and \index[cN]{$\varOmega_{\mathbb{P}^{1},D}$}$\varOmega_{\mathbb{P}^{1},D}$ denote 
the sheaves of rational functions and of rational 1-forms associated to $D$ respectively.

Recalling that a connection on the trivial $G$-bundle
corresponds to a $\mathfrak{g}$-valued 1-form,
we 
take a $\mathfrak{g}$-valued meromorphic algebraic $1$-form on $\mathbb{P}^{1}$\index[cN]{$\nabla_{A}$}
\[
	\nabla_{A}=A\,dz\in \mathfrak{g}\otimes_{\mathbb{C}}\varOmega_{\mathbb{P}^{1},D}(\mathbb{P}^{1})
\]
and call it a {\em meromorphic connection on the trivial $G$-bundle over $\mathbb{P}^{1}$}
or 
{\em meromorphic $G$-connection on $\mathbb{P}^{1}$} shortly.

By applying a projective transformation on $\mathbb{P}^{1}$, we may assume that $|D|\subset \mathbb{C}$.
Then we can write 
\[
	\nabla_{A}=\sum_{a\in |D|}\sum_{i=0}^{k_{a}}\frac{A^{(a)}_{i}}{(z-a)^{i}}\frac{dz}{z-a},\quad (A_{i}^{(a)}\in \mathfrak{g}).
\]
Here we note that the residues satisfy the relation
\begin{equation}\label{eq:rescond}
	\sum_{a\in |D|}A^{(a)}_{0}=0	
\end{equation}
since $\nabla_{A}$ is regular at $\infty$.
Now we moreover suppose that $\nabla_{A}$ has 
unramified truncated canonical form \index[cN]{$H^{(a)}$}$H^{(a)}\in \mathfrak{g}(\mathbb{C}[z_{a}^{-1}]_{k_{a}})$
at each singular point $a\in |D|$. Then the collection \index[cN]{$\mathbf{H}$}$\mathbf{H}=(H^{(a)})_{a\in |D|}$
of canonical forms should satisfy the following relation which is an analogue of the Fuchs relation.
\begin{prop}\label{prop:resss}
	Let us take a meromorphic $G$-connection 
	\[
		\sum_{a\in |D|}\sum_{i=0}^{k_{a}}\frac{A^{(a)}_{i}}{(z-a)^{i}}\frac{dz}{z-a}
	\]
	on $\mathbb{P}^{1}$  as above.
	Let us suppose that there exists a collection of canonical forms $\mathbf{H}=(H^{(a)})_{a\in |D|}
	\in \prod_{a\in |D|}\mathfrak{g}(\mathbb{C}[z_{a}^{-1}]_{k_{a}})$
	such that 
	\[
		\sum_{i=0}^{k_{a}}\frac{A^{(a)}_{i}}{(z-a)^{i+1}}\in \mathbb{O}_{H^{(a)}},\quad a\in |D|.
	\]
	Then we have 
	\[
		\sum_{a\in |D|}H^{(a)}_{\mathrm{res}}\in \mathfrak{g}_{\mathrm{ss}}.	
	\]
	Here \index[cN]{$\mathfrak{g}_{\mathrm{ss}}$}$\mathfrak{g}_{\mathrm{ss}}:=[\mathfrak{g},\mathfrak{g}]$ is the semisimple part of the reductive Lie algebra $\mathfrak{g}$.
\end{prop}
\begin{proof}
	Firstly, note that computations in this proof will be carried out in  
	$\mathrm{GL}_{N}(\mathbb{C}[z_{a}]_{k_{a}})$, $M_{N}(\mathbb{C}[z_{a}]_{k_{a}})$,
	and 
	$M_{N}(\mathbf{C}[z_{a}^{-1}]_{k_{a}})$
	through the embedding $G\hookrightarrow \mathrm{GL}_{N}$.
	Since $\sum_{i=0}^{k_{a}}\frac{A^{(a)}_{i}}{(z-a)^{i+1}}\in \mathbb{O}_{H^{(a)}}$, for each $a\in |D|$,
	we can find $g_{a}\in 	G(\mathbb{C}[z_{a}]_{k_{a}})$
	so that $\mathrm{Ad}^{*}(g_{a})(H^{(a)})=\sum_{i=0}^{k_{a}}\frac{A^{(a)}_{i}}{(z-a)^{i+1}}$.

	Let $\mathfrak{z}$ be the center of $\mathfrak{g}$. Let $\mathfrak{g}_{\mathrm{ss}}^{\bot}$ be the orthogonal complement of $\mathfrak{g}_{\mathrm{ss}}$ with respect to the trace pairing.
	Then $\mathfrak{g}_{\mathrm{ss}}^{\bot}= \mathfrak{z}$. Indeed, 
	for $X\in \mathfrak{g}_{\mathrm{ss}}^{\bot}$ and $Y_{1},Y_{2}\in \mathfrak{g}$, 
	we have 
	\[
		\mathrm{tr}([X,Y_{1}]\cdot Y_{2})=\mathrm{tr}(X\cdot [Y_{1},Y_{2}])=0	
	\]
	since $[Y_{1},Y_{2}]\in \mathfrak{g}_{ss}=[\mathfrak{g},\mathfrak{g}]$. Thus since the trace pairing
	is non-degenerate for $\mathfrak{g}$, we have 
	$[X,Y_{1}]=0$, i.e.,
	$\mathfrak{g}_{\mathrm{ss}}^{\bot}\subset \mathfrak{z}$. 
	Conversely, recalling that $\mathfrak{g}=\mathfrak{g}_{\mathrm{ss}}^{\bot}\oplus \mathfrak{g}_{\mathrm{ss}}\subset \mathfrak{z}\oplus \mathfrak{g}_{\mathrm{ss}}=\mathfrak{g}$,
	we obtain $\mathfrak{g}_{\mathrm{ss}}^{\bot}= \mathfrak{z}$.
	Then for $X\in \mathfrak{g}_{\mathrm{ss}}^{\bot}=\mathfrak{z}$, we have 
	\begin{align*}
		0&=\mathrm{tr}(X\cdot 0)=
		\mathrm{tr}\left(X\cdot \sum_{a\in |D|}A^{(a)}_{0}\right)=
		\mathrm{tr}\left(X\cdot \sum_{a\in |D|}\underset{z=a}{\mathrm{res\,}}\left(\sum_{i=0}^{k_{a}}\frac{A^{(a)}_{i}}{(z-a)^{i+1}}\right)\right)\\
		&=\mathrm{tr}\left(X\cdot \sum_{a\in |D|}\underset{z=a}{\mathrm{res\,}}(\mathrm{Ad}(g_{a})(H^{(a)}))\right)
		=
		\sum_{a\in |D|}\underset{z=a}{\mathrm{res\,}}\mathrm{tr}(X\cdot  \mathrm{Ad}(g_{a})(H^{(a)}))\\
		&=\sum_{a\in |D|}\underset{z=a}{\mathrm{res\,}}\mathrm{tr}(\mathrm{Ad}(g_{a}^{-1})(X)\cdot H^{(a)})
		=\sum_{a\in |D|}\underset{z=a}{\mathrm{res\,}}\mathrm{tr}(X \cdot H^{(a)})\\
		&=\mathrm{tr}\left(X\cdot \sum_{a\in |D|}\underset{z=a}{\mathrm{res\,}}(H^{(a)})\right).
	\end{align*}
	This equation implies that $\sum_{a\in |D|}\underset{z=a}{\mathrm{res\,}}(H^{(a)})\in \mathfrak{z}^{\bot}=\mathfrak{g}_{\mathrm{ss}}$
	as desired.
\end{proof}
\begin{df}[Index of rigidity of $\mathbf{H}$]\normalfont
	Let $\mathbf{H}=(H^{(a)})_{a\in |D|}$ be a collection of canonical forms satisfying 
	$\sum_{a\in |D|}\underset{z=a}{\mathrm{res\,}}(H^{(a)})\in \mathfrak{g}_{\mathrm{ss}}.$
	The integer defined by \index[cN]{$\mathrm{rig\,}(\mathbf{H})$}
	\[
		\mathrm{rig\,}(\mathbf{H}):=2\,\mathrm{dim\,}G-\sum_{a\in |D|}\delta(H^{(a)})
	\]
	is called the {\em index of rigidity of $\mathbf{H}$}.
\end{df}
Here we note that $\mathrm{rig\,}(\mathbf{H})$ depends only on 
the spectral type $\mathrm{sp}(\mathbf{H})$ of $\mathbf{H}$.
Thus we can define the index of rigidity $\mathrm{rig\,}(\mathbf{S})$
for a collection $\mathbf{S}$ of abstract spectral types as well.

We can also associate this quantity to meromorphic $G$-connections on $\mathbb{P}^{1}$. 
\begin{df}[Truncated index of rigidity]\normalfont
	Let $\mathbf{H}=(H^{(a)})_{a\in |D|}$ be a collection of canonical forms as above.
	Let $\displaystyle \nabla_{A}=A\,dz$, $A\,dz\in \varOmega^{\mathfrak{g}}_{\mathbb{P}^{1},D}(\mathbb{P}^{1})$,
	be a meromorphic $G$-connection on  $\mathbb{P}^{1}$ with 
	the truncated canonical form $H^{(a)}$ at each $a\in |D|$.
	Then the {\em truncated index of rigidity} of $\nabla_{A}$ is 
	defined by \index[cN]{$\mathrm{rig}^{\mathrm{tr}}(\nabla_{A})$}
	\[
		\mathrm{rig}^{\mathrm{tr}}(\nabla_{A}):=\mathrm{rig\,}(\mathbf{H}).	
	\]
\end{df}
This is an analogue of the index of rigidity originally defined by Katz in \cite{Katz} for local systems on $\mathbb{P}^{1}\backslash \{n\text{-points}\}$.
Under the Riemann-Hilbert correspondence, one can also define the index of rigidity for meromorphic connections, see Arinkin's paper \cite{Ari} and its references for instance,
and also see the paper  \cite{JY} by Jakob and Yun for $G$-connections.

\subsection{Example I: Heun equation and its confluent equations}\label{sec:exI}
In this and the next sections, we restrict our interest to the case $G=\mathrm{GL}_{n}$ 
and consider some explicit differential equations on $\mathbb{P}^{1}$,
and then see their canonical forms and spectral types and indices of rigidity.

The first example is the Heun-type equation
\[	
	\frac{dY}{dz}=\left(\frac{A^{(0)}_{0}}{z}+\frac{A^{(1)}_{0}}{z-1}+\frac{A^{(t)}_{0}}{z-t}\right)Y,\quad (A^{(a)}_{0}\in \mathfrak{gl}_{2}=M_{2}(\mathbb{C}),\,a\in \{0,1,t\}).
\]
This is an analogue 
of the usual Heun differential equation which is a 2nd order 
scalar differential equation with $4$ regular singular points, $z=0,1,t,\infty$.
We regard this differential equation as a $\mathrm{GL}_{2}$-connection 
on $\mathbb{P}^{1}$ and see its canonical forms at singular points.
For simplicity we may assume that eigen-values of 
each   
residue matrices $A^{(0)}_{0},\,A^{(1)}_{0},\,A^{(t)}_{0},\,A^{(\infty)}_{0}:=-(A^{(0)}_{0}+A^{(1)}_{0}+A^{(t)}_{0})$
does not differ by integers.
Then the canonical forms at these singular points are of the forms  
\[
	H^{(a)}=\begin{pmatrix}\theta^{(a)}_{1} & \\ &\theta^{(a)}_{2} \end{pmatrix}z_{a}^{-1}\quad 
	(a\in \{0,1,t,\infty\})	
\]
where $\theta^{(a)}_{1}$ and $\theta^{(a)}_{2}$ are distinct eigen-values 
of $A_{a}$.

Similarly we consider the confluent Heun-type equation
\[	
	\frac{dY}{dz}=\left(\frac{A^{(0)}_{0}}{z}+\frac{A^{(1)}_{0}}{z-1}-A^{(\infty)}_{1}\right)Y
\] 
with the canonical forms 
\begin{align*}
	H^{(a)}&=\begin{pmatrix}\theta^{(a)}_{1} & \\ &\theta^{(a)}_{2} \end{pmatrix}z_{a}^{-1}\quad 
	(a\in \{0,1\}),\\
	H^{(\infty)}&=	
	\begin{pmatrix}\alpha^{(\infty)}_{1} & \\ &\alpha^{(\infty)}_{2} \end{pmatrix}\zeta^{-2}
	+\begin{pmatrix}\theta^{(\infty)}_{1} & \\ &\theta^{(\infty)}_{2} \end{pmatrix}\zeta^{-1}
\end{align*}
where $\zeta:=1/z$ and we assume $\theta^{(a)}_{1}- \theta^{(a)}_{2}\notin\mathbb{Z}$ for $a\in \{0,1\}$ and 
 $\alpha^{(\infty)}_{1}\neq \alpha^{(\infty)}_{2}$.
Also consider the doubly-confluent Heun-type equation
\[	
	\frac{dY}{dz}=\left(\frac{A^{(0)}_{1}}{z^{2}}+\frac{A^{(0)}_{0}}{z}-A^{(\infty)}_{1}\right)Y
\] 
with the canonical forms 
\begin{align*}
	H^{(0)}&=
	\begin{pmatrix}\alpha^{(0)}_{1} & \\ &\alpha^{(0)}_{2} \end{pmatrix}z^{-2}
	+\begin{pmatrix}\theta^{(0)}_{1} & \\ &\theta^{(0)}_{2} \end{pmatrix}z^{-1},\\
	H^{(\infty)}&=	
	\begin{pmatrix}\alpha^{(\infty)}_{1} & \\ &\alpha^{(\infty)}_{2} \end{pmatrix}\zeta^{-2}
	+\begin{pmatrix}\theta^{(\infty)}_{1} & \\ &\theta^{(\infty)}_{2} \end{pmatrix}\zeta^{-1}
\end{align*}
where $\alpha^{(a)}_{1}\neq \alpha^{(a)}_{2}$ for $a=0,\infty$.
Further consider the biconfluent Heun-type equation 
\[	
	\frac{dY}{dz}=\left(\frac{A^{(0)}_{0}}{z}-A^{(\infty)}_{1}-A^{(\infty)}_{2}z\right)Y
\] 
with the canonical forms 
\begin{align*}
	H^{(0)}&=\begin{pmatrix}\theta^{(0)}_{1} & \\ &\theta^{(0)}_{2} \end{pmatrix}z^{-1},\\
	H^{(\infty)}&=	
	\begin{pmatrix}\alpha^{(\infty)}_{1} & \\ &\alpha^{(\infty)}_{2} \end{pmatrix}\zeta^{-3}
	+\begin{pmatrix}\beta^{(\infty)}_{1} & \\ &\beta^{(\infty)}_{2} \end{pmatrix}\zeta^{-2}
	+\begin{pmatrix}\theta^{(\infty)}_{1} & \\ &\theta^{(\infty)}_{2} \end{pmatrix}\zeta^{-1}
\end{align*}
where we assume $\theta^{(0)}_{1}-\theta^{(0)}_{2}\notin\mathbb{Z}$ and $\alpha^{(\infty)}_{1}\neq \alpha^{(\infty)}_{2}$.
Finally consider the triconfluent Heun-type equation
\[	
	\frac{dY}{dz}=\left(-A^{(\infty)}_{3}z^{2}-A^{(\infty)}_{2}z-A^{(\infty)}_{1}\right)Y
\] 
with the canonical form
\begin{align*}
	&H^{(\infty)}=\\	
	&\begin{pmatrix}\alpha^{(\infty)}_{1} & \\ &\alpha^{(\infty)}_{2} \end{pmatrix}\zeta^{-4}
	+\begin{pmatrix}\beta^{(\infty)}_{1} & \\ &\beta^{(\infty)}_{2} \end{pmatrix}\zeta^{-3}+
	\begin{pmatrix}\gamma^{(\infty)}_{1} & \\ &\gamma^{(\infty)}_{2} \end{pmatrix}\zeta^{-2}
	+\begin{pmatrix}\theta^{(\infty)}_{1} & \\ &\theta^{(\infty)}_{2} \end{pmatrix}\zeta^{-1}
\end{align*}
where we assume 
$\alpha^{(\infty)}_{1}\neq \alpha^{(\infty)}_{2}$.

Let us consider the spectral types and indices of rigidity of these canonical forms.
Since we are in the case $G=\mathrm{GL}_{2}$, 
as a set of simple roots $\Pi$, we can take $\Pi=\{e_{12}:=e_{1}-e_{2}\}$
where $e_{i}$ are $i$-th projection of $\mathbb{C}^{2}\cong \mathfrak{t}$.
Then for the Heun-type equation,
\[
	\mathrm{sp\,}(H^{(a)})=\left(\emptyset;[0]\right)\quad (a\in \{0,1,t,\infty\}),	
\]
for the confluent Heun-type equation,
\[
	\mathrm{sp\,}(H^{(a)})=\left(\emptyset;[0]\right)\ (a\in \{0,1\}),
	\quad 
	\mathrm{sp\,}(H^{(\infty)})=\left(\emptyset\supset \emptyset;[0]\right),
\]
for the doubly-confluent Heun-type equation,
\[
	\mathrm{sp\,}(H^{(0)})=\left(\emptyset\supset \emptyset;[0]\right),
	\quad 
	\mathrm{sp\,}(H^{(\infty)})=\left(\emptyset\supset \emptyset;[0]\right),
\]
for the biconfluent Heun-type equation,
\[
	\mathrm{sp\,}(H^{(0)})=\left(\emptyset;[0]\right),
	\quad 
	\mathrm{sp\,}(H^{(\infty)})=\left(\emptyset\supset \emptyset\supset \emptyset;[0]\right),
\] 
and finally for the triconfluent Heun-type equation,
\[
	\mathrm{sp\,}(H^{(\infty)})=\left(\emptyset\supset \emptyset\supset \emptyset\supset \emptyset;[0]\right).
\] 

It can be checked that all these cases have the same index of rigidity 
\[
	\mathrm{rig\,}(\mathbf{H})=0.	
\]
\subsection{Example II}\label{sec:exII}
The Heun-type equations in the previous section 
have the spectral types of quite simple forms 
since they are $\mathrm{GL}_{2}$-connections, i.e.,
the set of simple root $\Pi$ is the singleton set.
Here we shall consider the following little more complicated example
\[	
	\frac{dY}{dz}=\left(\frac{A^{(0)}_{0}}{z}+\frac{A^{(1)}_{0}}{z-1}+\frac{A^{(t)}_{0}}{z-t}\right)Y,\quad (A^{(a)}_{0}\in \mathfrak{gl}_{4}=M_{4}(\mathbb{C}),\,a\in \{0,1,t\})
\]
with the truncated canonical forms,
\begin{align*}
	H^{(a)}&=\begin{pmatrix}\theta^{(a)}_{1}I_{2} & \\ &\theta^{(a)}_{2}I_{2} \end{pmatrix}z_{a}^{-1}\quad 
	(a\in \{0,1,t\}),&
	H^{(\infty)}&=	
	\begin{pmatrix}\theta^{(\infty)}_{1}I_{2} & &\\ &\theta^{(\infty)}_{2}&\\&& \theta^{(\infty)}_{3}\end{pmatrix}\zeta^{-1}
\end{align*}
where $I_{2}$ is the identity matrix of size $2$ and 
we assume $\theta^{(a)}_{i}\neq \theta^{(a)}_{j}$ for $i\neq j$ and $a\in \{0,1,t,\infty\}$.
This differential equation appears in \cite{Bo} by Boalch and 
\cite{Sak} by Sakai to construct a higher dimensional analogue of the Painlev\'e VI equation, called the matrix 
Painlev\'e VI system in \cite{KNS}  by Kawakami-Nakamura-Sakai.
Furthermore, in the same paper, Kawakami-Nakamura-Sakai
classified  confluent type equations with unramified irregular singularities
arising from the above Fuchsian differential equation.
Now we recall these confluent type equations.

\noindent
\textbf{The confluent equation of type I}:
\[	
	\frac{dY}{dz}=\left(\frac{A^{(0)}_{1}}{z^{2}}+\frac{A^{(0)}_{0}}{z}+\frac{A^{(1)}_{0}}{z-1}\right)Y,\quad (A^{(a)}_{i}\in \mathfrak{gl}_{4}=M_{4}(\mathbb{C}),\,a\in \{0,1\})
\]
with the truncated canonical forms,
\begin{align*}
	H^{(0)}&=\begin{pmatrix}\alpha^{(0)}_{1}I_{2} & \\ &\alpha^{(0)}_{2}I_{2} \end{pmatrix}z^{-2}+
	\begin{pmatrix}\theta^{(0)}_{1}I_{2} & \\ &\theta^{(0)}_{2}I_{2} \end{pmatrix}z^{-1},	
	\\
	H^{(1)}&=\begin{pmatrix}\theta^{(1)}_{1}I_{2} & \\ &\theta^{(1)}_{2}I_{2} \end{pmatrix}(z-1)^{-1},\quad\quad
	H^{(\infty)}=	
	\begin{pmatrix}\theta^{(\infty)}_{1}I_{2} & &\\ &\theta^{(\infty)}_{2}&\\&& \theta^{(\infty)}_{3}\end{pmatrix}\zeta^{-1}
\end{align*}
where  
we assume $\alpha^{(0)}_{1}\neq \alpha^{(0)}_{2}$ and 
$\theta^{(a)}_{i}\neq \theta^{(a)}_{j}$ for $i\neq j$ and $a\in \{1,\infty\}$.

\noindent
\textbf{The confluent equation of type II}:
\[	
	\frac{dY}{dz}=\left(\frac{A^{(0)}_{0}}{z}+\frac{A^{(1)}_{0}}{z-1}-A^{(\infty)}_{1}\right)Y,\quad (A^{(a)}_{i}\in \mathfrak{gl}_{4}=M_{4}(\mathbb{C}),\,a\in \{0,1,\infty\})
\]
with the truncated canonical forms,
\begin{align*}
	H^{(a)}&=\begin{pmatrix}\theta^{(a)}_{1}I_{2} & \\ &\theta^{(a)}_{2}I_{2} \end{pmatrix}z_{a}^{-1}\quad (a\in \{0,1\}),\\
	H^{(\infty)}&=	\begin{pmatrix}\alpha^{(\infty)}_{1}I_{2} & \\ &\alpha^{(\infty)}_{2}I_{2} \end{pmatrix}\zeta^{-2}+
	\begin{pmatrix}\theta^{(\infty)}_{1}I_{2} & &\\ &\theta^{(\infty)}_{2}&\\&& \theta^{(\infty)}_{3}\end{pmatrix}\zeta^{-1}
\end{align*}
where 
we assume $\alpha^{(\infty)}_{1}\neq \alpha^{(\infty)}_{2}$,
$\theta^{(a)}_{1}\neq \theta^{(a)}_{2}$ for  $a\in \{0,1\}$,
and $\theta^{(\infty)}_{2}\neq \theta^{(\infty)}_{3}$.

\noindent
\textbf{The doubly-confluent equation}:
\[	
	\frac{dY}{dz}=\left(\frac{A^{(0)}_{1}}{z^{2}}+\frac{A^{(0)}_{0}}{z}-A^{(\infty)}_{1}\right)Y,\quad (A^{(a)}_{i}\in \mathfrak{gl}_{4}=M_{4}(\mathbb{C}),\,a\in \{0,\infty\})
\]
with the truncated canonical forms,
\begin{align*}
	H^{(0)}&=\begin{pmatrix}\alpha^{(0)}_{1}I_{2} & \\ &\alpha^{(0)}_{2}I_{2} \end{pmatrix}z^{-2}+
	\begin{pmatrix}\theta^{(0)}_{1}I_{2} & \\ &\theta^{(0)}_{2}I_{2} \end{pmatrix}z^{-1},	
	\\
	H^{(\infty)}&=	\begin{pmatrix}\alpha^{(\infty)}_{1}I_{2} & \\ &\alpha^{(\infty)}_{2}I_{2} \end{pmatrix}\zeta^{-2}+
	\begin{pmatrix}\theta^{(\infty)}_{1}I_{2} & &\\ &\theta^{(\infty)}_{2}&\\&& \theta^{(\infty)}_{3}\end{pmatrix}\zeta^{-1}
\end{align*}
where 
we assume $\alpha^{(a)}_{1}\neq \alpha^{(a)}_{2}$,
 for  $a\in \{0,\infty\}$ and $\theta^{(\infty)}_{2}\neq \theta^{(\infty)}_{3}$.

 \noindent
\textbf{The biconfluent equation of type I}:
\[	
	\frac{dY}{dz}=\left(\frac{A^{(0)}_{2}}{z^{3}}+\frac{A^{(0)}_{1}}{z^{2}}+\frac{A^{(0)}_{0}}{z}\right)Y,\quad (A^{(a)}_{i}\in \mathfrak{gl}_{4}=M_{4}(\mathbb{C}),\,a\in \{0,1\})
\]
with the truncated canonical forms,
\begin{align*}
	H^{(0)}&=\begin{pmatrix}\alpha^{(0)}_{1}I_{2} & \\ &\alpha^{(0)}_{2}I_{2} \end{pmatrix}z^{-3}+
	\begin{pmatrix}\beta^{(0)}_{1}I_{2} & \\ &\beta^{(0)}_{2}I_{2} \end{pmatrix}z^{-2}+
	\begin{pmatrix}\theta^{(0)}_{1}I_{2} & \\ &\theta^{(0)}_{2}I_{2} \end{pmatrix}z^{-1},	
	\\
	H^{(\infty)}&=	
	\begin{pmatrix}\theta^{(\infty)}_{1}I_{2} & &\\ &\theta^{(\infty)}_{2}&\\&& \theta^{(\infty)}_{3}\end{pmatrix}\zeta^{-1}
\end{align*}
where  
we assume $\alpha^{(0)}_{1}\neq \alpha^{(0)}_{2}$ and 
$\theta^{(\infty)}_{i}\neq \theta^{(\infty)}_{j}$ for $i\neq j$.

\noindent
\textbf{The biconfluent equation of type II}:
\[	
	\frac{dY}{dz}=\left(\frac{A^{(0)}_{0}}{z}-A^{(\infty)}_{2}z-A^{(\infty)}_{1}\right)Y,\quad (A^{(a)}_{i}\in \mathfrak{gl}_{4}=M_{4}(\mathbb{C}),\,a\in \{0,\infty\})
\]
with the truncated canonical forms,
\begin{align*}
	&H^{(0)}=\begin{pmatrix}\theta^{(0)}_{1}I_{2} & \\ &\theta^{(0)}_{2}I_{2} \end{pmatrix}z^{-1},\\
	&H^{(\infty)}=\\	
	&\begin{pmatrix}\alpha^{(\infty)}_{1}I_{2} & \\ &\alpha^{(\infty)}_{2}I_{2} \end{pmatrix}\zeta^{-2}+
	\begin{pmatrix}\beta^{(\infty)}_{1}I_{2} & \\ &\beta^{(\infty)}_{2}I_{2} \end{pmatrix}\zeta^{-2}+
	\begin{pmatrix}\theta^{(\infty)}_{1}I_{2} & &\\ &\theta^{(\infty)}_{2}&\\&& \theta^{(\infty)}_{3}\end{pmatrix}\zeta^{-1}
\end{align*}
where 
we assume $\alpha^{(\infty)}_{1}\neq \alpha^{(\infty)}_{2}$, 
$\theta^{(0)}_{1}\neq \theta^{(0)}_{2}$ and $\theta^{(\infty)}_{2}\neq \theta^{(\infty)}_{3}$.

\noindent
\textbf{The triconfluent equation}:
\[	
	\frac{dY}{dz}=\left(-A^{(\infty)}_{3}z^{2}-A^{(\infty)}_{2}z-A^{(\infty)}_{1}\right)Y,\quad (A^{(a)}_{i}\in \mathfrak{gl}_{4}=M_{4}(\mathbb{C}),\,a\in \{0,\infty\})
\]
with the truncated canonical form,
\begin{align*}
	H^{(\infty)}&=	\begin{pmatrix}\alpha^{(\infty)}_{1}I_{2} & \\ &\alpha^{(\infty)}_{2}I_{2} \end{pmatrix}\zeta^{-4}+
	\begin{pmatrix}\beta^{(\infty)}_{1}I_{2} & \\ &\beta^{(\infty)}_{2}I_{2} \end{pmatrix}\zeta^{-3}\\
	&\quad\quad+
	\begin{pmatrix}\gamma^{(\infty)}_{1}I_{2} & \\ &\gamma^{(\infty)}_{2}I_{2} \end{pmatrix}\zeta^{-2}+
	\begin{pmatrix}\theta^{(\infty)}_{1}I_{2} & &\\ &\theta^{(\infty)}_{2}&\\&& \theta^{(\infty)}_{3}\end{pmatrix}\zeta^{-1}
\end{align*}
where 
we assume $\alpha^{(\infty)}_{1}\neq \alpha^{(\infty)}_{2}$ and $\theta^{(\infty)}_{2}\neq \theta^{(\infty)}_{3}$.

Let us consider the spectral types and indices of rigidity of these canonical forms.
Since we are in the case $G=\mathrm{GL}_{4}$, 
as a set of simple roots $\Pi$ we can take $\Pi=\{e_{12},e_{23},e_{34}\}$
where $e_{ij}:=e_{i}-e_{j}$ and  $e_{i}$ are $i$-th projection of $\mathbb{C}^{4}\cong \mathfrak{t}$.
Then for the first Fuchsian equation, the collection of spectral types is
\[
	\mathrm{sp\,}(H^{(a)})=\left(\{e_{12},e_{34}\};[0]\right)\quad (a\in \{0,1,t\}),\quad 
	\mathrm{sp\,}(H^{(\infty)})=\left(\{e_{12}\};[0]\right),
\]
also for the confluent equation of type I,
\begin{align*}
	\mathrm{sp\,}(H^{(0)})&=\left(\{e_{12},e_{34}\}\supset \{e_{12},e_{34}\};[0]\right),&	
	\mathrm{sp\,}(H^{(1)})&=\left(\{e_{12},e_{34}\};[0]\right),\\
	\mathrm{sp\,}(H^{(\infty)})&=\left(\{e_{12}\};[0]\right),
\end{align*}
for the confluent equation of type II,
\begin{align*}
	\mathrm{sp\,}(H^{(a)})&=\left(\{e_{12},e_{34}\};[0]\right)\ (a\in \{0,1\}),& 	
	\mathrm{sp\,}(H^{(\infty)})&=\left(\{e_{12},e_{34}\}\supset \{e_{12}\};[0]\right),
\end{align*}
for the doubly-confluent equation,
\begin{align*}
	\mathrm{sp\,}(H^{(0)})&=\left(\{e_{12},e_{34}\}\supset \{e_{12},e_{34}\};[0]\right),&
	\mathrm{sp\,}(H^{(\infty)})&=\left(\{e_{12},e_{34}\}\supset \{e_{12}\};[0]\right),
\end{align*}
for the biconfluent equation of type I,
\begin{align*}
	\mathrm{sp\,}(H^{(0)})&=\left(\{e_{12},e_{34}\}\supset\{e_{12},e_{34}\}\supset \{e_{12},e_{34}\};[0]\right),&
	\mathrm{sp\,}(H^{(\infty)})&=\left(\{e_{12}\};[0]\right),
\end{align*}
for the biconfluent equation of type II,
\begin{align*}
	\mathrm{sp\,}(H^{(0)})&=\left(\{e_{12},e_{34}\};[0]\right),&
	\mathrm{sp\,}(H^{(\infty)})&=\left(\{e_{12},e_{34}\}\supset\{e_{12},e_{34}\}\supset \{e_{12}\};[0]\right),
\end{align*}
 and finally for the triconfluent  equation,
\[
	\mathrm{sp\,}(H^{(\infty)})=\left(\{e_{12},e_{34}\}\supset\{e_{12},e_{34}\}\supset\{e_{12},e_{34}\}\supset \{e_{12}\};[0]\right).
\] 

It can be checked that all these cases have the same index of rigidity 
\[
	\mathrm{rig\,}(\mathbf{H})=-2.	
\]

\section{$\delta$-constant deformation of canonical forms
}\label{sec:unfoldHTL}
In this section, we shall introduce a deformation of 
unramified canonical forms which preserves their $\delta$-invariants.

Let us consider an unramified canonical form
\[
	H\,dz=\left(\frac{H_{k}}{z^{k}}+\cdots+\frac{H_{1}}{z}+H_{\mathrm{res}}\right)\frac{dz}{z}
	\in \mathfrak{g}(\mathbb{C}[z^{-1}]_{k})dz.
\]
Let $H_{\mathrm{res}}=H_{0}+J_{0}$ be the 
Jordan decomposition with the semisimple $H_{0}$ and nilpotent $J_{0}$.
We may assume $H$ satisfies the assumptions in Section \ref{sec:spectral}.

\subsection{Complement of hypersurfaces associated with $H$}\label{sec:openset}
To each simple root $\alpha\in \Pi\backslash \Pi_{0}$, we associate the positive integer
\[
	d(\alpha):=\mathrm{max}\{i\in \{1,\ldots,k\}\mid \alpha\notin \Pi_{i}\}.
\]
Then it follows that  
$\alpha(H_{d(\alpha)})\neq 0$, and $\alpha(H_{i})=0$ for all $d(\alpha) < j\le k$.	
We also define the following 
polynomials 
\[
	f_{\alpha}^{(i)}(x_{0},x_{1},\ldots,x_{k}):=\sum_{j=i}^{d(\alpha)}\left(\alpha(H_{j})\prod_{\nu=j+1}^{d(\alpha)}(x_{i}-x_{\nu})\right)
\]
for $i=0,1,\ldots,d(\alpha)$.
Here we formally put $\prod_{\nu=d(\alpha)+1}^{d(\alpha)}(x_{i}-x_{\nu}):=1$.
Then we denote the complement of hypersurfaces defined by these polynomials $f_{\alpha}^{(i)}(x)$ by\index[cN]{$\mathbb{B}_{H}$}
\begin{equation*}\label{eq:deformationspace}
	\mathbb{B}_{H}:=\left\{\mathbf{c}\in \mathbb{C}^{k+1}\,\middle|\, f_{\alpha}^{(i)}(\mathbf{c})\neq 0\text{ for all }i=0,1,\ldots,d(\alpha),\,
	\alpha\in \Pi\backslash \Pi_{0}
	\right\}.
\end{equation*}
Since $f_{\alpha}^{(i)}(0,0,\ldots,0)=\alpha(H_{d(\alpha)})\neq 0$,
we have $\mathbf{0}\in \mathbb{B}_{H}$, thus in particular, $\mathbb{B}_{H}\neq \emptyset$.

\subsection{Stratification of $\mathbb{C}^{k+1}$ associated to partitions}\label{sec:partstra}
Let\index[cN]{$\mathcal{I}$}
$\mathcal{I}\colon I_{1},I_{2},\ldots,I_{r}$  be a partition of $\{0,1,2,\ldots,k\}$,
i.e., the direct sum decomposition
\[
	I_{0}\sqcup I_{1}\sqcup\cdots\sqcup I_{r}=\{0,1,2,\ldots,k\}
\]
with nonempty direct summands \index[cN]{$I_{i}$}$I_{i}$, $i=0,1,\ldots,r$.
Without loss of generality, we may assume that  $0\in I_{0}$.
Let us introduce an embedding of $\mathbb{C}^{r+1}$ into $\mathbb{C}^{k+1}$ associated to $\mathcal{I}$,
\[
	\iota_{\mathcal{I}}\colon \mathbb{C}^{r+1}\ni\mathbf{c}=(c_{I_{0}},c_{I_{1}},\ldots,c_{I_{r}})\mapsto 
	(\iota_{\mathcal{I}}(\mathbf{c})_{0},\iota_{\mathcal{I}}(\mathbf{c})_{1},\ldots,\iota_{\mathcal{I}}(\mathbf{c})_{k})\in \mathbb{C}^{k+1}
\]
by setting 
$\iota_{\mathcal{I}}(\mathbf{c})_{i}:=c_{I_{l}}$ if $i\in I_{l}.$
Let us denote the configuration space of $r+1$ points in $\mathbb{C}$ by 
\[
	C_{r+1}(\mathbb{C}):=\{(c_{I_{0}},c_{I_{1}},\ldots,c_{I_{r}})\in\mathbb{C}^{r+1}\mid c_{I_{i}}\neq c_{I_{j}}\text{ for }i\neq j\}.	
\]
We consider the subspace of $\mathbb{C}^{k+1}$, \index[cN]{$C(\mathcal{I})$}
\[
	C(\mathcal{I}):=\iota_{\mathcal{I}}(C_{r+1}(\mathbb{C}))=
	\left\{
		(c_{0},c_{1},\ldots,c_{k})\in \mathbb{C}^{k+1}\,\middle|\,
		\begin{array}{ll}
			c_{i}=c_{j}&\text{if }i,j\in I_{l}\text{ for some }l,\\
			c_{i}\neq c_{j}&\text{otherwise}
		\end{array}
	\right\}.
\]
Then we obtain the direct sum decomposition of $\mathbb{C}^{k+1}$,
\[
		\mathbb{C}^{k+1}=\bigsqcup_{\mathcal{I}\in \mathcal{P}_{[k+1]}}C(\mathcal{I}),
\]
where \index[cN]{$\mathcal{P}_{[k+1]}$} $\mathcal{P}_{[k+1]}$ is the set of all partitions of $\{0,1,\ldots,k\}$.
\begin{df}[$\mathcal{P}$-decomposed space]\label{df:decompsp}\normalfont
	Let $\mathcal{P}$ be a poset.
	An {\em $\mathcal{P}$-decomposition} of a Hausdorff paracompact topological 
	space $X$ is a locally finite correction $(S_{i})_{i\in \mathcal{I}}$
	of disjoint locally closed manifolds such that 
	\begin{enumerate}
		\item $\displaystyle X=\bigsqcup_{i\in \mathcal{I}}S_{i}$,
		\item $\displaystyle S_{i}\cap \overline{S_{j}}\neq\emptyset\ \Longleftrightarrow\ S_{i}\subset \overline{S_{j}}\ \Longleftrightarrow\ 
		i\le j$.
	\end{enumerate}
	We call the space $X$ an {\em $\mathcal{P}$-decomposed space} and each $S_{i}$ a {\em piece} of $X$.
\end{df}
According to the natural ordering of $\mathcal{P}_{[k+1]}$ defined by the refinement of partitions,
the above decomposition defines a $\mathcal{P}_{[k+1]}$-decomposed space structure on $\mathbb{C}^{k+1}$.
This decomposed space structure on $\mathbb{C}^{k+1}$ moreover
defines the stratification. Thus we call the pieces of $\mathbb{C}^{k+1}$ the {\em strata} of 
the stratified space $\mathbb{C}^{k+1}$.

The lemma below assures that the open subset $\mathbb{B}_{H}\subset \mathbb{C}^{k+1}$ defined above 
intersects with 
every stratum $\mathcal{C}(\mathcal{I})$.
Therefore the decomposition $\mathbb{B}_{H}=\bigsqcup _{\mathcal{I}\in \mathcal{P}_{[k+1]}}\mathcal{C}(\mathcal{I})\cap \mathbb{B}_{H}$ 
defines the stratification on $\mathbb{B}_{H}$.
\begin{lem}\label{lem:zeronbd}
	Let $U$ be an open neighborhood of $\mathbf{0}\in \mathbb{C}^{k+1}$.
	Then  $U\cap \mathcal{C}(\mathcal{I})\neq \emptyset$ for 
	any $\mathcal{I}\in \mathcal{P}_{[k+1]}$.
\end{lem}
\begin{proof}
	Let us take $\mathcal{I}\in \mathcal{P}_{[k+1]}$.
	Then the inclusion map $\iota_{\mathcal{I}}\colon \mathbb{C}^{r+1}\hookrightarrow \mathbb{C}^{k+1}$
	defined above satisfies $\mathbf{0}\in \mathrm{Im\,}\iota_{\mathcal{I}}$.
	Thus $\iota_{\mathcal{I}}^{-1}(U)\neq \emptyset$. 
	The configuration space $C_{r+1}(\mathbb{C})$ of $r+1$ points in $\mathbb{C}$
	is a Zariski open subset of the connected space $\mathbb{C}^{r+1}$ 
	and thus it is dense in $\mathbb{C}^{r+1}$.
	Therefore we have 
	\[
		\emptyset\neq C_{r+1}(\mathbb{C})\cap \iota_{\mathcal{I}}^{-1}(U)
		=\iota_{\mathcal{I}}^{-1}(\mathcal{C}(\mathcal{I}))\cap \iota_{\mathcal{I}}^{-1}(U)
		=\iota_{\mathcal{I}}^{-1}(\mathcal{C}(\mathcal{I})\cap U),
	\]
	which shows $\mathcal{C}(\mathcal{I})\cap U\neq \emptyset$ as desired.
\end{proof}

\subsection[Unfolding of canonical forms]{Unfolding of canonical forms and decomposition of the spectral types}\label{sec:unfoldspec}

\begin{df}[Deformation of unramified canonical forms]\normalfont
Let $H\,dz$ be the canonical form as above.
Then the  {\em unfolding} of $H$ is the function from $\mathbb{C}^{k+1}$
to $\mathfrak{g}(\mathbb{C}(z))$
of the form,\index[cN]{$H(\mathbf{c})$}
\begin{multline*}
	H(c_{0},c_{1},\ldots,c_{k}):=\\
	\left(\frac{H_{k}}{(z-c_{1})(z-c_{2})\cdots(z-c_{k})}
	+\frac{H_{k-1}}{(z-c_{1})(z-c_{2})\cdots(z-c_{k-1})}+\cdots 
	+H_{\mathrm{res}}\right)
	\times \frac{1}{z-c_{0}}.
\end{multline*}
\end{df}

Let us take a partition 
$\mathcal{I}\colon I_{0},I_{1},\ldots,I_{r} \in \mathcal{P}_{[k+1]}$, and 
consider $H(\mathbf{c})$ for $\mathbf{c}\in C(\mathcal{I})$.
Since $H(\mathbf{c})$ has poles at 
$c_{I_{0}},c_{I_{1}},\ldots,c_{I_{r}}$
 as the function of $z$,
we obtain the partial fraction decomposition 
of $H(\mathbf{c})$,
\[
	H(\mathbf{c})=
	\sum_{j=0}^{r}\sum_{\nu=0}^{k_{j}}\frac{A^{[j]}_{\nu}}{(z-c_{I_{j}})^{\nu+1}}
	\quad (A^{[j]}_{\nu}\in \mathfrak{g}).
\]
This is the sum of the canonical forms  
\index[cN]{$H(\mathbf{c})_{j}$}
\[
	H(\mathbf{c})_{j}:=\sum_{\nu=0}^{k_{j}}\frac{A^{[j]}_{\nu}}{(z-c_{I_{j}})^{\nu+1}}
\]
at $z=c_{I_{j}}$ for $j=0,1,\ldots,r$. 
The coefficients $A^{[j]}_{\nu}\in \mathfrak{g}$
are linear combinations of $H_{i}$ and $H_{\text{res}}$
which might be complicated to write down explicitly.
However, the spectral types of $H(\mathbf{c})_{j}$
are rather easily described in terms of the partition $\mathcal{I}$ as follows.

\begin{prop}\label{prop:hyper}
	Let $(\Pi_{k}\supset \Pi_{k-1}\supset \cdots \supset \Pi;[J_{0}])$
	be the spectral type of $H$.
	Suppose that $\mathbf{c}\in \mathbb{B}_{H}\cap C(\mathcal{I})$,
	and write $I_{j}=\{i_{[j,0]}<i_{[j,1]}<\cdots < i_{[j,k_{j}]}\}$.
	Then we have 
		\[
			\Pi_{i_{[j,\nu]}}=\left\{\alpha\in \Pi\,\middle|\, 
			\alpha((A^{[j]}_{\mu})_{\mathrm{ss}})=0\text{ for all }\mu=\nu,\nu+1,\ldots,k_{j}\right\}
		\]
		for $\nu=0,1,\ldots,k_{j},\,j=0,1,\ldots,r.$
		Here \index[cN]{$X_{\mathrm{ss}}$}
		$X_{\mathrm{ss}}$ denotes the semisimple part in the Jordan decomposition of $X\in \mathfrak{g}$.
\end{prop}
\begin{proof}
Recalling that 
$
		A_{\mu}^{[j]}=\underset{z=c_{I_{j}}}{\mathrm{res}}\left((z-c_{I_{j}})^{\mu}H(\mathbf{c})\right),
$
we have
\begin{align*}
	&A_{\mu}^{[j]}=\\
	&\ \sum_{i=i_{[j,\mu]}}^{i_{[j,\mu+1]}-1}\frac{H_{i}}{\prod_{\substack{0\le l\le i\\
	l\notin I_{j}}}(c_{I_{j}}-c_{l})}
	+\left(
		\text{linear combination of }H_{i}\text{ for } i_{[j,\mu+1]}\le i\le k\right)
\end{align*}
for $(\mu,j)\neq (0,0)$, and 
\begin{equation}\label{eq:numzero}
	\begin{aligned}
		&A_{0}^{[0]}=\\
		&\ H_{\text{res}}+\sum_{i=1}^{i_{[0,1]}-1}\frac{H_{i}}{\prod_{\substack{0\le l\le i\\
		l\notin I_{0}}}(c_{I_{0}}-c_{l})}
		+\left(
			\text{linear combination of }H_{i}\text{ for } i_{[0,1]}\le i\le k\right).	
	\end{aligned}	
\end{equation}
Here we formally put $i_{[j,k_{j}+1]}:=k+1$.
Thus, in any cases, semisimple parts are written by  
\begin{align*}
	\begin{aligned}
	&(A_{\mu}^{[j]})_{\mathrm{ss}}=\\
	&\quad \sum_{i=i_{[j,\mu]}}^{i_{[j,\mu+1]}-1}\frac{H_{i}}{\prod_{\substack{0\le l\le i\\
	l\notin I_{j}}}(c_{I_{j}}-c_{l})}
	+\left(
		\text{linear combination of }H_{i}\text{ for } i_{[j,\mu+1]}\le i\le k\right)
	\end{aligned}	
\end{align*}
for all $\mu=0,1,\ldots,k$ and $j=0,1,\ldots,r$.

Therefore if  $\alpha(H_{i})=0$ for all $i\ge i_{[j,\mu]}$, then $\alpha((A_{\mu}^{[j]})_{\mathrm{ss}})=0$.
Namely, we have the inclusion 
\[
	\Pi_{i_{[j,\nu]}}\subset \left\{\alpha\in \Pi\,\middle|\, 
	\alpha((A^{[j]}_{\mu})_{\mathrm{ss}})=0\text{ for all }\mu=\nu,\nu+1,\ldots,k_{j}\right\}.
\]

To obtain the converse inclusion, we take a simple root $\alpha \notin \Pi_{i_{[j,\nu]}}$.
Then we can find $\mu\in \{\nu,\nu+1,\ldots,k_{j}\}$ such that 
$i_{[j,\mu]}\le d(\alpha)< i_{[j,\mu+1]}$.
Since we have $\alpha(H_{i})=0$ for all $d(\alpha)<i \le k$,
we obtain
\begin{align*}
	\alpha((A^{[j]}_{\mu})_{\mathrm{ss}})&=\sum_{i=i_{[j,\mu]}}^{i_{[j,\mu+1]}-1}\frac{\alpha(H_{i})}{\prod_{\substack{0\le l\le i\\
	l\notin I_{j}}}(c_{I_{j}}-c_{l})}
	=\sum_{i=i_{[j,\mu]}}^{d(\alpha)}\frac{\alpha(H_{i})}{\prod_{\substack{0\le l\le i\\
	l\notin I_{j}}}(c_{I_{j}}-c_{l})}\\
	&=\sum_{i=i_{[j,\mu]}}^{d(\alpha)}\frac{\alpha(H_{i})\prod_{l=i+1}^{d(\alpha)}(c_{I_{j}}-c_{l})}
	{\prod_{\substack{0\le l\le d(\alpha)\\
	l\notin I_{j}}}(c_{I_{j}}-c_{l})}
	=\frac{f_{\alpha}^{(i_{[j,\mu]})}(\mathbf{c})}
	{\prod_{\substack{0\le l\le d(\alpha)\\
	l\notin I_{j}}}(c_{I_{j}}-c_{l})}\neq 0.
\end{align*}
Here we used the assumption $\mathbf{c}\in \mathbb{B}_{H}$.
Therefore it follows that 
$$\alpha\notin \left\{\alpha\in \Pi\,\middle|\, 
\alpha((A^{[j]}_{\mu})_{\mathrm{ss}})=0\text{ for all }\mu=k_{j},k_{j}-1,\ldots,\nu	\right\}$$
as desired.
\end{proof}
Therefore we obtain the following decomposition of the spectral type of $H$.
\begin{cor}\label{cor:specdecomp}
	Suppose that $\mathbf{c}\in \mathbb{B}_{H}\cap C(\mathcal{I})$.
	Then the spectral types of $H(\mathbf{c})_{j}$
	are
	\[
	\begin{array}{ll}
		(\Pi_{I_{0}};[J_{0}])&\text{ if }j=0,\\
		(\Pi_{I_{j}};[0])&\text{otherwise}.	
	\end{array}
	\]
	Here $\Pi_{I_{j}}$ denotes the subsequence indexed by $I_{j}=\{i_{[j,0]}<i_{[j,1]}<\cdots < i_{[j,k_{j}]}\}$, 
	\[
		\Pi_{i_{[j,k_{j}]}}\supset \cdots \supset \Pi_{i_{[j,1]}}\supset \Pi_{i_{[j,0]}},
	\]
	of $\Pi_{k}\supset\Pi_{k-1}\supset \cdots \supset \Pi_{0}$.
\end{cor}
\subsection{$\delta$-constant deformation of canonical form}
Let us show that the unfolding $H(\mathbf{c})$ of the canonical form $H$ preserves 
its $\delta$-invariant.
\begin{prop}\label{prop:stabdecomp}
	Suppose that $\mathbf{c}\in \mathbb{B}_{H}\cap C(\mathcal{I})$.
	Write $I_{j}=\{i_{[j,0]}<i_{[j,1]}<\cdots < i_{[j,k_{j}]}\}$ for the components of the partition $\mathcal{I}$.	
	Then we have the following descriptions of the stabilizers of $H(\mathbf{c})_{j}$.
	\begin{enumerate}
		\item For $j=0,1,\ldots,r$ we have 
		\begin{multline*}
			\mathrm{Stab}_{G(\mathbb{C}[z_{c_{I_{j}}}]_{k_{j}})_{1}}(H(\mathbf{c})_{j})\\
			=
			\left\{
				e^{X_{k_{j}} z_{c_{I_{j}}}^{k_{j}}}\cdots e^{X_{1} z_{c_{I_{j}}}}\in G(\mathbb{C}[z_{c_{I_{j}}}]_{k_{j}})_{1}\,\middle|\,
				X_{\nu}\in \mathfrak{l}_{i_{[j,\nu]}},\ \nu=1,\ldots,k_{j}
			\right\}.
		\end{multline*}
		\item For $j=1,2,\ldots,r$, we have 
		\[
			\mathrm{Stab}_{G}(H(\mathbf{c})_{j})=L_{i_{[j,0]}}.	
		\]
		For $j=0$, we also have
		\[
			\mathrm{Stab}_{G}(H(\mathbf{c})_{0})=\mathrm{Stab}_{G}(H).
		\]
	\end{enumerate}
\end{prop}
\begin{proof}
	The first assertion 
	follows directly from Propositions \ref{prop:stabh} and  \ref{prop:hyper}.
	
	For $j=1,2,\ldots,r$,
	Proposition \ref{prop:hyper} 
	shows 
	\[
			\left\{\alpha\in \Pi\,\middle|\, 
			\alpha(A^{[j]}_{\mu})=0\text{ for all }\mu=0,1,\ldots,k_{j}\right\}
			=\Pi_{i_{[j,0]}},
	\]
	which implies $\mathrm{Stab}_{G}(H(\mathbf{c})_{j})=L_{i_{[j,0]}}$.

	For $j=0$, let us recall the equation $(\ref{eq:numzero})$.
	Then it follows that the stabilizer of $A_{0}^{[0]}-J_{0}$
	in $L_{i_{[0,1]}}$ is $L_{1}\subset L_{i_{[0,1]}}$.
	Here we recall that $H_{\mathrm{res}}=J_{0}+H_{0}$ is the Jordan decomposition.
	This implies that $\mathrm{Stab}_{L_{i_{[0,1]}}}(A_{0}^{[0]})=\mathrm{Stab}_{L_{1}}(H_{\text{res}})$, where we postpone the proof
	for a while.
	Then 
	we obtain 
	\begin{align*}
		\mathrm{Stab}_{G}(H(\mathbf{c})_{0})&=\mathrm{Stab}_{G}((H(\mathbf{c})_{0})_{\text{res}})\cap \mathrm{Stab}_{G}((H(\mathbf{c})_{0})_{\text{irr}})\\
		&=\mathrm{Stab}_{G}(A_{0}^{[0]})\cap L_{i_{[0,1]}}
		=\mathrm{Stab}_{L_{1}}(H_{\text{res}})=\mathrm{Stab}_{G}(H),
	\end{align*}
	which shows the second assertion for $j=0$.

	Thus we finally show the equation $\mathrm{Stab}_{L_{i_{[0,1]}}}(A_{0}^{[0]})=\mathrm{Stab}_{L_{1}}(H_{\text{res}})$.
	Firstly, we can obtain the inclusion $\mathrm{Stab}_{L_{i_{[0,1]}}}(A_{0}^{[0]})\subset 
	\mathrm{Stab}_{L_{i_{[0,1]}}}(A_{0}^{[0]}-H_{\text{res}})$ as follows.
	The equation $(\ref{eq:numzero})$ shows that the  
	nilpotent part of the Jordan decomposition of 
	$A_{0}^{[0]}$ equals $J_{0}$.
	Thus for $l\in \mathrm{Stab}_{L_{i_{[0,1]}}}(A_{0}^{[0]})$, we obtain
	\begin{align*}
	\mathrm{Ad}(l)(A_{0}^{[0]}-H_{\mathrm{res}})
	&=\mathrm{Ad}(l)(A_{0}^{[0]}-J_{0}-H_{0})\\
	&=\mathrm{Ad}(l)(A_{0}^{[0]}-J_{0})-\mathrm{Ad}(l)(H_{0})\\
	&=A_{0}^{[0]}-J_{0}-H_{0}=A_{0}^{[0]}-H_{\mathrm{res}}.
	\end{align*}
	Here the third equation follows from the uniqueness of the Jordan decomposition 
	and the fact that 
	$H_{0}$ is stabilized by $L_{i_{[0,1]}}$.
	Thus we have the desired inclusion $\mathrm{Stab}_{L_{i_{[0,1]}}}(A_{0}^{[0]})\subset 
	\mathrm{Stab}_{L_{i_{[0,1]}}}(A_{0}^{[0]}-H_{\mathrm{res}})$.
	
	On the other hand, the equation $(\ref{eq:numzero})$
	tells us that $\mathrm{Stab}_{L_{i_{[0,1]}}}(A_{0}^{[0]}-H_{\mathrm{res}})=L_{1}$.
	Thus obviously we have  $\mathrm{Stab}_{L_{1}}(H_{\text{res}})\subset 
	\mathrm{Stab}_{L_{i_{[0,1]}}}(A_{0}^{[0]})$.
	Moreover the relation 
	$\mathrm{Stab}_{L_{i_{[0,1]}}}(A_{0}^{[0]})\subset \mathrm{Stab}_{L_{i_{[0,1]}}}(A_{0}^{[0]}-H_{\text{res}})$ 
	gives the converse inclusion.
\end{proof}
\begin{cor}\label{cor:deltaconst}
	Let us take $\mathcal{I}\in \mathcal{P}_{[k+1]}$ 
	and consider $H(\mathbf{c})$ for $\mathbf{c}\in C(\mathcal{I})\cap \mathbb{B}_{H}$.
	Let  $H(\mathbf{c})=\sum_{j=0}^{r}H(\mathbf{c})_{j}$ be 
	the partial fraction decomposition 
of $H(\mathbf{c})$ as above.
	Then we have the equation
	\[
		\sum_{i=0}^{r}\delta(H(\mathbf{c})_{j})=\delta(H).
	\]
\end{cor}
\begin{proof}
	By Corollary \ref{cor:specdecomp} and Proposition \ref{prop:stabdecomp},
	we have 
	\begin{align*}
		\sum_{i=0}^{r}\delta(H(\mathbf{c})_{j})&=
		\sum_{i=1}^{r}(\mathrm{dim\,}G+\sum_{j=1}^{k_{i}}(\mathrm{dim\,}G-L_{i_{[i,j]}})-\mathrm{dim\,}L_{i_{[i,0]}})\\
		&\quad +(\mathrm{dim\,}G+\sum_{j=1}^{k_{0}}(\mathrm{dim\,}G-L_{i_{[0,j]}})-\mathrm{dim\,}\mathrm{Stab}_{G}(H))\\
		&=\mathrm{dim\,}G-\sum_{l=1}^{k}(\mathrm{dim\,}G-\mathrm{dim\,}L_{l})-\mathrm{Stab}_{G}(H)\\
		&=\mathrm{dim\,}G-\mathrm{Irr\,}(H)-\mathrm{Stab}_{G}(H)=\delta(H).
	\end{align*}
\end{proof}
\subsection{$\mathrm{rig}$-constant deformation of a family of canonical forms}\label{sec:riginv}
As well as in Section \ref{sec:rigidindex},
let us take an effective divisor 
$D=\sum_{a\in \{a_{1},\ldots,a_{d}\}}(k_{a}+1)\cdot a$
in $\mathbb{C}$ and a collection of canonical forms
$\mathbf{H}=(H^{(a)})_{a\in |D|}\in \prod_{a\in |D|}\mathfrak{g}(\mathbb{C}[z_{a}^{-1}]_{k_{a}})$
with the relation
$\sum_{a\in |D|}H^{(a)}_{\mathrm{res}}\in \mathfrak{g}_{\mathrm{ss}}.$
Then we can consider the unfolding\index[cN]{$\mathbf{H}(\mathbf{c})$}
\[
	\mathbf{H}(\mathbf{c}):=(H^{(a)}(\mathbf{c}_{a}))_{\mathbf{c}=(\mathbf{c}_{a})_{a\in |D|}\in \prod_{a\in |D|}\mathbb{C}^{k_{a}+1}}	
\] 
of the collection $\mathbf{H}$. 

Recall 
the stratification $\mathbb{C}^{k_{a}+1}=\bigsqcup_{\mathcal{I}^{(a)}
\in \mathcal{P}_{[k_{a}+1]}}C(\mathcal{I}^{(a)})$
of each $\mathbb{C}^{k_{a}+1}$,
then the product $\prod_{a\in |D|}\mathbb{C}^{k_{a}+1}$ also has the stratification with 
the strata $\prod_{a\in |D|}C(\mathcal{I}^{(a)})$
for $(\mathcal{I}^{(a)})\in \prod_{a\in |D|}\mathcal{P}_{[k_{a}+1]}$.

Let us take a stratum $\prod_{a\in |D|}C(\mathcal{I}^{(a)})$
of $\prod_{a\in |D|}\mathbb{C}^{k_{a}+1}$
associated with the collection of partitions \index[cN]{$\mathcal{I}^{(a)}$}
$\mathcal{I}^{(a)}\colon I^{(a)}_{0},\ldots,I^{(a)}_{r_{a}}
\in \mathcal{P}_{[k_{a}+1]}$ for $a\in |D|$.
Then for $\mathbf{c}=(\mathbf{c}_{a})_{a\in |D|}\in \prod_{a\in |D|}C(\mathcal{I}_{a})$,
we can write
$H^{(a)}(\mathbf{c}_{a})=\sum_{i=0}^{r_{a}}H^{(a)}(\mathbf{c}_{a})_{i}$ \index[cN]{$H^{(a)}(\mathbf{c}_{a})_{i}$}
according to the partial fraction decomposition
and regard $\mathbf{H}(\mathbf{c})$ as the collection of canonical forms 
\[
	\mathbf{H}(\mathbf{c})=(H^{(a)}(\mathbf{c}_{a})_{i})_{\substack{i=0,1,\ldots,r_{a}\\a\in |D|}}.
\]
\begin{lem}
	Let $(\mathcal{I}^{{(a)}})_{a\in |D|}\in \prod_{a\in |D|}\mathcal{P}_{[k_{a}+1]}$ be a collection 
	of partitions.
	We write the components of the partition by 
	$\mathcal{I}^{(a)}\colon I^{(a)}_{0},\ldots,I^{(a)}_{r_{a}}$.
	Let us take $\mathbf{c}=(\mathbf{c}_{a})_{a\in |D|}\in \prod_{a\in |D|}\mathcal{C}(\mathcal{I}^{(a)})$
	and write $\mathbf{c}_{a}=(c^{(a)}_{0},\ldots,c^{(a)}_{k_{a}})$.

	Then the  collection $\mathbf{H}(\mathbf{c})=(H^{(a)}(\mathbf{c}_{a})_{i})_{\substack{i=0,1,\ldots,r_{a}\\a\in |D|}}$
	of canonical forms satisfies the residue condition  
	\[
		\sum_{a\in |D|}\sum_{j=0}^{r_{a}}\underset{z_{a}-c^{(a)}_{I^{(a)}_{j}}}{\mathrm{res\,}}(H^{(a)}(\mathbf{c}_{a})_{i})
		\in \mathfrak{g}_{\mathrm{ss}}.	
	\]
\end{lem}
\begin{proof}
	Lemma \ref{lem:residue} in the latter section tells us that 
	\[
		\sum_{a\in |D|}\sum_{j=0}^{r_{a}}\underset{z_{a}=c^{(a)}_{I^{(a)}_{j}}}{\mathrm{res\,}}H^{(a)}(\mathbf{c}_{a})_{i}
		=\sum_{a\in |D|}\underset{z=a}{\mathrm{res\,}}(H^{(a)})\in \mathfrak{g}_{\mathrm{ss}}.	
	\]
\end{proof}
Now we can consider the index of rigidity of $\mathbf{H}(\mathbf{c})$,
\[
	\mathrm{rig\,}(\mathbf{H}(\mathbf{c}))=2\,\mathrm{dim\,}G-\sum_{a\in |D|}\sum_{i=0}^{r_{a}}\delta(H^{(a)}(\mathbf{c}_{a})_{i}).
\]
For each $a\in |D|$,
let \index[cN]{$\mathbb{B}_{H^{(a)}}$}$\mathbb{B}_{H^{(a)}}$ be the complement of hypersurfaces 
associated with the canonical form $H^{(a)}$ defined 
in Section \ref{sec:openset}. 
Then we consider the product of them,\index[cN]{$\mathbb{B}_{\mathbf{H}}$}
\[
	\mathbb{B}_{\mathbf{H}}:=\prod_{a\in |D|}\mathbb{B}_{H^{(a)}}.	
\]
\begin{cor}
	For any $\mathbf{c}\in \prod_{a\in |D|}\mathbb{B}_{H^{(a)}}$,
	we have 
	\[
		\mathrm{rig\,}(\mathbf{H}(\mathbf{c}))=\mathrm{rig\,}(\mathbf{H}).
	\]
\end{cor}
\begin{proof}
	This is a direct consequence of Corollary \ref{cor:deltaconst}.
\end{proof}
\subsection{Examples}\label{sec:exIII}
Let us consider the canonical form 
\begin{align*}
	H&=	
	\begin{pmatrix}\alpha_{1} & \\ &\alpha_{2} \end{pmatrix}z^{-4}
	+\begin{pmatrix}\beta_{1} & \\ &\beta_{2} \end{pmatrix}z^{-3}+
	\begin{pmatrix}\gamma_{1} & \\ &\gamma_{2} \end{pmatrix}z^{-2}
	+\begin{pmatrix}\theta_{1} & \\ &\theta_{2} \end{pmatrix}z^{-1}
\end{align*}
in $\mathfrak{gl}_{2}(\mathbb{C}[z^{-1}]_{3})\,dz$
with $\alpha_{1}\neq \alpha_{2}$ which is the canonical form of 
the triconfluent Heun-type equation in Section \ref{sec:exI},
and also consider its unfolding 
\begin{align*}
	H(\mathbf{c})
	&=	
	\frac{\begin{pmatrix}\alpha_{1} & \\ &\alpha_{2} \end{pmatrix}}{(z-c_{0})(z-c_{1})(z-c_{2})(z-c_{3})}
	+\frac{\begin{pmatrix}\beta_{1} & \\ &\beta_{2} \end{pmatrix}}{(z-c_{0})(z-c_{1})(z-c_{2})}\\
	&\quad\quad  +
	\frac{\begin{pmatrix}\gamma_{1} & \\ &\gamma_{2} \end{pmatrix}}{(z-c_{0})(z-c_{1})}
	+\frac{\begin{pmatrix}\theta_{1} & \\ &\theta_{2} \end{pmatrix}}{z-c_{0}}.
\end{align*}
On the original setting in Section \ref{sec:exI},
$H$ is the canonical form at $z=\infty$.
Here we however consider at $z=0$ under some projective transformation 
of $\mathbb{P}^{1}$ for simplicity.
Let us see the partial fraction decomposition of $H(\mathbf{c})$
on each stratum of $\mathbb{B}_{H}\subset \mathbb{C}^{4}$.

Let us take the trivial partition $\mathcal{I}\colon \{0,1,2,3\}$.
Then for $\mathbf{c}\in \mathcal{C}(\mathcal{I})\cap \mathbb{B}_{H}$,
\begin{align*}
	H(\mathbf{c})
	=	
	\begin{pmatrix}\alpha_{1} & \\ &\alpha_{2} \end{pmatrix}z_{c_{0}}^{-4}
	\begin{pmatrix}\beta_{1} & \\ &\beta_{2} \end{pmatrix}z_{c_{0}}^{-3}
	+
	\begin{pmatrix}\gamma_{1} & \\ &\gamma_{2} \end{pmatrix}z_{c_{0}}^{-2}
	+\begin{pmatrix}\theta_{1} & \\ &\theta_{2} \end{pmatrix}z_{c_{0}}^{-1},
\end{align*}
which is the canonical form of the triconfluent Heun-type equation.

Let us consider the partition $\mathcal{I}\colon \{0\}\sqcup \{1,2,3\}$.
Then Proposition \ref{prop:hyper} and its proof imply that 
for $\mathbf{c}\in \mathcal{C}(\mathcal{I})\cap \mathbb{B}_{H}$, we have 
\begin{align*}
	&H(\mathbf{c})
	=	
	\begin{pmatrix}\theta^{(c_{0})}_{1}(\mathbf{c}) & \\ &\theta^{(c_{0})}_{2}(\mathbf{c}) \end{pmatrix}z_{c_{0}}^{-1}+
	\\& 
	\begin{pmatrix}\alpha^{(c_{1})}_{1}(\mathbf{c}) & \\ &\alpha^{(c_{1})}_{2}(\mathbf{c}) \end{pmatrix}z_{c_{1}}^{-3}+
	\begin{pmatrix}\beta^{(c_{1})}_{1}(\mathbf{c}) & \\ &\beta^{(c_{1})}_{2}(\mathbf{c}) \end{pmatrix}z_{c_{1}}^{-2}+
	\begin{pmatrix}\theta_{(c_{1})}^{(1)}(\mathbf{c}) & \\ &\theta^{(c_{1})}_{2}(\mathbf{c}) \end{pmatrix}z_{c_{1}}^{-1},
\end{align*}
with some $\alpha^{(c_{i})}_{j}(\mathbf{c}),\,\beta^{(c_{i})}_{j}(\mathbf{c}),\,\theta^{(c_{i})}_{j}(\mathbf{c})\in \mathbb{C}$
satisfying
$\theta^{(c_{0})}_{1}(\mathbf{c})\neq \theta^{(c_{0})}_{2}(\mathbf{c})$
and $\alpha^{(c_{1})}_{1}(\mathbf{c})\neq \alpha^{(c_{1})}_{2}(\mathbf{c})$.
The canonical forms in the right hand side are those of the biconfluent Heun-type equation.
We obtain the same canonical forms for the partitions 
$\{1\}\sqcup \{0,2,3\}$, $\{2\}\sqcup \{0,1,3\}$, and $\{3\}\sqcup \{0,1,2\}$
as well.

Let us consider the partition $\mathcal{I}\colon \{0,1\}\sqcup \{2,3\}$.
Then for $\mathbf{c}\in \mathcal{C}(\mathcal{I})\cap \mathbb{B}_{H}$,
\begin{align*}
	H(\mathbf{c})
	=&	
	\begin{pmatrix}\alpha^{(c_{0})}_{1}(\mathbf{c}) & \\ &\alpha^{(c_{0})}_{2}(\mathbf{c}) \end{pmatrix}z_{c_{0}}^{-2}+
	\begin{pmatrix}\beta_{(c_{0})}^{(1)}(\mathbf{c}) & \\ &\beta^{(c_{0})}_{2}(\mathbf{c}) \end{pmatrix}z_{c_{0}}^{-1}
	+\\&
	\begin{pmatrix}\alpha^{(c_{2})}_{1}(\mathbf{c}) & \\ &\alpha^{(c_{2})}_{2}(\mathbf{c}) \end{pmatrix}z_{c_{2}}^{-2}+
	\begin{pmatrix}\beta_{(c_{2})}^{(1)}(\mathbf{c}) & \\ &\beta^{(c_{2})}_{2}(\mathbf{c}) \end{pmatrix}z_{c_{2}}^{-1},
\end{align*}
with $\alpha^{(c_{i})}_{1}(\mathbf{c})\neq \alpha^{(c_{i})}_{2}(\mathbf{c})$
for $i=0,2$.
The canonical forms in the right hand side are those of the doubly-confluent Heun-type equation.
We can obtain the same canonical forms for the partitions 
$\{0,2\}\sqcup \{1,3\}$, and $\{0,3\}\sqcup \{1,2\}$
as well.

Let us consider the partition $\mathcal{I}\colon \{0\}\sqcup\{1\}\sqcup \{2,3\}$.
Then for $\mathbf{c}\in \mathcal{C}(\mathcal{I})\cap \mathbb{B}_{H}$,
\begin{align*}
	H(\mathbf{c})
	=&	
	\begin{pmatrix}\theta^{(c_{0})}_{1}(\mathbf{c}) & \\ &\theta^{(c_{0})}_{2}(\mathbf{c}) \end{pmatrix}z_{c_{0}}^{-1}+
	\\
	&\begin{pmatrix}\theta^{(c_{1})}_{1}(\mathbf{c}) & \\ &\theta^{(c_{1})}_{2}(\mathbf{c}) \end{pmatrix}z_{c_{1}}^{-1}+
	\\& 
	\begin{pmatrix}\alpha^{(c_{2})}_{1}(\mathbf{c}) & \\ &\alpha^{(c_{2})}_{2}(\mathbf{c}) \end{pmatrix}z_{c_{2}}^{-2}+
	\begin{pmatrix}\theta_{(c_{2})}^{(1)}(\mathbf{c}) & \\ &\theta^{(c_{2})}_{2}(\mathbf{c}) \end{pmatrix}z_{c_{2}}^{-1},
\end{align*}
with
$\theta^{(c_{i})}_{1}(\mathbf{c})\neq \theta^{(c_{i})}_{2}(\mathbf{c})$ for $i=0,1$
and $\alpha^{(c_{2})}_{1}(\mathbf{c})\neq \alpha^{(c_{2})}_{2}(\mathbf{c})$.
The canonical forms in the right hand side are those of the confluent Heun-type equation.
We can obtain the same canonical forms for the partitions 
$\{0\}\sqcup \{2\}\sqcup \{1,3\}$, $\{0\}\sqcup \{3\}\sqcup \{1,2\}$, 
$\{0,1\}\sqcup \{2\}\sqcup \{3\}$, $\{0,2\}\sqcup \{1\}\sqcup \{3\}$,
and $\{0,3\}\sqcup \{1\}\sqcup \{2\}$
as well.

Finally consider the finest partition $\mathcal{I}\colon \{0\}\sqcup\{1\}\sqcup \{2\}\sqcup\{3\}$.
Then for $\mathbf{c}\in \mathcal{C}(\mathcal{I})\cap \mathbb{B}_{H}$,
\begin{align*}
	H(\mathbf{c})
	=\sum_{i=0}^{3}
	\begin{pmatrix}\theta^{(c_{i})}_{1}(\mathbf{c}) & \\ &\theta^{(c_{i})}_{2}(\mathbf{c}) \end{pmatrix}z_{c_{i}}^{-1},
	\end{align*}
with
$\theta^{(c_{i})}_{1}(\mathbf{c})\neq \theta^{(c_{i})}_{2}(\mathbf{c})$ for $i=0,1,2,3$.
The canonical forms in the right hand side are those of the Heun-type equation.
Therefore all canonical forms for confluent Heun-type equations 
in Section \ref{sec:exI} can be obtained from 
the unfolding of the canonical form of the triconfluent Heun-type equation.

Let us see another example,
\begin{align*}
	H=&	\begin{pmatrix}\alpha_{1}I_{2} & \\ &\alpha_{2}I_{2} \end{pmatrix}z^{-4}+
	\begin{pmatrix}\beta_{1}I_{2} & \\ &\beta_{2}I_{2} \end{pmatrix}z^{-3}	
	\\
	&\quad\quad +\begin{pmatrix}\gamma_{1}I_{2} & \\ &\gamma_{2}I_{2} \end{pmatrix}z^{-2}+
	\begin{pmatrix}\theta_{1}I_{2} & &\\ &\theta_{2}&\\&& \theta_{3}\end{pmatrix}z^{-1}
	\in \mathfrak{gl}_{4}(\mathbb{C}[z^{-1}]_{3})\,dz
\end{align*}
with $\alpha_{1}\neq \alpha_{2}$ and $\theta_{2}\neq \theta_{3}$
which is the canonical form 
of the triconfluent equation in Section \ref{sec:exII}, 
and also consider  its unfolding 
\begin{align*}
	H(\mathbf{c})
	&=	
	\frac{\begin{pmatrix}\alpha_{1}I_{2} & \\ &\alpha_{2}I_{2} \end{pmatrix}}{(z-c_{0})(z-c_{1})(z-c_{2})(z-c_{3})}
	+\frac{\begin{pmatrix}\beta_{1}I_{2} & \\ &\beta_{2}I_{2} \end{pmatrix}}{(z-c_{0})(z-c_{1})(z-c_{2})}\\
	&\quad\quad  +
	\frac{\begin{pmatrix}\gamma_{1}I_{2} & \\ &\gamma_{2}I_{2} \end{pmatrix}}{(z-c_{0})(z-c_{1})}
	+\frac{\begin{pmatrix}\theta_{1}I_{2} & &\\ &\theta_{2}&\\&& \theta_{3}\end{pmatrix}}{z-c_{0}}.
\end{align*}

Take the trivial partition $\mathcal{I}\colon \{0,1,2,3\}$,
then for $\mathbf{c}\in \mathcal{C}(\mathcal{I})\cap \mathbb{B}_{H}$,
\begin{align*}
	H(\mathbf{c})
	&=	
	\begin{pmatrix}\alpha_{1}I_{2} & \\ &\alpha_{2}I_{2} \end{pmatrix}(z-c_{0})^{-4}
	+\begin{pmatrix}\beta_{1}I_{2} & \\ &\beta_{2}I_{2} \end{pmatrix}(z-c_{0})^{-3}\\
	&\quad\quad  +
	\begin{pmatrix}\gamma_{1}I_{2} & \\ &\gamma_{2}I_{2} \end{pmatrix}(z-c_{0})^{-2}
	+\begin{pmatrix}\theta_{1}I_{2} & &\\ &\theta_{2}&\\&& \theta_{3}\end{pmatrix}(z-c_{0})
\end{align*}
which is the canonical form of the triconfluent equation in Section \ref{sec:exII}.

Let us consider the partition $\mathcal{I}\colon \{0\}\sqcup \{1,2,3\}$.
Then 
for $\mathbf{c}\in \mathcal{C}(\mathcal{I})\cap \mathbb{B}_{H}$, we have \small
\begin{align*}
	&H(\mathbf{c})=\\
	&	
	\frac{\begin{pmatrix}\theta_{1}^{(c_{0})}(\mathbf{c})I_{2} & &\\ &\theta_{2}^{(c_{0})}(\mathbf{c})&\\&& \theta_{3}^{(c_{0})}(\mathbf{c})\end{pmatrix}}{z_{c_{0}}}+\\
	& 
	\frac{\begin{pmatrix}\alpha^{(c_{1})}_{1}(\mathbf{c})I_{2} & \\ &\alpha^{(c_{1})}_{2}(\mathbf{c})I_{2} \end{pmatrix}}{z_{c_{1}}^{3}}+
	\frac{\begin{pmatrix}\beta^{(c_{1})}_{1}(\mathbf{c})I_{2} & \\ &\beta^{(c_{1})}_{2}(\mathbf{c})I_{2} \end{pmatrix}}{z_{c_{1}}^{2}}+
	\frac{\begin{pmatrix}\theta_{1}^{(c_{1})}(\mathbf{c})I_{2} & &\\ &\theta_{2}^{(c_{1})}(\mathbf{c})I_{2}\end{pmatrix}}{z_{c_{1}}},
	\end{align*}
\normalsize
with $\alpha^{(c_{1})}_{1}(\mathbf{c})\neq \alpha^{(c_{1})}_{2}(\mathbf{c})$ and 
$\theta^{(c_{0})}_{i}(\mathbf{c})\neq \theta^{(c_{0})}_{j}(\mathbf{c})$ for $i\neq j$.
The canonical forms in the right hand side are those of the biconfluent equation of type I 
in Section \ref{sec:exII}.

Let us consider the partition $\mathcal{I}\colon \{0,1,2\}\sqcup \{3\}$.
Then  
for $\mathbf{c}\in \mathcal{C}(\mathcal{I})\cap \mathbb{B}_{H}$,  \scriptsize
\begin{align*}
	&H(\mathbf{c})=\\
	&\frac{\begin{pmatrix}\alpha^{(c_{0})}_{1}(\mathbf{c})I_{2} & \\ &\alpha^{(c_{0})}_{2}(\mathbf{c})I_{2} \end{pmatrix}}{z_{c_{0}}^{3}}+
	\frac{\begin{pmatrix}\beta^{(c_{0})}_{1}(\mathbf{c})I_{2} & \\ &\beta^{(c_{0})}_{2}(\mathbf{c})I_{2} \end{pmatrix}}{z_{c_{0}}^{2}}+
	\frac{\begin{pmatrix}\theta_{1}^{(c_{0})}(\mathbf{c})I_{2} & &\\ &\theta_{2}^{(c_{0})}(\mathbf{c})&\\&& \theta_{3}^{(c_{0})}(\mathbf{c})\end{pmatrix}}{z_{c_{0}}}+\\
	&	
	\frac{\begin{pmatrix}\theta^{(c_{3})}_{1}(\mathbf{c})I_{2} & \\ &\theta^{(c_{3})}_{2}(\mathbf{c})I_{2} \end{pmatrix}}{z_{c_{3}}},
\end{align*}\normalsize
with 
$\alpha^{(c_{0})}_{1}(\mathbf{c})\neq \alpha^{(c_{0})}_{2}(\mathbf{c})$,
$\theta^{(c_{0})}_{2}(\mathbf{c})\neq \theta^{(c_{0})}_{3}(\mathbf{c})$,
and 
$\theta^{(c_{3})}_{1}(\mathbf{c})\neq \theta^{(c_{3})}_{2}(\mathbf{c})$.
The canonical forms in the right hand side are those of the biconfluent equation of type II 
in Section \ref{sec:exII}.
We can obtain the same canonical forms for the partitions 
$\{0,2,3\}\sqcup \{1\}$ and  $\{0,1,3\}\sqcup \{2\}$
as well.

Let us consider the partition $\mathcal{I}\colon \{0,1\}\sqcup \{2,3\}$.
Then for $\mathbf{c}\in \mathcal{C}(\mathcal{I})\cap \mathbb{B}_{H}$,
\begin{align*}
	H(\mathbf{c})
	=&	
	\frac{\begin{pmatrix}\alpha^{(c_{0})}_{1}(\mathbf{c})I_{2} & \\ &\alpha^{(c_{0})}_{2}(\mathbf{c})I_{2} \end{pmatrix}}{z_{c_{0}}^{2}}+
	\frac{\begin{pmatrix}\theta_{1}^{(c_{0})}(\mathbf{c})I_{2} & &\\ &\theta_{2}^{(c_{0})}(\mathbf{c})&\\&& \theta_{3}^{(c_{0})}(\mathbf{c})\end{pmatrix}}{z_{c_{0}}}
	+\\&
	\frac{\begin{pmatrix}\alpha^{(c_{2})}_{1}(\mathbf{c})I_{2} & \\ &\alpha^{(c_{2})}_{2}(\mathbf{c})I_{2} \end{pmatrix}}{z_{c_{2}}^{2}}+
	\frac{\begin{pmatrix}\theta_{1}^{(c_{2})}(\mathbf{c})I_{2} & \\ &\theta^{(c_{2})}_{2}(\mathbf{c})I_{2} \end{pmatrix}}{z_{c_{2}}},
\end{align*}
with $\alpha^{(c_{i})}_{1}(\mathbf{c})\neq \alpha^{(c_{i})}_{2}(\mathbf{c})$
for $i=0,2$ and $\theta^{(c_{0})}_{2}(\mathbf{c})\neq \theta^{(c_{0})}_{3}(\mathbf{c})$.
The canonical forms in the right hand side are those of the doubly-confluent equation
in Section \ref{sec:exII}.
We can obtain the same canonical forms for the partitions 
$\{0,2\}\sqcup \{1,3\}$, and $\{0,3\}\sqcup \{1,2\}$
as well.

Let us consider the partition $\mathcal{I}\colon \{0\}\sqcup\{1\}\sqcup \{2,3\}$.
Then for $\mathbf{c}\in \mathcal{C}(\mathcal{I})\cap \mathbb{B}_{H}$,
\begin{align*}
	&H(\mathbf{c})
	=\\&	
	\frac{\begin{pmatrix}\theta_{1}^{(c_{0})}(\mathbf{c})I_{2} & &\\ &\theta_{2}^{(c_{0})}(\mathbf{c})&\\&& \theta_{3}^{(c_{0})}(\mathbf{c})\end{pmatrix}}{z_{c_{0}}}+
	\\
	&
	\frac{\begin{pmatrix}\theta^{(c_{1})}_{1}(\mathbf{c})I_{2} & \\ &\theta^{(c_{1})}_{2}(\mathbf{c})I_{2} \end{pmatrix}}{z_{c_{1}}}+
	\\& 
	\frac{\begin{pmatrix}\alpha^{(c_{2})}_{1}(\mathbf{c})I_{2} & \\ &\alpha^{(c_{2})}_{2}(\mathbf{c})I_{2} \end{pmatrix}}{z_{c_{2}}^{2}}+
	\frac{\begin{pmatrix}\theta_{(c_{2})}^{(1)}(\mathbf{c})I_{2} & \\ &\theta^{(c_{2})}_{2}(\mathbf{c})I_{2} \end{pmatrix}}{z_{c_{2}}},
\end{align*}
with
$\theta^{(c_{i})}_{j}(\mathbf{c})\neq \theta^{(c_{i})}_{k}(\mathbf{c})$ for $i=0,1$ and $j\neq k$
and $\alpha^{(c_{2})}_{1}(\mathbf{c})\neq \alpha^{(c_{2})}_{2}(\mathbf{c})$.
The canonical forms in the right hand side are those of the confluent equation of type I
in Section \ref{sec:exII}.
We can obtain the same canonical forms for the partitions 
$\{0\}\sqcup \{2\}\sqcup \{1,3\}$, $\{0\}\sqcup \{3\}\sqcup \{1,2\}$ as well.

Let us consider the partition $\mathcal{I}\colon \{0,1\}\sqcup \{2\}\sqcup \{3\}$.
Then for $\mathbf{c}\in \mathcal{C}(\mathcal{I})\cap \mathbb{B}_{H}$,
\begin{align*}
	&H(\mathbf{c})
	=
	\\& 
	\frac{\begin{pmatrix}\alpha^{(c_{0})}_{1}(\mathbf{c})I_{2} & \\ &\alpha^{(c_{0})}_{2}(\mathbf{c})I_{2} \end{pmatrix}}{z_{c_{0}}^{2}}+
	\frac{\begin{pmatrix}\theta_{1}^{(c_{0})}(\mathbf{c})I_{2} & &\\ &\theta_{2}^{(c_{0})}(\mathbf{c})&\\&& \theta_{3}^{(c_{0})}(\mathbf{c})\end{pmatrix}}{z_{c_{0}}}
	\\&	
	\frac{\begin{pmatrix}\theta_{1}^{(c_{2})}(\mathbf{c})I_{2} & &\\ &\theta_{2}^{(c_{2})}(\mathbf{c})I_{2}\end{pmatrix}}{z_{c_{2}}}+
	\\
	&
	\frac{\begin{pmatrix}\theta^{(c_{3})}_{1}(\mathbf{c})I_{2} & \\ &\theta^{(c_{3})}_{2}(\mathbf{c})I_{2} \end{pmatrix}}{z_{c_{3}}},
\end{align*}
with
$\alpha^{(c_{0})}_{1}(\mathbf{c})\neq \alpha^{(c_{0})}_{2}(\mathbf{c})$,
$\theta^{(c_{0})}_{2}(\mathbf{c})\neq \theta^{(c_{0})}_{3}(\mathbf{c})$,
and 
$\theta^{(c_{i})}_{1}(\mathbf{c})\neq \theta^{(c_{i})}_{2}(\mathbf{c})$ for $i=2,3$.
The canonical forms in the right hand side are those of the confluent equation of type II 
in Section \ref{sec:exII}.
We can obtain the same canonical forms for the partitions 
$\{0,2\}\sqcup \{1\}\sqcup \{3\}$
and $\{0,3\}\sqcup \{1\}\sqcup \{2\}$
as well.

Finally consider the finest partition $\mathcal{I}\colon \{0\}\sqcup\{1\}\sqcup \{2\}\sqcup \{3\}$.
Then for $\mathbf{c}\in \mathcal{C}(\mathcal{I})\cap \mathbb{B}_{H}$,
\begin{align*}
	H(\mathbf{c})
	=
	\frac{\begin{pmatrix}\theta_{1}^{(c_{0})}(\mathbf{c})I_{2} & &\\ &\theta_{2}^{(c_{0})}(\mathbf{c})&\\&& \theta_{3}^{(c_{0})}(\mathbf{c})\end{pmatrix}}{z_{c_{0}}}+
	\sum_{i=1}^{3}
	\frac{\begin{pmatrix}\theta^{(c_{i})}_{1}(\mathbf{c})I_{2} & \\ &\theta^{(c_{i})}_{2}(\mathbf{c})I_{2} \end{pmatrix}}{z_{c_{i}}},
	\end{align*}
with
$\theta^{(c_{i})}_{j}(\mathbf{c})\neq \theta^{(c_{i})}_{k}(\mathbf{c})$ for $i=0,1,2,3$ and $j\neq k$.
The canonical forms in the right hand side are those of the first equation in Section \ref{sec:exII}.

Therefore as well as the Heun-type equations,
the unfolding of $H$ recovers the canonical forms for all the confluent equations appeared in Section \ref{sec:exII}.

\section{Additive Deligne-Simpson problem and a conjecture by Oshima}
In the previous section, we introduced 
a deformation of unramified canonical forms and 
showed that this deformation preserves some important invariants of the canonical forms.
In this section, we propose a problem to find   
a deformation of irreducible meromorphic $G$-connection
defined over $\mathbb{P}^{1}$ which realizes 
the given deformation of canonical forms. 

\subsection{Additive Deligne-Simpson problem for unfolding families of 
canonical forms}
Let us consider an  effective divisor  
$D=\sum_{a\in \{a_{1},\ldots,a_{d}\}}(k_{a}+1)\cdot a$
as in Section \ref{sec:rigidindex}.
\begin{df}[Irreducible connection, cf. Arinkin \cite{Ari0}]\normalfont
	For 
	a  meromorphic $G$-connection $\nabla_{A}=A\,dz$ with a 
	$\mathfrak{g}$-valued meromorphic $1$-form $A\,dz\in \varOmega^{\mathfrak{g}}_{\mathbb{P}^{1},D}(\mathbb{P}^{1})$,
	we say the $\nabla_{A}$ is {\em irreducible} if 
	there is no proper parabolic subalgebra $\mathfrak{p}$ of $\mathfrak{g}$ such that 
	$A\,dz\in \mathfrak{p}\otimes_{\mathbb{C}}\varOmega_{\mathbb{P}^{1},D}(\mathbb{P}^{1})$ 
\end{df}

We now consider the following problem which asks the existence of 
an irreducible $G$-connection with the prescribed local isomorphism classes.
We assume $|D|\subset \mathbb{C}$ under a projective transformation for simplicity.
\begin{prob}[Additive Deligne-Simpson problem for $\mathbf{H}$]\label{prob:addDS}\normalfont
	Let us take a collection of unramified canonical forms
	$\mathbf{H}=(H^{(a)})_{a\in |D|}\in \prod_{a\in |D|}\mathfrak{g}(\mathbb{C}[z_{a}^{-1}]_{k_{a}})$
	satisfying the residue condition
	\(\displaystyle 
		\sum_{a\in |D|}H^{(a)}_{\mathrm{res}}\in \mathfrak{g}_{\mathrm{ss}}.
	\)
	Then find an irreducible meromorphic connection on the trivial $G$-bundle over $\mathbb{P}^{1}$ 
	\[
	\nabla_{A}=\sum_{a\in |D|}\sum_{i=0}^{k_{a}}\frac{A^{(a)}_{i}}{(z-a)^{i}}\frac{dz}{z-a}
	\]
	with singularities only on $|D|$
	such that 
	\[
		\sum_{i=0}^{k_{a}}\frac{A^{(a)}_{i}}{(z-a)^{i+1}}\in \mathbb{O}_{H^{(a)}}
	\]
	for all $a\in |D|$.
\end{prob}
As it is explained in Section \ref{sec:intads},
this problem was considered by Deligne, Simpson, and Kostov,
and after their pioneering works, there are many subsequent 
developments around this problem,
see \cite{C}, \cite{Kos}, \cite{Boarx}, \cite{HY}, \cite{H2}, \cite{KLMNS}, \cite{LSN}, and \cite{JY}.

Now we also consider a generalization of this problem as follows.
For a collection $\mathbf{H}=(H^{(a)})_{a\in |D|}\in \prod_{a\in |D|}\mathfrak{g}(\mathbb{C}[z_{a}^{-1}]_{k_{a}})$
 of unramified canonical forms
satisfying the residue condition $\sum_{a\in |D|}H^{(a)}_{\mathrm{res}}\in \mathfrak{g}_{\mathrm{ss}}$,
let $\mathbf{H}(\mathbf{c})=(H^{(a)}(\mathbf{c}_{a}))_{\mathbf{c}=(\mathbf{c}_{a})\in \prod_{a\in |D|}\mathbb{C}^{k_{a}+1}}$
be the unfolding of $\mathbf{H}$.
Then as in Section \ref{sec:riginv}, 
for each $\mathbf{c}\in \prod_{a\in |D|}\mathbb{C}^{k_{a}+1}$,
we regard $\mathbf{H}(\mathbf{c})$
as the collection of the canonical forms $\mathbf{H}(\mathbf{c})=(H^{(a)}(\mathbf{c}_{a})_{i})_{\substack{i=0,1,\ldots,r_{a}\\a\in |D|}}$.
\begin{prob}[Additive Deligne-Simpson problem for an unfolding 
	family of canonical forms]\label{prob:addDSfamily}\normalfont
	Let $U$ be an open neighborhood of $\mathbf{0}\in \prod_{a\in |D|}\mathbb{C}^{k_{a}+1}$.
	Let $(\mathbf{H}(\mathbf{c}))_{\mathbf{c}\in U}$ be 
	the unfolding of a collection of canonical forms $\mathbf{H}=(H^{(a)})_{a\in |D|}$
	as above.
	Then find a family of meromorphic connections on the trivial $G$-bundle over $\mathbb{P}^{1}$,
	\[
	\nabla_{A(\mathbf{c})}=A(\mathbf{c})\,dz
	\]
	satisfying all the following conditions.
	\begin{enumerate}
		\item The $\mathfrak{g}$-valued $1$-form 
		$A(\mathbf{c})\,dz$ depends holomorphically  on $\mathbf{c}\in U$. 
		\item For each $\mathbf{c}\in U$, 
		 $\nabla_{A(\mathbf{c})}$ is a solution 
		to the additive Deligne-Simpson problem for 
		$\mathbf{H}(\mathbf{c})$.
	\end{enumerate}
\end{prob}
In regard to Problems \ref{prob:addDS} and \ref{prob:addDSfamily},
we obtain the following theorem whose proof is postponed until Section \ref{sec:deform}.
\begin{thm}\label{thm:addDS}
	Let $\mathbf{H}$ be a collection of unramified canonical forms as above.
    Let $\nabla_{A}=A\,dz$ be a solution to the additive Deligne-Simpson problem 
    for $\mathbf{H}$.
    Then there exists an open neighborhood $U$ of $\mathbf{0}\in \prod_{a\in |D|}\mathbb{C}^{k_{a}+1}$
    and a holomorphic family of $G$-connections $\nabla_{A(\mathbf{c})}=A(\mathbf{c})\,dz$\ $(\mathbf{c}\in U)$ on $\mathbb{P}^{1}$
    such that 
    \begin{itemize}
    \item $\nabla_{A(\mathbf{0})}=\nabla_{A}$,
    \item $\nabla_{A(\mathbf{c})}=A(\mathbf{c})\,dz$ is a solution to
    the additive Deligne-Simpson problem for the family $(\mathbf{H}(\mathbf{c}))_{\mathbf{c}\in U}$.
	\end{itemize}
    In particular, the following are equivalent.
	\begin{enumerate}
		\item The additive Deligne-Simpson problem for $\mathbf{H}$ has a solution.
		\item There exists an open neighborhood $U$ of $\mathbf{0}\in\prod_{a\in |D|}\mathbb{C}^{k_{a}+1}$ and 
		the additive Deligne-Simpson problem for the family $(\mathbf{H}(\mathbf{c}))_{\mathbf{c}\in U}$
		has a solution.
	\end{enumerate}
\end{thm}

\subsection{Unfolding of spectral types}\label{sec:unfspc}
As we saw in Corollary \ref{cor:specdecomp},
the unfolding $H(\mathbf{c})$ of an unramified canonical form $H$
induces decompositions of the spectral type of $H$.
We give a diagrammatic description of these decompositions.

According to the decomposition in Corollary \ref{cor:specdecomp},
we define the unfolding of abstract spectral types as follows.
Let $S$ be an abstract spectral type as above.
Take a partition $\mathcal{I}\colon I_{0},\ldots,I_{r}\in \mathcal{P}_{[k+1]}$ of the index set $\{0,1,\ldots,k\}$,
and assume $0\in I_{0}$.
Then we
define the collection \index[cN]{$S^{\mathcal{I}}$}$S^{\mathcal{I}}=(S^{I_{j}})_{j=1,\ldots,r}$ \index[cN]{$S^{I_{j}}$}of spectral types
by 
\begin{align*}
    &S^{I_{0}}:=(\Pi_{I_{0}}; [J_{0}]),
    &S^{I_{j}}:=(\Pi_{I_{j}}; [0]),\ 
    j=1,\ldots,r,
\end{align*}
where for each $j=0,1,\ldots,r$, $\Pi_{I_{j}}$ 
is the subsequences of $\Pi_{k}\supset \cdots \Pi_{0}$ indexed by $I_{j}$, namely,
$\Pi_{I_{j}}$ stands for the subsequence 
$(\Pi_{i_{[j,k_{j}]}}\supset \Pi_{i_{[j,k_{j}-1]}}\supset \cdots \Pi_{i_{[j,0]}})$ with 
indices in $I_{j}=\{i_{[j,0]}<\cdots <i_{[j,k_{j}-1]}<i_{[j,k_{j}]}\}$.
We call this collection $S^{\mathcal{I}}=(S^{I_{j}})_{j=1,\ldots,r}$ 
an {\em unfolding} of $S$ associated to $\mathcal{I}$.

Also for a collection \index[cN]{$\mathbf{S}$}
\[
\mathbf{S}=(S_{i})_{i=1,2,\ldots,d}=
\left((\Pi^{(i)}_{k_{i}}\supset \Pi^{(i)}_{k_{i}-1}\supset\cdots \supset\Pi^{(i)}_{0};[J^{(i)}_{0}])\right)_{i=1,2,\ldots,d}
\]
of abstract spectral types,
we can define the unfolding of $\mathbf{S}$
\[
	\mathbf{S}^{(\mathcal{I}^{(i)})_{i=1,2,\ldots,d}}:=(S_{i}^{I_{j}^{(i)}})_{\substack{i=1,2,\ldots,d,\\ j=1,2,\ldots,r_{i}}}
\]
of $\mathbf{S}$ associated to $(\mathcal{I}^{(i)})_{i=1,2,\ldots,d}=(I_{0}^{(i)}\sqcup\cdots \sqcup I_{r_{i}}^{(i)})_{i=1,2,\ldots,d}\in \prod_{i=1}^{d}\mathcal{P}_{[k_{i}+1]}$.

In particular for the finest partitions 
$(\{0\}\sqcup\cdots \sqcup \{k_{i}\})_{i=1,\ldots,d}\in \prod_{i=1}^{d}\mathcal{P}_{[k_{i}+1]}$,
the corresponding unfolding \index[cN]{$\mathbf{S}^{\mathrm{reg}}$}\small
\begin{align*}
	&\mathbf{S}^{\mathrm{reg}}:=\mathbf{S}^{(\{0\}\sqcup\cdots \{k_{i}\})_{i=1,\ldots,d}}=\\
	&\ \left((\Pi^{(1)}_{0};[J^{(1)}_{0}]),(\Pi^{(1)}_{1};[0]),\ldots,(\Pi^{(1)}_{k_{1}};[0]),\ldots,(\Pi^{(d)}_{0};[J^{(d)}_{0}]),(\Pi^{(d)}_{1};[0]),\ldots,(\Pi^{(d)}_{k_{d}};[0])\right)
\end{align*}\normalsize
of $
\mathbf{S}$
which is a collection of spectral types for regular singular canonical forms 
is called 
the {\em finest unfolding} of $\mathbf{S}$.

Let us introduce the unfolding diagram of a collection $\mathbf{S}$
of spectral types.
Recall that  $\prod_{i=1}^{d}\mathcal{P}_{[k_{i}+1]}$ becomes a product poset, i.e.,
for $(\mathcal{I}^{(i)})_{i=1,\ldots,d},(\mathcal{J}^{(i)})_{i=1,\ldots,d}\in \prod_{i=1}^{d}\mathcal{P}_{[k_{i}+1]}$,
we have $(\mathcal{I}^{(i)})_{i=1,\ldots,d}\le (\mathcal{J}^{(i)})_{i=1,\ldots,d}$ 
if and only if $\mathcal{I}^{(i)}\le \mathcal{J}^{(i)}$ for all $i=1,\ldots,d.$
Thus we obtain the Hasse diagram of the poset $\prod_{i=1}^{d}\mathcal{P}_{[k_{i}+1]}$.
Then we attach to each vertex $(\mathcal{I}^{(i)})_{i=1,\ldots,d}\in \prod_{i=1}^{d}\mathcal{P}_{[k_{i}+1]}$
the corresponding unfolding $\mathbf{S}^{(\mathcal{I}^{(i)})_{i=1,2,\ldots,d}}$
of $\mathbf{S}$.
Then we call the resulting diagram the {\em unfolding diagram} of $\mathbf{S}$.

For example, let us recall 
the spectral type of the triconfluent Heun-type equation 
$(\emptyset\supset \emptyset\supset\emptyset\supset\emptyset;[0])$.
Then we obtain the following unfolding diagram.\vspace{3mm}\\
\tiny
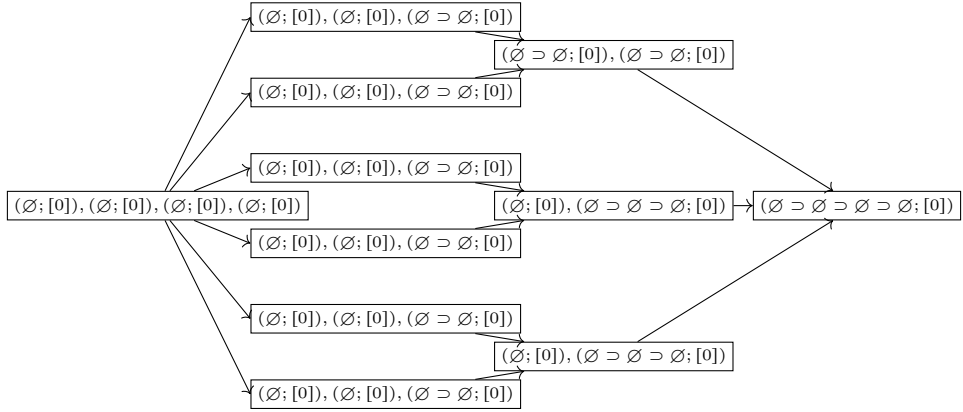
\begin{center}
    \begin{tikzpicture}[auto]
		\node[shape=rectangle, draw] (a1) at (-7, 0) {$(\emptyset;[0]),(\emptyset;[0]),(\emptyset;[0]),(\emptyset;[0])$};
		\node[shape=rectangle, draw] (b1) at (-4, -2.5) {$(\emptyset;[0]),(\emptyset;[0]),(\emptyset\supset\emptyset;[0])$};
		\node[shape=rectangle, draw] (b2) at (-4, -1.5) {$(\emptyset;[0]),(\emptyset;[0]),(\emptyset\supset\emptyset;[0])$};
		\node[shape=rectangle, draw] (b3) at (-4, -0.5) {$(\emptyset;[0]),(\emptyset;[0]),(\emptyset\supset\emptyset;[0])$};
		\node[shape=rectangle, draw] (b4) at (-4, 0.5) {$(\emptyset;[0]),(\emptyset;[0]),(\emptyset\supset\emptyset;[0])$};
		\node[shape=rectangle, draw] (b5) at (-4, 1.5) {$(\emptyset;[0]),(\emptyset;[0]),(\emptyset\supset\emptyset;[0])$};
		\node[shape=rectangle, draw] (b6) at (-4, 2.5) {$(\emptyset;[0]),(\emptyset;[0]),(\emptyset\supset\emptyset;[0])$};
        \node[shape=rectangle, draw] (c1) at (-1, 0) {$(\emptyset;[0]),(\emptyset\supset\emptyset\supset\emptyset;[0])$};
        \node[shape=rectangle, draw] (c2) at (-1, 2) {$(\emptyset\supset \emptyset ;[0]),(\emptyset\supset\emptyset;[0])$};
        \node[shape=rectangle, draw] (c3) at (-1, -2) {$(\emptyset;[0]),(\emptyset\supset\emptyset\supset\emptyset;[0])$};
        \node[shape=rectangle, draw] (e) at (2.2, 0) {$(\emptyset\supset \emptyset\supset\emptyset\supset\emptyset;[0])$};
		\draw[->] (a1) to (node cs:name=b1,anchor=west);
		\draw[->] (a1) to (node cs:name=b2,anchor=west);
		\draw[->] (a1) to (node cs:name=b3,anchor=west);
		\draw[->] (a1) to (node cs:name=b4,anchor=west);
		\draw[->] (a1) to (node cs:name=b5,anchor=west);
		\draw[->] (a1) to (node cs:name=b6,anchor=west);
		\draw[->] (b1) to (c3);
		\draw[->] (b2) to (c3);
		\draw[->] (b3) to (c1);
		\draw[->] (b4) to (c1);
		\draw[->] (b5) to (c2);
		\draw[->] (b6) to (c2);
		\draw[->] (c1) to (e);
        \draw[->] (c2) to (e);
		\draw[->] (c3) to (e);
\end{tikzpicture}
\end{center}
\normalsize\vspace{3mm}

\noindent
In this diagram, we notice that same collections of spectral types appear in several vertices.
This implies that we can reduce the diagram to a smaller one as follows.

Let $\mathbf{S}=(S_{i})_{i=1,2,\ldots,d}$ be a collection 
of abstract spectral types as above.
Then there exist 
positive integers $0\le l^{(i)}_{1}<l^{(i)}_{2}<\cdots <l^{(i)}_{t_{i}}=k_{i}$
such that 
\[
	\Pi^{(i)}_{k_{i}}=\cdots =\Pi^{(i)}_{l^{(i)}_{t_{i}-1}+1}\supsetneq 
	\cdots \supsetneq\Pi^{(i)}_{l^{(i)}_{2}}\cdots=\Pi^{(i)}_{l^{(i)}_{1}+1}
	\supsetneq\Pi^{(i)}_{l^{(i)}_{1}}=\cdots=\Pi^{(i)}_{0}
\]
for $i=1,2,\ldots,d$.
Let us consider the 
subgroup $\mathfrak{S}_{l^{(i)}_{0}+1}\times \mathfrak{S}_{l^{(i)}_{1}-l^{(i)}_{0}}
\times \cdots \mathfrak{S}_{l^{(i)}_{t_{i}}-l^{(i)}_{t_{i}-1}}$
of the symmetric group $\mathfrak{S}_{k_{i}+1}$
of degree $k_{i}+1$
which acts on $\{0,1,\ldots,k_{i}\}$
as the permutation.
Then we can define quotient 
posets 
\[
\mathcal{P}_{[k_{i}+1]}^{S_{i}}:=\mathcal{P}_{[k_{i}+1]}/(\mathfrak{S}_{l^{(i)}_{0}+1}\times \mathfrak{S}_{l^{(i)}_{1}-l^{(i)}_{0}}
\times \cdots \mathfrak{S}_{l^{(i)}_{t_{i}}-l^{(i)}_{t_{i}-1}})
\]
and the product $\prod_{i=1}^{d}\mathcal{P}_{[k_{i}+1]}^{S_{i}}$ of them. 
If  $(\mathcal{I}_{i})_{i=1,\ldots,d}, (\mathcal{J}_{i})_{i=1,\ldots,d}
\in \prod_{i=1}^{d}\mathcal{P}_{[k_{i}+1]}$
are equal in the quotient $\prod_{i=1}^{d}\mathcal{P}_{[k_{i}+1]}^{S_{i}}$,
then we have $\mathbf{S}^{(\mathcal{I}_{i})_{i=1,\ldots,d}}=\mathbf{S}^{(\mathcal{J}_{i})_{i=1,\ldots,d}}.$
Thus the following diagram is well-defined.
Namely, let us consider the Hasse diagram 
of the poset $\prod_{i=1}^{d}\mathcal{P}_{[k_{i}+1]}^{S_{i}}$
and attach to the each vertex 
$[(\mathcal{I}_{i})_{i=1,\ldots,d}]$
the unfolding $\mathbf{S}^{(\mathcal{I}_{i})_{i=1,\ldots,d}}$
of $\mathbf{S}$.
We call this diagram the {\em reduced unfolding diagram} of $\mathbf{S}$.

For the above spactral type $(\emptyset\supset \emptyset\supset\emptyset\supset\emptyset;[0])$,
the reduced unfolding diagram is drawn as follows.\vspace{3mm}\\
\tiny
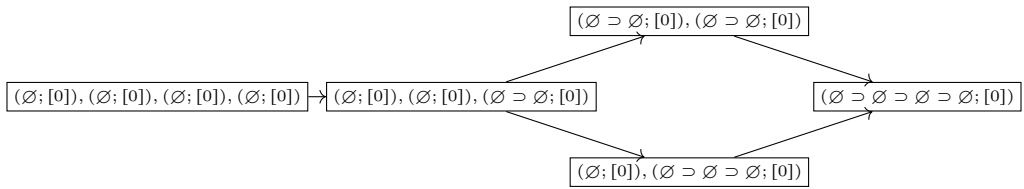
\begin{center}
	\begin{tikzpicture}[auto]
		\node[shape=rectangle, draw] (a) at (-8, 0) {$(\emptyset;[0]),(\emptyset;[0]),(\emptyset;[0]),(\emptyset;[0])$};
		\node[shape=rectangle, draw] (b) at (-4, 0) {$(\emptyset;[0]),(\emptyset;[0]),(\emptyset\supset\emptyset;[0])$};
		\node[shape=rectangle, draw] (c) at (-1, 1) {$(\emptyset\supset \emptyset ;[0]),(\emptyset\supset\emptyset;[0])$};
		\node[shape=rectangle, draw] (d) at (-1, -1) {$(\emptyset;[0]),(\emptyset\supset\emptyset\supset\emptyset;[0])$};
		\node[shape=rectangle, draw] (e) at (1.5, 0) {$(\emptyset\supset \emptyset\supset\emptyset\supset\emptyset;[0])$};
		\draw[->] (a) to (b);
		\draw[->] (b) to (c);
		\draw[->] (b) to (d);
		\draw[->] (c) to (e);
		\draw[->] (d) to (e);
	\end{tikzpicture}
\end{center}
\normalsize\vspace{3mm}
	
Also consider the spectral type 
$
\left(\{e_{12},e_{34}\}\supset\{e_{12},e_{34}\}\supset\{e_{12},e_{34}\}\supset \{e_{12}\};[0]\right)
$
of 
the triconfluent  equation in Section \ref{sec:exII}.
Then we obtain the reduced unfolding diagram,\vspace{3mm}\\
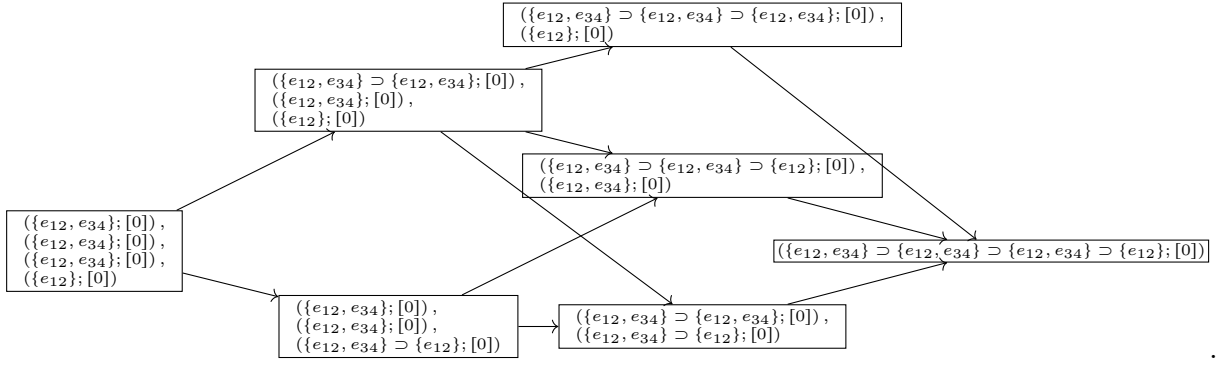
\begin{center}
    \begin{tikzpicture}[auto]
        \node[shape=rectangle, draw, font=\tiny, inner sep=1.2] (a) at (-8, 0.5) {$\begin{array}{l}\left(\{e_{12},e_{34}\};[0]\right),\\\left(\{e_{12},e_{34}\};[0]\right),\\ \left( \{e_{12},e_{34}\};[0]\right),\\\left(\{e_{12}\};[0]\right)\end{array}$};
        \node[shape=rectangle, draw, font=\tiny, inner sep=1.2] (b) at (-7, -1) {$\begin{array}{l}\left(\{e_{12},e_{34}\};[0]\right),\\ \left(\{e_{12},e_{34}\};[0]\right),\\ \left( \{e_{12},e_{34}\}\supset\{e_{12}\};[0]\right)\end{array}$};
        \node[shape=rectangle, draw, font=\tiny, inner sep=1.2] (c) at (-7, 2) {$\begin{array}{l}\left(\{e_{12},e_{34}\}\supset\{e_{12},e_{34}\};[0]\right),\\ \left(\{e_{12},e_{34}\};[0]\right),\\ \left(\{e_{12}\};[0]\right)\end{array}$};
        \node[shape=rectangle, draw, font=\tiny, inner sep=1.2] (d) at (-2.5, -1) {$\begin{array}{l}\left(\{e_{12},e_{34}\}\supset\{e_{12},e_{34}\};[0]\right),\\ \left( \{e_{12},e_{34}\}\supset\{e_{12}\};[0]\right)\end{array}$};
        \node[shape=rectangle, draw, font=\tiny, inner sep=1.2] (e) at (-2.3, 1) {$\begin{array}{l}\left(\{e_{12},e_{34}\}\supset\{e_{12},e_{34}\}\supset\{e_{12}\};[0]\right),\\ \left( \{e_{12},e_{34}\};[0]\right)\end{array}$};
        \node[shape=rectangle, draw, font=\tiny, inner sep=1.2] (f) at (-2.3, 3) {$\begin{array}{l}\left(\{e_{12},e_{34}\}\supset\{e_{12},e_{34}\}\supset\{e_{12},e_{34}\};[0]\right),\\ \left(\{e_{12}\};[0]\right)\end{array}$};
        \node[shape=rectangle, draw, font=\tiny, inner sep=1.2] (g) at (1.4, 0) {$\left(\{e_{12},e_{34}\}\supset\{e_{12},e_{34}\}\supset\{e_{12},e_{34}\}\supset \{e_{12}\};[0]\right)$};
        \draw[->] (a) to (b);
        \draw[->] (a) to (c);
        \draw[->] (b) to (d);
        \draw[->] (b) to (e);
        \draw[->] (c) to (d);
        \draw[->] (c) to (e);
        \draw[->] (c) to (f);
        \draw[->] (e) to (g);
        \draw[->] (f) to (g);
        \draw[->] (d) to (g);
\end{tikzpicture}.
\end{center}

\subsection{A conjecture by Oshima}
Let us explain a conjecture proposed by Oshima in the paper \cite{Oshi2}
which asks the existence 
of differential equations on $\mathbb{P}^{1}$
with a prescribed collection of spectral types.  
A collection 
\[
\mathbf{S}=(S_{i})_{i=1,2,\ldots,d}=
\left((\Pi^{(i)}_{k_{i}}\supset \Pi^{(i)}_{k_{i}-1}\supset\cdots \supset\Pi^{(i)}_{0};[J^{(i)}_{0}])\right)_{i=1,2,\ldots,d}
\]
of abstract spectral types are said to be {\em irreducibly realizable}
when there exists an irreducible  meromorphic connection on the trivial $G$-bundle 
over $\mathbb{P}^{1}$ with poles only at $\{a_{1},a_{2},\ldots,a_{d}\}\subset \mathbb{P}^{1}$,
which moreover has the spectral type $(\Pi^{(i)}_{k}\supset \Pi^{(i)}_{k-1}\supset\cdots \supset\Pi^{(i)}_{0};[J^{(i)}_{0}])$
at each $z=a_{i}$ for $i=1,2,\ldots,d$.
This can be restated as follows.
A  collection $\mathbf{S}=(S_{i})_{i=1,2,\ldots,d}$
of abstract spectral types is irreducibly realizable 
if and only if there exists a collection $\mathbf{H}=(H^{(a_{i})})_{i=1,2,\ldots,d}$
of unramified canonical forms 
satisfying that
\begin{itemize}
	\item $\mathrm{sp\,}(\mathbf{H})=\mathbf{S}$, where $\mathrm{sp\,}(\mathbf{H}):=(\mathrm{sp\,}(H^{(a_{i})}))_{i=1,2,\ldots,d}$,
	\item the additive Deligne-Simpson problem for $\mathbf{H}$ has a solution.
\end{itemize}
We call this solution $A\,dz$ a {\em realization} of $\mathbf{S}$.

For a collection $\mathbf{S}=(S_{i})_{i=1,2,\ldots,d}$ of 
spectral types, we can consider the 
reduced unfolding diagram of $\mathbf{S}$ as in the previous section.
Then we say that the unfolding diagram is {\em irreducibly realizable}
if there exists a  collection $\mathbf{H}=(H^{(a_{i})})_{i=1,2,\ldots,d}$
of unramified canonical forms 
satisfying 
\begin{itemize}
	\item there exists an open neighborhood $U$ of $\mathbf{0}\in \prod_{i=1}^{d}\mathbb{C}^{k_{i}+1}$
	such that   
	the deformation $\mathbf{H}(\mathbf{c})_{\mathbf{c}\in U}$
	realizes the unfolding diagram, namely, 
	for each $(\mathcal{I}^{(i)})_{i=1,\ldots,d}\in \prod_{i=1}^{d}\mathcal{P}_{[k_{i}+1]}$
	and $\mathbf{c}\in U\cap (\prod_{i=1}^{d}\mathcal{C}(\mathcal{I}^{(i)}))$,
	we have 
	\[
		\mathrm{sp}(\mathbf{H}(\mathbf{c}))=\mathbf{S}^{(\mathcal{I}^{(i)})_{i=1,\ldots,d}}.
	\]
	\item there exists an open neighborhood of $\mathbf{0}\in \mathbb{B}_{\mathbf{H}}$ and the additive Deligne-Simpson 
	problem for the family $(\mathbf{H}(\mathbf{c}))_{\mathbf{c}\in U}$ has a solution.
\end{itemize}
In \cite{Oshi2}, $\mathbf{S}$ is said to be {\em versally realizable}
if the unfolding diagram of $\mathbf{S}$ is irreducibly realizable.

Then as a corollary of Theorem \ref{thm:addDS} we obtain the following.
\begin{thm}\label{thm:specdiag}
	Let $\mathbf{S}=(S_{i})_{i=1,\ldots,d}$ be a collection of abstract spectral types.
	Suppose that $\mathbf{S}$ is irreducibly realizable 
	and let $\nabla_{A}=A\,dz$ be a realization of $\mathbf{S}$.
	Then there exists an open neighborhood $U$ of $\mathbf{0}\in \prod_{i=1}^{d}\mathbb{C}^{k_{i}+1}$
	and a holomorphic family of $G$-connections $\nabla_{A(\mathbf{c})}=A(\mathbf{c})\,dz$\ $(\mathbf{c}\in U)$ on $\mathbb{P}^{1}$
	such that 
	\begin{itemize}
	\item $\nabla_{A(\mathbf{0})}=\nabla_{A}$,
	\item for each $(\mathcal{I}_{i})_{i=1,\ldots,d}\in \prod_{i=1}^{d}\mathcal{P}_{[k_{i}+1]}$
	and $\mathbf{c}\in U\cap \prod_{i=1}^{d}\mathcal{C}(\mathcal{I}^{(i)})$,
	$\nabla_{A(\mathbf{c})}=A(\mathbf{c})\,dz$ is a realization of 
	$\mathbf{S}^{(\mathcal{I}_{i})_{i=1,\ldots,d}}$.
	\end{itemize}
	In particular, the following are equivalent.
	\begin{enumerate}
	\item The collection $\mathbf{S}$ of abstract spectral types is irreducibly 
	realizable.
	\item For every $(\mathcal{I}^{(i)})_{i=1,\ldots,d}\in \prod_{i=1}^{d}\mathcal{P}_{[k_{i}+1]}$,
	the collection $\mathbf{S}^{(\mathcal{I}^{(i)})_{i=1,\ldots,d}}$
	is irreducibly realizable.
	\item The unfolding diagram of $\mathbf{S}$ is irreducibly realizable.
	\end{enumerate}
\end{thm}
\begin{proof}
	Conditions $2$ and $3$ are obviously equivalent.
	From $2$ or $3$, $1$ immediately follows.
	Suppose $1$ holds. Then there exists a collection 
	$\mathbf{H}=(H^{(i)})_{i=1,\ldots,d}$ of unramified canonical forms 
	such that $\mathrm{sp\,}(\mathbf{H})=\mathbf{S}$
	and the additive Deligne-Simpson problem 
	for $\mathbf{H}$ admits a solution $\nabla_{A}=A\,dz$.
	Then Theorem \ref{thm:addDS}
	shows that there exists an open neighborhood $U$ of $\mathbf{0}\in \prod_{i=1}^{d}\mathbb{C}^{k_{i}+1}$
    and a holomorphic family of $G$-connections $\nabla_{A(\mathbf{c})}=A(\mathbf{c})\,dz$\ $(\mathbf{c}\in U)$ on $\mathbb{P}^{1}$
    such that 
    $\nabla_{A(\mathbf{0})}=\nabla_{A}$, and
    $\nabla_{A(\mathbf{c})}=A(\mathbf{c})\,dz$ is a solution to
    the additive Deligne-Simpson problem for the family $(\mathbf{H}(\mathbf{c}))_{\mathbf{c}\in U}$.
	Let $\mathbb{B}_{\mathbf{H}}=\prod_{i=1}^{d}\mathbb{B}_{H^{(i)}}$
	be open neighborhood of $\mathbf{0}\in \prod_{i=1}^{d}\mathbb{C}^{k_{i}+1}$
	defined in Section \ref{sec:riginv}
	and set
	$V:=U\cap \mathbb{B}_{\mathbf{H}}\neq \emptyset.$
	Then since $V$ is an open neighborhood of $\mathbf{0}\in \prod_{i=1}^{d}\mathbb{C}^{k_{i}+1}$
	Lemma \ref{lem:zeronbd} assures that 
	$V\cap \prod_{i=1}^{d}\mathcal{C}(\mathcal{I}^{(i)})\neq \emptyset$
	for any $(\mathcal{I}^{(i)})_{i=1,\ldots,d}\in \prod_{i=1}^{d}\mathcal{P}_{[k_{i}+1]}$.
	By Corollary \ref{cor:specdecomp}
	we have 
	$\mathrm{sp}(\mathbf{H}(\mathbf{c}))=\mathbf{S}^{(\mathcal{I}^{(i)})_{i=1,\ldots,d}}$
	for all  $\mathbf{c}\in V\cap \prod_{i=1}^{n}\mathcal{C}(\mathcal{I}^{(i)})$ and 
	$(\mathcal{I}^{(i)})_{i=1,\ldots,d}\in \prod_{i=1}^{d}\mathcal{P}_{[k_{i}+1]}$.
	Namely, the unfolding diagram of $\mathbf{S}$ is realizable.
\end{proof}

Now we are ready to state the conjecture by Oshima.
\begin{conj}[Oshima \cite{Oshi2}]\label{conj:osh}\normalfont
	Let $\mathbf{S}=(S_{i})_{i=1,\ldots,d}$
	 be a 
	 collection of abstract spectral types.
	Then the following are equivalent.
	\begin{enumerate}
		\item The collection $\mathbf{S}$ is irreducibly realizable.
		\item The collection $\mathbf{S}$ is versally realizable.
		\item The finest unfolding $\mathbf{S}^{\mathrm{reg}}$ of $\mathbf{S}$ is irreducibly realizable.
	\end{enumerate}
\end{conj}

As a direct consequence of Theorem \ref{thm:specdiag},
we obtain the following.
\begin{cor}\label{cor:conjosh}
	In Conjecture \ref{conj:osh}, $1$ and $2$ are equivalent and 
	the implication $1 \Rightarrow 3$ holds true.
\end{cor}
\begin{proof}
	The first assertion directly follows from Theorem \ref{thm:specdiag}.
	The direction $2 \Rightarrow 3$ holds obviously.
	Thus $1 \Rightarrow 3$ is also true.
\end{proof}

\section[Moduli spaces of meromorphic connections on $\mathbb{P}^{1}$]{Moduli spaces of meromorphic connections on $\mathbb{P}^{1}$ with unramified irregular singularities}\label{sec:defmoduli}
We recall the definition of moduli spaces of algebraic meromorphic connections
on a trivial bundle on the projective line $\mathbb{P}^{1}=\mathbb{P}^{1}(\mathbb{C})$ with
unramified irregular singularities.

\subsection{Lie groupoids, Orbifolds, and symplectic stratified spaces}\label{sec:symstra}
Before giving the definition of moduli spaces of meromorphic connections,
we recall some notions of Lie groupoids, orbifolds, and symplectic stratified spaces,
see \cite{Mack} by Mackenzie and its references 
for more detail.
\begin{df}[Complex Lie groupoid]\label{df:groupoid}\normalfont
	Let us consier a complex manifold $M$ and its submanifold $M_{0}$
	with the inclusion map $i\colon M_{0}\hookrightarrow M$. 
	Let us also consider holomorphic surjective submersions $s,t\colon M\rightarrow M_{0}$
	satisfying $t\circ i=s\circ i=\mathrm{id}_{M_{0}}$.
	Further, we consider a holomorphic map \index[cN]{$M^{(2)}$}
	$
	m\colon M^{(2)}:=\{(g_{1},g_{2})\in M\times M\mid s(g_{1})=t(g_{2})\}
	\rightarrow M	
	$ satisfying $s\circ m(g_{1},g_{2})=s(g_{2})$ and $t\circ m(g_{1},g_{2})=t(g_{1})$.
	Then if the following conditions are satisfied, the tuple 
	$(M,M_{0},s,t,m)$ is called a {\em complex Lie groupoid}.
	\begin{enumerate}
		\item \textbf{Asssociativity:} $m(m(g_{1},g_{2}),g_{3})=m(g_{1},m(g_{2},g_{3}))$ for all $(g_{1},g_{2},g_{3})\in M^{3}$
		satisfying $s(g_{1})=t(g_{2})$ and $s(g_{2})=t(g_{3})$.
		\item \textbf{Units:} $m(i\circ t(g),g)=g=m(g,i\circ s(g))$ for all $g\in M$.
		\item \textbf{Inverces:} For all $g\in M$, there exists $h\in M$ such that 
		$s(h)=t(g)$, $t(h)=s(g)$ and $m(g,h), m(h,g)\in M_{0}$.
	\end{enumerate}
	Here the maps $s$ and $t$ are called the {\em source map} and {\em target map}
	respectively, and $m$ is called the {\em multiplication map}. Also elements of the submanifold $i\colon M_{0}
	\hookrightarrow M$ are called {\em units}. 
\end{df}
\begin{df}[Bisection]\normalfont
	Let $M=(M,M_{0},s,t,m)$ be a Lie groupoid.
	A {\em bisection} of $M$ is a submanifold $S\subset M$ 
	such that both $s$ and $t$ restrict to biholomorphic map 
	$S\rightarrow M_{0}$.
\end{df}

Let us recall the notion of orbifold groupoids.
We refer \cite{ALR} by Adem-Leida-Ruan, for more details.
\begin{df}[Orbit space of a Lie groupoid]\normalfont
Let $M=(M,M_{0},s,t,m)$ be a complex Lie groupoid. 
For $m\in M_{0}$, the {\em orbit} of $m$ is the subset of $M_{0}$ defined by  
\[
	M(m):=\{t(g)\in M_{0}\mid s(g)=m \}=t(s^{-1}(m)).
\]
Then we can define the equivalent relation $m_{1}\sim m_{2}$ 
on $M_{0}$ by $m_{1}\in M(m_{2})$.
The quotient space by this equivalent relation 
$|M|:=M_{0}/\sim$
is called the {\em orbit space} of $M$.
\end{df}

As a typical example of Lie groupoids, we introduce action groupoids.
\begin{df}[Action groupoids]\normalfont
Let $X$ be a complex manifold and $H$ a complex Lie group acting 
holomorphically on $X$ from the left.
Then let $M:=H\times X$ and $M_{0}=X$. We define 
the source map $s\colon H\times X\rightarrow X$ 
by the second projection $\mathrm{pr}_{2}$
and the target map by  
\[
	t\colon H\times X\ni (h,x)\longmapsto h\cdot x\in X.
\]
Also define the inclusion by $i\colon X\ni x\mapsto (e,x)\in H\times X$.
Then the multiplication map 
\[
	m\colon M^{(2)}\ni ((h_{1},x),(h_{2},h_{2}^{-1}\cdot x))\longmapsto (h_{1}h_{2},h_{2}^{-1}\cdot x)\in M
\]
defines the groupoid structure on $(H\times X, X,s,t,m)$.
We call this Lie groupoid an {\em action groupoid}. 
In this case, the orbit space of the action groupoid
coincides with the usual orbit space $H\,\backslash X$. 
\end{df}
Let us recall the proper, \'etale and foliation groupoids.
\begin{df}[Proper, \'etale and foliation groupoids]\normalfont
Let us take a Lie groupoid $M=(M,M_{0},s,t,m)$.
Then $M$ is called a 
{\em proper groupoid} if 
the product map $s\times t\colon M\rightarrow M_{0}\times M_{0}$ is a proper map.
If both $s$ and $t$ are locally homeomorphisms, then $M$ is called an 
{\em \'etale groupoid}.
If {\em isotropy groups} $M_{m}:=s^{-1}(m)\cap t^{-1}(m)$
are discrete for all $m\in M_{0}$,
then $M$ is called a {\em foliation groupoid}.
\end{df} 
Then we can define orbifold groupoids.
\begin{df}[Orbifold groupoids]\normalfont
	We call a proper \'etale foliation groupoid 
	an {\em orbifold groupoid}.
\end{df}
If a Lie group $H$ acts on a complex manifold $X$ almost freely and properly,
then one can check that 
the action groupoid $H\times X$ is an example of orbifold groupoids.

\begin{df}[Orbifolds]\normalfont
	An {\em orbifold structure} on a paracompact Hausdorff space $X$ is given by an orbifold groupoid $M$ and a homeomorphism
	$|M|\rightarrow X$.
	Equivalent classes of orbifold structures are defined by Morita equivalence, see \cite{ALR}.
	An {\em orbifold} $\mathcal{X}$ is a paracompact Hausdorff space with an equivalence class of orbifold structures.
\end{df}
Therefore 
if $H$ acts on $X$ almost freely and properly as above,
the quotient space $H\,\backslash X$ becomes 
an orbifold with the orbifold structure coming from the 
action groupoid $H\times X$.

Next let us recall symplectic stratified spaces and 
symplectic reductions.
Firstly, according to the paper \cite{GM} by Goresky and MacPherson, recall topological stratified spaces.
\begin{df}[Stratified space]\normalfont
	Let $X$ be a $\mathcal{P}$-decomposed space with the pieces $(S_{i})_{i\in \mathcal{P}}$.
	Then $X$ is called a {\em stratified space} if the pieces 
	of $X$, called {\em strata}, satisfies the following 
	{\em local normal triviality}.
	Namely, for each point $x\in S_{i}$,
	there is a compact stratified space $L$,
	called the {\em link} of $x$, and a homeomorphism $h$
	of an open neighborhood $U$ of $x$ in X
	on the product $B\times \mathring{c}L$.
	Here $B$ is an open ball in $S_{i}$
	and $\mathring{c}L$ is the open cone 
	$L\times [0,1)/L\times \{0\}$ over $L$.
	Moreover $h$ preserves the decomposition.
\end{df}
As it is also well-known, 
a quotient orbifold $H\,\backslash X$ 
has the following natural stratification.
For a subgroup $\Gamma\subset H$,
we consider the subset $X_{(\Gamma)}$ of $X$
consisting of all points whose stabilizer group is conjugate 
to $\Gamma$.
Then the collection of the quotient spaces of 
these subset $(H\,\backslash X_{(\Gamma)})_{\Gamma\subset H}$
defines $H\,\backslash X$ as a stratified space.

The quotient orbifold $H\,\backslash X$ 
has the following natural holomorphic structure
\[
	\mathcal{O}_{H\,\backslash X}(U):=
	\{
		f\in C^{0}(U)\mid f\circ \pi|_{\pi^{-1}(U)}\text{ is holomorphic}
	\}
\]
for open subsets $U\subset H\,\backslash X$,
where $\pi \colon X\rightarrow H\,\backslash X$ is the quotient map.
In general,
for a stratified space $X$
with the strata $(S_{i})_{i\in \mathcal{I}}$
which are complex manifolds,
a {\em holomorphic structure} $\mathcal{O}_{X}$
is the sheaf 
generated by the correspondence
\[
	\mathcal{O}_{X}(U):=\{
		f\in C^{0}(U)\mid f|_{S_{i}\cap U} \text{ is holomorphic for every strata } S_{i}
	\}
\]
for open subsets $U\subset X$.

\begin{df}[Stratified holomorphic symplectic space]\normalfont
	Let $X$ be a stratified space with a holomorphic structure 
	$\mathcal{O}_{X}$.
	Then $X$ is called a {\em stratified holomorphic symplectic space}
	if the following are satisfied,
	\begin{enumerate}
		\item each stratum $S_{i}$ is a holomorphic symplectic manifolds,
		\item $\mathcal{O}_{X}$ is a sheaf of Poisson algebras,
		\item the embeddings $S_{i}\hookrightarrow X$ are Poisson.
	\end{enumerate}
\end{df}
Then the following is 
a generalization 
of the symplectic reduction theory by Marsden and Weinstein \cite{MarsWein}
to stratified spaces given by Sjammar and Lerman \cite{SL}.
\begin{thm}[Marsden-Weinstein and Sjamaar-Lerman]\label{thm:symplred}
Let $X$ be a holomorphic symplectic manifold on which 
a complex Lie group $H$ acts properly and almost freely.
Let $\mu\colon X\rightarrow \mathfrak{h}^{*}$
be a moment map.
Then the quotient space $H\,\backslash\,\mu^{-1}(0)$ 
becomes an orbifold whose canonical stratification 
gives the structure of stratified holomorphic symplectic space.
\end{thm}
\begin{rem}\normalfont
	This is a consequence of Theorem 2.1 in \cite{SL}
	in which the Lie group $H$ is assumed to be compact.
	However as it is stated in the introduction 
	of the same paper \cite{SL}, Theorem 2.1 is valid for proper actions of arbitrary Lie groups.
\end{rem}

\subsection{Moduli spaces of meromorphic connections on a trivial bundle over $\mathbb{P}^{1}(\mathbb{C})$}\label{sec:defmodl}
Let  $D=\sum_{a\in \{a_{1},a_{2},\ldots,a_{d}\}}(k_{a}+1)\cdot a$ be an effective divisor 
on $\mathbb{P}^{1}$. Under the projective transformation, we may assume  
$|D|\subset \mathbb{C}$ for simplicity.
Let us consider the space of meromorphic connections on the trivial $G$-bundles on $\mathbb{P}^{1}$\index[cN]{$\overline{\mathcal{M}}_{D}$}
\[
	\overline{\mathcal{M}}_{D}:=\left\{\nabla_{A}=A\,dz\,\middle|\, A\,dz\in \varOmega^{\mathfrak{g}}_{\mathbb{P}^{1},D}(\mathbb{P}^{1})\right\}.  
\]
Since each $A\,dz\in \overline{\mathcal{M}}_{D}$ is of the form 
\[
	A\,dz=\sum_{a\in |D|}\sum_{i=0}^{k_{a}}\frac{A^{(a)}_{i}}{(z-a)^{i}}\frac{dz}{z-a}	\quad (A^{(a)}_{i}\in \mathfrak{g})
\]
with 
 the relation 
$
	\sum_{a\in |D|}A^{(a)}_{0}=0,	
$
we can identify $\overline{\mathcal{M}}_{D}$ with the space 
\[
	\left\{\left(X^{(a)}_{i}\right)_{\substack{a\in |D|,\\ i=0,1,\ldots,k_{a}}}\in \prod_{a\in |D|}\mathfrak{g}^{k_{a}+1}\,\middle|\, 
	\sum_{a\in |D|}X^{(a)}_{0}=0\right\}.
\]	
The $G$-action on $\overline{\mathcal{M}}_{D}$ by the gauge transformation 
is translated into the diagonal adjoint
action, i.e.,
\[
	G\times \overline{\mathcal{M}}_{D}\ni \left(g, \left(X^{(a)}_{i}\right)_{\substack{a\in |D|,\\ i=0,1,\ldots,k_{a}}}\right)
	\longmapsto \left(\mathrm{Ad}(g)(X^{(a)}_{i})\right)_{\substack{a\in |D|,\\ i=0,1,\ldots,k_{a}}}\in \overline{\mathcal{M}}_{D}.
\]
Under this identification, we regard $\overline{\mathcal{M}}_{D}$ as a complex manifold with the algebraic $G$-action.

Then by following the paper \cite{Boa1} by Boalch, we define the moduli 
space of meromorphic $G$-connections on $\mathbb{P}^{1}$ with given 
local canonical forms, see also 
\cite{HY} by Hiroe-Yamakawa and \cite{Yam1} by Yamakawa,
and also see \cite{BS} by Bremer-Sage for ramified cases.
\begin{df}[Moduli spaces of meromorphic $G$-connections]\label{df:irredmod}\normalfont
	For a collection of canonical forms $\mathbf{H}=(H^{(a)})_{a\in |D|}\in \prod_{a\in |D|}\mathfrak{g}(\mathbb{C}[z_{a}^{-1}]_{k_{a}})$
	with $\sum_{a\in |D|}H^{(a)}_{\mathrm{res}}\in \mathfrak{g}_{\mathrm{ss}}$,
	we consider the quotient space\index[cN]{$\mathcal{M}_{\mathbf{H}}$}
	\begin{align*}
		&\mathcal{M}_{\mathbf{H}}:=\\
		&\quad G\,\Bigg\backslash \left\{
			\nabla_{A}=\sum_{a\in |D|}\sum_{i=0}^{k_{a}}\frac{A^{(a)}_{i}}{(z-a)^{i}}\frac{dz}{z-a}\in \overline{\mathcal{M}}_{D}
			\,\middle|\, 
					\sum_{i=0}^{k_{a}}\frac{A^{(a)}_{i}}{(z-a)^{i+1}}\in \mathbb{O}_{H_{a}}, a\in |D|
			\right\}
	\end{align*}
	and  call this space 
	 the
	{\em moduli spaces of connections} associated with $\mathbf{H}$.
	We also consider the subspace \index[cN]{$\mathcal{M}^{ir}_{\mathbf{H}}$}$\mathcal{M}^{ir}_{\mathbf{H}}$ of $\mathcal{M}_{\mathbf{H}}$ consists of irreducible connections
	and call this subspace the {\em moduli spaces of irreducible connections} associated with $\mathbf{H}$.
\end{df}
Although the moduli space $\mathcal{M}_{\mathbf{H}}$
itself is a naive quotient space, 
we shall later explain that 
the subspace $\mathcal{M}^{ir}_{\mathbf{H}}$ of $\mathcal{M}_{\mathbf{H}}$ 
has 
a structure of holomorphic symplectic orbifold.

\subsection{Stability of $G/Z$-action on the space of irreducible connections}\label{sec:stabirred}
Let us consider \index[cN]{$(\prod_{i=1}^{n}\mathfrak{g})^{\text{ir}}$}
\[
	\left(\prod_{i=1}^{n}\mathfrak{g}\right)^{\text{ir}}:=\left\{ (X_{i})_{i=1,\ldots,n}\in  \prod_{i=1}^{n}\mathfrak{g}\,\middle|\,(X_{i})_{i=1,\ldots,n}\text{ is irreducible}\right\},
\]
the subspace of $\prod_{i=1}^{n}\mathfrak{g}$
consisting of all irreducible elements,
where $(X_{i})_{i=1,\ldots,n}\in \prod_{i=1}^{n}\mathfrak{g}$
is said to be {\em irreducible} if there exists no proper parabolic subalgebra of $\mathfrak{g}$
containing all $X_{i}$ for $i=1,\ldots,n$.

Since the center $Z$ of $G$ acts trivially on $\prod_{i=1}^{n}\mathfrak{g}$,
we can consider the action of $G/Z$ on $\prod_{i=1}^{n}\mathfrak{g}$.
Let us explain that $\left(\prod_{i=1}^{n}\mathfrak{g}\right)^{\text{ir}}$
coincides with the space of stable points in $\prod_{i=1}^{n}\mathfrak{g}$
with respect to the 
$G/Z$-action. 

We follow the formulations in \cite{Kem} by Kempf. 
Let $\lambda\colon \mathbb{G}_{m}\rightarrow G/Z$ be a one-parameter subgroup of the algebraic group $G/Z$.
Let $\chi$ be a character of the image of $\lambda$. Then we can 
write 
\[
	\chi(\lambda(t))=t^{\chi(\lambda)}\quad (t\in \mathbb{G}_{m})	
\]
by some integer $\chi(\lambda)$ which characterizes $\chi$.

Let $V$ be a $\mathbb{C}$-vector space with a linear $G/Z$-action. Then
we consider the weight decomposition $V=\oplus V^{\chi}$
of $V$ where $\chi$ runs through the set of characters of the image of $\lambda$,
namely, $V^{\chi}$ are $\chi$-eigenspaces of $V$ under the action of $\mathbb{G}_{m}$ through $\lambda$.
Then we can associate the integer to each $v\in V$ defined by 
\[
	m(v,\lambda):=\mathrm{min}\{\chi(\lambda)\mid \text{projection of $v$ onto $V^\chi$ is non-zero}\}.
\]
Then we can see that 
the limit $\lim_{t\to 0}\lambda(t)\cdot v$ exists if and only if $m(v,\lambda)\ge 0$.

On the other hand, to a one-parameter subgroup $\lambda$, we can 
associate the parabolic subalgebra of $\mathfrak{g}_{\mathrm{ss}}=\mathrm{Lie}(G/Z)$ defined by 
\[
	\mathfrak{p}(\lambda):=\{X\in \mathfrak{g}_{\mathrm{ss}}\mid m(X,\lambda)\ge 0\}.	
\]
\begin{prop}\label{prop:HM}
	For 
	$X\in \prod_{i=1}^{n}\mathfrak{g}$, 
	$X$ is irreducible if and only if it is  
	stable under the action of $G/Z$, i.e., the $G/Z$-orbit of 
	$X$ is Zariski closed and the stabilizer $\mathrm{Stab}_{G/Z}(X)$
	of $X$ has the finite cardinality.
\end{prop}
\begin{proof}
Obviously any element $X$ in the center $\mathfrak{z}$ of $\mathfrak{g}$
satisfies $m(X,\lambda)=0$.
Thus by Proposition 8.4.5 in \cite{Sp},
any parabolic subalgebra of $\mathfrak{g}$ is of the form $\mathfrak{z}\oplus \mathfrak{p}(\lambda)$.
Therefore we can conclude that  
 $X\in \prod_{i=1}^{n}\mathfrak{g}$
is irreducible if and only if $X$ satisfies
$m(X,\lambda)<0$ for any one-parameter subgroup $\lambda$ of $G/Z$.
Then the desired result follows from the Hilbert-Mumford criterion (see Theorem 2.1 in \cite{MFK}).
\end{proof}
It follows from this
proposition that $\left(\prod_{i=1}^{n}\mathfrak{g}\right)^{\text{ir}}$
coincides with the space of stable points in $\prod_{i=1}^{n}\mathfrak{g}$
as desired.
In particular, $\left(\prod_{i=1}^{n}\mathfrak{g}\right)^{\text{ir}}$
is a Zariski open subset of $\prod_{i=1}^{n}\mathfrak{g}$.

\subsection{Symplectic orbifold $\mathcal{M}^{\mathrm{ir}}_{\mathbf{H}}$}\label{sec:symporb}
Let us consider the injective immersion\index[cN]{$\iota_{\mathbf{H}}$}
\[
	\iota_{\mathbf{H}}\colon \prod_{a\in |D|}\mathbb{O}_{H_{a}}\longrightarrow \prod_{a\in |D|}\mathfrak{g}(\mathbb{C}[z_{a}^{-1}]_{k_{a}})	
\]
defined 
as the product of the natural immersions\index[cN]{$\iota_{\mathbb{O}_{H_{a}}}$} 
\[
	\iota_{\mathbb{O}_{H_{a}}}\colon G(\mathbb{C}[z_{a}]_{k_{a}})/G(\mathbb{C}[z_{a}]_{k_{a}})_{H_{a}}\ni [g]
	\longmapsto \mathrm{Ad}^{*}(g)(H_{a})\in \mathfrak{g}(\mathbb{C}[z_{a}^{-1}]_{k_{a}})\quad (a\in |D|)
\]
under the identification $\mathbb{O}_{H_{a}}\cong G(\mathbb{C}[z_{a}]_{k_{a}})/G(\mathbb{C}[z_{a}]_{k_{a}})_{H_{a}}.$
We consider the open subset of $\prod_{a\in |D|}\mathbb{O}_{H_{a}}$ defined by\index[cN]{$(\prod_{a\in |D|}\mathbb{O}_{H_{a}})^{\text{ir}}$}
\[
	\left(\prod_{a\in |D|}\mathbb{O}_{H_{a}}\right)^{\text{ir}}:=\iota_{\mathbf{H}}^{-1}\left(\left(
		\prod_{a\in |D|}\mathfrak{g}(\mathbb{C}[z_{a}^{-1}]_{k_{a}})
	\right)^{\mathrm{ir}}\right).
\]

Under the identification $\mathfrak{g}(\mathbb{C}[z_{a}^{-1}]_{k_{a}})\cong \mathfrak{g}(\mathbb{C}[z_{a}]_{k_{a}})^{*}$
via the trace pairing,
the injective immersion \index[cN]{$\mu_{\mathbb{O}_{H_{a}}}$}
\[
	\mu_{\mathbb{O}_{H_{a}}}=\iota_{\mathbb{O}_{H_{a}}}\colon \mathbb{O}_{H_{a}} \longmapsto \mathfrak{g}(\mathbb{C}[z_{a}^{-1}]_{k_{a}})
\]
is considered as 
the moment map 
with respect to the coadjoint action of $G(\mathbb{C}[z]_{k})$.
Furthermore, since 
the residue map $\mathfrak{g}(\mathbb{C}[z_{a}^{-1}]_{k_{a}})
\ni X(z_{a})\mapsto \underset{z_{a}=0}{\mathrm{res\,}}X(z_{a})\in \mathfrak{g}$
is  
the dual map of the inclusion $\mathfrak{g}\hookrightarrow \mathfrak{g}(\mathbb{C}[z_{a}]_{k_{a}})$,
the map \index[cN]{$\mu_{\mathbb{O}_{H_{a}}\downarrow G}$}
\[
	\mu_{\mathbb{O}_{H_{a}}\downarrow G}:=\underset{z_{a}=0}{\mathrm{res\,}}\circ 
	\iota_{\mathbb{O}_{H_{a}}}\colon \mathbb{O}_{H_{a}}\longrightarrow \mathfrak{g}
\]
is the moment map with respect to the coadjoint action of $G$
under the identification $\mathfrak{g}\cong \mathfrak{g}^{*}$ via the trace pairing.
As the product of these moment maps, we obtain the moment map \index[cN]{$\mu_{\mathbf{H}}$}
\[
	\begin{array}{cccc}
		\mu_{\mathbf{H}}\colon &\prod_{a\in |D|}\mathbb{O}_{H_{a}}&\longrightarrow &\mathfrak{g}^{*}\\
		&(X_{a})_{a\in |D|}&\longmapsto &\sum_{a\in |D|}\mu_{\mathbb{O}_{H_{a}}\downarrow G}(X_{a})
	\end{array}
\]
with respect to the diagonal action of $G$.
Also we denote the restriction 
of this moment map $\mu_{\mathbf{H}}$
on $\left(\prod_{a\in |D|}\mathbb{O}_{H_{a}}\right)^{\text{ir}}$
by \index[cN]{$\mu_{\mathbf{H}}^{\mathrm{ir}}$}$\mu_{\mathbf{H}}^{\mathrm{ir}}$
which is still a moment map 
since $\left(\prod_{a\in |D|}\mathbb{O}_{H_{a}}\right)^{\text{ir}}$
is a open submanifold of $\prod_{a\in |D|}\mathbb{O}_{H_{a}}$
and closed under the $G$-action. 
Then
under the injective immersion $\iota_{\mathbf{H}}$,
we obtain the following identification,
\begin{align*}
	&\left\{
			\nabla_{A}=\sum_{a\in |D|}\sum_{i=0}^{k_{a}}\frac{A^{(a)}_{i}}{(z-a)^{i}}\frac{dz}{z-a}\in \overline{\mathcal{M}}_{D}
			\,\middle|\, \begin{array}{l}
				\nabla_{A}\text{ is irreducible},\\
					\sum_{i=0}^{k_{a}}\frac{A^{(a)}_{i}}{(z-a)^{i+1}}\in \mathbb{O}_{H_{a}}\text{ for all }a\in |D|
			\end{array}
			\right\}\\
	&=\left\{
			\left(\sum_{i=0}^{k_{a}}\frac{A^{(a)}_{i}}{(z-a)^{i+1}}\right)_{a\in |D|}\in \left(\prod_{a\in |D|}\mathbb{O}_{H_{a}}\right)^{\text{ir}}
			\,\middle|\,
			\sum_{a\in |D|}\underset{z_{a}=0}{\mathrm{res\,}}\sum_{i=0}^{k_{a}}\frac{A^{(a)}_{i}}{(z-a)^{i+1}}=0
	\right\}\\
	&=\left\{(\xi_{a})_{a\in |D|}\in \left(\prod_{a\in |D|}\mathbb{O}_{H_{a}}\right)^{\text{ir}}\,
	\middle|\, \mu_{\mathbf{H}}^{\mathrm{ir}}((\xi_{a})_{a\in |D|})=0\right\}
	=\left(\mu_{\mathbf{H}}^{\mathrm{ir}}\right)^{-1}(0).
\end{align*}
Namely,  
the moduli space $\mathcal{M}^{\text{ir}}_{\mathbf{H}}$
of irreducible connections is regarded as the symplectic reduction 
\[
	\mathcal{M}^{\text{ir}}_{\mathbf{H}}=G\,\backslash(\mu_{\mathbf{H}}^{\text{ir}})^{-1}(0).	
\]
Now let us look closely at the $G$-action on the level $0$ set $(\mu_{\mathbf{H}}^{\text{ir}})^{-1}(0)$.
\begin{lem}\label{lem:ss}
	Let us consider the decomposition $\mathfrak{g}^{*}
	=\mathfrak{z}^{*}\oplus \mathfrak{g}_{\mathrm{ss}}^{*}$
	induced by the 
	decomposition $\mathfrak{g}=\mathfrak{z}\oplus \mathfrak{g}_{\mathrm{ss}}$.
	Then we have $\mathrm{Im\,}\mu_{\mathbf{H}}\subset \mathfrak{g}_{\mathrm{ss}}^{*}$.
\end{lem}
\begin{proof}
	As in Proposition \ref{prop:resss},
	for $X\in \mathfrak{g}_{\mathrm{ss}}^{\bot}=\mathfrak{z}$ and $(g_{a})\in \prod_{a\in |D|}G(\mathbb{C}[z_{a}]_{k_{a}})$,
	we have 
	\begin{align*}
		&\mathrm{tr}(X\cdot \sum_{a\in |D|}\underset{z=a}{\mathrm{res\,}}(\mathrm{Ad}(g_{a})(H^{(a)})))\\
		&\ =
		\mathrm{tr}\sum_{a\in |D|}\underset{z=a}{\mathrm{res\,}}(X\cdot \mathrm{Ad}(g_{a})(H^{(a)}))
		=\mathrm{tr}\sum_{a\in |D|}\underset{z=a}{\mathrm{res\,}}(\mathrm{Ad}(g_{a}^{-1})(X)H^{(a)})\\
		&\ =\mathrm{tr}\sum_{a\in |D|}\underset{z=a}{\mathrm{res\,}}(X\cdot H^{(a)})
		=\mathrm{tr}(X \sum_{a\in |D|}\underset{z=a}{\mathrm{res\,}}(H^{(a)}))=0,
	\end{align*}
	since $\sum_{a\in |D|}\underset{z=a}{\mathrm{res\,}}(H^{(a)})\in \mathfrak{g}_{\mathrm{ss}}.$
	Thus $\sum_{a\in |D|}\mathrm{res}_{z=a}(\mathrm{Ad}(g_{a})(H_{a}))\in \mathfrak{g}_{\mathrm{ss}}$.
	Then the correspondence between the residue map and 
	the moment map $\mu_{\mathbf{H}}$
	show the result.
\end{proof}
\begin{lem}\label{lem:denpan}
	Let $L$ be a locally compact topological group. 
	Let us consider locally compact topological spaces $V$ and $W$ with continuous $L$-actions.
	Suppose that $L$ acts on $W$ properly and almost freely, and there exists a $L$-equivariant 
	continuous map $\phi\colon V\rightarrow W$. Then the $L$-action on $V$ is also 
	proper and almost free. 
	In particular if $L$-action on $W$ is free, so is that on $V$.
\end{lem}
\begin{proof}
	Let us consider the stabilizer group $\mathrm{Stab}_{L}(v)$ of $v\in V$.
	If $g\in \mathrm{Stab}_{L}(v)$, then $g\cdot\phi(v)=\phi(g\cdot v)=\phi(v)$. Thus 
	we have $\mathrm{Stab}_{L}(v)\subset \mathrm{Stab}_{L}(\phi(v))$.
	Then 
	since the $L$-action on $W$ is  almost free, $\mathrm{Stab}_{L}(v)$ is a finite set.

	Let us take a compact subset $K\subset V$ and consider  $L_{K}:=\{g\in L\mid g\cdot K\cap K\neq \emptyset\}$.
	Since $\phi$ is continuous, the image $\phi(K)\subset W$ is a compact subset.
	Let us take $g\in L_{K}$. Then we have $g\cdot \phi(K)\cap \phi(K)=\phi(g\cdot K)\cap \phi(K)
	\supset \phi(g\cdot K\cap K)\neq \emptyset$ since $g\cdot K\cap K\neq \emptyset$.
	Hence $L_{K}$ is a subset of $L_{\phi(K)}:=\{g\in G\mid g\cdot \phi(K)\cap \phi(K)\neq \emptyset\}$
	which has the compact closure. 
	Therefore $L_{K}$ has the compact closure as well.
\end{proof}

\begin{prop}\label{prop:orbifold}
	The moduli space $\mathcal{M}^{\text{ir}}_{\mathbf{H}}$
	has a structure 
	of complex symplectic orbifold of dimension 
	\[
		\mathcal{M}^{\text{ir}}_{\mathbf{H}}=2\,\mathrm{dim\,}Z-\mathrm{rig\,}(\mathbf{H}),
	\]
	if it is nonempty.

\end{prop}
\begin{proof}
	Let us identify $\mathfrak{g}^{*}\cong \mathfrak{g}$ via the trace pairing 
	and denote the dual map of the inclusion 
	$\iota_{\mathrm{ss}}\colon \mathfrak{g}_{\mathrm{ss}}\hookrightarrow \mathfrak{g}=\mathfrak{z}\oplus \mathfrak{g}_{\mathrm{ss}}$
	by $\iota_{\mathrm{ss}}^{*}\colon \mathfrak{g}^{*}\rightarrow \mathfrak{g}_{\mathrm{ss}}^{*}$.
	Then
	we obtain the moment map 
	\[
		\iota_{\mathrm{ss}}^{*}\circ \mu_{\mathbf{H}}^{\mathrm{ir}}\colon \left(\prod_{a\in |D|}\mathbb{O}_{H_{a}}\right)^{\mathrm{ir}}\rightarrow \mathfrak{g}_{\mathrm{ss}}^{*}
	\]
	of $G/Z$-action. 
	Since we know that $\iota_{\mathrm{ss}}^{*}\circ \mu_{\mathbf{H}}^{\mathrm{ir}}=\mu_{\mathbf{H}}^{\mathrm{ir}}$ by Lemma \ref{lem:ss}, 
	we obtain
	\[
		G\,\backslash(\mu_{\mathbf{H}}^{\mathrm{ir}})^{-1}(0)=(G/Z)\,\backslash(\iota_{\mathrm{ss}}^{*}\circ \mu_{\mathbf{H}}^{\mathrm{ir}})^{-1}(0).
	\]
	Thus it suffices to show that the right hand side has a structure 
	of complex symplectic orbifold.

	Proposition \ref{prop:HM} and Lemma \ref{lem:denpan} show that 
	$G/Z$ action on $\left(\prod_{a\in |D|}\mathbb{O}_{H_{a}}\right)^{\mathrm{ir}}$ is 
	proper and almost free. Therefore since $\iota_{\mathrm{ss}}^{*}\circ \mu_{\mathbf{H}}^{\mathrm{ir}}$
	is a moment map, the almost free-ness of $G/Z$-action 
	assures that $0$ is a regular value of $\iota_{\mathrm{ss}}^{*}\circ \mu_{\mathbf{H}}^{\mathrm{ir}}$.
	Namely,
	$ (\iota_{\mathrm{ss}}^{*}\circ \mu_{\mathbf{H}}^{\mathrm{ir}})^{-1}(0)$
	becomes a submanifold of $\left(\prod_{a\in |D|}\mathbb{O}_{H_{a}}\right)^{\mathrm{ir}}$.
	Thus 
	$(G/Z)\backslash(\iota_{\mathrm{ss}}^{*}\circ \mu_{\mathbf{H}}^{\mathrm{ir}})^{-1}(0)$
	is the quotient space of the complex manifold $(\iota_{\mathrm{ss}}^{*}\circ \mu_{\mathbf{H}}^{\mathrm{ir}})^{-1}(0)$
	by the proper and almost free action of $G/Z$, namely,
	it has the complex symplectic orbifold structure by Theorem \ref{thm:symplred}.

	Let us consider the dimension of $\mathcal{M}_{\mathbf{H}}^{\mathrm{ir}}$
	under the assumption $(\iota_{\mathrm{ss}}^{*}\circ \mu_{\mathbf{H}}^{\mathrm{ir}})^{-1}(0)\neq 
	\emptyset$.
	Since the $G/Z$-action is almost free, we have 
	\begin{align*}
		\mathrm{dim\,}\mathcal{M}_{\mathbf{H}}^{\mathrm{ir}}&=\mathrm{dim\,}(G/Z)\,\backslash(\iota_{\mathrm{ss}}^{*}\circ \mu_{\mathbf{H}}^{\mathrm{ir}})^{-1}(0)
		=\sum_{a\in |D|}\mathrm{dim\,}\mathbb{O}_{H^{(a)}}-2\,\mathrm{dim\,}G/Z\\
		&=\sum_{a\in |D|}\delta(H^{(a)})-2\,\mathrm{dim\,}G+2\,\mathrm{dim\,}Z
		=2\,\mathrm{dim\,}Z-\mathrm{rig}(\mathbf{H})
	\end{align*}
	as desired.
\end{proof}
\section{Triangular decompositions of  truncated orbits}\label{sec:triangle}
We shall recall the triangular decompositions of truncated orbits. 
This will play an important role to construct the deformation of 
truncated orbits in latter sections.
A similar decomposition 
can also be found in \cite{HY} by Hiroe-Yamakawa, \cite{H2} by Hiroe, and \cite{Yam} by Yamakawa.

Let us fix an unramified canonical form 
\[
	H=\left(
		\frac{H_{k}}{z^{k}}+\cdots +\frac{H_{1}}{z}+H_{\mathrm{res}}
	\right)\frac{1}{z}\in\mathfrak{g}(\mathbb{C}[z^{-1}]_{k}) 
\]
throughout this section.

\subsection{A fiber bundle 
with the fiber $\mathbb{O}^{L_{1}\ltimes G(\mathbb{C}[z]_{k})_{1}}_{H}$ 
and the truncated orbit}
For a positive integer $l$ with $0\le l\le k$ and a 
closed subgroup $S\subset G(\mathbb{C}[z]_{l})$,
we denote the $S$-orbit of $\xi \in \mathfrak{g}(\mathbb{C}[z]_{l})^{*}$ by\index[cN]{$\mathbb{O}_{\xi}^{S}$}
\[
	\mathbb{O}_{\xi}^{S}:=\mathrm{Ad}^{*}(S)(\xi).
\]
Now let us consider the orbit $\mathbb{O}_{H}^{L_{1}\ltimes G(\mathbb{C}[z]_{k})_{1}}$
of $H$ under the $L_{1}\ltimes G(\mathbb{C}[z]_{k})_{1}$-action.
Let us also consider the product manifold 
$
	G\times \mathbb{O}_{H}^{L_{1}\ltimes G(\mathbb{C}[z]_{k})_{1}}	
$
with the $L_{1}$-action defined by 
$h\cdot (g,\xi):=(gh^{-1},\mathrm{Ad}^{*}(h)(\xi))$
for $h\in L_{1}$, $g\in G$, and $\xi\in \mathbb{O}_{H}^{L_{1}\ltimes G(\mathbb{C}[z]_{k})_{1}}$.
Then the quotient space
\[
	G\times_{L_{1}}\mathbb{O}_{H}^{L_{1}\ltimes G(\mathbb{C}[z]_{k})_{1}}:=L_{1}\,\backslash(G\times \mathbb{O}_{H}^{L_{1}\ltimes G(\mathbb{C}[z]_{k})_{1}})	
\]
is the fiber bundle on $G/L_{1}$ with the fiber $\mathbb{O}_{H}^{L_{1}\ltimes G(\mathbb{C}[z]_{k})_{1}}$
which is associated with the principal $L_{1}$-bundle $G\rightarrow G/L_{1}$.
Then we shall show that  the total space 
$G\times_{L_{1}}\mathbb{O}_{H}^{L_{1}\ltimes G(\mathbb{C}[z]_{k})_{1}}$ is isomorphic 
to the truncated orbit $\mathbb{O}_{H}$.

Consider the multiplication map 
$\widetilde{\psi}\colon 
G\times (L_{1}\ltimes G(\mathbb{C}[z]_{k})_{1})\ni (g,(h,f))\mapsto 
ghf\in G(\mathbb{C}[z]_{k})$.
This is obviously $\mathrm{Stab}_{L_{1}\ltimes G(\mathbb{C}[z]_{k})_{1}}(H)$-equivariant 
under the multiplications from the right.
Therefore it induces the map 
\[
	\overline{\psi}\colon 	G\times \mathbb{O}_{H}^{L_{1}\ltimes G(\mathbb{C}[z]_{k})_{1}}
	\ni (g,\xi)\longmapsto \mathrm{Ad}^{*}(g)(\xi)\in \mathbb{O}_{H}.
\]
Moreover since this $\overline{\psi}$ is invariant under the $L_{1}$-action,
it factors through the map 
\[
	\psi\colon 	G\times_{L_{1}} \mathbb{O}_{H}^{L_{1}\ltimes G(\mathbb{C}[z]_{k})_{1}}
	\ni [(g,\xi)]\longmapsto \mathrm{Ad}^{*}(g)(\xi)\in \mathbb{O}_{H}
\] 
which does not depend on the choice of representatives of $[(g,\xi)]$.
Namely, we obtain 
the commutative diagram
\begin{equation}\label{eq::comglh}
	\begin{tikzcd}
		G\times (L_{1}\ltimes G(\mathbb{C}[z]_{k})_{1})
		\arrow[r,"\widetilde{\psi}"]
		\arrow[d,"\mathrm{id}_{G}\times \pi"]
		&G(\mathbb{C}[z]_{k})
		\arrow[d,"\pi"]\\
		G\times \mathbb{O}_{H}^{L_{1}\ltimes G(\mathbb{C}[z]_{k})_{1}}
		\arrow[r,"\overline{\psi}"]
		\arrow[d,"p"]
		&\mathbb{O}_{H}\\
		G\times_{L_{1}}\mathbb{O}_{H}^{L_{1}\ltimes G(\mathbb{C}[z]_{k})_{1}}
		\arrow[ru,"\psi"]&
	\end{tikzcd}.
\end{equation}
Here vertical arrows are natural projections.
\begin{prop}[cf. Proposition 4.9 in \cite{H2}, and also 
	Proposition 2.12 in \cite{HY}]\label{prop:pullback2form}
	The above map 
	\[
		\psi\colon G\times_{L_{1}}\mathbb{O}_{H}^{L_{1}\ltimes G(\mathbb{C}[z]_{k})_{1}}
		\overset{\sim}{\longrightarrow}
		\mathbb{O}_{H}
	\]
	is an isomorphism
	as complex manifolds.
\end{prop}
\begin{proof}
	Obviously the map  
	$\overline{\psi}\colon G\times \mathbb{O}_{H}^{L_{1}\ltimes G(\mathbb{C}[z]_{k})_{1}}
	\rightarrow \mathbb{O}_{H}$
	is surjective, and thus $\psi$ is surjective as well.
	Then let us see the injectivity.
	Suppose that  
	$(g_{1},\xi_{1}),(g_{2},\xi_{2})\in G\times \mathbb{O}_{H}^{L_{1}\ltimes G(\mathbb{C}[z]_{k})_{1}}$
	are sent to the same image by $\overline{\psi}$, i.e.,
	$\mathrm{Ad}^{*}(g_{1})(\xi_{1})=\mathrm{Ad}^{*}(g_{2})(\xi_{2})$.
	Since
	there exist $(h_{i},f_{i})\in L_{1}\ltimes G(\mathbb{C}[z]_{k})_{1}$, $i=1,2$,
	such that $\mathrm{Ad}^{*}(h_{i}f_{i})(H)=\xi_{i}$,
	we obtain 
	\[	
		H=\mathrm{Ad}^{*}((h_{1}f_{1})^{-1})\circ \mathrm{Ad}^{*}(g_{1}^{-1})\circ
		\mathrm{Ad}^{*}(g_{2})\circ \mathrm{Ad}^{*}(h_{2}f_{2})(H).
	\]
	This implies that 
	\[
		(h_{1}^{-1}g_{1}^{-1}g_{2}h_{2}, \mathrm{Ad}(h_{2}^{-1}g_{2}^{-1}g_{1}h_{1})(f_{1}^{-1})\cdot f_{2})
		\in G\ltimes G(\mathbb{C}[z]_{k})_{1}	
	\]
	stabilizes $H$.
	Then Proposition \ref{prop:stabh}
	shows that  $h_{1}^{-1}g_{1}^{-1}g_{2}h_{2}\in \mathrm{Stab}_{G}(H)\subset L_{1}$.
	Thus we have $h:=g_{1}^{-1}g_{2}\in L_{1}$
	and $(g_{2},\xi_{2})=(g_{1}h,\mathrm{Ad}^{*}(h^{-1})(\xi_{1}))$
	which show the injectivity of $\psi$.
\end{proof}
In the remaining of this section, we shall give a proof 
of the following triangular decomposition theorem of the fiber 
$\mathbb{O}_{H}^{L_{1}\ltimes G(\mathbb{C}[z]_{k})_{1}}$.
\begin{thm}[cf. Theorem 3.6 \cite{HY}, Proposition 4.15 \cite{H2}, and Corollary 3.2 \cite{Yam}]\label{thm:triangdecompfb}
There exists an isomorphism 
	\[
		\Psi\colon \left(\prod_{l=1}^{k}\left(N_{l,l+1}(\mathbb{C}[z]_{l})_{1}\times \mathfrak{u}_{l,l+1}(\mathbb{C}[z^{-1}]_{l-1})\right)\right)
		\times \mathbb{O}_{H_{\mathrm{res}}}^{L_{1}}
		\longrightarrow \mathbb{O}_{H}^{L_{1}\ltimes G(\mathbb{C}[z]_{k})_{1}}
	\]
	as complex manifolds.
\end{thm}
Here for the notations $N_{l,l+1}$ and $\mathfrak{n}_{l,l+1}$, see Section \ref{sec:spectral}.
\subsection{Decompositions of $G(\mathbb{C}[z]_{l})_{1}$}
Recall that $G(\mathbb{C}[z]_{l})_{1}$ is a connected and simply connected nilpotent 
Lie group and moreover it has the grading with respect to the degrees of $z^{m}$.
Therefore $G(\mathbb{C}[z]_{l})_{1}$ admits the following decomposition property.
\begin{prop}\label{prop:LU}
	Let $\mathfrak{h}_{1}$, $\mathfrak{h}_{2}$, and $\mathfrak{l}$ be Lie subalgebras of $\mathfrak{g}$,
	which satisfy 
	$\mathfrak{l}=\mathfrak{h}_{1}\oplus \mathfrak{h}_{2}$.
	Let us define the subgroups of $G(\mathbb{C}[z]_{l})_{1}$ by 
	$H_{i}(\mathbb{C}[z]_{l})_{1}:=\mathrm{exp}(\mathfrak{h}_{i}(\mathbb{C}[z]_{l})_{1})$
	for $i=1,2$, and 
	$L(\mathbb{C}[z]_{l})_{1}:=\mathrm{exp}(\mathfrak{l}(\mathbb{C}[z]_{l})_{1})$.
	Then we obtain the following decomposition,
	\[
		L(\mathbb{C}[z]_{l})_{1}=H_{1}(\mathbb{C}[z]_{l})_{1}\cdot H_{2}(\mathbb{C}[z]_{l})_{1}.
	\]
\end{prop}
\begin{proof}
	Take $g\in L(\mathbb{C}[z]_{l})_{1}$, and write it as    
	$g=e^{X_{l}z^{l}}\cdots e^{X_{2}z^{2}}e^{X_{1}z}$
	by $X_{i}\in \mathfrak{l}$, $i=1,2,\ldots,l$
	according to Proposition \ref{prop:explicitdesc}.
	First let us decompose 
	$X_{1}=H_{1}^{(1)}+H^{(2)}_{1}$ by $H_{1}^{(i)}\in \mathfrak{h}_{i}$.
	Then the Campbell-Baker-Hausdorff formula shows that
	\[
		e^{X_{1}z}=e^{(H_{1}^{(1)}+H^{(2)}_{1})z}\equiv e^{H_{1}^{(1)}z}e^{H_{1}^{(2)}z}\quad (\mathrm{mod\,}z^{2}).
	\]
	Thus we can write 
	\[
	g\cdot (e^{H_{1}^{(1)}z}e^{H_{1}^{(2)}z})^{-1}=e^{X'_{l}z^{l}}\cdots e^{X'_{2}z^{2}}.
	\]
	Next, we decompose $X'_{2}=H_{2}^{(1)}+H^{(2)}_{2}$ by $H_{2}^{(i)}\in \mathfrak{h}_{i}$
	and we have 
	\[
		e^{X'_{2}z^{2}}=e^{(H_{2}^{(1)}+H^{(2)}_{2})z^{2}}\equiv e^{H_{2}^{(1)}z^{2}}e^{H_{2}^{(2)}z^{2}}\quad (\mathrm{mod\,}z^{3})
	\]
	as well.
	Since $e^{H_{1}^{(i)}z}$ and $e^{H_{2}^{(i')}z^{2}}$ are commutative modulo $z^{3}$
	by the Campbell-Baker-Hausdorff formula,
	we obtain
	\[
	g\cdot (e^{H_{2}^{(1)}z^{2}}e^{H_{1}^{(1)}z}e^{H_{2}^{(2)}z^{2}}e^{H_{1}^{(2)}z})^{-1}=e^{X''_{l}z^{l}}\cdots e^{X''_{3}z^{3}}.
	\] 
	Repeating this procedure, we obtain $h_{i}\in H_{i}(\mathbb{C}[z]_{l})_{1}$, $i=1,2,$
	such that $g=h_{1}h_{2}$.
	
	Let us notice that $H_{1}(\mathbb{C}[z]_{l})_{1}\cap H_{2}(\mathbb{C}[z]_{l})_{1}=\{e\}$
	since $\mathfrak{h}_{1}(\mathbb{C}[z]_{l})_{1}\cap \mathfrak{h}_{2}(\mathbb{C}[z]_{l})_{1}=\{0\}$
	and the exponential map is an analytic diffeomorphism.
	Thus the description $g=h_{1}h_{2}$ is unique.
\end{proof}

\subsection{Filtration of $\mathbb{O}^{L_{1}\ltimes G(\mathbb{C}[z]_{k})_{1}}_{H}$}\label{sec:filtext}
Let us consider the sequence of subgroups 
\[
	L_{1}\ltimes L_{1}(\mathbb{C}[z]_{k})_{1}\subset L_{1}\ltimes L_{2}(\mathbb{C}[z]_{k})_{1}
	\subset \cdots \subset L_{1}\ltimes L_{k+1}(\mathbb{C}[z]_{k})_{1}=L_{1}\ltimes G(\mathbb{C}[z]_{k})_{1} 	
\]
of $L_{1}\ltimes G(\mathbb{C}[z]_{k})_{1}$, and the induced 
filtration 
 of $\mathbb{O}^{L_{1}\ltimes G(\mathbb{C}[z]_{k})_{1}}_{H}$,
\[
	\mathbb{O}_{H}^{L_{1}\ltimes L_{1}(\mathbb{C}[z]_{k})_{1}}
		\subset \mathbb{O}_{H}^{L_{1}\ltimes L_{2}(\mathbb{C}[z]_{k})_{1}}\subset \cdots 
	\subset \mathbb{O}_{H}^{L_{1}\ltimes L_{k+1}(\mathbb{C}[z]_{k})_{1}}=\mathbb{O}^{L_{1}\ltimes G(\mathbb{C}[z]_{k})_{1}}_{H}.
\]
Here we notice that Lemma 3.3 in \cite{Yam} or 
computations carried out in the proof of Proposition \ref{prop:stabh}
shows that 
\[
	\mathbb{O}_{H}^{L_{1}\ltimes L_{1}(\mathbb{C}[z]_{k})_{1}}=H_{\mathrm{irr}}+\mathrm{Ad}^{*}(L_{1})(H_{\mathrm{res}})\cong 
	\mathbb{O}_{H_{\text{res}}}^{L_{1}},
\]
see also Lemma 4.10 in \cite{H2}.

Recall that for a pair of positive integers $0\le l<l'\le k$, there is the  natural projection map
\[
	\pi_{l',l}\colon \mathbb{C}[z]_{l'}\rightarrow \mathbb{C}[z]_{l'}/(\mathfrak{m}_{z}^{(l')})^{l+1}=\mathbb{C}[z]_{l}
\]
where $\mathfrak{m}_{z}^{(l')}=\langle z\rangle_{\mathbb{C}[z]_{l'}}$
is the unique maximal ideal of $\mathbb{C}[z]_{l'}$.
Then we obtain exact sequences of Lie algebras
\[
	0\rightarrow \mathfrak{h}((\mathfrak{m}_{z}^{(l')})^{l+1})
	\rightarrow \mathfrak{h}(\mathbb{C}[z]_{l'})
	\xrightarrow{\pi_{l',l}} \mathfrak{h}(\mathbb{C}[z]_{l})
	\rightarrow 0
\]
and nilpotent Lie groups 
\[
	1\rightarrow \mathrm{exp\,}(\mathfrak{h}((\mathfrak{m}_{z}^{(l')})^{l+1}))	
	\rightarrow H(\mathbb{C}[z]_{l'})_{1}
	\xrightarrow{\pi_{l',l}} H(\mathbb{C}[z]_{l})_{1}
	\rightarrow 1
\]
for  a Lie subgroup $H\subset G$ and its Lie algebra $\mathfrak{h}$. 

For a positive integer $0\le l\le k$, let us define \index[cN]{$H^{[l]}$}
	\[
		H^{[l]}:=
		\left(
			\frac{H_{l}}{z^{l}}+\cdots +H_{\mathrm{res}}
		\right)\frac{1}{z}\in \mathfrak{g}(\mathbb{C}[z^{-1}]_{l}).
	\]
\begin{prop}\label{prop:smallorbit}
	Let us take a pair of positive integers $l,\,l'$ with $0\le l\le l'\le k$.
	Then the projection map $\pi_{l',l}\colon L_{1}\ltimes L_{l+1}(\mathbb{C}[z]_{l'})_{1}
	\rightarrow L_{1}\ltimes L_{l+1}(\mathbb{C}[z]_{l})_{1}$ induces 
	the isomorphism
	\[
		\pi_{l',l}^{\mathbb{O}}\colon \mathbb{O}_{H^{[l']}}^{L_{1}\ltimes L_{l+1}(\mathbb{C}[z]_{l'})_{1}}\overset{\sim}{\longrightarrow} \mathbb{O}_{H^{[l]}}^{L_{1}\ltimes L_{l+1}(\mathbb{C}[z]_{l})_{1}}.
	\]
\end{prop}
\begin{proof}
	From
	Proposition \ref{prop:stabh}, we have 
	\begin{align*}
		&\mathrm{Stab}_{L_{l+1}(\mathbb{C}[z]_{l'})_{1}}(H^{[l']})=\\
		&\quad\left\{
			e^{X_{l'}z^{l'}}\cdots e^{X_{2}z^{2}}e^{X_{1}z}\in G(\mathbb{C}[z]_{l'})_{1}\,\middle|\,
			X_{i}\in \mathfrak{l}_{i}\cap \mathfrak{l}_{l+1},\ i=1,\ldots,l'
		\right\},\\
		&\mathrm{Stab}_{L_{1}\ltimes L_{l+1}(\mathbb{C}[z]_{l'})_{1}}(H^{[l']})
		=\mathrm{Stab}_{L_{1}}(H_{\mathrm{res}})\ltimes \mathrm{Stab}_{L_{l+1}(\mathbb{C}[z]_{l'})_{1}}(H^{[l']}).
	\end{align*}	
	Here we note that the above condition 
	$X_{i}\in \mathfrak{l}_{i}\cap \mathfrak{l}_{l+1},\ i=1,\ldots,l'$
	is equivalent to 
	\[
		X_{i}\in \begin{cases}
			\mathfrak{l}_{i}& i=1,2,\ldots,l,\\
			\mathfrak{l}_{l+1}&i=l+1,\ldots,l'
		\end{cases}.
	\]
	Since we moreover have 
	\begin{align*}
			\mathrm{Stab}_{L_{l+1}(\mathbb{C}[z]_{l})_{1}}(H^{[l]})&=
			\left\{
				e^{X_{l}z^{l}}\cdots e^{X_{2}z^{2}}e^{X_{1}z}\in G(\mathbb{C}[z]_{l})_{1}\,\middle|\,
				X_{i}\in \mathfrak{l}_{i},\ i=1,\ldots,l
			\right\},\\
		\mathrm{Stab}_{L_{1}\ltimes L_{l+1}(\mathbb{C}[z]_{l})_{1}}(H^{[l]})
		&=\mathrm{Stab}_{L_{1}}(H_{\mathrm{res}})\ltimes \mathrm{Stab}_{L_{l+1}(\mathbb{C}[z]_{l})_{1}}(H^{[l]}),
	\end{align*}	
	we obtain the commutative diagram
	\scriptsize
	\[
		\begin{tikzcd}
			1\arrow[r]& \mathrm{exp}(\mathfrak{l}_{l+1}\otimes_{\mathbb{C}}(\mathfrak{m}_{z}^{(l')})^{l+1})
	\arrow[r]& 	L_{1}\ltimes L_{l+1}(\mathbb{C}[z]_{l'})_{1}
	\arrow[r,"\pi_{l',l}"]& L_{1}\ltimes L_{l+1}(\mathbb{C}[z]_{l})_{1}
	\arrow[r]&1\\
	1\arrow[r] &\mathrm{exp}(\mathfrak{l}_{l+1}\otimes_{\mathbb{C}}(\mathfrak{m}_{z}^{(l')})^{l+1})
	\arrow[r]\arrow[u,equal]& 	\mathrm{Stab}_{L_{1}\ltimes L_{l+1}(\mathbb{C}[z]_{l'})_{1}}(H^{[l']})
	\arrow[r,"\pi_{l',l}"]\arrow[u,hookrightarrow]& \mathrm{Stab}_{L_{1}\ltimes L_{l+1}(\mathbb{C}[z]_{l})_{1}}(H^{[l]})\arrow[u,hookrightarrow]
	\arrow[r]&1
		\end{tikzcd}
	\]
	\normalsize
	whose vertical arrows are natural inclusions and horizontal sequences are exact.
	By a diagram chase shows $\pi_{l,l'}$ induces the bijective holomorphic map 
	\begin{multline*}
	\mathbb{O}_{H^{[l']}}^{L_{1}\ltimes L_{l+1}(\mathbb{C}[z]_{l'})_{1}} = L_{l+1}(\mathbb{C}[z]_{l'})_{1}/\mathrm{Stab}_{L_{1}\ltimes L_{l+1}(\mathbb{C}[z]_{l'})_{1}}(H^{[l']})
		\overset{\pi_{l',l}}{\longrightarrow}\\
	 L_{1}\ltimes L_{l+1}(\mathbb{C}[z]_{l})_{1}/\mathrm{Stab}_{L_{1}\ltimes L_{l+1}(\mathbb{C}[z]_{l})_{1}}(H^{[l]})=\mathbb{O}_{H^{[l]}}^{L_{1}\ltimes L_{l+1}(\mathbb{C}[z]_{l})_{1}}.
	\end{multline*} 
\end{proof}
By the isomorphisms in this proposition, the above filtration 
is replaced by 
\[
	\mathbb{O}_{H_{\mathrm{res}}}^{L_{1}\ltimes L_{1}(\mathbb{C}[z]_{0})_{1}}
		\subset \mathbb{O}_{H^{[1]}}^{L_{1}\ltimes L_{2}(\mathbb{C}[z]_{1})_{1}}\subset \cdots 
	\subset\mathbb{O}_{H^{[l]}}^{L_{1}\ltimes L_{l+1}(\mathbb{C}[z]_{l})_{1}}
	\subset \cdots \subset 
	\mathbb{O}^{L_{1}\ltimes G(\mathbb{C}[z]_{k})_{1}}_{H}.
\]

\subsection{LU-type decomposition and Levi-type decomposition}\label{sec:ludecomp}
For each pair $1\le l< l'\le k$
of positive integers,
Proposition \ref{prop:LU} gives us the Levi-type decomposition
\[
	P_{l,l'}(\mathbb{C}[z]_{l})_{1}=L_{l}(\mathbb{C}[z]_{l})_{1}\ltimes U_{l,l'}(\mathbb{C}[z]_{l})_{1}
\]
and the LU-type 
decomposition 
\[
	L_{l'}(\mathbb{C}[z]_{l})_{1}=N_{l,l'}(\mathbb{C}[z]_{l})_{1}\cdot P_{l,l'}(\mathbb{C}[z]_{l})_{1}.
\]
Since $L_{1}$ normalizes $N_{l,l'}(\mathbb{C}[z]_{l})_{1}$, we further obtain 
\begin{align*}
	L_{1}\ltimes 	L_{l'}(\mathbb{C}[z]_{l})_{1}
	&=N_{l,l'}(\mathbb{C}[z]_{l})_{1}\cdot \left(L_{1}\ltimes P_{l,l'}(\mathbb{C}[z]_{l})_{1}\right)\\
	&=N_{l,l'}(\mathbb{C}[z]_{l})_{1}\cdot \left((L_{1}\ltimes L_{l}(\mathbb{C}[z]_{l})_{1})\ltimes U_{l,l'}(\mathbb{C}[z]_{l})_{1}\right).
\end{align*}

\begin{prop}\label{prop:NP}
	Let us take a pair $1\le l\le l'\le k$
	of positive integers.
	The multiplication map
	\[
		\overline{\phi}\colon N_{l,l'}(\mathbb{C}[z]_{l})_{1}\times 	(L_{1}\ltimes P_{l,l'}(\mathbb{C}[z]_{l})_{1})
		\longrightarrow L_{1}\ltimes 	L_{l'}(\mathbb{C}[z]_{l})_{1}
	\] 
	induces the 
	isomorphism 
	\[
		\phi\colon N_{l,l'}(\mathbb{C}[z]_{l})_{1}\times \mathbb{O}_{H^{[l]}}^{L_{1}\ltimes P_{l,l'}(\mathbb{C}[z]_{l})_{1}}
		\longrightarrow \mathbb{O}_{H^{[l]}}^{L_{1}\ltimes L_{l'}(\mathbb{C}[z]_{l})_{1}}
	\]
	as complex manifolds.
\end{prop}
\begin{proof}
	Let 
	\[
	\pi_{1} \colon L_{1}\ltimes 	L_{l'}(\mathbb{C}[z]_{l})_{1}
	\rightarrow \mathbb{O}_{H^{[l]}}^{L_{1}\ltimes L_{l'}(\mathbb{C}[z]_{l})_{1}},
	\pi_{2}\colon 
	L_{1}\ltimes P_{l,l'}(\mathbb{C}[z]_{l})_{1}
	\rightarrow 
	\mathbb{O}_{H^{[l]}}^{L_{1}\ltimes P_{l,l'}(\mathbb{C}[z]_{l})_{1}}
	\]
	be natural quotient maps.
	Since we obviously have 
	\[
	\mathrm{Stab}_{L_{1}\ltimes 	P_{l,l'}(\mathbb{C}[z]_{l})_{1}}(H^{[l]})\subset \mathrm{Stab}_{L_{1}\ltimes 	L_{l'}(\mathbb{C}[z]_{l})_{1}}(H^{[l]}),
	\]
	the composition map 
	\[
	\pi_{1}\circ \overline{\phi}\colon 
	N_{l,l'}(\mathbb{C}[z]_{l})_{1}\times 	(L_{1}\ltimes P_{l,l'}(\mathbb{C}[z]_{l})_{1})
	\rightarrow \mathbb{O}_{H^{[l]}}^{L_{1}\ltimes L_{l'}(\mathbb{C}[z]_{l})_{1}}
	\]
	factors through the 
	quotient map 
	\[
	N_{l,l'}(\mathbb{C}[z]_{l})_{1}\times 	(L_{1}\ltimes P_{l,l'}(\mathbb{C}[z]_{l})_{1})
	\xrightarrow{\mathrm{id}_{N_{l,l'}(\mathbb{C}[z]_{l})_{1}}\times \pi_{2}}
	N_{l,l'}(\mathbb{C}[z]_{l})_{1}\times \mathbb{O}_{H^{[l]}}^{L_{1}\ltimes P_{l,l'}(\mathbb{C}[z]_{l})_{1}}.
	\]
	Namely,
	we obtain the commutative diagram 
	\[
		\begin{tikzcd}
			N_{l,l'}(\mathbb{C}[z]_{l})_{1}\times 	(L_{1}\ltimes P_{l,l'}(\mathbb{C}[z]_{l})_{1})
			\arrow[r,"\overline{\phi}"]\arrow[d,"\mathrm{id}_{N_{l,l'}(\mathbb{C}[z]_{l})_{1}}\times \pi_{2}"]&
			L_{1}\ltimes 	L_{l'}(\mathbb{C}[z]_{l})_{1}\arrow[d,"\pi_{1}"]\\
			N_{l,l'}(\mathbb{C}[z]_{l})_{1}\times \mathbb{O}_{H^{[l]}}^{L_{1}\ltimes P_{l,l'}(\mathbb{C}[z]_{l})_{1}}
			\arrow[r,"\phi"]
			&\mathbb{O}_{H^{[l]}}^{L_{1}\ltimes L_{l'}(\mathbb{C}[z]_{l})_{1}}
		\end{tikzcd}.	
	\]
	Since $\bar{\phi}$ is isomorphism, $\phi$ is surjective.
	Let us notice that Proposition \ref{prop:stabh}
	shows that 
	$\mathrm{Stab}_{L_{1}\ltimes 	L_{l'}(\mathbb{C}[z]_{l})_{1}}(H^{[l]})
	\subset  L_{1}\ltimes P_{l,l'}(\mathbb{C}[z]_{l})_{1}.$
	Thus the above inclusion becomes the equation
	$\mathrm{Stab}_{L_{1}\ltimes P_{l,l'}(\mathbb{C}[z]_{l})_{1}}(H^{[l]})= \mathrm{Stab}_{L_{1}\ltimes 	L_{l'}(\mathbb{C}[z]_{l})_{1}}(H^{[l]})$
	which shows the injectivity of $\phi$ by a 
	diagram chase.
\end{proof}

Let us focus on the case $l'=l+1$ and consider  
the multiplication map 
\[
	\chi\colon (L_{1}\ltimes L_{l}(\mathbb{C}[z]_{l})_{1})\times 	U_{l,l+1}(\mathbb{C}[z]_{l})_{1}\longrightarrow L_{1}\ltimes P_{l,l+1}(\mathbb{C}[z]_{l})_{1}.
\]
Also
we 
 define the \index[cN]{$a_{H^{[l]}}$}
map
\[
	\begin{array}{cccc}
	a_{H^{[l]}}\colon &(L_{1}\ltimes L_{l}(\mathbb{C}[z]_{l})_{1})\times U_{l,l+1}(\mathbb{C}[z]_{l})_{1}&\longrightarrow &(L_{1}\ltimes L_{l}(\mathbb{C}[z]_{l})_{1})
	\times \mathfrak{u}_{l,l+1}(\mathbb{C}[z^{-1}]_{l-1})\\
	&(g,u)&\longmapsto &(g,\mathrm{Ad}^{*}(gu)(H^{[l]})-\mathrm{Ad}^{*}(g)(H^{[l]}))
	\end{array}.
\]\normalsize
\begin{lem}\label{lem:prep2}
	The above map $a_{H^{[l]}}$ is an isomorphism as complex manifolds.
\end{lem}
\begin{proof}
	Take
	$u=e^{z^{l}X_{l}}\cdots e^{zX_{1}}\in U_{l,l+1}(\mathbb{C}[z]_{l})_{1}$
	and $B=\sum_{i=0}^{l}B_{i}z^{-i-1}\in \mathfrak{l}_{l}(\mathbb{C}[z^{-1}]_{l})$.
	Then 
	we have $\mathrm{Ad}^{*}(u)(B)=\sum_{i=0}^{l}C_{i}z^{-i-1}$ with 
	\begin{align*}
	&C_{i}
	=\\
	&B_{i}+\sum_{\substack{m_{1},m_{2},\ldots,m_{l}\ge 0\\
	(m_{1},m_{2},\ldots,m_{l})\neq \mathbf{0}}}
	\frac{\mathrm{ad}(X_{l})^{m_{l}}\circ \cdots\circ \mathrm{ad}(X_{2})^{m_{2}}\circ \mathrm{ad}(X_{1})^{m_{1}}
	(B_{i+m_{1}+2m_{2}+\cdots +lm_{l}})}
	{m_{1}!m_{2}!\cdots m_{l}!}.
	\end{align*}
	Then we notice that that the second term in the right hand side is contained in $\mathfrak{u}_{l,l+1}$,
	and $C_{l}=B_{l}$.
	Thus $a_{H^{[l]}}$ is well-defined.
	
	Next we show that $a_{H^{[l]}}$ is injective.
	Suppose that $(g,u)$ and $(g',u')$ have the same image by $a_{H^{[l]}}$.
	Then obviously $g=g'$ and also
	$\mathrm{Ad}^{*}(gu)(H^{[l]})=\mathrm{Ad}^{*}(gu')(H^{[l]})$.
	Thus $u^{-1}u'\in U_{l,l+1}(\mathbb{C}[z]_{l})_{1}$ belongs to $\mathrm{Stab}_{L_{l+1}(\mathbb{C}[z]_{l})_{1}}(H^{[l]})$ which means 
	that $u^{-1}u'=e$ by Proposition \ref{prop:stabh}, i.e., $u=u'$ as desired.

	Next we show that $a_{H^{[l]}}$ is surjective.
	Recall that we have $\mathfrak{u}_{l,l+1}=\mathrm{Ker\,}(\mathrm{ad}(H_{m})|_{\mathfrak{u}_{l,l+1}})\oplus \mathrm{ad}(H_{m})(\mathfrak{u}_{l,l+1})
	=\mathrm{ad}(H_{m})(\mathfrak{u}_{l,l+1})$ for $m=0,1,\ldots,l$,
	since $\mathrm{Ker\,}(\mathrm{ad}(H_{m})|_{\mathfrak{u}_{l,l+1}})=0$.
	Then we can see that 
	the above equations for given $C_{l}=B_{l}$ and 
	$C_{i}\in B_{i}+\mathfrak{u}_{l,l+1}$ for $i=0,1,\ldots,l-1$
	can be solved inductively from $i=l-1$ to $i=0$ and we find $u\in U_{l,l+1}(\mathbb{C}[z]_{l})_{1}$
	such that $\mathrm{Ad}^{*}(u)(B)=\sum_{i=0}^{l}C_{i}z^{-i-1}$.
	In particular when $B\in \mathbb{O}_{H}^{L_{1}\ltimes L_{l}(\mathbb{C}[z]_{l})_{1}}$,
	there exists $g\in L_{1}\ltimes L_{l}(\mathbb{C}[z]_{l})_{1}$
	such that $\mathrm{Ad}^{*}(g)(H^{[l]})=B$.
	Thus in this case we have 
	\[
		\mathrm{Ad}^{*}(g(g^{-1}ug))(H^{[l]})=\mathrm{Ad}^{*}(ug)(H^{[l]})=\mathrm{Ad}^{*}(u)(B)=C.
	\]
	This shows that $a_{H^{[l]}}$ is surjective.
\end{proof}
\begin{prop}\label{prop:lu}
	The map \index[cN]{$X$}
	\[
		\begin{array}{cccc}	
		X\colon 	&\mathbb{O}_{H^{[l]}}^{L_{1}\ltimes L_{l}(\mathbb{C}[z]_{l})_{1}}\times \mathfrak{u}_{l,l+1}(\mathbb{C}[z^{-1}]_{l-1})&
		\longrightarrow &\mathbb{O}_{H^{[l]}}^{L_{1}\ltimes P_{l,l+1}(\mathbb{C}[z]_{l})_{1}}\\
		&(\eta,\xi)&\longmapsto 
		&\eta+\xi
		\end{array}
	\]
	is an isomorphism of complex manifolds.
\end{prop}
\begin{proof}
	One can check that the following diagram is commutative,
	\[
		\begin{tikzcd}
			(L_{1}\ltimes L_{l}(\mathbb{C}[z]_{l})_{1})\times U_{l,l+1}(\mathbb{C}[z]_{l})_{1}
			\arrow[r,"\chi"]
			\arrow[d,"a_{H^{[l]}}"]
			&L_{1}\ltimes P_{l,l+1}(\mathbb{C}[z]_{l})_{1}
			\arrow[dd,"\pi_{P}"]\\
			(L_{1}\ltimes L_{l}(\mathbb{C}[z]_{l})_{1})
	\times \mathfrak{u}_{l,l+1}(\mathbb{C}[z^{-1}]_{l-1})
			\arrow[d,"\pi_{L}\times \mathrm{id}"]&\\
			\mathbb{O}_{H^{[l]}}^{L_{1}\ltimes L_{l}(\mathbb{C}[z]_{l})_{1}}\times \mathfrak{u}_{l,l+1}(\mathbb{C}[z^{-1}]_{l-1})
			\arrow[r,"X"]
			&\mathbb{O}_{H^{[l]}}^{L_{1}\ltimes P_{l,l+1}(\mathbb{C}[z]_{l})_{1}}
		\end{tikzcd}.
	\]
	Here 
	\[\pi_{L}\colon L_{1}\ltimes L_{l}(\mathbb{C}[z]_{l})_{1}\rightarrow \mathbb{O}_{H^{[l]}}^{L_{1}\ltimes L_{l}(\mathbb{C}[z]_{l})_{1}},
	\pi_{P}\colon L_{1}\ltimes P_{l,l+1}(\mathbb{C}[z]_{l})_{1}\rightarrow \mathbb{O}_{H^{[l]}}^{L_{1}\ltimes P_{l,l+1}(\mathbb{C}[z]_{l})_{1}}
	\]
	are natural projections.
	Since $\chi$ and $a_{H^{[l]}}$ are isomorphisms, and $\pi_{L}\times \mathrm{id}$ and $\pi_{P}$ are surjective,
	it follows that $X$ is surjective.

	Equip the action of $L_{1}\ltimes L_{l}(\mathbb{C}[z]_{l})_{1}$
	with $(L_{1}\ltimes L_{l}(\mathbb{C}[z]_{l})_{1})
	\times \mathfrak{u}_{l,l+1}(\mathbb{C}[z^{-1}]_{l-1})$ by the multiplication on the first component.
	Then the isomorphism $a_{H^{[l]}}$ becomes $(L_{1}\ltimes L_{l}(\mathbb{C}[z]_{l})_{1})_{H}$-equivariant
	and also 
	the isomorphism $\chi\circ a_{H^{[l]}}^{-1}$ becomes $(L_{1}\ltimes L_{l}(\mathbb{C}[z]_{l})_{1})_{H}$-equivariant.
	Since $(L_{1}\ltimes L_{l}(\mathbb{C}[z]_{l})_{1})_{H}=(L_{1}\ltimes P_{l,l+1}(\mathbb{C}[z]_{l})_{1})_{H}$,
	we can see that $X$ is injective.
\end{proof}
\subsection{A proof of Theorem \ref{thm:triangdecompfb}}\label{sec:proofthm}
Now we are ready to give a proof of Theorem \ref{thm:triangdecompfb}.
\begin{proof}[Proof of Theorem \ref{thm:triangdecompfb}]
	As a combination of Propositions \ref{prop:NP} and \ref{prop:lu},
	we obtain isomorphisms
	\[
		\overline{\psi}_{l}\colon N_{l,l+1}(\mathbb{C}[z]_{l})_{1}\times \mathfrak{u}_{l,l+1}(\mathbb{C}[z^{-1}]_{l-1})
		\times \mathbb{O}_{H^{[l]}}^{L_{1}\ltimes L_{l}(\mathbb{C}[z]_{l})_{1}}
		\longrightarrow \mathbb{O}_{H^{[l]}}^{L_{1}\ltimes L_{l+1}(\mathbb{C}[z]_{l})_{1}}
	\]
	for 
	$l=1,2,\ldots,k$.
	Since $\mathbb{O}_{H^{[l]}}^{L_{1}\ltimes L_{l}(\mathbb{C}[z]_{l})_{1}}
	\cong \mathbb{O}_{H^{[l-1]}}^{L_{1}\ltimes L_{l}(\mathbb{C}[z]_{l-1})_{1}}$
	by Proposition \ref{prop:smallorbit},
	above isomorphisms moreover give us following isomorphisms
	\[
		\psi_{l}\colon (N_{l,l+1}(\mathbb{C}[z]_{l})_{1}\times \mathfrak{u}_{l,l+1}(\mathbb{C}[z^{-1}]_{l-1}))
		\times \mathbb{O}_{H^{[l-1]}}^{L_{1}\ltimes L_{l}(\mathbb{C}[z]_{l-1})_{1}}
		\longrightarrow \mathbb{O}_{H^{[l]}}^{L_{1}\ltimes L_{l+1}(\mathbb{C}[z]_{l})_{1}}
	\]
	for $l=1,2,\ldots,k$.
	Therefore, as the combination of these isomorphisms $\psi_{l}$	
	we obtain the desired isomorphism
	\[
		\Psi\colon \left(\prod_{l=1}^{k}\left(N_{l,l+1}(\mathbb{C}[z]_{l})_{1}\times \mathfrak{u}_{l,l+1}(\mathbb{C}[z^{-1}]_{l-1})\right)\right)
		\times \mathbb{O}_{H_{\mathrm{res}}}^{L_{1}}
		\longrightarrow \mathbb{O}_{H}^{L_{1}\ltimes G(\mathbb{C}[z]_{k})_{1}}
	\]
	by recalling $\mathbb{O}_{H^{[0]}}^{L_{1}\ltimes L_{1}(\mathbb{C}[z]_{0})_{1}}=\mathbb{O}_{H_{\mathrm{res}}}^{L_{1}}$.
\end{proof}
\subsection{Triangular decomposition of $\mathbb{O}_{H}$}\label{sec:tridec}
As a consequence of Theorem \ref{thm:triangdecompfb},
we moreover obtain the triangular decomposition of the truncated orbit $\mathbb{O}_{H}$ as follows.

Define an $L_{1}$-action on 
$\left(\prod_{l=1}^{k}(N_{l,l+1}(\mathbb{C}[z]_{l})_{1}\times \mathfrak{u}_{l,l+1}(\mathbb{C}[z^{-1}]_{l-1}))\right)
	\times \mathbb{O}_{H_{\mathrm{res}}}^{L_{1}}$
	as follows,
	\[
	g\cdot ((n_{l,l+1},\xi_{l})_{l=1,\ldots,k},\xi):=
	((\mathrm{Ad}(g)(n_{l}),\mathrm{Ad}^{*}(g)(\xi_{l}))_{l=1,2,\ldots,k},\mathrm{Ad}^{*}(g)(\xi))
	\]
	for $g\in L_{1}$ and $((n_{l},\xi_{l})_{l=1,\ldots,k},\xi)
	\in \left(\prod_{l=1}^{k}(N_{l,l+1}(\mathbb{C}[z]_{l})_{1}\times \mathfrak{u}_{l,l+1}(\mathbb{C}[z^{-1}]_{l-1}))\right)
	\times \mathbb{O}_{H_{\mathrm{res}}}^{L_{1}}$.
Then   
\begin{align*}
	&\mathrm{id}_{G}\times \Psi\colon\\
	&G\times \left(\left(\prod_{l=1}^{k}(N_{l,l+1}(\mathbb{C}[z]_{l})_{1}\times \mathfrak{u}_{l,l+1}(\mathbb{C}[z^{-1}]_{l-1}))\right)
	\times \mathbb{O}_{H_{\mathrm{res}}}^{L_{1}}\right)
	\longrightarrow G\times \mathbb{O}_{H}^{L_{1}\ltimes G(\mathbb{C}[z]_{k})_{1}}
\end{align*}
is $L_{1}$-equivariant isomorphism, and therefore Proposition \ref{prop:pullback2form} shows that 
the map 
\begin{multline*}
	G\times \left(\left(\prod_{l=1}^{k}(N_{l,l+1}(\mathbb{C}[z]_{l})_{1}\times \mathfrak{u}_{l,l+1}(\mathbb{C}[z^{-1}]_{l-1}))\right)
	\times \mathbb{O}_{H_{\mathrm{res}}}^{L_{1}}\right)
	\xrightarrow{\mathrm{id}_{G}\times \Psi} \\
	G\times \mathbb{O}_{H}^{L_{1}\ltimes G(\mathbb{C}[z]_{k})_{1}}
	\xrightarrow{\psi}
	\mathbb{O}_{H}
\end{multline*}
factors through 
the isomorphism 
\[
		L_{1}\Bigg\backslash \left(
			G\times \left(\left(\prod_{l=1}^{k}(N_{l,l+1}(\mathbb{C}[z]_{l})_{1}\times \mathfrak{u}_{l,l+1}(\mathbb{C}[z^{-1}]_{l-1}))\right)
	\times \mathbb{O}_{H_{\mathrm{res}}}^{L_{1}}\right)
		\right)
		\cong \mathbb{O}_{H}.
\] 

According to this triangular decomposition of $\mathbb{O}_{H}$,
let us give a description of the moment map 
\[
	\mu_{\mathbb{O}_{H}\downarrow G}\colon \mathbb{O}_{H}
	\longrightarrow \mathfrak{g}
\]
defined in Section \ref{sec:symporb},
which will play an important role in the latter sections.
Let 
us take 
$\Xi\in \mathbb{O}_{H}$ and 
choose its representative 
$$(g,(n_{l},\nu_{l})_{l=1,\ldots,k},\eta)
\in G\times \prod_{l=1}^{k}\left(
	N_{l,l+1}(\mathbb{C}[z]_{l})_{1}\times 
	\mathfrak{u}_{l,l+1}(\mathbb{C}[z^{-1}]_{l-1})	
	\right)
\times \mathbb{O}_{H_{\mathrm{res}}}^{L_{1}}.$$
Then 
under the injective immersion $\iota_{\mathbb{O}_{H}}\colon \mathbb{O}_{H}\hookrightarrow \mathfrak{g}(\mathbb{C}[z^{-1}]_{k})$,
the image $\iota_{\mathbb{O}_{H}}(\Xi)$ is computed by the following steps.
\\
\noindent
\underline{Step 1}. 
Let us define 
\[
	\iota_{\mathbb{O}_{H}}^{(1)}(\Xi):=
	\mathrm{Ad}^{*}(n_{1})(\nu_{1}+\eta)\in \mathfrak{g}(\mathbb{C}[z^{-1}]_{k}).
\]
\\
\noindent
\underline{Step $l$ for $1<l\le k$}.
Let us define
\[
	\iota_{\mathbb{O}_{H}}^{(l)}(\Xi):=
	\mathrm{Ad}^{*}(n_{l})(\nu_{l}+\iota_{\mathbb{O}_{H}}^{(l-1)}(\Xi))\in \mathfrak{g}(\mathbb{C}[z^{-1}]_{k}).
\] 
\\
Then we obtain 
\[
	\iota_{\mathbb{O}_{H}}(\Xi)=
	\mathrm{Ad}^{*}(g)(\iota_{\mathbb{O}_{H}}^{(k)}(\Xi))\in \mathfrak{g}(\mathbb{C}[z^{-1}]_{k}).
\] 
Therefore, we finally  
obtain the inductive formula for the moment map, 
\[
	\mu_{\mathbb{O}_{H}\downarrow G}(\Xi)=
	\underset{z=0}{\mathrm{res\,}}\circ \mathrm{Ad}^{*}(g)(\iota_{\mathbb{O}_{H}}^{(k)}(\Xi))\in \mathfrak{g}.
\]

Another representative of $\Xi$ leads to the same result. 
Indeed for another representative there exists $h\in L_{1}$
and it is written by 
\[
(gh^{-1},(\mathrm{Ad}(h)(n_{l}),\mathrm{Ad}^{*}(h)(\nu_{l}))_{l=1,\ldots,k},\mathrm{Ad}^{*}(h)(\eta)).
\]
Then we can also inductively define 
\begin{align*}
	\widetilde{\iota}_{\mathbb{O}_{H}}^{(1)}(\Xi)&:=
	\mathrm{Ad}^{*}(\mathrm{Ad}(h)(n_{1}))(\mathrm{Ad}^{*}(h)(\nu_{1})+\mathrm{Ad}^{*}(h)(\eta))\\
	&=
	\mathrm{Ad}^{*}(\mathrm{Ad}(h)(n_{1})h)(\nu_{1}+\eta)\\
	&=\mathrm{Ad}^{*}(h\cdot n_{1})(\nu_{1}+\eta),
\end{align*}
\begin{align*}
	\widetilde{\iota}_{\mathbb{O}_{H}}^{(2)}(\Xi)&:=
	\mathrm{Ad}^{*}(\mathrm{Ad}(h)(n_{2}))(\mathrm{Ad}^{*}(h)(\nu_{2})+\widetilde{\iota}_{\mathbb{O}_{H}}^{(1)}(\Xi))\\
	&=
	\mathrm{Ad}^{*}(\mathrm{Ad}(h)(n_{2}))(\mathrm{Ad}^{*}(h)(\nu_{2})+\mathrm{Ad}^{*}(h\cdot n_{1})(\nu_{1}+\eta))\\
	&=\mathrm{Ad}^{*}(\mathrm{Ad}(h)(n_{2})\cdot h)(\nu_{2}+\mathrm{Ad}^{*}(n_{1})(\nu_{1}+\eta))\\
	&=\mathrm{Ad}^{*}(h\cdot n_{2})(\nu_{2}+\iota_{\mathbb{O}_{H}}^{(1)}(\Xi)),
\end{align*}
\begin{align*}
	\widetilde{\iota}_{\mathbb{O}_{H}}^{(l)}(\Xi)&:=
	\mathrm{Ad}^{*}(\mathrm{Ad}(h)(n_{l}))(\mathrm{Ad}^{*}(h)(\nu_{l})+\widetilde{\iota}_{\mathbb{O}_{H}}^{(l-1)}(\Xi))\\
	&=
	\mathrm{Ad}^{*}(\mathrm{Ad}(h)(n_{l}))(\mathrm{Ad}^{*}(h)(\nu_{l})+\mathrm{Ad}^{*}(h\cdot n_{l-1})(\nu_{l-2}+\iota_{\mathbb{O}_{H}}^{(l-2)}(\Xi)))\\
	&=
	\mathrm{Ad}^{*}(h\cdot n_{l})(\nu_{l}+\iota_{\mathbb{O}_{H}}^{(l-1)}(\Xi))
\end{align*}
for $l\le k$, and finally we obtain 
\begin{align*}
	\mathrm{Ad}^{*}(gh^{-1})(\widetilde{\iota}_{\mathbb{O}_{H}}^{(l)}(\Xi))&=
	\mathrm{Ad}^{*}(gh^{-1})(\mathrm{Ad}^{*}(h\cdot n_{k})(\nu_{k}+\iota_{\mathbb{O}_{H}}^{(k-1)}(\Xi)))\\
	&=
	\mathrm{Ad}^{*}(g)(\mathrm{Ad}^{*}(n_{k})(\nu_{k}+\iota_{\mathbb{O}_{H}}^{(k-1)}(\Xi)))\\
	&=\mathrm{Ad}^{*}(g)(\iota_{\mathbb{O}_{H}}^{(l)}(\Xi))=\iota_{\mathbb{O}_{H}}(\Xi)
\end{align*}
as well.

\subsection[A Zariski open subset of the truncated orbit]{A Zariski open subset of the truncated orbit of a semisimple canonical form}\label{sec:zaropsemi}
In this section, we particularly consider a semisimple canonical form
\[
	H\,dz=\left(\frac{H_{k}}{z^{k}}+\cdots+\frac{H_{1}}{z}+H_{0}\right)\frac{dz}{z}\in
	\mathfrak{g}(\mathbb{C}[z^{-1}]_{k})\,dz,
\]
$H_{i}\in \mathfrak{t}$, $i=0,1,\ldots,k$.
Then we have $\mathrm{Stab}_{G}(H)=L_{0}$ and
$
	\mathrm{Stab}_{G(\mathbb{C}[z]_{k})}(H)=L_{0}\ltimes \mathrm{Stab}_{G(\mathbb{C}[z]_{k})_{1}}(H).	
$
Let us consider the  
Zariski open subset of $G$ defined by\index[cN]{$G_{0}$}
\[	
	G_{0}:=N_{0}\cdot P_{0}
\]
which is the open cell of the Bruhat decomposition of $G$ with respect to $P_{0}$.
Then for a non negative integer $j$,
we can also define the Zariski open subset of $G(\mathbb{C}[z]_{j})$ by \index[cN]{$G_{0}(\mathbb{C}[z]_{j})$}
\[
	G_{0}(\mathbb{C}[z]_{j}):=\{g_{0}\cdot g_{1}\in G\ltimes G(\mathbb{C}[z]_{j})_{1}\mid g_{0}\in G_{0}\}.
\]
Then we can consider the Zariski open subset \index[cN]{$\mathbb{O}^{G_{0}(\mathbb{C}[z]_{k})}_{H}$}
\[
	\mathbb{O}^{G_{0}(\mathbb{C}[z]_{k})}_{H}:=\{
		\mathrm{Ad}^{*}(g)(H)\in \mathfrak{g}(\mathbb{C}[z]_{k})^{*}\mid g\in G_{0}(\mathbb{C}[z]_{k}) 
	\}
\]
of the truncated orbit $\mathbb{O}_{H}$
and we shall show an analogue of Theorem \ref{thm:triangdecompfb} for $\mathbb{O}^{G_{0}(\mathbb{C}[z]_{k})}_{H}$.

Recall that $G_{0}(\mathbb{C}[z]_{k})$ contains $\mathrm{Stab}_{G(\mathbb{C}[z]_{k})}(H)$
and moreover $\mathrm{Stab}_{G(\mathbb{C}[z]_{k})}(H)$ acts on $G_{0}(\mathbb{C}[z]_{k})$
via the multiplication from the right. Then
we have 
\[
	\mathbb{O}^{G_{0}(\mathbb{C}[z]_{k})}_{H}\cong G_{0}(\mathbb{C}[z]_{k})/\mathrm{Stab}_{G(\mathbb{C}[z]_{k})}(H).
\]

\begin{prop}\label{prop:decompzero}
	Let $j$ be a non negative integer.
	We have the following decomposition
\begin{multline*}
	G_{0}(\mathbb{C}[z]_{j})=N_{k}(\mathbb{C}[z]_{j})\cdot N_{k-1,k}(\mathbb{C}[z]_{j})\cdot 
	\cdots N_{0,1}(\mathbb{C}[z]_{j})\cdot L_{0}(\mathbb{C}[z]_{j})\\
	\cdot U_{0,1}(\mathbb{C}[z]_{j})\cdot 
	\cdots \cdot U_{k-1,k}(\mathbb{C}[z]_{j})\cdot U_{k}(\mathbb{C}[z]_{j}).
\end{multline*}
\end{prop}
\begin{proof}
	If $j=0$, this follows from the decompositions of $N_{0}$ and $P_{0}=L_{0}\cdot U_{0}.$
	Let us consider the case $j\ge 1$.
	Let us take $g\in G_{0}(\mathbb{C}[z]_{j})$
	and write $g=g_{0}g_{1}$ by $g_{0}\in G_{0}$ and $g_{1}\in G(\mathbb{C}[z]_{j})_{1}$.
	Then along the decomposition $G_{0}=N_{0}\cdot P_{0}$
	we can write $g_{0}=n_{0}p_{0}$
	and thus we have $g=n_{0}g'_{1}p_{0}$
	with $g'_{1}:=p_{0}g_{1}p_{0}^{-1}\in G(\mathbb{C}[z]_{j})_{1}$. 
	Proposition \ref{prop:LU}
	gives us the decomposition 
	$g'_{1}=n_{1}p_{1}$
	by $n_{1}\in N_{0}(\mathbb{C}[z]_{j})_{1}$
	and $p_{1}\in P_{0}(\mathbb{C}[z]_{j})_{1}$.
	By putting $n:=n_{0}n_{1}\in N_{0}(\mathbb{C}[z]_{j})$ and 
	$p:=p_{0}(p_{0}^{-1}p_{1}p_{0})\in P_{0}(\mathbb{C}[z]_{j})$
	we obtain the decomposition 
	$g=np$.
	Namely we obtain that 
	the multiplication map 
	$N_{0}(\mathbb{C}[z]_{k})\times P_{0}(\mathbb{C}[z]_{j})
	\ni (n,p)\mapsto np\in G_{0}(\mathbb{C}[z]_{j})$
	is an isomorphism as complex manifolds, 
	where the injectivity comes from the fact
	$N_{0}(\mathbb{C}[z]_{j})\cap P_{0}(\mathbb{C}[z]_{j})=\{e\}$.

	Similarly from the decompositions 
	$N_{0}=N_{k}\cdot N_{k-1,k}\cdot \cdots N_{0,1}$
	and $P_{0}=L_{0}\cdot U_{0,1}\cdot \cdots U_{k-1,k}\cdot U_{k}$
	together with Proposition \ref{prop:LU},
	we obtain the decompositions 
	\begin{align*}
		N_{0}(\mathbb{C}[z]_{j})&=N_{k}(\mathbb{C}[z]_{j})\cdot N_{k-1,k}(\mathbb{C}[z]_{j})\cdot 
		\cdots N_{0,1}(\mathbb{C}[z]_{j}),\\
		P_{0}(\mathbb{C}[z]_{j})&=L_{0}(\mathbb{C}[z]_{j})
		\cdot U_{0,1}(\mathbb{C}[z]_{j})\cdot 
		\cdots \cdot U_{k-1,k}(\mathbb{C}[z]_{j})\cdot U_{k}(\mathbb{C}[z]_{j}).
	\end{align*}
	Thus we obtain the desired decomposition of $G_{0}(\mathbb{C}[z]_{j})$.
\end{proof}
Let us set $L_{l}^{G_{0}}:=L_{l}\cap G_{0}$\index[cN]{$L_{l}^{G_{0}}$}
and define 
\[
	L_{l}^{G_{0}}(\mathbb{C}[z]_{j}):=\{g_{0}\cdot g_{1}\in L_{l}\ltimes L_{l}(\mathbb{C}[z]_{j})_{1}
	\mid g_{0}\in L_{l}^{G_{0}}\}
\]
for each $0\le l\le k+1$.
Also define \label{cN}{$P_{l,m}^{G_{0}}(\mathbb{C}[z]_{j})$}
\begin{multline*}
	P_{l,m}^{G_{0}}(\mathbb{C}[z]_{j}):=
	N_{l-1,l}(\mathbb{C}[z]_{j})\cdot N_{l-2,l-1}(\mathbb{C}[z]_{j})\cdot 
	\cdots N_{0,1}(\mathbb{C}[z]_{j})\cdot L_{0}(\mathbb{C}[z]_{j})\\
	\cdot U_{0,1}(\mathbb{C}[z]_{j})\cdot 
	\cdots \cdot U_{l'-2,l'-1}(\mathbb{C}[z]_{j})\cdot U_{l'-1,l'}(\mathbb{C}[z]_{j})
\end{multline*}
for each pair $0\le l<l'\le k$.
Then we obtain the Levi-type decomposition
\[
	P_{l,l'}^{G_{0}}(\mathbb{C}[z]_{j})=L_{l}^{G_{0}}(\mathbb{C}[z]_{j})\cdot U_{l,l'}(\mathbb{C}[z]_{j})	
\]
and also
the $\mathrm{LU}$-type decomposition 
\[
	L_{l'}^{G_{0}}(\mathbb{C}[z]_{j})=N_{l,l'}(\mathbb{C}[z]_{j})\cdot P_{l,l'}^{G_{0}}(\mathbb{C}[z]_{j})
\]
by Proposition \ref{prop:decompzero}.

Put $\mathbb{O}^{L_{l}^{G_{0}}(\mathbb{C}[z]_{k})}_{H}:=\{
	\mathrm{Ad}^{*}(g)(H)\in \mathfrak{g}(\mathbb{C}[z]_{k})^{*}\mid g\in L_{l}^{G_{0}}(\mathbb{C}[z]_{l'}) 
\}$, and then we obtain a filtration \index[cN]{$\mathbb{O}^{L_{l}^{G_{0}}(\mathbb{C}[z]_{k})}_{H}$}
\[
	\{H\}=\mathbb{O}^{L_{0}^{G_{0}}(\mathbb{C}[z]_{k})}_{H}
	\subset \mathbb{O}^{L_{1}^{G_{0}}(\mathbb{C}[z]_{k})}_{H}
	\subset \cdots 
	\subset \mathbb{O}^{L_{k+1}^{G_{0}}(\mathbb{C}[z]_{k})}_{H}=\mathbb{O}^{G_{0}(\mathbb{C}[z]_{k})}_{H}
\] 
of $\mathbb{O}^{G_{0}(\mathbb{C}[z]_{k})}_{H}$ as well as that of $\mathbb{O}_{H}^{L_{1}\ltimes G(\mathbb{C}[z]_{k})_{1}}$
in Section \ref{sec:filtext}.

\begin{prop}\label{prop:smallorbit2}
	For each pair of integers $0\le l\le l' \le k$,
	the projection $\pi_{l',l}\colon \mathbb{C}[z]_{l'}\rightarrow \mathbb{C}[z]_{l}$
	induces the isomorphism 
	\[
		\pi^{\mathbb{O}}_{l',l}\colon \mathbb{O}_{H^{[l']}}^{L_{l+1}^{G_{0}}(\mathbb{C}[z]_{l'})}\overset{\sim}{\longrightarrow}
		\mathbb{O}_{H^{[l]}}^{L_{l+1}^{G_{0}}(\mathbb{C}[z]_{l})}.
	\]
\end{prop}
\begin{proof}
	This directly follows from the same argument as in Proposition \ref{prop:smallorbit}.
\end{proof}
Thus the above filtration becomes as follows,
\[
	\{H_{\mathrm{res}}\}=\mathbb{O}^{L_{1}^{G_{0}}(\mathbb{C}[z]_{0})}_{H^{[0]}}
	\subset \mathbb{O}^{L_{2}^{G_{0}}(\mathbb{C}[z]_{1})}_{H^{[1]}}
	\subset \cdots 
	\subset \mathbb{O}^{L_{l+1}^{G_{0}}(\mathbb{C}[z]_{l})}_{H^{[l]}}
	\subset \cdots 
	\subset \mathbb{O}^{L_{k+1}^{G_{0}}(\mathbb{C}[z]_{k})}_{H^{[k]}}=\mathbb{O}^{G_{0}(\mathbb{C}[z]_{k})}_{H}.
\]

As we saw in Section \ref{sec:ludecomp},
we can obtain decompositions of $\mathbb{O}_{H^{[l]}}^{L_{l}^{G_{0}}(\mathbb{C}[z]_{l})}$
along the Levi-type decomposition of $P_{l,l'}^{G_{0}}(\mathbb{C}[z]_{l})$
and the $\mathrm{LU}$-type one of $L_{l}^{G_{0}}(\mathbb{C}[z]_{l})$ as follows.
\begin{prop}\label{prop:NPII}
	Let us take a pair $0\le l\le l'\le k$
	of positive integers.
	The multiplication map
	\[
		\overline{\phi}\colon N_{l,l'}(\mathbb{C}[z]_{l})\times 	P^{G_{0}}_{l,l'}(\mathbb{C}[z]_{l})
		\longrightarrow L_{l'}^{G_{0}}(\mathbb{C}[z]_{l})
	\] 
	induces the 
	isomorphism 
	\[
		\phi\colon N_{l,l'}(\mathbb{C}[z]_{l})_{1}\times \mathbb{O}_{H^{[l]}}^{P_{l,l'}^{^{G_{0}}}(\mathbb{C}[z]_{l})}
		\longrightarrow \mathbb{O}_{H^{[l]}}^{ L_{l'}^{G_{0}}(\mathbb{C}[z]_{l})}
	\]
	as complex manifolds.
	Here $\mathbb{O}_{H^{[l]}}^{P_{l,l'}^{^{G_{0}}}(\mathbb{C}[z]_{l})}
	:=\{
		\mathrm{Ad}^{*}(g)(H^{[l]})\in \mathfrak{g}(\mathbb{C}[z]_{l})^{*}\mid g\in P_{l,l'}^{^{G_{0}}}(\mathbb{C}[z]_{l})
	\}$.
\end{prop}
\begin{proof}
	This directly follows from the same argument as in Proposition \ref{prop:NP}.
\end{proof}

\begin{prop}\label{prop:lu2}
	The map 
	\[
		\begin{array}{cccc}	
		X\colon 	&\mathbb{O}_{H^{[l]}}^{L_{l}^{G_{0}}(\mathbb{C}[z]_{l})}\times \mathfrak{u}_{l,l+1}(\mathbb{C}[z^{-1}]_{l})&
		\longrightarrow &\mathbb{O}_{H^{[l]}}^{P^{G_{0}}_{l,l+1}(\mathbb{C}[z]_{l})}\\
		&(\eta,\xi)&\longmapsto 
		&\eta+\xi
		\end{array}
	\]
	is an isomorphism of complex manifolds.
\end{prop}
\begin{proof}
	Let us 
 define the 
map
\[
	\begin{array}{cccc}
	a_{H^{[l]}}\colon &L_{l}^{G_{0}}(\mathbb{C}[z]_{l})\times U_{l,l+1}(\mathbb{C}[z]_{l})&\longrightarrow & L_{l}^{G_{0}}(\mathbb{C}[z]_{l})
	\times \mathfrak{u}_{l,l+1}(\mathbb{C}[z^{-1}]_{l})\\
	&(g,u)&\longmapsto &(g,\mathrm{Ad}^{*}(gu)(H^{[l]})-\mathrm{Ad}^{*}(g)(H^{[l]}))
	\end{array}
\]
and the argument in Lemma \ref{lem:prep2} tells us that this $a_{H^{[l]}}$ is an isomorphism.
	Then as well as in the proof of Proposition \ref{prop:lu}, 
	the multiplication map 
	$\chi \colon L_{l}^{G_{0}}(\mathbb{C}[z]_{l})\times U_{l,l+1}(\mathbb{C}[z])
	\rightarrow P^{G_{0}}_{l,l+1}(\mathbb{C}[z]_{l})$
	induces the commutative diagram
	\[
		\begin{tikzcd}
			L^{G_{0}}_{l}(\mathbb{C}[z]_{l})_{1}\times U_{l,l+1}(\mathbb{C}[z]_{l})
			\arrow[r,"\chi"]
			\arrow[d,"a_{H^{[l]}}"]
			&P^{G_{0}}_{l,l+1}(\mathbb{C}[z]_{l})
			\arrow[dd,"\pi_{P}"]\\
			L^{G_{0}}_{l}(\mathbb{C}[z]_{l})
	\times \mathfrak{u}_{l,l+1}(\mathbb{C}[z^{-1}]_{l})
			\arrow[d,"\pi_{L}\times \mathrm{id}"]&\\
			\mathbb{O}_{H^{[l]}}^{L^{G_{0}}_{l}(\mathbb{C}[z]_{l})}\times \mathfrak{u}_{l,l+1}(\mathbb{C}[z^{-1}]_{l})
			\arrow[r,"X"]
			&\mathbb{O}_{H^{[l]}}^{P^{G_{0}}_{l,l+1}(\mathbb{C}[z]_{l})}
		\end{tikzcd}
	\]
	which shows the surjectivity of $X$.
	
	Moreover since 
	 $\mathrm{Stab}_{L^{G_{0}}_{l}(\mathbb{C}[z]_{l})}(H^{[l]})=\mathrm{Stab}_{P^{G_{0}}_{l,l+1}(\mathbb{C}[z]_{l})}(H^{[l]})$,
	 $X$ is injective by the same argument as in the proof of Proposition \ref{prop:lu}.
\end{proof}
Then we finally obtain an analogue of 
Theorem \ref{thm:triangdecompfb}.
\begin{thm}\label{thm:triangdecompfb2}
	There exists an isomorphism 
		\[
			\Psi\colon \prod_{l=1}^{k}\left(N_{l,l+1}(\mathbb{C}[z]_{l})\times \mathfrak{u}_{l,l+1}(\mathbb{C}[z^{-1}]_{l})\right)
			\overset{\sim}{\longrightarrow} \mathbb{O}_{H}^{G_{0}(\mathbb{C}[z]_{k})}
		\]
		as complex manifolds.
\end{thm}
\begin{proof}
	As a combination of Propositions \ref{prop:smallorbit2}, \ref{prop:NPII}, and \ref{prop:lu2},
	we obtain isomorphisms
	\[
		\psi_{l}\colon N_{l,l+1}(\mathbb{C}[z]_{l})\times \mathfrak{u}_{l,l+1}(\mathbb{C}[z^{-1}]_{l})
		\times \mathbb{O}_{H^{[l-1]}}^{L^{G_{0}}_{l}(\mathbb{C}[z]_{l})}
		\longrightarrow \mathbb{O}_{H^{[l]}}^{L^{G_{0}}_{l+1}(\mathbb{C}[z]_{l})}
	\]
	for 
	$l=1,2,\ldots,k$.
	Therefore, as the combination of these isomorphisms $\psi_{l}$	
	we obtain the desired isomorphism
	\[
		\Psi\colon \prod_{l=1}^{k}\left(N_{l,l+1}(\mathbb{C}[z]_{l})\times \mathfrak{u}_{l,l+1}(\mathbb{C}[z^{-1}]_{l})\right)
		\longrightarrow \mathbb{O}_{H}^{G_{0}(\mathbb{C}[z]_{k})_{1}}
	\]
	by recalling $\mathbb{O}_{H^{[0]}}^{L_{1}^{G_{0}}(\mathbb{C}[z]_{0})}=\{H_{\mathrm{res}}\}$.
\end{proof}
	
\section{Deformation of $N_{l,l+1}(\mathbb{C}[z]_{l})_{1}$ and $\mathfrak{u}_{l,l+1}(\mathbb{C}[z^{-1}]_{l-1})$}\label{sec:deformn}
We saw in Section \ref{sec:triangle} that
$N_{l,l+1}(\mathbb{C}[z]_{l})_{1}\times \mathfrak{u}(\mathbb{C}[z^{-1}]_{l-1})$ are main
building blocks of truncated orbit $\mathbb{O}_{H}$.
In this section, we shall introduce deformations of these $N_{l,l+1}(\mathbb{C}[z]_{l})_{1}\times \mathfrak{u}(\mathbb{C}[z^{-1}]_{l-1})$
which will define the deformation of the truncated orbit $\mathbb{O}_{H}$
later.

\subsection{Deformation of $\mathbb{C}[z]_{l}$ and $\mathbb{C}[z^{-1}]_{l}$}
To each point $\mathbf{c}=(c_{0},c_{1},\ldots,c_{k})\in \mathbb{C}^{k+1}$, we associate 
the effective divisor of $\mathbb{A}^{1}(\mathbb{C})=\mathbb{C}$ defined by \index[cN]{$D(\mathbf{c})$}
\[
	D(\mathbf{c}):=c_{0}+c_{1}+\cdots +c_{k}.
\]
Let $|D(\mathbf{c})|$ be the underlying 
set of points in $\mathbb{C}$.
Then we can consider the subspace of rational functions\index[cN]{$\mathrm{L}(D(\mathbf{c}))$}
\begin{equation*}
	\mathrm{L}(D(\mathbf{c})):=\left\{f\in \mathbb{C}(z)\,\middle|\,
	\mathrm{ord}_{P}(f)\ge -n_{P}\text{ for all }P\in \mathbb{C}\right\}\\
\end{equation*}
where $n_{P}$ is the coefficient of $P\in \mathbb{C}$ in 
the formal sum $D(\mathbf{c})=\sum_{P\in\mathbb{C}}n_{P}P$, and 
$\mathrm{ord}_{P}(f)$ is the order of $f$ at $P$.
We can regard 
$\mathrm{L}(D(\mathbf{c}))$ as $\mathbb{C}[z]$-submodule of $\mathbb{C}(z)$
containing 
$\mathbb{C}[z]$. Thus we can consider 
the quotient module\index[cN]{$\widehat{\mathrm{L}}(\mathbf{c})$} 
\[
	\widehat{\mathrm{L}}(\mathbf{c}):=\mathrm{L}(D(\mathbf{c}))/\mathbb{C}[z].	
\]
We also consider 
\[
	\mathrm{L}(-D(\mathbf{c})):=\left\{f\in \mathbb{C}(z)\mid \mathrm{ord}_{P}(f)
	\ge n_{P}\text{ for all }P\in \mathbb{C}\right\}.
\]
which may be regarded
as the ideal of $\mathbb{C}[z]$ generated by
$
\prod_{i=0}^{k}(z-c_{i}). 	
$
Thus we can consider the 
quotient ring \index[cN]{$\widehat{\Gamma}(\mathbf{c})$}
\[
	\widehat{\Gamma}(\mathbf{c}):=\mathbb{C}[z]/\mathrm{L}(-D(\mathbf{c}))
	=\mathbb{C}[z]/\langle \prod_{i=0}^{k}(z-c_{i})\rangle_{\mathbb{C}[z]}.	
\]
Then $\widehat{\Gamma}(\mathbf{c})$ is 
a finite-dimensional $\mathbb{C}$-algebra with 
maximal ideals \index[cN]{$\mathfrak{m}_{z_{c_{i}}}^{\widehat{\Gamma}(\mathbf{c})}$}$\mathfrak{m}_{z_{c_{i}}}^{\widehat{\Gamma}(\mathbf{c})}:=\langle z-c_{i}\rangle_{\widehat{\Gamma}(\mathbf{c})}$
for $i=0,1,\ldots,k$.
Let us note that these maximal ideals $\mathfrak{m}_{z_{c_{i}}}^{\widehat{\Gamma}(\mathbf{c})}$
are not nilpotent in general, unlike the maximal ideal $\mathfrak{m}_{z}^{(l)}$
of the local ring $\mathbb{C}[z]_{l}$.
We notice that since $\mathrm{L}(-D(\mathbf{c}))
\subset\mathbb{C}[z]$ acts trivially on $\widehat{\mathrm{L}}(\mathbf{c})$,
$\widehat{\mathrm{L}}(\mathbf{c})$ has the  
$\widehat{\Gamma}(\mathbf{c})$-module structure.
Also we note that 
if $\mathbf{c}=\mathbf{0}$, then  
\[
	\widehat{\Gamma}(\mathbf{0})\cong\mathbb{C}[z]_{k},\quad\quad
	\widehat{\mathrm{L}}(\mathbf{0})\cong \mathbb{C}[z^{-1}]_{k}.	
\]
Therefore $\widehat{\Gamma}(\mathbf{c})$ and $\widehat{\mathrm{L}}(\mathbf{c})$
can be regarded  
as deformations of $\mathbb{C}[z]_{k}$ and $\mathbb{C}[z^{-1}]_{k}$ respectively.

Let us define a filtration on  $\widehat{\mathrm{L}}(\mathbf{c})$
and cofiltration on $\widehat{\Gamma}(\mathbf{c})$ as follows.
Consider projection maps $\mathrm{pr}^{(l)}\colon \mathbb{C}^{k+1}\ni(x_{0},x_{1},\ldots,x_{k})\mapsto 
(x_{0},x_{1},\ldots,x_{l})\in \mathbb{C}^{l+1}$
and define divisors \index[cN]{$D(\mathbf{c})_{l}$}
$D(\mathbf{c})_{l}:=D(\mathrm{pr}^{(l)}(\mathbf{c}))$  for $l=0,1,\ldots,k$.
Then we obtain the filtration
\[
	\mathrm{L}(D(\mathbf{c})_{0})\subset \mathrm{L}(D(\mathbf{c})_{1})\subset 
	\cdots \subset \mathrm{L}(D(\mathbf{c})_{k})=\mathrm{L}(D(\mathbf{c})).	
\]
\begin{df}[Standard filtration and basis of $\widehat{\mathrm{L}}(\mathbf{c})$]\normalfont
	Let \index[cN]{$\widehat{\mathrm{L}}(\mathbf{c})_{l}$}
	\[
		\widehat{\mathrm{L}}(\mathbf{c})_{l}:=\mathrm{L}(D(\mathbf{c})_{l})/
		\mathbb{C}[z]\text{ for } l=0,1,\ldots,k.	
	\] The 
	filtration 
	\[
		\widehat{\mathrm{L}}(\mathbf{c})_{0}\subset \widehat{\mathrm{L}}(\mathbf{c})_{1}
		\subset \cdots \subset \widehat{\mathrm{L}}(\mathbf{c})_{k}=\widehat{\mathrm{L}}(\mathbf{c})
	\]
	is called the {\em standard filtration} of $\widehat{\mathrm{L}}(\mathbf{c})$.
	The associated basis of $\widehat{\mathrm{L}}(\mathbf{c})$,
	\[
		\frac{1}{z-c_{0}},\,\frac{1}{(z-c_{0})(z-c_{1})},\ldots,\frac{1}{(z-c_{0})(z-c_{1})\cdots(z-c_{k})}	
	\]
	as $\mathbb{C}$-vector space is called the {\em standard basis} of $\widehat{\mathrm{L}}(\mathbf{c})$.
\end{df} 
Notice that $$\frac{1}{z-c_{0}},\,\frac{1}{(z-c_{0})(z-c_{1})},\ldots,\frac{1}{(z-c_{0})(z-c_{1})\cdots(z-c_{l})}$$
becomes a basis of the $l$-th component $\widehat{\mathrm{L}}(\mathbf{c})_{l}$
of the standard filtration for each $l=0,1,\ldots,k$.

Similarly, we can consider the filtration 
\[
	\mathrm{L}(-D(\mathbf{c})_{0})
	\supset 	\mathrm{L}(-D(\mathbf{c})_{1})
	\supset \cdots 
	\supset 
	\mathrm{L}(-D(\mathbf{c})_{k})=\mathrm{L}(-D(\mathbf{c})).
\]
\begin{df}[Standard cofiltration and basis of $\widehat{\Gamma}(\mathbf{c})$]\normalfont
	Let \index[cN]{$\widehat{\Gamma}(\mathbf{c})_{l}$}
	\[
		\widehat{\Gamma}(\mathbf{c})_{l}:=\mathbb{C}[z]/L(-D(\mathbf{c})_{l})
		=\mathbb{C}[z]/\langle\prod_{i=0}^{l}(z-c_{i})\rangle_{\mathbb{C}[z]}
		\text{ for }l=0,1,\ldots,k.
	\]
	Then the sequence of projection maps 
	\[
		\widehat{\Gamma}(\mathbf{c})_{0}\rightarrow \widehat{\Gamma}(\mathbf{c})_{1}
		\rightarrow\cdots\rightarrow \widehat{\Gamma}(\mathbf{c})_{k}=\widehat{\Gamma}(\mathbf{c})
	\]
	induced by the above filtration is called the 
	{\em standard cofiltration} of $\widehat{\Gamma}(\mathbf{c})$.
	The basis of $\widehat{\Gamma}(\mathbf{c})$,
	\[
	1, (z-c_{0}), (z-c_{0})(z-c_{1}),\ldots, (z-c_{0})(z-c_{1})\cdots (z-c_{k-1})
	\]
	as $\mathbb{C}$-vector space is called the {\em standard basis} of $\widehat{\Gamma}(\mathbf{c})$.
\end{df}
We can see that
$$1, (z-c_{0}), (z-c_{0})(z-c_{1}),\ldots, (z-c_{0})(z-c_{1})\cdots (z-c_{l-1})$$
becomes a basis of $\widehat{\Gamma}(\mathbf{c})_{l}$ for each $l=0,1,\ldots,k$.

Next we introduce the pairing of $\widehat{\Gamma}(\mathbf{c})_{l}$ and $\widehat{\mathrm{L}}(\mathbf{c})_{l}$
as $\mathbb{C}$-vector spaces defined by \index[cN]{$\langle\ ,\ \rangle_{\mathbf{c},l}$}
\[
	\begin{array}{cccc}
	\langle\ ,\ \rangle_{\mathbf{c},l}\colon 	&\widehat{\Gamma}(\mathbf{c})_{l} \times \widehat{\mathrm{L}}(\mathbf{c})_{l}&
	\longrightarrow &\mathbb{C}\\
	&(f(z),g(z))&\longmapsto &\displaystyle \sum_{c\,\in |D(\mathbf{c})|}\underset{z=c}{\mathrm{res\,}}(f(z)g(z))
	\end{array}
\]
for each $l=0,1,\ldots,k$. 
\begin{lem}[Residue formula]\label{lem:residue}
	Let us take $f(z)\in \widehat{\mathrm{L}}(\mathbf{c})_{l}$ and write  
	$$f(z)=\sum_{i=0}^{l}\frac{f_{i}}{\prod_{\nu=0}^{i}(z-c_{\nu})}$$
	under the standard basis.
	Then we have 
	\[
		\sum_{c\,\in |D(\mathbf{c})|}\underset{z=c}{\mathrm{res\,}}(f(z))=f_{0}.
	\]
\end{lem}
\begin{proof}
	It follows from a direct computation.
\end{proof}
\begin{prop}\label{prop:nondeg}
	The pairing $\langle\ ,\ \rangle_{\mathbf{c},l}$ is non-degenerate.
	Moreover, the bases 
	$$1, (z-c_{0}), (z-c_{0})(z-c_{1}),\ldots, (z-c_{0})(z-c_{1})\cdots (z-c_{l-1})$$ 
	of 
	$\widehat{\Gamma}(\mathbf{c})_{l}$
	and 
	$$\frac{1}{z-c_{0}},\,\frac{1}{(z-c_{0})(z-c_{1})},\ldots,\frac{1}{(z-c_{0})(z-c_{1})\cdots(z-c_{l})}$$
	of 
	$\widehat{\mathrm{L}}(\mathbf{c})_{l}$
	are dual bases with respect to this pairing.
\end{prop}
\begin{proof}
	It follows that 
	\[
		\left\langle 1, \frac{1}{(z-c_{0})(z-c_{1})\cdots(z-c_{j'})}\right\rangle_{\mathbf{c},l}=\delta_{0,j'},
	\] for $0\le j'\le l$ and 
	\[
		\left\langle (z-c_{0})(z-c_{1})\cdots (z-c_{j-1}),	\frac{1}{(z-c_{0})(z-c_{1})\cdots(z-c_{j'})}\right\rangle_{\mathbf{c},l}=\delta_{j,j'},
	\]
	for $0<j\le l,\,0\le j'\le l,$
	from Lemma \ref{lem:residue}. Here $\delta_{j,j'}$ is Kronecker delta.
\end{proof}
\begin{cor}\label{cor:nondeg}
	Let us take $\mathbf{c}=(c_{i})_{i=0,1,\ldots,l}$, $\mathbf{x}=(x_{i})_{i=0,1,\ldots,l}$, and $\mathbf{y}=(y_{i})_{i=0,1,\ldots,l}$
	from $\mathbb{C}^{\oplus (l+1)}$.
	Define 
	\begin{align*}
		f(\mathbf{x},\mathbf{c})&:=x_{0}+x_{1}(z-c_{0})+\cdots+ x_{l}(z-c_{0})(z-c_{1})\cdots (z-c_{l-1}),\\
		g(\mathbf{x},\mathbf{c})&:=\frac{x_{0}}{z-c_{0}}+\frac{x_{1}}{(z-c_{0})(z-c_{1})}+\cdots+ \frac{x_{l}}{(z-c_{0})(z-c_{1})\cdots (z-c_{l})}.
	\end{align*}
	Then  
	\[
		\langle f(\mathbf{x},\mathbf{c}),g(\mathbf{x},\mathbf{c})\rangle_{\mathbf{c},l}
	\]
	is constant as a function of $\mathbf{c}\in \mathbb{C}^{\oplus (l+1)}$.  
\end{cor}
\begin{proof}
	This is a direct consequence of the above proposition.
\end{proof}
\subsection{Partial fraction decomposition and Chinese remainder theorem}\label{sec:partialfrac}
Let us take a partition $\mathcal{I}\colon I_{0},I_{1},\ldots,I_{r}$
of $\{0,1,\ldots,k\}$. . 
For each $l=0,1,\ldots,k$, define a partition\index[cN]{$\mathcal{I}^{(l)}$} \index[cN]{$I_{j}^{(l)}$}
 $\mathcal{I}^{(l)}\colon I_{0}^{(l)},I_{1}^{(l)},\ldots,I_{r}^{(l)}$ of $\{0,1,\ldots,l\}\subset
\{0,1,\ldots,k\}$ by 
$I^{(l)}_{j}:=I_{j}\cap \{0,1,\ldots,l\}$.
Here we allow the case $I^{(l)}_{j}=\emptyset$.
Let us set \index[cN]{$l_{j}$}
$
	l_{j}:=|I^{(l)}_{j}|-1.	
$

Then for  $\mathbf{c}=(c_{0},c_{1},\ldots,c_{k})\in C(\mathcal{I})\cong 
\{(c_{I_{0}},\ldots,c_{I_{r}})\in \mathbb{C}^{r+1}\}$,
we can write  
$D(\mathbf{c})_{l}=\sum_{\nu=0}^{r}(l_{j}+1)\cdot c_{i_{[j,0]}}$.
The algorithm of the partial fractional decomposition of rational functions gives us the direct sum decomposition
\[
	\widehat{\mathrm{L}}(\mathbf{c})_{l}=
\bigoplus_{j=0}^{r}\mathbb{C}[z_{c_{I_{j}}}^{-1}]_{l_{j}}.
\]
Here we formally put $\mathbb{C}[z_{c_{I_{j}}}^{-1}]_{l_{j}}:=\{0\}$
if
$l_{j}<0$, i.e., $I_{j}^{(l)}=\emptyset$.
We denote natural projection maps 
by $\mathrm{pr}_{c_{I_{j}}}\colon \widehat{\mathrm{L}}(\mathbf{c})_{l}\rightarrow \mathbb{C}[z_{c_{I_{j}}}^{-1}]_{l_{j}}$
for $j=0,1,\ldots,r.$

We also have the following decomposition of $\widehat{\Gamma}(\mathbf{c})_{l}$.
Since $c_{I_{j}}\neq c_{I_{j'}}$ for $j\neq j'$, 
it follows that 
\begin{align*}
	&\langle (z-c_{I_{j}})^{l_{j}} \rangle_{\mathbb{C}[z]}+\langle (z-c_{I_{j'}})^{l_{j'}} \rangle_{\mathbb{C}[z]}
	=\mathbb{C}[z],\quad j\neq j'\text{ with }I_{j}\neq \emptyset, I_{j'}\neq \emptyset,\\
	&\bigcap_{j=0}^{r}\langle (z-c_{I_{j}})^{l_{j}} \rangle_{\mathbb{C}[z]}
	=\mathrm{L}(-D(\mathbf{c})_{l}).
\end{align*}
Therefore the Chinese remainder theorem implies that 
natural projection maps 
\[
	\mathrm{pr}_{c_{I_{j}}}\colon \widehat{\Gamma}(\mathbf{c})_{l}\longrightarrow
\mathbb{C}[z]/\langle(z-c_{I_{j}})^{l_{j}+1}\rangle_{\mathbb{C}[z]}=
\mathbb{C}[z_{c_{I_{j}}}]_{l_{j}}
\]
define the algebra isomorphism
\[
	\begin{array}{cccc}
	\prod_{j=0}^{r}\mathrm{pr}_{c_{I_{j}}}\colon&\widehat{\Gamma}(\mathbf{c})_{l}&\longrightarrow &\prod_{j=0}^{r}\mathbb{C}[z_{c_{I_{j}}}]_{l_{j}}\\
	&f(z)& \longmapsto & (\mathrm{pr}_{c_{I_{j}}}(f(z)))_{j=0,1,\ldots,r}
	\end{array}.
\]
Here 
we formally put $\mathbb{C}[z_{c_{I_{j}}}]_{l_{j}}:=\{0\}$
if
$l_{j}<0$, i.e., $I_{j}^{(l)}=\emptyset$.

Under this isomorphism, maximal ideals $\mathfrak{m}_{z_{c_{I_{j}}}}^{\widehat{\Gamma}(\mathbf{c})_{l}}$
for $j=0,1,\ldots,r$ correspond to 
\[
	\mathbb{C}[z_{c_{I_{0}}}]_{l_{0}}\times \cdots \times \mathbb{C}[z_{c_{I_{j-1}}}]_{l_{j-1}}\times \mathfrak{m}_{z_{c_{I_{j}}}}^{(l_{j})}	
	\times \mathbb{C}[z_{c_{I_{j+1}}}]_{l_{j+1}}\times \cdots \times \mathbb{C}[z_{c_{I_{r}}}]_{l_{r}}
\]
in $\prod_{j=0}^{r}\mathbb{C}[z_{c_{I_{j}}}]_{l_{j}}$.

These decompositions of $\widehat{\Gamma}(\mathbf{c})_{l}$ and $\widehat{\mathrm{L}}(\mathbf{c})_{l}$
are compatible with 
the pairing $\langle\ ,\ \rangle_{\mathbf{c},l}$.
Namely, 
for each $j=0,1,\ldots,r$, we define the non-degenerate pairing 
\[
	\begin{array}{cccc}
	\langle\ ,\ \rangle_{c_{I_{j}},l}&\mathbb{C}[z_{c_{I_{j}}}]_{l_{j}}\times \mathbb{C}[z_{c_{I_{j}}}^{-1}]_{l_{j}}	&\longrightarrow &\mathbb{C}\\
	&(f(z_{c_{I_{j}}}),g(z_{c_{I_{j}}}))&\longmapsto &
	\underset{z_{c_{I_{j}}}=0}{\mathrm{res}}(f(z_{c_{I_{j}}})g(z_{c_{I_{j}}}))
	\end{array}.
\]
Then a direct computation shows the following.
\begin{prop}
	Let us take $\mathcal{I}\in \mathcal{P}_{[k+1]}$,
	 $\mathbf{c}\in C(\mathcal{I})$, and an integer $0\le l\le k$ as above.
	Then for $(f(z),g(z))\in \widehat{\Gamma}(\mathbf{c})_{l} \times \widehat{\mathrm{L}}(\mathbf{c})_{l}$,
	we have 
	\[
		\langle f(z),g(z)\rangle_{\mathbf{c},l}=\sum_{j=0}^{r}\langle\mathrm{pr}_{c_{I_{j}}}(f(z)),\mathrm{pr}_{c_{I_{j}}}(g(z)) \rangle_{c_{I_{j}},l}.
	\]
\end{prop}

Let us notice that the decomposition $\widehat{\mathrm{L}}(\mathbf{c})_{l}=
\bigoplus_{j=0}^{r}\mathbb{C}[z_{c_{I_{j}}}^{-1}]_{l_{j}}
$ is not just as a vector space but as a $\widehat{\Gamma}(\mathbf{c})_{l}$-module.
Furthermore, this is compatible with the decomposition of the algebra
$\widehat{\Gamma}(\mathbf{c})_{l}=\prod_{j=0}^{r}\mathbb{C}[z_{c_{I_{j}}}]_{l_{j}}$
in the following sense. Namely, we can easily check that
the diagrams
\begin{equation}\label{eq:fracchin}
	\begin{tikzcd}
		\widehat{\Gamma}(\mathbf{c})_{l}\times \widehat{\mathrm{L}}(\mathbf{c})_{l}
		\arrow[d,"\mathrm{pr}_{c_{I_{j}}}\times \mathrm{pr}_{c_{I_{j}}}"]\arrow[r,"\cdot "] &\widehat{\mathrm{L}}(\mathbf{c})_{l} \arrow[d,"\mathrm{pr}_{c_{I_{j}}}"]\\
			\mathbb{C}[z_{c_{I_{j}}}]_{l_{j}}\times\mathbb{C}[z_{c_{I_{j}}}^{-1}]_{l_{j}}  \arrow[r,"\cdot "]&\mathbb{C}[z_{c_{I_{j}}}^{-1}]_{l_{j}}
	\end{tikzcd}	
\end{equation}
are commutative for $j=0,1,\ldots,r$ and $l=0,1,\ldots,k$. 
Here horizontal arrows are scalar multiplications of left modules.

\subsection{Lie groupoid $N_{l,l+1}(\widehat{\Gamma}(\mathbb{B})_{l})$}\label{sec:liegroupoid}
Let us  recall the notion of Lie algebroids.
\begin{df}[Complex Lie algebroid]\normalfont
	A {\em complex Lie algebroid over a complex manifold }M is a holomorphic vector bundle 
	$A\rightarrow M$, together with a Lie bracket $[\cdot,\cdot]$ on its space of holomorphic 
	sections, such that there exists a vector bundle map
	$\mathbf{a}\colon A\rightarrow TM$ called the {\em anchor map} satisfying the Leibniz rule
	\[
		[\sigma,f\tau]=f[\sigma,\tau]+(\mathbf{a}(\sigma)f)\tau	
	\]
	for all global sections $\sigma,\tau$ of $A\rightarrow M$ and holomorphic functions $f$ on $M$.
\end{df}
As well as the theory of Lie groups, we can associate the Lie algebroid to a Lie groupoid as follows.
\begin{df}[Lie algebroid of a Lie groupoid]\normalfont
	 For a complex Lie groupoid $M=(M,M_{0},s,t,m)$,
	 we consider the vector bundle 
	 \[
		\mathrm{Lie}(M):=TM|_{M_{0}}/TM_{0}
	 \]
	 defined as the normal bundle to $M_{0}$ in $M$. 
	 As the anchor map, we consider the unique map
	 $\mathbf{a}\colon \mathrm{Lie}(M)\rightarrow TM_{0}$
	 which makes the following diagram commutative, 
	 \[
	\begin{tikzcd}
		TM \arrow[d] \arrow[r,"s_{*}-t_{*} "] &TM_{0}\\
		\mathrm{Lie}(M)\arrow[ru,"\mathbf{a}"]&
	\end{tikzcd}.
	\]

	The Lie bracket is defined as follows.
	Let $(S_{t})$ and $(R_{t})$ be families of bisections of $M$ 
	such that 	
	$\sigma=\frac{\partial}{\partial t}|_{t=0}S$ 
	of bisections such that $\tau=\frac{\partial}{\partial t}|_{t}R_{t}$.
	Then the Lie bracket of $\sigma$ and $\tau$ is defined by 
	\[
		[\sigma,\tau]:=\frac{\partial}{\partial t_{2}}\Big|_{t_{2}=0}\left(
			\frac{\partial}{\partial t_{1}}\Big|_{t_{1}=0}(S_{t_{2}}R_{t_{1}}S^{-1}_{t_{2}})
		\right).
	\]
\end{df}

Now let us fix a positive integer $1\le l\le k$
and  an open neighborhood $\mathbb{B}\subset \mathbb{C}^{k+1}$  of $\mathbf{0}$.
Then we consider the trivial vector bundle 
\[
	\mathfrak{g}^{\oplus (l+1)}\times \mathbb{B} \xrightarrow[]{\mathrm{pr}_{2}}	\mathbb{B}.
\]
This can be identified with the family 
$\displaystyle \bigsqcup_{\mathbf{c}\in \mathbb{B}}\mathfrak{g}(\widehat{\Gamma}(\mathbf{c})_{l})$
of Lie algebras by the following bijection,
\[
	\begin{array}{ccc}
	\mathfrak{g}^{\oplus (l+1)}\times \mathbb{B}&\longmapsto &\bigsqcup_{\mathbf{c}\in \mathbb{B}}\mathfrak{g}(\widehat{\Gamma}(\mathbf{c})_{l})\\
	((X_{i})_{i=0,1,\ldots,l},\mathbf{c})&\longmapsto&
	X_{0}+X_{1}(z-c_{0})+\cdots +X_{l}(z-c_{0})\cdots(z-c_{l-1})
	\end{array}.
\]
Then the Lie algebra structure of each direct summand $\mathfrak{g}(\widehat{\Gamma}(\mathbf{c})_{l})$
allows us to 
regard the trivial bundle $\mathfrak{g}^{\oplus (l+1)}\times \mathbb{B}\rightarrow \mathbb{B}$
as a Lie algebroid with the trivial anchor map.
We denote this Lie algebroid by \index[cN]{$\mathfrak{g}(\widehat{\Gamma}(\mathbb{B})_{l})$}
\[
	\mathfrak{g}(\widehat{\Gamma}(\mathbb{B})_{l}).	
\]
Furthermore, let us 
consider the nilpotent Lie subalgebra 
$\mathfrak{n}_{l,l+1}\subset \mathfrak{g}$ and 
the subbundle $\left(\mathfrak{n}_{l,l+1}\right)^{\oplus (l+1)}\times \mathbb{B} \xrightarrow[]{\mathrm{pr}_{2}}	\mathbb{B}$.
Then it also has a Lie algebroid structure.
Then we denote the Lie algebroid $\left(\mathfrak{n}_{l,l+1}\right)^{\oplus (l+1)}\times \mathbb{B} \xrightarrow[]{\mathrm{pr}_{2}}	\mathbb{B}$
by  \index[cN]{$\mathfrak{n}_{l,l+1}(\widehat{\Gamma}(\mathbb{B})_{l})$}
\[
	\mathfrak{n}_{l,l+1}(\widehat{\Gamma}(\mathbb{B})_{l}).
\]
Since each fiber of $\mathfrak{n}_{l,l+1}(\widehat{\Gamma}(\mathbb{B})_{l})$
 is nilpotent Lie algebra $\mathfrak{n}_{l,l+1}(\widehat{\Gamma}(\mathbf{c})_{l})$,
we can moreover equip this trivial bundle with a Lie groupoid structure as follows.
As the source $s$ and target $t$, we take the projection $\mathrm{pr}_{2}\colon \left(\mathfrak{n}_{l,l+1}\right)^{\oplus (l+1)}\times \mathbb{B}\rightarrow 
\mathbb{B}$, i.e., $s=t=\mathrm{pr}_{2}$.
As the unit $i$, we take the zero section $\mathbf{0}\colon \mathbb{B}\ni \mathbf{c}\mapsto (\mathbf{0},\mathbf{c})\in \left(\mathfrak{n}_{l,l+1}\right)^{\oplus (l+1)}\times \mathbb{B}.$
Then for each $\mathbf{c}\in \mathbb{B}$
we identify the fiber $\left(\mathfrak{n}_{l,l+1}\right)^{\oplus (l+1)}\times \{\mathbf{c}\}$
with the nilpotent Lie algebra $\mathfrak{n}_{l,l+1}(\widehat{\Gamma}(\mathbf{c}))$,
and define 
the multiplication of $X_{\mathbf{c}}, Y_{\mathbf{c}}\in \left(\mathfrak{n}_{l,l+1}\right)^{\oplus (l+1)}\times \{\mathbf{c}\}$ 
by Dynkin's formula 
\begin{equation*}
	m(X_{\mathbf{c}},Y_{\mathbf{c}}):=
		\sum_{n=1}^{\infty}\frac{(-1)^{n-1}}{n}\sum_{(p_{1},q_{1})\in (\mathbb{Z}_{\ge 0})^{2}\backslash\{0\}}
	\cdots \sum_{(p_{n},q_{n})\in (\mathbb{Z}_{\ge 0})^{2}\backslash\{0\}}
	\frac{[X_{\mathbf{c}}^{p_{1}}Y_{\mathbf{c}}^{q_{1}}\cdots X_{\mathbf{c}}^{p_{n}}Y_{\mathbf{c}}^{q_{n}}]}
	{(\sum_{i=1}^{n}(p_{i}+q_{i}))\prod_{i=1}^{n}p_{i}!q_{i}!},
\end{equation*}
where 
\[
	[X_{\mathbf{c}}^{p_{1}}Y_{\mathbf{c}}^{q_{1}}\cdots X_{\mathbf{c}}^{p_{n}}Y_{\mathbf{c}}^{q_{n}}]:=
	(\mathrm{ad}X_{\mathbf{c}})^{p_{1}}(\mathrm{ad}Y_{\mathbf{c}})^{q_{1}}
	\cdots (\mathrm{ad}X_{\mathbf{c}})^{p_{n}}(\mathrm{ad}Y_{\mathbf{c}})^{q_{n}-1}Y_{\mathbf{c}}.
\]
Since $\mathfrak{n}_{l,l+1}(\widehat{\Gamma}(\mathbf{c})_{l})$ is nilpotent,
Dynkin's formula becomes a finite sum, namely, the multiplication is well-defined.
This multiplication map obviously depends holomorphically on $\mathbf{c}\in \mathbb{B}$,
and therefore the Campbell-Baker-Hausdorff formula shows that three properties in Definition \ref{df:groupoid} are satisfied. 
We denote this Lie groupoid by \index[cN]{$N_{l,l+1}(\widehat{\Gamma}(\mathbb{B})_{l})$}
\[
	N_{l,l+1}(\widehat{\Gamma}(\mathbb{B})_{l}).
\]
The fiber at each $\mathbf{c}\in \mathbb{B}$
is isomorphic to the nilpotent Lie group $N_{l,l+1}(\widehat{\Gamma}(\mathbf{c})_{l})$ by definition.
Obviously the Lie algebroid of this Lie groupoid is $\mathfrak{n}_{l,l+1}(\widehat{\Gamma}(\mathbb{B})_{l})$.
We denote the identity map 
by $\mathrm{exp}\colon \mathfrak{n}_{l,l+1}(\widehat{\Gamma}(\mathbb{B})_{l}) \rightarrow N_{l,l+1}(\widehat{\Gamma}(\mathbb{B})_{l})$
for the later convenience.

The adjoint action of the Lie group $N_{l,l+1}(\widehat{\Gamma}(\mathbf{c})_{l})$
on the Lie algebra $\mathfrak{g}(\widehat{\Gamma}(\mathbf{c})_{l})$ at each $\mathbf{c}\in \mathbb{B}$
defines 
the adjoint action of the Lie groupoid $N_{l,l+1}(\widehat{\Gamma}(\mathbb{B})_{l})$
on the Lie algebroid $\mathfrak{g}(\widehat{\Gamma}(\mathbb{B})_{l})$.
Namely we obtain the bundle map 
\[
	\begin{array}{cccc}
		\mathrm{Ad}\colon &N_{l,l+1}(\widehat{\Gamma}(\mathbb{B})_{l})\times \mathfrak{g}(\widehat{\Gamma}(\mathbb{B})_{l})
		&\longrightarrow &\mathfrak{g}(\widehat{\Gamma}(\mathbb{B})_{l})\\
		&(n_{\mathbf{c}},X_{\mathbf{c}})&
		\longmapsto&\mathrm{Ad}(n_{\mathbf{c}})(X_{\mathbf{c}})
	\end{array}	
\]
where $N_{l,l+1}(\widehat{\Gamma}(\mathbb{B})_{l})\times\mathfrak{g}(\widehat{\Gamma}(\mathbb{B})_{l})$
denotes the product bundle over $\mathbb{B}$.

Let us introduce another realization of the trivial bundle $\mathfrak{g}^{\oplus (l+1)}\times 
\mathbb{B}\rightarrow \mathbb{B}$. We 
identify it with the family
$\bigsqcup_{\mathbf{c}\in \mathbb{B}}\mathfrak{g}(\widehat{\mathrm{L}}(\mathbf{c})_{l})$
by 
\[
	\begin{array}{ccc}
	\mathfrak{g}^{\oplus (l+1)}\times \mathbb{B}&\longmapsto &\bigsqcup_{\mathbf{c}\in \mathbb{B}}\mathfrak{g}(\widehat{\mathrm{L}}(\mathbf{c})_{l})\\
	((X_{i})_{i=0,1,\ldots,l},\mathbf{c})&\longmapsto&
	\displaystyle \frac{X_{0}}{z-c_{0}}+\frac{X_{1}}{(z-c_{0})(z-c_{1})}+\cdots +\frac{X_{l}}{(z-c_{0})\cdots(z-c_{l})}
	\end{array}.
\]
Then we denote this trivial bundle by \index[cN]{$\mathfrak{g}(\widehat{\mathrm{L}}(\mathbb{B})_{l})$}
\[
	\mathfrak{g}(\widehat{\mathrm{L}}(\mathbb{B})_{l}).	
\]
For each $\mathbf{c}\in \mathbb{B}$, we have the non-degenerate 
trace pairing 
\[
	\begin{array}{cccc}
		\langle\,,\,\rangle_{\mathbf{c},l}
		\colon &\mathfrak{g}(\widehat{\Gamma}(\mathbf{c})_{l})
		\times \mathfrak{g}(\widehat{\mathrm{L}}(\mathbf{c})_{l})
		&\longrightarrow &\mathbb{C}\\
		&(f(z),g(z))&\longmapsto&\displaystyle
		\mathrm{tr\,}\left(\sum_{c\in |D(\mathbf{c})|}\underset{z=c}{\mathrm{res\,}}(f(z)g(z))\right)
	\end{array}	.
\]
Thus we can regard $\mathfrak{g}(\widehat{\mathrm{L}}(\mathbb{B})_{l})$
as the dual bundle of the Lie algebroid
$\mathfrak{g}(\widehat{\Gamma}(\mathbb{B})_{l})$
through these pairings.

The residue map 
\[
	\begin{array}{cccc}
		\mathrm{res}\colon &\mathfrak{g}(\widehat{\mathrm{L}}(\mathbf{c})_{l})&
		\longrightarrow &\mathfrak{g}\\
		&f(z)&\longmapsto&\displaystyle \sum_{c\in |D(\mathbf{c})|}\underset{z=c}{\mathrm{res\,}}(f(z))
	\end{array}	
\]
at each $\mathbf{c}\in \mathbb{B}$ defines the 
residue map for the family  
\[
	\mathrm{res}\colon \mathfrak{g}(\widehat{\mathrm{L}}(\mathbb{B})_{l})\longrightarrow 
	\mathfrak{g}.
\]

The fiber-wise coadjoint action defines the 
action of $N_{l,l+1}(\widehat{\Gamma}(\mathbb{B})_{l})$
on  $\mathfrak{g}(\widehat{\mathrm{L}}(\mathbb{B})_{l})
\cong \mathfrak{g}(\widehat{\Gamma}(\mathbb{B})_{l})^{*}$,
\[
	\begin{array}{cccc}
		\mathrm{Ad}^{*}\colon &N_{l,l+1}(\widehat{\Gamma}(\mathbb{B})_{l})\times \mathfrak{g}(\widehat{\mathrm{L}}(\mathbb{B})_{l})
		&\longrightarrow &\mathfrak{g}(\widehat{\mathrm{L}}(\mathbb{B})_{l})\\
		&(n_{\mathbf{c}},\xi_{\mathbf{c}})&
		\longmapsto&\mathrm{Ad}^{*}(n_{\mathbf{c}})(\xi_{\mathbf{c}})
	\end{array}.
\]
Here $N_{l,l+1}(\widehat{\Gamma}(\mathbb{B})_{l})\times \mathfrak{g}(\widehat{\mathrm{L}}(\mathbb{B})_{l})$
denotes product bundle over $\mathbb{B}$.

Moreover
take the opposite $\mathfrak{u}_{l,l+1}$ of $\mathfrak{n}_{l,l+1}$. Then we consider
the subbundle $\left(\mathfrak{u}_{l,l+1}\right)^{\oplus (l+1)}\times 
\mathbb{B}\rightarrow \mathbb{B}$ of $\mathfrak{g}(\widehat{\mathrm{L}}(\mathbb{B})_{l})$
and denote it by \index[cN]{$\mathfrak{u}_{l,l+1}(\widehat{\mathrm{L}}(\mathbb{B})_{l})$}
\[
	\mathfrak{u}_{l,l+1}(\widehat{\mathrm{L}}(\mathbb{B})_{l}).	
\]

\subsection{Decompositions of fibers of $N_{l,l+1}(\widehat{\Gamma}(\mathbb{B})_{l})$}\label{sec:decompfib}
Let us 
take a partition $\mathcal{I}\colon I_{0},I_{1},\ldots,I_{r}$ of $\{0,1,\ldots,k\}$
 as in Section \ref{sec:partstra}.
Then as we saw in Section \ref{sec:partialfrac},
we can define the partition $\mathcal{I}^{(l)}\colon I_{0}^{(l)},I_{1}^{(l)},\ldots,I_{r}^{(l)}$
of $\{0,1,\ldots,l\}$ from $\mathcal{I}$. Also we take a fixed 
$\mathbf{c}\in C(\mathcal{I})\cap \mathbb{B}$.

Then
the isomorphism 
$\prod_{j=0}^{r}\mathrm{pr}_{c_{I_{j}}}\colon \widehat{\Gamma}(\mathbf{c})_{l}\overset{\sim}{\rightarrow} 
\prod_{j=0}^{r}\mathbb{C}[z_{c_{I_{j}}}]_{l_{j}}$
given in Section \ref{sec:partialfrac}
induces the isomorphism 
\[
	\prod_{j=0}^{r}\mathrm{pr}_{c_{I_{j}}}\colon \mathfrak{g}(\widehat{\Gamma}(\mathbf{c})_{l})
	\overset{\sim}{\longrightarrow}
	\prod_{j=0}^{r}\mathfrak{g}(\mathbb{C}[z_{c_{I_{j}}}]_{l_{j}})
\]
of Lie algebras which restricts to the isomorphism
\[
	\prod_{j=0}^{r}\mathrm{pr}_{c_{I_{j}}}\colon \mathfrak{n}_{l,l+1}(\widehat{\Gamma}(\mathbf{c})_{l})
	\overset{\sim}{\longrightarrow}
	\prod_{j=0}^{r}\mathfrak{n}_{l,l+1}(\mathbb{C}[z_{c_{I_{j}}}]_{l_{j}}).
\]
Under the natural isomorphisms 
\[
	\mathfrak{n}_{l,l+1}(\widehat{\Gamma}(\mathbf{c})_{l})
\cong N_{l,l+1}(\widehat{\Gamma}(\mathbf{c})_{l}),\quad
\mathfrak{n}_{l,l+1}(\mathbb{C}[z_{c_{I_{j}}}]_{l_{j}})\cong 
N_{l,l+1}(\mathbb{C}[z_{c_{I_{j}}}]_{l_{j}}) 	
\]
by exponential maps, we also obtain 
the isomorphism 
\[
	\prod_{j=0}^{r}\mathrm{pr}_{c_{I_{j}}}\colon N_{l,l+1}(\widehat{\Gamma}(\mathbf{c})_{l})
	\overset{\sim}{\longrightarrow}
	\prod_{j=0}^{r}N_{l,l+1}(\mathbb{C}[z_{c_{I_{j}}}]_{l_{j}})
\]
of complex manifolds.
\begin{prop}\label{prop:lieisom}
	The above map $\prod_{j=0}^{r}\mathrm{pr}_{c_{I_{j}}}$
	is an isomorphism as Lie groups. 
\end{prop}
\begin{proof}
	It suffices to check that the above map 
	is a group homomorphism.
	This  
	follows from the fact 
	that linear algebraic groups defined over $\mathbb{C}$ are 
	functors form the category of $\mathbb{C}$-algebras
	to that of groups, which preserves finite products.  
	
	Nevertheless, we give a direct proof of it, which will be useful 
	for Proposition \ref{prop:liedecomp2} appearing later.
	Let us take $X,Y\in \mathfrak{n}_{l,l+1}(\widehat{\Gamma}(\mathbf{c})_{l})$.
	Then Dynkin's formula gives us 
	\begin{multline*}
		\prod_{j=0}^{r}\mathrm{pr}_{c_{I_{j}}}(\mathrm{exp\,}(X)\cdot \mathrm{exp\,}(Y))
		=\\
		\mathrm{exp\,}\left(
			\prod_{j=0}^{r}\mathrm{pr}_{c_{I_{j}}}\left(
				\sum_{n=1}^{\infty}\frac{(-1)^{n-1}}{n}\sum_{\substack{(p_{1},q_{1})\in (\mathbb{Z}_{\ge 0})^{2}\backslash\{0\}\\\vdots\\\substack{(p_{n},q_{n})\in (\mathbb{Z}_{\ge 0})^{2}\backslash\{0\}}}}
				\frac{[X^{p_{1}}Y^{q_{1}}\cdots X^{p_{n}}Y^{q_{n}}]}
				{(\sum_{i=1}^{n}(p_{i}+q_{i}))\prod_{i=1}^{n}p_{i}!q_{i}!}	
			\right)
		\right).
	\end{multline*}
Then since $\prod_{j=0}^{r}\mathrm{pr}_{c_{I_{j}}}$ is a Lie algebra isomorphism,
we moreover have
\scriptsize
\begin{multline*}
	\prod_{j=0}^{r}\mathrm{pr}_{c_{I_{j}}}\left(
				\sum_{\substack{(p_{1},q_{1})\in (\mathbb{Z}_{\ge 0})^{2}\backslash\{0\}\\\vdots\\\substack{(p_{n},q_{n})\in (\mathbb{Z}_{\ge 0})^{2}\backslash\{0\}}}}
				\frac{[X^{p_{1}}Y^{q_{1}}\cdots X^{p_{n}}Y^{q_{n}}]}
				{(\sum_{i=1}^{n}(p_{i}+q_{i}))\prod_{i=1}^{n}p_{i}!q_{i}!}	
			\right)=\\
	\left(
		\sum_{\substack{(p_{1},q_{1})\in (\mathbb{Z}_{\ge 0})^{2}\backslash\{0\}\\\vdots\\\substack{(p_{n},q_{n})\in (\mathbb{Z}_{\ge 0})^{2}\backslash\{0\}}}}
		\frac{[\mathrm{pr}_{c_{I_{j}}}(X)^{p_{1}}\mathrm{pr}_{c_{I_{j}}}(Y)^{q_{1}}\cdots \mathrm{pr}_{c_{I_{j}}}(X)^{p_{n}}\mathrm{pr}_{c_{I_{j}}}(Y)^{q_{n}}]}
		{(\sum_{i=1}^{n}(p_{i}+q_{i}))\prod_{i=1}^{n}p_{i}!q_{i}!}
	\right)_{j=0,\ldots,r}.
\end{multline*}\normalsize 
Therefore we obtain 
\begin{align*}
	\prod_{j=0}^{r}\mathrm{pr}_{c_{I_{j}}}(\mathrm{exp\,}(X)\cdot \mathrm{exp\,}(Y))
	&=\left(\mathrm{exp\,}\left(\mathrm{pr}_{c_{I_{j}}}(X)\right)\cdot \mathrm{exp\,}\left(\mathrm{pr}_{c_{I_{j}}}(Y)\right)\right)_{j=0,\ldots,r}\\
	&=\left(\prod_{j=0}^{r}\mathrm{pr}_{c_{I_{j}}}(\mathrm{exp\,}(X))\right)\cdot \left(\prod_{j=0}^{r}\mathrm{pr}_{c_{I_{j}}}(\mathrm{exp\,}(X))\right)
\end{align*}
as desired.	
\end{proof}

Let us also recall the  isomorphism 
$\prod_{j=0}^{r}\mathrm{pr}_{c_{I_{j}}}\colon \widehat{\mathrm{L}}(\mathbf{c})_{l}\overset{\sim}{\rightarrow}
\prod_{j=0}^{r}\mathbb{C}[z^{-1}_{c_{I_{j}}}]$
given in Section \ref{sec:partialfrac}.
Then this isomorphism induces
another isomorphism 
\[
	\prod_{j=0}^{r}\mathrm{pr}_{c_{I_{j}}}
\colon \mathfrak{g}(\widehat{\mathrm{L}}(\mathbf{c})_{l})
\overset{\sim}{\longrightarrow }
\prod_{j=0}^{r}\mathfrak{g}(\mathbb{C}[z^{-1}_{c_{I_{j}}}])
\]
which restricts to the isomorphism
\[
	\prod_{j=0}^{r}\mathrm{pr}_{c_{I_{j}}}
\colon \mathfrak{u}_{l,l+1}(\widehat{\mathrm{L}}(\mathbf{c})_{l})
\overset{\sim}{\longrightarrow }
\prod_{j=0}^{r}\mathfrak{u}_{l,l+1}(\mathbb{C}[z^{-1}_{c_{I_{j}}}]).
\]
\begin{prop}\label{prop:commmult}
	The following diagram is commutative.
	\[
	\begin{tikzcd}
		N_{l,l+1}(\widehat{\Gamma}(\mathbf{c})_{l})\times \mathfrak{g}(\widehat{\mathrm{L}}(\mathbf{c})_{l})
		\arrow[d,"\prod_{j=0}^{r}(\mathrm{pr}_{c_{I_{j}}}\times \mathrm{pr}_{c_{I_{j}}})"]
		\arrow[r,"\mathrm{Ad}^{*}"]
		&\mathfrak{g}(\widehat{\mathrm{L}}(\mathbf{c})_{l})
		\arrow[d,"\prod_{j=0}^{r}\mathrm{pr}_{c_{I_{j}}}"]
		\\
		\prod_{j=0}^{r}
		\left(
			N_{l,l+1}(\mathbb{C}[z_{c_{I_{j}}}]_{l_{j}})\times \mathfrak{g}(\mathbb{C}[z^{-1}_{c_{I_{j}}}]_{l_{j}})
		\right)
		\arrow[r,"\prod_{j=0}^{r}\mathrm{Ad}^{*}"]
		&
		\prod_{j=0}^{r}\mathfrak{g}(\mathbb{C}[z^{-1}_{c_{I_{j}}}]_{l_{j}})
	\end{tikzcd}
	\]
\end{prop}
\begin{proof}
	Let us take $X\in \mathfrak{n}_{l,l+1}(\widehat{\Gamma}(\mathbf{c})_{l})$
	and $Y\in \mathfrak{g}(\widehat{\Gamma}(\mathbf{c})_{l})$.
	Then since 
	\begin{align*}
		&\prod_{j=0}^{r}\mathrm{pr}_{c_{I_{j}}}\colon \mathfrak{g}(\widehat{\Gamma}(\mathbf{c})_{l})
		\rightarrow
		\prod_{j=0}^{r}\mathfrak{g}(\mathbb{C}[z_{c_{I_{j}}}]_{l_{j}}),\\
		&\prod_{j=0}^{r}\mathrm{pr}_{c_{I_{j}}}\colon \mathfrak{n}_{l,l+1}(\widehat{\Gamma}(\mathbf{c})_{l})
		\overset{\sim}{\longrightarrow}
		\prod_{j=0}^{r}\mathfrak{n}_{l,l+1}(\mathbb{C}[z_{c_{I_{j}}}]_{l_{j}})
	\end{align*} are Lie algebra isomorphisms, 
	we have 
	\begin{align*}
		\prod_{j=0}^{r}\mathrm{pr}_{c_{I_{j}}}\left(\mathrm{Ad}(\mathrm{exp\,}(X))(Y)\right)
		&=\prod_{j=0}^{r}\mathrm{pr}_{c_{I_{j}}}\left(
			\sum_{n=0}^{\infty}\frac{(\mathrm{ad}(X))^{n}(Y)}{n!}
			\right)\\
		&=\left(
			\sum_{n=0}^{\infty}\frac{(\mathrm{ad}(\mathrm{pr}_{c_{I_{j}}}(X)))^{n}(\mathrm{pr}_{c_{I_{j}}}(Y))}{n!}
			\right)_{j=0,\ldots,r}	\\
		&=\left(\mathrm{Ad}\left(\mathrm{pr}_{c_{I_{j}}}(\mathrm{exp\,}(X))\right)
			\left(\mathrm{pr}_{c_{I_{j}}}(Y)\right)
			\right)_{j=0,\ldots,r}.
	\end{align*}
	Namely, the following diagram is commutative,
	\[
	\begin{tikzcd}
		N_{l,l+1}(\widehat{\Gamma}(\mathbf{c})_{l})\times \mathfrak{g}(\widehat{\Gamma}(\mathbf{c})_{l})
		\arrow[d,"\prod_{j=0}^{r}(\mathrm{pr}_{c_{I_{j}}}\times \mathrm{pr}_{c_{I_{j}}})"]
		\arrow[r,"\mathrm{Ad}"]
		&\mathfrak{g}(\widehat{\Gamma}(\mathbf{c})_{l})
		\arrow[d,"\prod_{j=0}^{r}\mathrm{pr}_{c_{I_{j}}}"]
		\\
		\prod_{j=0}^{r}
		\left(
			N_{l,l+1}(\mathbb{C}[z_{c_{I_{j}}}]_{l_{j}})\times \mathfrak{g}(\mathbb{C}[z_{c_{I_{j}}}]_{l_{j}})
		\right)
		\arrow[r,"\prod_{j=0}^{r}\mathrm{Ad}"]
		&
		\prod_{j=0}^{r}\mathfrak{g}(\mathbb{C}[z_{c_{I_{j}}}]_{l_{j}})
	\end{tikzcd}.
	\]
	Then since vertical arrows are isomorphisms, 
	the dual of this diagram gives our desired commutative diagram.
\end{proof}

Notice that we have \begin{equation}\label{eq:rescom}
	\mathrm{res\,}(X)=\sum_{j=0}^{r}\underset{z=c_{I_{j}}}{\mathrm{res\,}}(\mathrm{pr}_{c_{I_{j}}}(X))	
\end{equation}
for $X\in \mathfrak{g}(\widehat{\mathrm{L}}(\mathbf{c})_{l})$,
since 
$\prod_{j=0}^{r}\mathrm{pr}_{c_{I_{j}}}
\colon \mathfrak{g}(\widehat{\mathrm{L}}(\mathbf{c})_{l})
\rightarrow
\prod_{j=0}^{r}\mathfrak{g}(\mathbb{C}[z^{-1}_{c_{I_{j}}}])$
is defined by the partial fractional decomposition.

\subsection{Lie subgroupoid $N_{l,l+1}(\widehat{\Gamma}(\mathbb{B})_{l})_{1}$}
Let us consider the subbundle of $\mathfrak{n}_{l,l+1}(\widehat{\Gamma}(\mathbb{B})_{l})=\bigsqcup_{\mathbf{c}\in \mathbb{B}}\mathfrak{n}_{l,l+1}(\widehat{\Gamma}(\mathbf{c})_{l})$
defined by \index[cN]{$\mathfrak{n}_{l,l+1}(\widehat{\Gamma}(\mathbb{B})_{l})_{1}$}
\begin{align*}
	&\mathfrak{n}_{l,l+1}(\widehat{\Gamma}(\mathbb{B})_{l})_{1}:=\\
	&\left\{X_{0}+X_{1}(z-c_{0})+\cdots +X_{l}(z-c_{0})\cdots(z-c_{l-1})
	\in \bigsqcup_{\mathbf{c}\in \mathbb{B}}\mathfrak{g}(\widehat{\Gamma}(\mathbf{c}))\,\middle|\, X_{0}=0\right\}
\end{align*}
which has the natural Lie algebroid structure induced from that of $\mathfrak{n}_{l,l+1}(\widehat{\Gamma}(\mathbb{B})_{l})$.
Then as well as for the Lie groupoid $N_{l,l+1}(\widehat{\Gamma}(\mathbb{B})_{l})$,
Dynkin's formula defines the groupoid structure on this vector bundle $\mathfrak{n}_{l,l+1}(\widehat{\Gamma}(\mathbb{B})_{l})_{1}$,
and we denote this Lie groupoid by \index[cN]{$N_{l,l+1}(\widehat{\Gamma}(\mathbb{B})_{l})_{1}$}
\[
	N_{l,l+1}(\widehat{\Gamma}(\mathbb{B})_{l})_{1}
\]
which is a Lie subgroupoid of $N_{l,l+1}(\widehat{\Gamma}(\mathbb{B})_{l})$.

By definition,
the fiber of the Lie algebroid $\mathfrak{n}_{l,l+1}(\widehat{\Gamma}(\mathbb{B})_{l})_{1}$
at each $\mathbf{c}\in \mathbb{B}$ is isomorphic to 
$\mathfrak{n}_{l,l+1}\otimes_{\mathbb{C}}\mathfrak{m}_{z_{c_{0}}}^{\widehat{\Gamma}(\mathbf{c})}$
where $\mathfrak{m}_{z_{c_{0}}}^{\widehat{\Gamma}(\mathbf{c})}=\langle z-c_{0}\rangle_{\widehat{\Gamma}(\mathbf{c})}$
is a maximal ideal of $\widehat{\Gamma}(\mathbf{c})$.
We denote this nilpotent Lie algebra by \index[cN]{$\mathfrak{n}_{l,l+1}(\widehat{\Gamma}(\mathbf{c})_{l})_{1}$}
\[
	\mathfrak{n}_{l,l+1}(\widehat{\Gamma}(\mathbf{c})_{l})_{1}:=\mathfrak{n}_{l,l+1}\otimes_{\mathbb{C}}\mathfrak{m}_{z_{c_{0}}}^{\widehat{\Gamma}(\mathbf{c})}
\]
and also denote the corresponding nilpotent Lie group by 
\[
	N_{l,l+1}(\widehat{\Gamma}(\mathbf{c})_{l})_{1}.
\]
Then $\mathfrak{n}_{l,l+1}(\widehat{\Gamma}(\mathbb{B})_{l})_{1}$ and $N_{l,l+1}(\widehat{\Gamma}(\mathbb{B})_{l})_{1}$ 
can be 
regarded as the holomorphic families of nilpotent Lie algebras $\bigsqcup_{\mathbf{c}\in \mathbb{B}}\left(\mathfrak{n}_{l,l+1}(\widehat{\Gamma}(\mathbf{c})_{l})_{1} \right)$
and groups $\bigsqcup_{\mathbf{c}\in \mathbb{B}}\left(N_{l,l+1}(\widehat{\Gamma}(\mathbf{c})_{l})_{1} \right)$ respectively.

Now we take a partition $\mathcal{I}\colon I_{0},\ldots,I_{r}$ of $\{0,1,\ldots,k\}$
and an element
$\mathbf{c}\in \mathcal{C}(\mathcal{I})\cap \mathbb{B}$ of $\mathbb{B}$.
Recall that the 
algebra isomorphism $\prod_{j=0}^{r}\mathrm{pr}_{c_{I_{j}}}\colon \widehat{\Gamma}(\mathbf{c})_{l}\overset{\sim}{\rightarrow}
\prod_{j=0}^{r}\mathbb{C}[z_{c_{I_{j}}}]_{l_{j}}$
restricts to the isomorphism 
\[
	\prod_{j=0}^{r}\mathrm{pr}_{c_{I_{j}}}\colon \mathfrak{m}_{z_{c_{0}}}^{\widehat{\Gamma}(\mathbf{c})}\overset{\sim}{\longrightarrow}
\mathfrak{m}_{z_{c_{0}}}^{(l_{0})}\times \prod_{j=1}^{r}\mathbb{C}[z_{c_{i_{[j,0]}}}]_{l_{j}}	
\]
of non-unital $\mathbb{C}$-algebras. Then
we see that the 
isomorphism 
\[
\prod_{j=0}^{r}\mathrm{pr}_{c_{I_{j}}}\colon \mathfrak{n}_{l,l+1}(\widehat{\Gamma}(\mathbf{c})_{l})\overset{\sim}{\rightarrow}
\prod_{j=0}^{r}\mathfrak{n}_{l,l+1}(\mathbb{C}[z_{c_{I_{j}}}]_{l_{j}})
\]
of Lie algebras 
restricts to the isomorphism 
\[
	\prod_{j=0}^{r}\mathrm{pr}_{c_{I_{j}}}\colon \mathfrak{n}_{l,l+1}(\widehat{\Gamma}(\mathbf{c})_{l})_{1}\overset{\sim}{\longrightarrow}
\mathfrak{n}_{l,l+1}(\mathbb{C}[z_{c_{0}}]_{l_{0}})_{1}\times \prod_{j=1}^{r}\mathfrak{n}_{l,l+1}(\mathbb{C}[z_{c_{I_{j}}}]_{l_{j}}).
\]
Here we note that $c_{0}=c_{I_{0}}$ since we assumed $0\in I_{0}$.
Then through the exponential maps, we obtain the isomorphism as complex manifolds,
\[
	\prod_{j=0}^{r}\mathrm{pr}_{c_{i_{[j,0]}}}\colon N_{l,l+1}(\widehat{\Gamma}(\mathbf{c})_{l})_{1}\overset{\sim}{\longrightarrow}
N_{l,l+1}(\mathbb{C}[z_{c_{0}}]_{l_{0}})_{1}\times \prod_{j=1}^{r}N_{l,l+1}(\mathbb{C}[z_{c_{i_{[j,0]}}}]_{l_{j}}).
\]
\begin{prop}\label{prop:liedecomp2}
	The map
	\[
	\prod_{j=0}^{r}\mathrm{pr}_{c_{i_{[j,0]}}}\colon N_{l,l+1}(\widehat{\Gamma}(\mathbf{c})_{l})_{1}\overset{\sim}{\longrightarrow}
N_{l,l+1}(\mathbb{C}[z_{c_{0}}]_{l_{0}})_{1}\times \prod_{j=1}^{r}N_{l,l+1}(\mathbb{C}[z_{c_{i_{[j,0]}}}]_{l_{j}}).
\]
is a Lie group isomorphism. 
\end{prop}
\begin{proof}
	As well as Proposition \ref{prop:lieisom}, we can see  that 
	this map is compatible with the group multiplications.
\end{proof}
\subsection{Subbundle $\mathfrak{u}_{l,l+1}(\widehat{\mathrm{L}}(\mathbb{B})^{\mathrm{rev}}_{[l-1]})$}
In section \ref{sec:liegroupoid},
we gave a realization of the family $\bigsqcup_{\mathbf{c}\in \mathbb{B}}\mathfrak{u}_{l,l+1}(\widehat{\mathrm{L}}(\mathbf{c}))$
as the trivial vector bundle $\mathfrak{u}_{l,l+1}^{\oplus(l+1)}\times \mathbb{B}\rightarrow \mathbb{B}$
with respect to the standard basis
\[
	\frac{1}{z-c_{0}},\frac{1}{(z-c_{0})(z-c_{1})},\ldots,\frac{1}{(z-c_{0})(z-c_{1})\cdots(z-c_{l})}	
\]
of $\widehat{\mathrm{L}}(\mathbf{c})$.
This vector bundle was denoted by $\mathfrak{u}_{l,l+1}(\widehat{\mathrm{L}}(\mathbb{B})_{l}).$
Now by reversing the order of $c_{0},c_{1},\ldots,c_{l}$,
we obtain another basis
\[
	\frac{1}{(z-c_{l})},\frac{1}{(z-c_{l})(z-c_{l-1})},\ldots,\frac{1}{(z-c_{l})(z-c_{l-1})\cdots(z-c_{1})}		
\]
of $\widehat{\mathrm{L}}(\mathbf{c})$ for $\mathbf{c}\in \mathbb{B}$
and 
consider another realization of the family $\bigsqcup_{\mathbf{c}\in \mathbb{B}}\mathfrak{u}_{l,l+1}(\widehat{\mathrm{L}}(\mathbf{c}))$
as the vector bundle $\mathfrak{u}_{l,l+1}^{\oplus(l+1)}\times \mathbb{B}\rightarrow \mathbb{B}$
with respect to this new basis.
We denote this vector bundle by\index[cN]{$\mathfrak{u}_{l,l+1}(\widehat{\mathrm{L}}(\mathbb{B})_{l})^{\mathrm{rev}}$} $\mathfrak{u}_{l,l+1}(\widehat{\mathrm{L}}(\mathbb{B})_{l})^{\mathrm{rev}}.$
There uniquely exists 
base change matrix $M_{\mathbf{c}}\in \mathrm{GL}_{l+1}(\mathbb{C})$
between above two basis of $\widehat{\mathrm{L}}(\mathbf{c})$
for each $\mathbf{c}\in \mathbb{B}$,
and  $M_{\mathbf{c}}$
depends holomorphically on $\mathbf{c}\in \mathbb{B}$.
Thus we obtain the holomorphic map $\mathbb{B}\ni \mathbf{c}\mapsto M_{\mathbf{c}}\in\mathrm{GL}_{l+1}(\mathbb{C})$
and then it defines 
 the standard isomorphism 
\[
	\phi_{M}\colon \mathfrak{u}_{l,l+1}(\widehat{\mathrm{L}}(\mathbb{B})_{l})^{\mathrm{rev}}\overset{\sim}{\longrightarrow}\mathfrak{u}_{l,l+1}(\widehat{\mathrm{L}}(\mathbb{B})_{l})
\]
as vector bundles over $\mathbb{B}$.

Let us consider the subbundle of $\mathfrak{u}_{l,l+1}(\widehat{\mathrm{L}}(\mathbb{B}))^{\mathrm{rev}}$ defined by 
\begin{equation*}
	\mathfrak{u}_{l,l+1}(\widehat{\mathrm{L}}(\mathbb{B})_{[l-1]})^{\mathrm{rev}}:=
	\left\{
	\frac{X_{l}}{z-c_{l}}+\cdots+
	\frac{X_{0}}{(z-c_{l})\cdots(z-c_{0})}
	\in \bigsqcup_{\mathbf{c}\in \mathbb{B}}\mathfrak{u}_{l,l+1}(\widehat{\mathrm{L}}(\mathbf{c})_{l})\,\middle|\,
	X_{0}=0		
	\right\}.
\end{equation*}
Then we denote the corresponding subbundle of $\mathfrak{u}_{l,l+1}(\widehat{\mathrm{L}}(\mathbb{B}))$ under the isomorphism $\phi_{M}$ by \index[cN]{$\mathfrak{u}_{l,l+1}(\widehat{\mathrm{L}}(\mathbb{B})^{\mathrm{rev}}_{[l-1]})$}
\[
	\mathfrak{u}_{l,l+1}(\widehat{\mathrm{L}}(\mathbb{B})^{\mathrm{rev}}_{[l-1]}):=\phi_{M}(\mathfrak{u}_{l,l+1}(\widehat{\mathrm{L}}(\mathbb{B})_{l-1})^{\mathrm{rev}}).
\]
\begin{prop}\label{dualliedecomp}
	Let us take a partition $\mathcal{I}\colon I_{0},\ldots,I_{r}$ of $\{0,1,\ldots,k\}$
	and assume $0\in I_{0}$.
	Also take $\mathbf{c}\in C(\mathcal{I})\cap \mathbb{B}$.
	Let $\prod_{j=0}^{r}\mathrm{pr}_{c_{I_{j}}}\colon 
	\mathfrak{g}(\widehat{\mathrm{L}}(\mathbf{c})_{l})\xrightarrow{\sim}
	\prod_{j=0}^{r}\mathfrak{g}(\mathbb{C}[z^{-1}_{c_{I_{j}}}]_{l_{j}})$
	be the isomorphism introduced in Section \ref{sec:decompfib}.
	Let us denote the fiber of the vector bundle $\mathfrak{u}_{l,l+1}(\widehat{\mathrm{L}}(\mathbb{B})^{\mathrm{rev}}_{[l-1]})$
	at $\mathbf{c}$ by $\mathfrak{u}_{l,l+1}(\widehat{\mathrm{L}}(\mathbf{c})^{\mathrm{rev}}_{[l-1]})$.
	
	By restricting the map $\prod_{j=0}^{r}\mathrm{pr}_{c_{I_{j}}}$
	to $\mathfrak{u}_{l,l+1}(\widehat{\mathrm{L}}(\mathbf{c})^{\mathrm{rev}}_{[l-1]})$, we obtain the isomorphism 
	\[
		\prod_{j=0}^{r}\mathrm{pr}_{c_{I_{j}}}\colon \mathfrak{u}_{l,l+1}(\widehat{\mathrm{L}}(\mathbf{c})^{\mathrm{rev}}_{[l-1]})\overset{\sim}{\longrightarrow}
		\mathfrak{u}_{l,l+1}(\mathbb{C}[z^{-1}_{c_{0}}]_{l_{0}-1})\times \prod_{j=1}^{r}\mathfrak{u}_{l,l+1}(\mathbb{C}[z^{-1}_{c_{I_{j}}}]_{l_{j}}).
	\]
\end{prop}
\begin{proof}
Let us consider the vector subspace of $\widehat{\mathrm{L}}(\mathbf{c})_{l}$
defined by 
\begin{equation*}
	\widehat{\mathrm{L}}(\mathbf{c})^{\mathrm{rev}}_{[l-1]}:=
	\left\{
	\frac{x_{l}}{z-c_{l}}+\frac{x_{l-1}}{(z-c_{l})(z-c_{l-1})}+\cdots+
	\frac{x_{0}}{(z-c_{l})\cdots(z-c_{0})}
	\in \widehat{\mathrm{L}}(\mathbf{c})_{l}\,\middle|\,
	x_{0}=0		
	\right\}.
\end{equation*}
Since the zero-th component $c_{0}$ of $\mathbf{c}$ does not appear 
in this definition, by setting  
$\mathbf{c}_{\ge 1}:=(c_{1},c_{2},\ldots,c_{k})\in \mathbb{C}^{k}$,
we can regard  
\[
	\widehat{\mathrm{L}}(\mathbf{c})^{\mathrm{rev}}_{[l-1]}=
	\widehat{\mathrm{L}}(\mathbf{c}_{\ge 1})_{l-1}.
\]

For the partition
$\mathcal{I}\colon I_{0},I_{1},\ldots,I_{r}$, let $I^{\ge 1}_{j}:=I_{j}\cap \{1,2,\ldots,k\}$
for $j=0,1,\ldots,r$ and define the partition    
$\mathcal{I}_{\ge 1}\colon I_{0}^{\ge 1},I_{1}^{\ge 1},\ldots,I_{r}^{\ge 1}$
of $\{1,2,\ldots,k\}$ where we allow $I^{\ge 1}_{j}=\emptyset$. 
Recall that in our convention we always assume 
$0\in I_{0}$. Thus $I_{0}^{\ge 1}=I_{0}\backslash\{0\}$ and 
$I_{j}^{\ge 1}=I_{j}$ for $j\ge 1$, and we can write  
$\mathcal{I}_{\ge 1}\colon I_{0}^{\ge 1},I_{1},\ldots,I_{r}$.
Then we can regard $\mathbf{c}_{\ge 1}\in C(\mathcal{I}_{\ge 1})
\subset \{(a_{1},a_{2},\ldots,a_{k})\in \mathbb{C}^{k}\}$.
Therefore as we saw in Section \ref{sec:partialfrac},
we have the direct sum decomposition 
\[
	\widehat{\mathrm{L}}(\mathbf{c})^{\mathrm{rev}}_{[l-1]}=
\widehat{\mathrm{L}}(\mathbf{c}_{\ge 1})_{l-1}
=\mathbb{C}[z_{c_{0}}]_{l_{j}-1}\oplus \bigoplus_{j=1}^{r}\mathbb{C}[z_{c_{i_{[j,0]}}}]_{l_{j}}.
\]
This decomposition induces our desired isomorphism.
\end{proof}
\section{Deformation of truncated orbit}
Let us consider an unramified canonical normal form
\[
	H\,dz=\left(\frac{H_{k}}{z^{k}}+\cdots+\frac{H_{1}}{z}+H_{\mathrm{res}}\right)\frac{dz}{z}
	\in \mathfrak{g}(\mathbb{C}[z^{-1}]_{k})dz.
\]
Let $\mathbb{B}_{H}\subset \mathbb{C}^{k+1}$ be 
the complement of hypersurfaces defined in Section \ref{sec:openset}.
In this section, we shall construct a deformation 
of the truncated orbit $\mathbb{O}_{H}$ over the parameter space $\mathbb{B}_{H}.$

\subsection{Deformation of truncated orbit}\label{sec:deformtrunc}
Let us consider the product 
\[
	\prod_{l=1}^{k}\left(
		N_{l,l+1}(\widehat{\Gamma}(\mathbb{B}_{H})_{l})_{1}\times 
		\mathfrak{u}_{l,l+1}(\widehat{\mathrm{L}}(\mathbb{B}_{H})_{[l-1]}^{\mathrm{rev}})
	\right)
\]
of vector bundles over $\mathbb{B}_{H}$ defined in the previous section.
Let us define a left $L_{1}$-action on  this bundle.
Recall that each component of this product is a subbundle of 
the trivial bundles $(\mathfrak{n}_{l,l+1})^{\oplus (l+1)}\times \mathbb{B}_{H}$
or $(\mathfrak{u}_{l,l+1})^{\oplus (l+1)}\times \mathbb{B}_{H}$.
Then since
 $\mathfrak{n}_{l,l+1}$
and $\mathfrak{u}_{l,l+1}$ have the adjoint $L_{1}$-actions, 
we can extends these adjoint actions diagonally 
on $(\mathfrak{n}_{l,l+1})^{\oplus (l+1)}$ and $(\mathfrak{u}_{l,l+1})^{\oplus (l+1)}$.
Thus we obtain the $L_{1}$-actions on 
$(\mathfrak{n}_{l,l+1})^{\oplus (l+1)}\times \mathbb{B}_{H}$
and $(\mathfrak{u}_{l,l+1})^{\oplus (l+1)}\times \mathbb{B}_{H}$
where $L_{1}$ acts trivially on $\mathbb{B}_{H}$.
It can be directly checked that 
subbundles $N_{l,l+1}(\widehat{\Gamma}(\mathbb{B}_{H})_{l})_{1}$ and 
$\mathfrak{u}_{l,l+1}(\widehat{\mathrm{L}}(\mathbb{B}_{H})_{[l-1]}^{\mathrm{rev}})$
are preserved under these $L_{1}$-actions.
Therefore we can extend these $L_{1}$-actions to the whole product bundle
on which $L_{1}$ acts diagonally on each component through the above action.

Let us also consider the  
trivial fiber bundle $\mathbb{O}_{H_{\mathrm{res}}}^{L_{1}}\times \mathbb{B}_{H}
\rightarrow \mathbb{B}_{H}$
with the left $L_{1}$-action 
defined by multiplication from the left 
on $L_{1}/\mathrm{Stab}_{L_{1}}(H_{\mathrm{res}})\cong  \mathbb{O}_{H_{\mathrm{res}}}^{L_{1}}.$
We regard this bundle 
as a subspace of the vector bundle $\mathfrak{g}(\widehat{\mathrm{L}}(\mathbb{B}_{H}))$
by the injective immersion 
\begin{multline*}
	\left(L_{1}/\mathrm{Stab}_{L_{1}}(H_{\mathrm{res}})\right)\times \mathbb{B}_{H}\ni ([g],\mathbf{c})\longmapsto \\
		\frac{H_{k}}{\prod_{j=0}^{k}(z-c_{j})}
		+\frac{H_{k-1}}{\prod_{j=0}^{k-1}(z-c_{j})}
		+\cdots +
		\frac{\mathrm{Ad}(g)(H_{\mathrm{res}})}{z-c_{0}}
	\in \bigsqcup_{\mathbf{c}\in \mathbb{B}_{H}}\mathfrak{g}(\widehat{\mathrm{L}}(\mathbf{c}))
	=\mathfrak{g}(\widehat{\mathrm{L}}(\mathbb{B}_{H})).
\end{multline*}
Here we notice that $\mathrm{Ad}(g)(H_{i})=H_{i}$ for $g\in L_{1}$ and $i=1,2,\ldots,k$.

Let us moreover consider the 
trivial bundle $G\times \mathbb{B}_{H}\rightarrow \mathbb{B}_{H}$.
By the multiplication from the right, we define the left  
$L_{1}$-action on $G$, i.e.,
$l\cdot g:=gl^{-1}$ for $g\in G$ and $l\in L_{1}$,
and thus $G\times \mathbb{B}_{H}\rightarrow \mathbb{B}_{H}$
is equipped with the left $L_{1}$-action as well.

Now we define the product of all these fiber bundles over $\mathbb{B}_{H}$
with the left $L_{1}$-action by
\[
	G\times \prod_{l=1}^{k}\left(
		N_{l,l+1}(\widehat{\Gamma}(\mathbb{B}_{H})_{l})_{1}\times 
		\mathfrak{u}_{l,l+1}(\widehat{\mathrm{L}}(\mathbb{B}_{H})_{[l-1]}^{\mathrm{rev}})	
		\right)
		\times \mathbb{O}_{H_{\mathrm{res}}}^{L_{1}}.
\]

We shall consider the quotient of the total space of this 
bundle under the $L_{1}$-action.
Since $L_{1}$ acts freely and properly on $G$,  
Lemma \ref{lem:denpan} assures that 
the $L_{1}$-action on the total space of this fiber bundle is free and proper.
Thus
we can consider the quotient manifold\index[cN]{$\mathbb{O}_{H,\mathbb{B}_{H}}$}
\[
	\mathbb{O}_{H,\mathbb{B}_{H}}:=L_{1}\Bigg\backslash 
	\left(
		G\times \prod_{l=1}^{k}\left(
			N_{l,l+1}(\widehat{\Gamma}(\mathbb{B}_{H})_{l})_{1}\times 
			\mathfrak{u}_{l,l+1}(\widehat{\mathrm{L}}(\mathbb{B}_{H})_{[l-1]}^{\mathrm{rev}})	
			\right)
			\times \mathbb{O}_{H_{\mathrm{res}}}^{L_{1}}
	\right)
\]
with the projection map 
\[
	\pi_{\mathbb{B}_{H}}\colon \mathbb{O}_{H,\mathbb{B}_{H}}\longrightarrow \mathbb{B}_{H}
\]
induced by the 
bundle projection.

\begin{prop}\label{prop:deformorb}
	The fiber at $\mathbf{0}\in \mathbb{B}_{H}$ is isomorphic to the 
	truncated orbit $\mathbb{O}_{H}$, i.e.,
	\[
		\pi_{\mathbb{B}_{H}}^{-1}(\mathbf{0})\cong \mathbb{O}_{H}.	
	\]
\end{prop}
\begin{proof}
	From the definitions of 
	$N_{l,l+1}(\widehat{\Gamma}(\mathbb{B}_{H})_{l})_{1}$ and 
	$\mathfrak{u}_{l,l+1}(\widehat{\mathrm{L}}(\mathbb{B}_{H})_{[l-1]}^{\mathrm{rev}})$,
	we obtain 
	\[
		\pi_{\mathbb{B}_{H}}^{-1}(\mathbf{0})=
		L_{1}\Bigg\backslash 
	\left(
		G\times \prod_{l=1}^{k}\left(
			N_{l,l+1}(\mathbb{C}[z]_{l})_{1}\times 
			\mathfrak{u}_{l,l+1}(\mathbb{C}[z^{-1}]_{l-1})	
			\right)
			\times \mathbb{O}_{H_{\mathrm{res}}}^{L_{1}}
	\right).
	\]
	As we saw in Section \ref{sec:proofthm},
	the right hand side is isomorphic to $\mathbb{O}_{H}$. 
\end{proof}

As a fiber bundle, $\mathbb{O}_{H,\mathbb{B}_{H}}$
is trivial, and no deformation occurs to each fiber.
However group structures of $N_{l,l+1}(\widehat{\Gamma}(\mathbf{c})_{l})$
and their actions on $\mathfrak{u}_{l,l+1}(\widehat{\mathrm{L}}(\mathbf{c})_{[l-1]}^{\mathrm{rev}})$
are deformed according to 
the change of parameter $\mathbf{c}\in \mathbb{B}_{H}$.
Therefore the moment map on this bundle which will be defined in the next section,
gives a non-trivial deformation of the moment map
on the truncated orbit $\mathbb{O}_{H}$.	
\subsection{Deformation of the moment map $\mu_{\mathbb{O}_{H}\downarrow G}$}
We shall define the map 
\[
	\mu_{\mathbb{O}_{H}\downarrow G,\mathbb{B}_{H}}\colon \mathbb{O}_{H,\mathbb{B}_{H}}
	\longrightarrow \mathfrak{g}
\]
which gives a deformation of the moment map 
$\mu_{\mathbb{O}_{H}\downarrow G}\colon \mathbb{O}_{H}\rightarrow \mathfrak{g}$.
For this purpose we shall define the bundle map 
\[
	\iota_{\mathbb{O}_{H,\mathbb{B}_{H}}}
	\colon \mathbb{O}_{H,\mathbb{B}_{H}}
	\longrightarrow \mathfrak{g}(\widehat{\mathrm{L}}(\mathbb{B}_{H}))
\] 
which will be a deformation of the immersion 
$\iota_{\mathbb{O}_{H}}\colon \mathbb{O}_{H}\hookrightarrow 
\mathfrak{g}(\mathbb{C}[z^{-1}]_{k})$,
and the map $\mu_{\mathbb{O}_{H}\downarrow G,\mathbb{B}_{H}}$
will be defined as the composition 
of $\iota_{\mathbb{O}_{H,\mathbb{B}_{H}}}$
and 
$\mathrm{res\,}\colon \mathfrak{g}(\widehat{\mathrm{L}}(\mathbb{B}_{H}))
\rightarrow \mathfrak{g}$ defined in Section \ref{sec:liegroupoid}.

Let us take 
$\Xi\in \mathbb{O}_{H,\mathbb{B}_{H}}$ and 
choose its representative 
\[(g,(n_{l},\nu_{l})_{l=1,\ldots,k},\eta)
\in G\times \prod_{l=1}^{k}\left(
	N_{l,l+1}(\mathbb{C}[z]_{l})_{1}\times 
	\mathfrak{u}_{l,l+1}(\mathbb{C}[z^{-1}]_{l-1})	
	\right)
\times \mathbb{O}_{H_{\mathrm{res}}}^{L_{1}}.\]
Then let us define 
$\iota_{\mathbb{O}_{H,\mathbb{B}_{H}}}(\Xi)$ inductively as follows.
\\
\noindent
\underline{Step 1}. 
Let us define 
\[
	\iota_{\mathbb{O}_{H,\mathbb{B}_{H}}}^{(1)}(\Xi):=
	\mathrm{Ad}^{*}(n_{1})(\nu_{1}+\eta)\in \mathfrak{g}(\widehat{\mathrm{L}}(\mathbb{B}_{H}))
\]
where we regard $\nu_{1},\eta\in \mathfrak{g}(\widehat{\mathrm{L}}(\mathbb{B}_{H}))$
and $\mathrm{Ad}^{*}$ is the coadjoint action of $N_{1}(\widehat{\Gamma}(\mathbb{B}_{H}))$
on $\mathfrak{g}(\widehat{\mathrm{L}}(\mathbb{B}_{H}))$.
\\
\noindent
\underline{Step $l$ for $1<l\le k$}.
Let us define
\[
	\iota_{\mathbb{O}_{H,\mathbb{B}_{H}}}^{(l)}(\Xi):=
	\mathrm{Ad}^{*}(n_{l})(\nu_{l}+\iota_{\mathbb{O}_{H,\mathbb{B}_{H}}}^{(l-1)}(\Xi))\in \mathfrak{g}(\widehat{\mathrm{L}}(\mathbb{B}_{H})).
\] 
where we regard $\nu_{l}\in \mathfrak{g}(\widehat{\mathrm{L}}(\mathbb{B}_{H}))$
and $\mathrm{Ad}^{*}$ is the coadjoint action of $N_{l}(\widehat{\Gamma}(\mathbb{B}_{H}))$
on $\mathfrak{g}(\widehat{\mathrm{L}}(\mathbb{B}_{H}))$.
\\
\noindent
\underline{Step $k+1$}.
Let us regard the trivial bundle $G\times \mathbb{B}_{H}\rightarrow \mathbb{B}_{H}$
as a Lie groupoid in the trivial way.
Then the natural $G$-action on 
the each fiber $\mathfrak{g}(\widehat{\mathrm{L}}(\mathbf{c}))$
at $\mathbf{c}\in \mathbb{B}_{H}$
defines the coadjoint action $\mathrm{Ad}^{*}$
of $G\times \mathbb{B}_{H}$
on $\mathfrak{g}(\widehat{\mathrm{L}}(\mathbb{B}_{H})).$

Then we define
\[
	\iota_{\mathbb{O}_{H,\mathbb{B}_{H}}}^{(k+1)}(\Xi):=
	\mathrm{Ad}^{*}(g)(\iota_{\mathbb{O}_{H,\mathbb{B}_{H}}}^{(k)}(\Xi))\in \mathfrak{g}(\widehat{\mathrm{L}}(\mathbb{B}_{H})).
\] 

We can see that the definition of $\iota_{\mathbb{O}_{H,\mathbb{B}_{H}}}^{(k+1)}(\Xi)$
does not depend on the choice of representatives 
of $\Xi$ as we saw in Section \ref{sec:tridec}. Thus we obtain the bundle map \index[cN]{$\iota_{\mathbb{O}_{H,\mathbb{B}_{H}}}$}
\[
	\iota_{\mathbb{O}_{H,\mathbb{B}_{H}}}:=\iota^{(k+1)}_{\mathbb{O}_{H,\mathbb{B}_{H}}}
	\colon \mathbb{O}_{H,\mathbb{B}_{H}}
	\longrightarrow \mathfrak{g}(\widehat{\mathrm{L}}(\mathbb{B}_{H}))
\]
and the desired map \index[cN]{$\mu_{\mathbb{O}_{H}\downarrow G,\mathbb{B}_{H}}$}
\[
	\mu_{\mathbb{O}_{H}\downarrow G,\mathbb{B}_{H}}\colon
	\mathbb{O}_{H,\mathbb{B}_{H}}
	\overset{\iota_{\mathbb{O}_{H,\mathbb{B}_{H}}}}{\longrightarrow }
	\mathfrak{g}(\widehat{\mathrm{L}}(\mathbb{B}_{H}))
	\overset{\mathrm{res\,}}{\longrightarrow }
	\mathfrak{g}.
\]
Let us consider the left $G$-action on $\mathbb{O}_{H,\mathbb{B}_{H}}$
by the multiplication 
from the left on the first component of 
$G\times \left(\left(\prod_{l=1}^{k}(N_{l,l+1}(\mathbb{C}[z]_{l})_{1}\times \mathfrak{u}_{l,l+1}(\mathbb{C}[z^{-1}]_{l})_{1})\right)
\times \mathbb{O}_{H_{\mathrm{res}}}^{L_{1}}\right)$.
Then we can see that the map $\mu_{\mathbb{O}_{H}\downarrow G,\mathbb{B}_{H}}$
is $G$-equivariant. 

\begin{prop}\label{prop:momentspecfib}
	The map $\mu_{\mathbb{O}_{H}\downarrow G,\mathbb{B}_{H}}$
	is a deformation of the moment map $\mu_{\mathbb{O}_{H}\downarrow G}$,
	namely we have 
	\[
		\mu_{\mathbb{O}_{H}\downarrow G,\mathbb{B}_{H}}|_{\pi_{\mathbb{B}_{H}}^{-1}(\mathbf{0})}
		=\mu_{\mathbb{O}_{H}\downarrow G}
	\]
	under the identification $\pi_{\mathbb{B}_{H}}^{-1}(\mathbf{0})
	\cong \mathbb{O}_{H}$ given in Proposition \ref{prop:deformorb}.
\end{prop}
\begin{proof}
	By recalling the description of $\iota_{\mathbb{O}_{H}}$
	given in Section \ref{sec:tridec},
	we can see that  
	the restriction map $\iota_{\mathbb{O}_{H,\mathbb{B}_{H}}}|_{\pi_{\mathbb{B}_{H}}^{-1}(\mathbf{0})}$
	coincides with 
	the isomorphism 
	\begin{multline*}
		L_{1}\Bigg\backslash \left(
			G\times \left(\left(\prod_{l=1}^{k}(N_{l,l+1}(\mathbb{C}[z]_{l})_{1}\times \mathfrak{u}_{l,l+1}(\mathbb{C}[z^{-1}]_{l})_{1})\right)
	\times \mathbb{O}_{H_{\mathrm{res}}}^{L_{1}}\right)
		\right)\\
		\overset{\sim}{\longrightarrow } \mathbb{O}_{H}
		\subset \mathfrak{g}(\mathbb{C}[z^{-1}]_{k}).
	\end{multline*}
	Also the map $\mathrm{res\,}\colon \mathfrak{g}(\widehat{\mathrm{L}}(\mathbb{B}_{H}))
	\rightarrow \mathfrak{g}$
	coincides with $\mathrm{res\,}\colon \mathfrak{g}(\mathbb{C}[z^{-1}]_{k})
	\rightarrow \mathfrak{g}$
	on the fiber at $\mathbf{0}\in \mathbb{B}_{H}$.
	Thus we obtain the desired equation.
\end{proof}
\subsection{Fibers on each stratum of $\mathbb{B}_{H}$}
Let us take a partition $\mathcal{I}\colon I_{0},I_{1},\ldots,I_{r}$
of $\{0,1,\ldots,k\}$
and suppose $0\in I_{0}$ as usual.
Also take $\mathbf{c}\in C(\mathcal{I})\cap \mathbb{B}_{H}$
and consider 
$H(\mathbf{c})$.
Then we obtain the partial fraction decomposition 
\[
	H(\mathbf{c})=H(\mathbf{c})_{0}+H(\mathbf{c})_{1}+\cdots +H(\mathbf{c})_{r}
	\text{ with }H(\mathbf{c})_{j}\in \mathfrak{g}(\mathbb{C}[z_{c_{I_{j}}}^{-1}]_{k_{j}})
\]
and $H(\mathbf{c})$ is regarded as the collection of canonical forms 
$(H(\mathbf{c})_{j})_{j=0,\ldots,r}$.
\begin{prop}\label{prop:deformtrunc}
	Let $\pi_{\mathbb{B}_{H}}\colon \mathbb{O}_{H,\mathbb{B}_{H}}
	\rightarrow \mathbb{B}_{H}$
	be the deformation of $\mathbb{O}_{H}$.
	Then the fiber $\pi_{\mathbb{B}_{H}}^{-1}(\mathbf{c})$
	at $\mathbf{c}\in C(\mathcal{I})\cap \mathbb{B}_{H}$
	is isomorphic to 
	an open dense subset of 
	$\prod_{j=0}^{r}\mathbb{O}_{H(\mathbf{c})_{j}}$.
\end{prop}
Before giving a proof of this proposition,
we make some preparations.
Let us consider the direct product of Lie groups 
$\prod_{j=0}^{r}G(\mathbb{C}[z_{c_{I_{j}}}]_{k_{j}}).$
which naturally acts on the 
product space 
$\prod_{j=0}^{r}\mathfrak{g}(\mathbb{C}[z^{-1}_{c_{I_{j}}}]_{k_{j}}).$
By regarding 
$\mathbf{H}(\mathbf{c})=(H(\mathbf{c})_{j})_{j=0,\ldots,r}\in \prod_{j=0}^{r}\mathfrak{g}(\mathbb{C}[z^{-1}_{c_{I_{j}}}]_{k_{j}})$,
we obtain 
\[
	\prod_{j=0}^{r}\mathbb{O}_{H(\mathbf{c})_{j}}=\mathbb{O}_{\mathbf{H}(\mathbf{c})}^{\prod_{j=0}^{r}G(\mathbb{C}[z_{c_{I_{j}}}]_{k_{j}})}
\] 
where the right hand side is the $\prod_{j=0}^{r}G(\mathbb{C}[z_{c_{I_{j}}}]_{k_{j}})$-orbit 
through $\mathbf{H}(\mathbf{c})$.
We shall give a similar description of the orbit $\mathbb{O}_{\mathbf{H}(\mathbf{c})}^{\prod_{j=0}^{r}G(\mathbb{C}[z_{c_{I_{j}}}]_{k_{j}})}$
as that of the truncated orbit $\mathbb{O}_{H}$ given in Proposition \ref{prop:pullback2form}.

Let us consider the normal subgroup
\[
	G(\mathbb{C}[z_{c_{0}}]_{k_{0}})_{1}\times \prod_{j=1}^{r}G(\mathbb{C}[z_{c_{I_{j}}}]_{k_{j}})
\]
of $\prod_{j=0}^{r}G(\mathbb{C}[z_{c_{I_{j}}}]_{k_{j}})$
and the short exact sequence 
\[
	1\rightarrow G(\mathbb{C}[z_{c_{0}}]_{k_{0}})_{1}\times \prod_{j=1}^{r}G(\mathbb{C}[z_{c_{I_{j}}}]_{k_{j}})
	\rightarrow \prod_{j=0}^{r}G(\mathbb{C}[z_{c_{I_{j}}}]_{k_{j}})
	\rightarrow G\rightarrow 1	
\]
induced by the inclusion.
Then the diagonal embedding 
\[
	G\ni g\longmapsto (g,\ldots,g)\in 	\prod_{j=0}^{r}G(\mathbb{C}[z_{c_{I_{j}}}]_{k_{j}})
\]
gives a right splitting of this exact sequence and 
we obtain the semidirect product decomposition 
\[
	\prod_{j=0}^{r}G(\mathbb{C}[z_{c_{I_{j}}}]_{k_{1}})
	=G\ltimes \left(
	G(\mathbb{C}[z_{c_{0}}]_{k_{0}})_{1}\times \prod_{j=1}^{r}G(\mathbb{C}[z_{c_{I_{j}}}]_{k_{j}})
	\right).
\]
\begin{lem}\label{lem:prodquot}
	Let us consider 
	the subgroup 
	\[
	L_{1}\ltimes \left(
		G(\mathbb{C}[z_{c_{0}}]_{k_{0}})_{1}\times \prod_{j=1}^{r}G(\mathbb{C}[z_{c_{I_{j}}}]_{k_{j}})
		\right)
	\]
	of $\prod_{j=0}^{r}G(\mathbb{C}[z_{c_{I_{j}}}]_{k_{j}})$.
	Then the orbit of $\mathbf{H}(\mathbf{c})$
	under the action of  this subgroup is 
	\[
		\mathbb{O}_{H(\mathbf{c})_{0}}^{L_{1}\ltimes G(\mathbb{C}[z_{c_{0}}])_{1}}
		\times \prod_{j=1}^{r}\mathbb{O}_{H(\mathbf{c})_{j}}.
	\]
\end{lem}
\begin{proof}
	Recall that 
	$L_{1}\subset \mathrm{Stab}_{G(\mathbb{C}[z_{c_{I_{j}}}]_{k_{j}})}(H(\mathbf{c})_{j})$
	for $j=1,\ldots,r$ by Proposition \ref{prop:stabdecomp}.
	Then it follows that 
	the orbit of $\mathbf{H}(\mathbf{c})$
	under the action of 
	$L_{1}\ltimes \left(G(\mathbb{C}[z_{c_{0}}]_{k_{0}})_{1}\times \prod_{j=1}^{r}G(\mathbb{C}[z_{c_{I_{j}}}]_{k_{j}})\right)$
	coincides with that 
	of 
	$ \left(L_{1}\ltimes G(\mathbb{C}[z_{c_{0}}]_{k_{0}})_{1}\right)\times \prod_{j=1}^{r}G(\mathbb{C}[z_{c_{I_{j}}}]_{k_{j}})$,
	that is,
	$\mathbb{O}_{H(\mathbf{c})_{0}}^{L_{1}\ltimes G(\mathbb{C}[z_{c_{0}}])_{1}}
		\times \prod_{j=1}^{r}\mathbb{O}_{H(\mathbf{c})_{j}}.$
\end{proof}

Let us consider the left $L_{1}$-action on $G$
by the multiplication from the right.
Also on $\mathbb{O}_{H(\mathbf{c})_{0}}^{L_{1}\ltimes G(\mathbb{C}[z_{c_{0}}])_{1}}
\times \prod_{j=1}^{r}\mathbb{O}_{H(\mathbf{c})_{j}}$
we can consider the $L_{1}$-action by the diagonal coadjoint action.
Then we can consider 
\begin{equation*}
	G\times_{L_{1}}
	\left(
		\mathbb{O}_{H(\mathbf{c})_{0}}^{L_{1}\ltimes G(\mathbb{C}[z_{c_{0}}])_{1}}
\times \prod_{j=1}^{r}\mathbb{O}_{H(\mathbf{c})_{j}}
	\right)	
	:=
	L_{1}\bigg\backslash
	\left(G\times 
		\left(
			\mathbb{O}_{H(\mathbf{c})_{0}}^{L_{1}\ltimes G(\mathbb{C}[z_{c_{0}}])_{1}}
			\times \prod_{j=1}^{r}\mathbb{O}_{H(\mathbf{c})_{j}}
		\right)
	\right),
\end{equation*}
the fiber bundle over $G/L_{1}$.
Then as well as the diagram $(\ref{eq::comglh})$
we obtain the commutative diagram
\begin{equation}
	\begin{tikzcd}
		G\times \left(L_{1}\ltimes \left(
			G(\mathbb{C}[z_{c_{0}}]_{k_{0}})_{1}\times \prod_{j=1}^{r}G(\mathbb{C}[z_{c_{I_{j}}}]_{k_{j}})
			\right)\right)
		\arrow[r,"\widetilde{\psi}"]
		\arrow[d,"\mathrm{id}_{G}\times \pi"]
		&\prod_{j=0}^{r}G(\mathbb{C}[z_{c_{I_{j}}}]_{k_{j}})
		\arrow[d,"\pi"]\\
		G\times 
		\left(
			\mathbb{O}_{H(\mathbf{c})_{0}}^{L_{1}\ltimes G(\mathbb{C}[z_{c_{0}}])_{1}}
			\times \prod_{j=1}^{r}\mathbb{O}_{H(\mathbf{c})_{j}}
		\right)
		\arrow[r,"\overline{\psi}"]
		\arrow[d,"p"]
		&\mathbb{O}_{\mathbf{H}(\mathbf{c})}^{\prod_{j=0}^{r}G(\mathbb{C}[z_{c_{I_{j}}}]_{k_{j}})}\\
		G\times_{L_{1}}
	\left(
		\mathbb{O}_{H(\mathbf{c})_{0}}^{L_{1}\ltimes G(\mathbb{C}[z_{c_{0}}])_{1}}
\times \prod_{j=1}^{r}\mathbb{O}_{H(\mathbf{c})_{j}}
	\right)
		\arrow[ru,"\psi"]&
	\end{tikzcd},
\end{equation}
where $\widetilde{\psi}$, $\overline{\psi}$ and $\psi$
are similarly defined as them in the diagram $(\ref{eq::comglh})$,
and the upper left vertical arrow 
is the quotient map 
\[
	\pi
	\colon L_{1}\ltimes \left(
		G(\mathbb{C}[z_{c_{0}}]_{k_{0}})_{1}\times \prod_{j=1}^{r}G(\mathbb{C}[z_{c_{I_{j}}}]_{k_{j}})
		\right)
		\longrightarrow 
		\mathbb{O}_{H(\mathbf{c})_{0}}^{L_{1}\ltimes G(\mathbb{C}[z_{c_{0}}])_{1}}
		\times \prod_{j=1}^{r}\mathbb{O}_{H(\mathbf{c})_{j}}	
\]
defined by Lemma \ref{lem:prodquot}.
\begin{prop}\label{prop:prodisom}
The above map 
\[
	\psi\colon 
	G\times_{L_{1}}
	\left(
		\mathbb{O}_{H(\mathbf{c})_{0}}^{L_{1}\ltimes G(\mathbb{C}[z_{c_{0}}])_{1}}
\times \prod_{j=1}^{r}\mathbb{O}_{H(\mathbf{c})_{j}}
	\right)
	\longrightarrow 
	\mathbb{O}_{\mathbf{H}(\mathbf{c})}^{\prod_{j=0}^{r}G(\mathbb{C}[z_{c_{I_{j}}}]_{k_{j}})}
\]
is an isomorphism as complex manifolds.
\end{prop}
\begin{proof}
	We can prove similarly as Proposition \ref{prop:pullback2form}.
	The surjectivity comes from the definition.
	The injectivity comes from the fact 
	$\mathrm{Stab}_{G}(\mathbf{H}(\mathbf{c}))=\mathrm{Stab}_{G}(H)
	\subset L_{1}$.
\end{proof}
Now we are ready to give a proof of Proposition \ref{prop:deformtrunc}.
\begin{proof}[Proof of Proposition \ref{prop:deformtrunc}]
	By regarding 
	the trivial bundle $\mathbb{O}_{H_{\mathrm{res}}}^{L_{1}}\times \mathbb{B}_{H}$
	as a subbundle of $\mathfrak{g}(\widehat{\mathrm{L}}(\mathbb{B}_{H}))$
	as in Section \ref{sec:deformtrunc},
	we can see that  
	under the partial fraction decomposition 
	$\mathfrak{g}(\widehat{\mathrm{L}}(\mathbf{c}))
	\cong \prod_{j=0}^{r}\mathfrak{g}(\mathbb{C}[z_{c_{I_{j}}}^{-1}]_{k_{j}})$,
	the fiber $\mathbb{O}_{H_{\mathrm{res}}}^{L_{1}}\times \{\mathbf{c}\}
	\subset \mathfrak{g}(\widehat{\mathrm{L}}(\mathbf{c}))$
	corresponds to 
	\[
		\mathbb{O}_{H(\mathbf{c})_{0}}^{L_{1}}\times \prod_{j=1}^{r}\{H(\mathbf{c})_{j}\}
		\subset \prod_{j=0}^{r}\mathfrak{g}(\mathbb{C}[z_{c_{I_{j}}}^{-1}]_{k_{j}})
	\]	 	
	since $L_{1}\subset \mathrm{Stab}_{G(\mathbb{C}[z_{c_{I_{j}}}]_{k_{j}})}(H(\mathbf{c})_{j})$
	for $j=1,\ldots,r$ by Proposition \ref{prop:stabdecomp}.

	Also consider the fiber 
	$\prod_{l=1}^{k}(N_{l,l+1}(\widehat{\Gamma}(\mathbf{c})_{l})\times 
		\mathfrak{u}_{l,l+1}(\widehat{\mathrm{L}}(\mathbf{c})^{\mathrm{rev}}_{[l-1]}))$
	of the vector bundle $\prod_{l=1}^{k}(N_{l,l+1}(\widehat{\Gamma}(\mathbb{B}_{H})_{l})\times 
	\mathfrak{u}_{l,l+1}(\widehat{\mathrm{L}}(\mathbb{B}_{H})^{\mathrm{rev}}_{[l-1]})).$
	By combining the isomorphisms 
	in Propositions \ref{prop:liedecomp2} and \ref{dualliedecomp},
	we obtain the isomorphism
	\begin{multline*}
		\prod_{l=1}^{k}(N_{l,l+1}(\widehat{\Gamma}(\mathbf{c})_{l})\times 
		\mathfrak{u}_{l,l+1}(\widehat{\mathrm{L}}(\mathbf{c})^{\mathrm{rev}}_{[l-1]}))
		\overset{\sim}{\longrightarrow}\\
		\prod_{l=1}^{k_{0}}
			(N_{l,l+1}(\mathbb{C}[z_{c_{0}}]_{l_{0}})_{1}
			\times 
			\mathfrak{u}_{l,l+1}(\mathbb{C}[z^{-1}_{c_{0}}]_{l_{0}-1}))\\
			\times
		\prod_{j=1}^{r}
		\left(
			\prod_{l=1}^{k_{j}}
			(N_{l,l+1}(\mathbb{C}[z_{c_{I_{j}}}]_{l_{j}})
			\times 
			\mathfrak{u}_{l,l+1}(\mathbb{C}[z^{-1}_{c_{I_{j}}}]_{l_{j}}))
		\right).
	\end{multline*}

	Therefore the fiber 
	$\prod_{l=1}^{k}(N_{l,l+1}(\widehat{\Gamma}(\mathbf{c})_{l})\times 
		\mathfrak{u}_{l,l+1}(\widehat{\mathrm{L}}(\mathbf{c})^{\mathrm{rev}}_{[l-1]}))
		\times \mathbb{O}_{H_{\mathrm{res}}}^{L_{1}}$
	of the fiber bundle \[\prod_{l=1}^{k}(N_{l,l+1}(\widehat{\Gamma}(\mathbb{B}_{H})_{l})\times 
	\mathfrak{u}_{l,l+1}(\widehat{\mathrm{L}}(\mathbb{B}_{H})^{\mathrm{rev}}_{[l-1]}))
	\times \mathbb{O}_{H_{\mathrm{res}}}^{L_{1}}\]
	is isomorphic to 
	\begin{multline*}
		\left(\prod_{l=1}^{k_{0}}
			(N_{l,l+1}(\mathbb{C}[z_{c_{0}}]_{l_{0}})_{1}
			\times 
			\mathfrak{u}_{l,l+1}(\mathbb{C}[z^{-1}_{c_{0}}]_{l_{0}-1}))
			\times
			\mathbb{O}_{H(\mathbf{c})_{0}}^{L_{1}}
			\right)
			\times \\
		\prod_{j=1}^{r}\left(
		\left(
			\prod_{l=1}^{k_{j}}
			(N_{l,l+1}(\mathbb{C}[z_{c_{I_{j}}}]_{l_{j}})
			\times 
			\mathfrak{u}_{l,l+1}(\mathbb{C}[z^{-1}_{c_{I_{j}}}]_{l_{j}}))
		\right)
		\times\{H(\mathbf{c})_{j}\}
		\right).
	\end{multline*}
	By Theorems \ref{thm:triangdecompfb} and \ref{thm:triangdecompfb2},
	this is isomorphic to the dense open subset  
	\[
		\mathbb{O}_{H(\mathbf{c})_{0}}^{L_{1}\ltimes G(\mathbb{C}[z_{c_{0}}]_{k_{0}})_{1}}\times\prod_{j=1}^{r}\mathbb{O}_{H(\mathbf{c})_{j}}^{G_{0}(\mathbb{C}[z_{c_{I_{j}}}]_{k_{j}})}
	\] 
	of $\mathbb{O}_{H(\mathbf{c})_{0}}^{L_{1}\ltimes G(\mathbb{C}[z_{c_{0}}]_{k_{0}})_{1}}\times\prod_{j=1}^{r}\mathbb{O}_{H(\mathbf{c})_{j}}$.

	In conclusion, we obtain the diagram 
	\scriptsize
	\[
		\begin{tikzcd}
			G\times 
			\left(
				\mathbb{O}_{H(\mathbf{c})_{0}}^{L_{1}\ltimes G(\mathbb{C}[z_{c_{0}}]_{k_{0}})_{1}}\times\prod_{j=1}^{r}\mathbb{O}_{H(\mathbf{c})_{j}}^{G_{0}(\mathbb{C}[z_{c_{I_{j}}}]_{k_{j}})}
			\right)
			\arrow[r,hookrightarrow]
			\arrow[d,"p"]
			&
			G\times 
			\left(
				\mathbb{O}_{H(\mathbf{c})_{0}}^{L_{1}\ltimes G(\mathbb{C}[z_{c_{0}}]_{k_{0}})_{1}}\times\prod_{j=1}^{r}\mathbb{O}_{H(\mathbf{c})_{j}}	
			\right)
			\arrow[d,"p"]
			\\
			L_{1}\bigg\backslash
			\left(
				G\times 
			\left(
				\mathbb{O}_{H(\mathbf{c})_{0}}^{L_{1}\ltimes G(\mathbb{C}[z_{c_{0}}]_{k_{0}})_{1}}\times\prod_{j=1}^{r}\mathbb{O}_{H(\mathbf{c})_{j}}^{G_{0}(\mathbb{C}[z_{c_{I_{j}}}]_{k_{j}})}
			\right)
			\right)
			\arrow[r,hookrightarrow]
			&
			G\times_{L_{1}} 
			\left(
				\mathbb{O}_{H(\mathbf{c})_{0}}^{L_{1}\ltimes G(\mathbb{C}[z_{c_{0}}]_{k_{0}})_{1}}\times\prod_{j=1}^{r}\mathbb{O}_{H(\mathbf{c})_{j}}	
			\right)
		\end{tikzcd}	
	\]
	\normalsize
	which is checked to be commutative
	by the $L_{1}$-equivariance of the upper horizontal arrow. 
	Here the horizontal arrows are open embeddings and 
	the vertical arrows are quotient maps.
	Therefore by recalling the isomorphism 
	$$G\times_{L_{1}} 
	\left(
		\mathbb{O}_{H(\mathbf{c})_{0}}^{L_{1}\ltimes G(\mathbb{C}[z_{c_{0}}]_{k_{0}})_{1}}\times\prod_{j=1}^{r}\mathbb{O}_{H(\mathbf{c})_{j}}	
	\right)\cong \mathbb{O}_{\mathbf{H}(\mathbf{c})}^{\prod_{j=0}^{r}G(\mathbb{C}[z_{c_{I_{j}}}]_{k_{j}})}
	$$ in Proposition \ref{prop:prodisom},
	we obtain the desired open embedding 
	\begin{multline*}
		\pi_{H}^{-1}(\mathbf{c})
		=L_{1}\bigg\backslash\left(
			G\times 
			\left(
				\prod_{l=1}^{k}
				\left(
					N_{l,l+1}(\widehat{\Gamma}(\mathbf{c})_{l})\times 
					\mathfrak{u}_{l,l+1}(\widehat{\mathrm{L}}(\mathbf{c})^{\mathrm{rev}}_{[l-1]})	
				\right)
				\times \mathbb{O}_{H_{\mathrm{res}}}^{L_{1}}
			\right)
		\right)\overset{\sim}{\longrightarrow}\\
		L_{1}\bigg\backslash
			\left(
				G\times 
			\left(
				\mathbb{O}_{H(\mathbf{c})_{0}}^{L_{1}\ltimes G(\mathbb{C}[z_{c_{0}}]_{k_{0}})_{1}}\times\prod_{j=1}^{r}\mathbb{O}_{H(\mathbf{c})_{j}}^{G_{0}(\mathbb{C}[z_{c_{I_{j}}}]_{k_{j}})}
			\right)
			\right)
			\hookrightarrow 
			\mathbb{O}_{\mathbf{H}(\mathbf{c})}^{\prod_{j=0}^{r}G(\mathbb{C}[z_{c_{I_{j}}}]_{k_{j}})}.
	\end{multline*}
\end{proof}

Let us see the compatibility of the map $\mu_{\mathbb{O}_{H,\mathbb{B}_{H}}}$
and the embedding $\iota\colon \pi_{\mathbb{B}_{H}}^{-1}(\mathbf{c})\hookrightarrow \prod_{j=0}^{r}\mathbb{O}_{H(\mathbf{c})_{j}}$
in Proposition \ref{prop:deformtrunc}.
Define the map\index[cN]{$\mu_{\mathbf{H}(\mathbf{c})}$}
\[
	\begin{array}{cccc}
		\mu_{\mathbf{H}(\mathbf{c})}\colon &\prod_{j=0}^{r}\mathbb{O}_{H(\mathbf{c})_{j}}&\longrightarrow &\mathfrak{g}\\
		&(\xi_{j})_{j=0,\ldots,r}&\longmapsto& \sum_{j=0}^{r}\mu_{\mathbb{O}_{H(\mathbf{c})_{j}}\downarrow G}(\xi_{j})
	\end{array}
\]
which is considered as 
the moment map with respect to the diagonal $G$-action under the identification $\mathfrak{g}\cong \mathfrak{g}^{*}$
via the trace pairing as we explained in Section \ref{sec:symporb}.
\begin{prop}\label{prop:momentcompat}
	Let $\iota\colon \pi_{\mathbb{B}_{H}}^{-1}(\mathbf{c})\hookrightarrow \prod_{j=0}^{r}\mathbb{O}_{H(\mathbf{c})_{j}}$ be the open 
	embedding in Proposition \ref{prop:deformtrunc}.
	Then the following diagram is commutative,
	\[
		\begin{tikzcd}
			\pi_{\mathbb{B}_{H}}^{-1}(\mathbf{c})\arrow[r,"\mu_{\mathbb{O}_{H}\downarrow G,\mathbb{B}_{H}}|_{\pi_{\mathbb{B}_{H}}^{-1}(\mathbf{c})}"]
			\arrow[d,hookrightarrow,"\iota"]
			&\mathfrak{g}\\
			\prod_{j=0}^{r}\mathbb{O}_{H(\mathbf{c})_{j}}\arrow[ur,"\mu_{\mathbf{H}(\mathbf{c})}"']&
		\end{tikzcd}.
	\]
\end{prop}
\begin{proof}
	By Proposition \ref{prop:commmult} and the definition of $\iota_{\mathbb{O}_{H},\mathbb{B}_{H}}$, we obtain the commutative diagram,
	\footnotesize
	\[
        \begin{tikzcd}
            L_{1}\bigg\backslash\left(
            G\times 
            \left(
                \prod_{l=1}^{k}
                \left(
                    N_{l,l+1}(\widehat{\Gamma}(\mathbf{c})_{l})\times 
                    \mathfrak{u}_{l,l+1}(\widehat{\mathrm{L}}(\mathbf{c})^{\mathrm{rev}}_{[l-1]})   
                \right)
                \times \mathbb{O}_{H_{\mathrm{res}}}^{L_{1}}
            \right)
        \right)
        \arrow[r,"\iota_{\mathbb{O}_{H,\mathbb{B}_{H}}}|_{\pi_{\mathbb{B}_{H}}^{-1}(\mathbf{c})}"]
        \arrow[d,"\iota"]&
        \mathfrak{g}(\widehat{\mathrm{L}}(\mathbf{c}))\arrow[d,"\prod_{j=0}^{r}\mathrm{pr}_{c_{I_{j}}}"]\\
        \prod_{j=0}^{r}\mathbb{O}_{H(\mathbf{c})_{j}}
        \arrow[r,"\prod_{j=0}^{r}\iota_{\mathbb{O}_{H(\mathbf{c})_{j}}}"]
        &\prod_{j=0}^{r}\mathfrak{g}(\mathbb{C}[z^{-1}_{c_{i_{[j,0]}}}])
        \end{tikzcd}.
    \]\normalsize
	Then since 
	$\mu_{\mathbb{O}_{H}\downarrow G,\mathbb{B}_{H}}=\mathrm{res\,}\circ \iota_{\mathbb{O}_{H,\mathbb{B}_{H}}}$ and
	$\mu_{\mathbf{H}(\mathbf{c})}=\sum_{j=0}^{r}\underset{z=c_{I_{j}}}{\mathrm{res\,}}\circ \iota_{\mathbb{O}_{H(\mathbf{c})_{j}}}$,
	this diagram and the equation $(\ref{eq:rescom})$
	give us the desired commutative diagram.
\end{proof}

\section[Deformation of moduli spaces of $G$-connections]{Deformation of moduli spaces of $G$-connections
on $\mathbb{P}^{1}$ via unfolding of irregular singularities}\label{sec:deform}
Now we are ready to provide the precise statement of our main theorem 
and give a proof of it.

Let $D=\sum_{i=1}^{d}(k_{a_{i}}+1)\cdot a_{i}$
be an effective divisor of $\mathbb{P}^{1}$ and 
we may assume that the underlying set $|D|$ is contained 
in $\mathbb{C}$ under a projective transformation.
Let
$\mathbf{H}=(H^{(a)})_{a\in |D|}\in \prod_{a\in |D|}\mathfrak{g}(\mathbb{C}[z_{a}^{-1}]_{k_{a}})$
be a collection of canonical forms
with $\sum_{a\in |D|}\underset{z=a}{\mathrm{res}}(H^{(a)})\in \mathfrak{g}_{\mathrm{ss}}$.
Let $\mathbb{B}_{\mathbf{H}}=\prod_{a\in |D|}\mathbb{B}_{H^{(a)}}$.
We consider the 
product manifold $\prod_{a\in |D|}
\mathbb{O}_{H^{(a)}, \mathbb{B}_{H^{(a)}}}$
with the projection \index[cN]{$\pi_{\mathbb{B}_{\mathbf{H}}}$}
\[
	\pi_{\mathbb{B}_{\mathbf{H}}}
	\colon \prod_{a\in |D|}
	\mathbb{O}_{H^{(a)}, \mathbb{B}_{H^{(a)}}}\longrightarrow \prod_{a\in |D|}\mathbb{B}_{H^{(a)}}.
\]
Let 
\[
	\iota_{\mathbf{H},\mathbb{B}_{\mathbf{H}}}\colon \prod_{a\in |D|}
	\mathbb{O}_{H^{(a)}, \mathbb{B}_{H^{(a)}}}
	\longrightarrow
	\prod_{a\in |D|}\mathfrak{g}(\widehat{\mathrm{L}}(\mathbb{B}_{H^{(a)}})) 	
\]
be the product map of the immersions
$\iota_{\mathbb{O}_{H_{a}},\mathbb{B}_{H^{(a)}}}\colon 
\mathbb{O}_{H^{(a)},\mathbb{B}_{H^{(a)}}}\rightarrow \mathfrak{g}(\widehat{\mathrm{L}}(\mathbb{B}_{H^{(a)}}))$
for $a\in |D|$.
Let 
\[
	\left(\prod_{a\in |D|}\mathfrak{g}(\widehat{\mathrm{L}}(\mathbb{B}_{H^{(a)}}))\right)^{\text{ir}}	
\]
be the totality of stable points under the diagonal coadjoint action of $G$.
Here we notice that 
under the trivialization 
$\prod_{a\in |D|}\mathfrak{g}(\widehat{\mathrm{L}}(\mathbb{B}_{H^{(a)}}))\cong
\left(\bigoplus_{a\in |D|}\mathfrak{g}^{\oplus(k_{a}+1)}\right)\times \mathbb{B}_{\mathbf{H}}$,
we have 
\[
\left(\prod_{a\in |D|}\mathfrak{g}(\widehat{\mathrm{L}}(\mathbb{B}_{H^{(a)}}))\right)^{\text{ir}}
\cong\left(\bigoplus_{a\in |D|}\mathfrak{g}^{\oplus(k_{a}+1)}\right)^{\mathrm{ir}}\times \mathbb{B}_{\mathbf{H}}
\]
where $\left(\bigoplus_{a\in |D|}\mathfrak{g}^{\oplus(k_{a}+1)}\right)^{\mathrm{ir}}$
is the set of irreducible elements 
defined in Section \ref{sec:stabirred}.
We define the open subset of $\prod_{a\in |D|}
\mathbb{O}_{H^{(a)}, \mathbb{B}^{H_{(a)}}}$
by \index[cN]{$(\prod_{a\in |D|}
	\mathbb{O}_{H^{(a)}, \mathbb{B}_{H^{(a)}}})^{\text{ir}}$}
\[
	\left(\prod_{a\in |D|}
	\mathbb{O}_{H^{(a)}, \mathbb{B}_{H^{(a)}}}\right)^{\text{ir}}:=
	\iota_{\mathbf{H},\mathbb{B}_{H}}^{-1}\left(
		\left(\prod_{a\in |D|}\mathfrak{g}(\widehat{\mathrm{L}}(\mathbb{B}_{H^{(a)}}))\right)^{\text{ir}}	
	\right).
\]

For the moment map 
\[
	\mu_{\mathbf{H},\mathbb{B}_{\mathbf{H}}}\colon \prod_{a\in |D|}
	\mathbb{O}_{H_{a}, \mathbb{B}_{H^{(a)}}}\ni(X_{a})_{a\in |D|}\longmapsto 
	\sum_{a\in |D|}\mu_{\mathbb{O}_{H^{(a)}\downarrow G,\mathbb{B}_{H^{(a)}}}}(X_{a})\in \mathfrak{g}^{*},
\]
we denote the restriction of the map on 
$\left(\prod_{a\in |D|}
	\mathbb{O}_{H^{(a)},\mathbb{B}_{H^{(a)}}}\right)^{\text{ir}}$
by 
$\mu_{\mathbf{H},\mathbb{B}_{\mathbf{H}}}^{\text{ir}}$.

\begin{df}[Deformation of $\mathcal{M}^{\text{ir}}_{\mathbf{H}}$]\normalfont
	We denote the quotient spaces of the level zero sets of 
	$\mu_{\mathbf{H},\mathbb{B}_{\mathbf{H}}}$ and 
	$\mu_{\mathbf{H},\mathbb{B}_{\mathbf{H}}}^{\text{ir}}$
	under the $G$-action
	by 
	\begin{align*}
		\mathcal{M}_{\mathbf{H},\mathbb{B}_{\mathbf{H}}}&:=G\,\backslash \mu_{\mathbf{H},\mathbb{B}_{\mathbf{H}}}^{-1}(0),&
		\mathcal{M}_{\mathbf{H},\mathbb{B}_{\mathbf{H}}}^{\mathrm{ir}}&:=G\,\backslash (\mu_{\mathbf{H},\mathbb{B}_{\mathbf{H}}}^{\text{ir}})^{-1}(0).
	\end{align*}
\end{df}
The projection $\pi_{\mathbb{B}_{\mathbf{H}}}\colon \prod_{a\in |D|}\mathbb{O}_{H^{(a)},\mathbb{B}_{H^{(a)}}}\rightarrow 
\mathbb{B}_{\mathbf{H}}$ induces 
projections on the spaces defined above, $\pi_{\mathbb{B}_{\mathbf{H}}}\colon \mathcal{M}_{\mathbf{H},\mathbb{B}_{\mathbf{H}}}\rightarrow \mathbb{B}_{\mathbf{H}}$
and $\pi_{\mathbb{B}_{H}}\colon \mathcal{M}_{\mathbf{H},\mathbb{B}_{\mathbf{H}}}^{\mathrm{ir}}\rightarrow \mathbb{B}_{\mathbf{H}}.$ 
\begin{prop}\label{prop:main0}
	Suppose $\mathcal{M}_{\mathbf{H}}^{\text{ir}}\neq \emptyset$.
	Then  $\mathcal{M}_{\mathbf{H},\mathbb{B}_{\mathbf{H}}}^{\text{ir}}$ 
		is a complex orbifold and 
		the projection map 
		$\pi_{\mathbb{B}_{\mathbf{H}}}\colon \mathcal{M}_{\mathbf{H},\mathbb{B}_{\mathbf{H}}}^{\text{ir}}
		\rightarrow \mathbb{B}_{\mathbf{H}}$
		is a submersion.
\end{prop}
We make some preparations for the proof of this proposition.
\begin{lem}\label{lem:subm0}
	The projection $\pi_{\mathbb{B}_{\mathbf{H}}}\colon \left(\prod_{a\in |D|}
	\mathbb{O}_{H^{(a)}, \mathbb{B}_{H^{(a)}}}\right)^{\text{ir}}
	\rightarrow \mathbb{B}_{\mathbf{H}}$ is a submersion.
\end{lem}
\begin{proof}
	Since the fiber bundle $\mathbb{O}_{H^{(a)},\mathbb{B}_{H^{(a)}}}$
	has the trivialization $\mathbb{O}_{H^{(a)}}\times \mathbb{B}_{H^{(a)}}$
	for each $a\in |D|$,
	we can identify 
	$\prod_{a\in |D|}\mathbb{O}_{H^{(a)},\mathbb{B}_{H^{(a)}}}
	\cong \left(\prod_{a\in |D|}\mathbb{O}_{H^{(a)}}\right)\times \mathbb{B}_{\mathbf{H}}$
	under which the projection $\pi_{\mathbb{B}_{\mathbf{H}}}\colon \prod_{a\in |D|}\mathbb{O}_{H^{(a)},\mathbb{B}_{H^{(a)}}}
	\rightarrow \mathbb{B}_{\mathbf{H}}$
	corresponds to the second projection
	that is obviously a submersion.
	Then since $\left(\prod_{a\in |D|}
	\mathbb{O}_{H^{(a)}, \mathbb{B}_{H^{(a)}}}\right)^{\text{ir}}$
	is an open submanifold of $\prod_{a\in |D|}
	\mathbb{O}_{H^{(a)}, \mathbb{B}_{H^{(a)}}}$, we can see that 
	the restriction map $\pi_{\mathbb{B}_{\mathbf{H}}}\colon \left(\prod_{a\in |D|}
	\mathbb{O}_{H^{(a)}, \mathbb{B}_{H^{(a)}}}\right)^{\text{ir}}
	\rightarrow \mathbb{B}_{\mathbf{H}}$ is still a submersion.
\end{proof}
\begin{lem}\label{lem:regval0}
	Suppose that $\mathcal{M}_{\mathbf{H}}^{\mathrm{ir}}\neq \emptyset$.
	Then $0\in \mathfrak{g}^{*}$ is a regular value of 
	$\mu_{\mathbf{H},\mathbb{B}_{\mathbf{H}}}^{\mathrm{ir}}$.
\end{lem}
\begin{proof}
	Note that $(\mu_{\mathbf{H},\mathbb{B}_{\mathbf{H}}}^{\mathrm{ir}})^{-1}(0)\neq \emptyset$
	by the assumption.
	Let us take $m\in (\mu_{\mathbf{H},\mathbb{B}_{\mathbf{H}}}^{\mathrm{ir}})^{-1}(0)$
	and put $\mathbf{c}=\pi_{\mathbb{B}_{\mathbf{H}}}(m)\in \mathbb{B}_{\mathbf{H}}$.
	Then we have  
	\[
	\mu_{\mathbf{H},\mathbb{B}_{\mathbf{H}}}^{\mathrm{ir}}|_{\pi_{\mathbb{B}_{\mathbf{H}}}^{-1}(\mathbf{c})}
	=\mu_{\mathbf{H}(\mathbf{c})}^{\mathrm{ir}}
	\]
	by Proposition \ref{prop:momentcompat}.
	Since $\mathrm{dim\,}\mathrm{Stab}_{G}(m)=0$
	and $\mu_{\mathbf{H}(\mathbf{c})}^{\mathrm{ir}}$ is
	a moment map with respect to the 
	$G$-action,
	$m\in (\mu_{\mathbf{H}(\mathbf{c})}^{\mathrm{ir}})^{-1}(0)$ 
	is a regular point of $\pi_{\mathbb{B}_{\mathbf{H}}}^{-1}(\mathbf{c})$
	with respect to $\mu_{\mathbf{H}(\mathbf{c})}^{\mathrm{ir}}$.
	Since  
	$\pi_{\mathbb{B}_{\mathbf{H}}}^{-1}(\mathbf{c})$
	is a closed submanifold of $\left(\prod_{a\in |D|}
	\mathbb{O}_{H_{a}, \mathbb{B}_{H_{a}}}\right)^{\mathrm{ir}}$
	by Lemma \ref{lem:subm0},
	$m$ is also a regular point of  $\left(\prod_{a\in |D|}
	\mathbb{O}_{H_{a}, \mathbb{B}_{H_{a}}}\right)^{\mathrm{ir}}$
	with respect to $\mu_{\mathbf{H},\mathbb{B}_{\mathbf{H}}}^{\mathrm{ir}}$.
\end{proof}
\begin{proof}[Proof of Proposition \ref{prop:main0}]
	The first assertion follows from 
	Lemma \ref{lem:regval0} and the 
	fact that the $G$-action on 
	$\left(\prod_{a\in |D|}
	\mathbb{O}_{H_{a}, \mathbb{B}_{H_{a}}}\right)^{\text{ir}}$
	is proper and almost free by Lemma \ref{lem:denpan}.
	For the second assertion,
	it suffices to show that 
	the projection $\pi:=\pi_{\mathbb{B}_{\mathbf{H}}}|_{(\mu_{\mathbf{H},\mathbb{B}_{\mathbf{H}}}^{\text{ir}})^{-1}(0)}\colon (\mu_{\mathbf{H},\mathbb{B}_{\mathbf{H}}}^{\text{ir}})^{-1}(0)
	\rightarrow \mathbb{B}_{\mathbf{H}}$
	is a submersion.
	Let us take $m\in (\mu_{\mathbf{H},\mathbb{B}_{\mathbf{H}}}^{\text{ir}})^{-1}(0)$
	and put $\mathbf{c}=\pi(m)$.
	Then 
	by Proposition \ref{prop:momentcompat},
	we obtain the commutative diagram
	\[
		\begin{tikzcd}
			(\mu_{\mathbf{H}(\mathbf{c})}^{\mathrm{ir}})^{-1}(0)
			\arrow[r,hookrightarrow,"\iota_{\mathbf{c}}"]
			\arrow[d,hookrightarrow]&
			\pi_{\mathbb{B}_{\mathbf{H}}}^{-1}(\mathbf{c})
			\arrow[d,hookrightarrow]\\
			(\mu_{\mathbf{H},\mathbb{B}_{\mathbf{H}}}^{\mathrm{ir}})^{-1}(0)
			\arrow[r,hookrightarrow,"\iota_{\mathbb{B}_{\mathbf{H}}}"]
			&\prod_{a\in |D|}\mathbb{O}_{H^{(a)},\mathbb{B}_{H^{(a)}}}
		\end{tikzcd}	
	\]
	where all the arrows are natural inclusions.
	This diagram induces the following commutative diagram, 
	\[
		\begin{tikzcd}
			&&0\arrow[d]&0\arrow[d]&\\
			0\arrow[r]&T_{m}(\mu_{\mathbf{H}(\mathbf{c})}^{\text{ir}})^{-1}(0)
			\arrow[r,"T_{m}\iota_{\mathbf{c}}"]\arrow[d]&
			T_{m}\pi_{\mathbb{B}_{\mathbf{H}}}^{-1}(\mathbf{c})
			\arrow[r,"T_{m}\mu_{\mathbf{H}(\mathbf{c})}^{\text{ir}}"]
			\arrow[d]&
			T_{0}\mathfrak{g}\arrow[r]\arrow[d,equal]&0\\
			0\arrow[r]&T_{m}(\mu_{\mathbf{H},\mathbb{B}_{\mathbf{H}}}^{\text{ir}})^{-1}(0)
			\arrow[r,"T_{m}\iota_{\mathbb{B}_{\mathbf{H}}}"']&
			T_{m}\left(\prod_{a\in |D|}
			\mathbb{O}_{H_{a}, \mathbb{B}_{H_{a}}}\right)^{\text{ir}}
			\arrow[r,"T_{m}\mu_{\mathbf{H},\mathbb{B}_{\mathbf{H}}}^{\text{ir}}"']
			\arrow[d,"T_{m}\pi_{\mathbb{B}_{\mathbf{H}}}"]&
			T_{0}\mathfrak{g}\arrow[r]\arrow[d]&0\\
			&&T_{\mathbf{c}}\mathbb{B}_{\mathbf{H}}\arrow[d]&0\arrow[d]&\\
			&&0&0&
		\end{tikzcd}.	
	\]
	Here the horizontal sequences are exact, since as we saw in Lemma \ref{lem:regval0},
	$m$ is a regular point of $\pi_{\mathbb{B}_{\mathbf{H}}}^{-1}(\mathbf{c})$
	with respect to the map $\mu_{\mathbf{H}(\mathbf{c})}^{\mathrm{ir}}$ 
	and also of $\left(\prod_{a\in |D|}
	\mathbb{O}_{H_{a}, \mathbb{B}_{H_{a}}}\right)^{\text{ir}}$
	with respect to the map $\mu_{\mathbf{H},\mathbb{B}_{\mathbf{H}}}$.
	Also the middle vertical sequence is exact
	by Lemma \ref{lem:subm0}.
	The right vertical sequence is obviously exact.
	By the snake lemma, we obtain 
	$T_{\mathbf{c}}\mathbb{B}_{H}\cong 
	\mathrm{Coker\,}(T_{m}(\mu_{\mathbf{H}(\mathbf{c})}^{\text{ir}})^{-1}(0)
	\rightarrow T_{m}(\mu_{\mathbf{H},\mathbb{B}_{\mathbf{H}}}^{\text{ir}})^{-1}(0)).$

	Since $T_{m}\pi=T_{m}\pi_{\mathbb{B}_{\mathbf{H}}}\circ T_{m}\iota_{\mathbb{B}_{\mathbf{H}}}$,
	the above diagram shows that 
	the composition map 
	\[
		T_{m}(\mu_{\mathbf{H}(\mathbf{c})}^{\text{ir}})^{-1}(0)
	\rightarrow T_{m}(\mu_{\mathbf{H},\mathbb{B}_{\mathbf{H}}}^{\text{ir}})^{-1}(0)
	\overset{T_{m}\pi}{\rightarrow} T_{\mathbf{c}}\mathbb{B}_{\mathbf{H}}
	\]
	is the zero map.
	Thus 
	the map $T_{m}\pi\colon T_{m}(\mu_{\mathbf{H},\mathbb{B}_{\mathbf{H}}}^{\text{ir}})^{-1}(0)
	\rightarrow T_{\mathbf{c}}\mathbb{B}_{\mathbf{H}}$
	factors through 
	the natural projection 
	\[
		T_{m}(\mu_{\mathbf{H},\mathbb{B}_{\mathbf{H}}}^{\text{ir}})^{-1}(0)
	\rightarrow \mathrm{Coker\,}(T_{m}(\mu_{\mathbf{H}(\mathbf{c})}^{\text{ir}})^{-1}(0)
	\rightarrow T_{m}(\mu_{\mathbf{H},\mathbb{B}_{\mathbf{H}}}^{\text{ir}})^{-1}(0)).
	\]
	Namely, we have the commutative diagram
	\[
		\begin{tikzcd}
			T_{m}(\mu_{\mathbf{H},\mathbb{B}_{\mathbf{H}}}^{\text{ir}})^{-1}(0)
			\arrow[r,"T_{m}\pi"]
			\arrow[d]
			&T_{\mathbf{c}}\mathbb{B}_{\mathbf{H}}\\
			\mathrm{Coker\,}(T_{m}(\mu_{\mathbf{H}(\mathbf{c})}^{\text{ir}})^{-1}(0)
	\rightarrow T_{m}(\mu_{\mathbf{H},\mathbb{B}_{\mathbf{H}}}^{\text{ir}})^{-1}(0))
	\arrow[ur]&
		\end{tikzcd}	
	\]
	where the arrow from the bottom to the upper right is 
	the above isomorphism.
	Thus $T_{m}\pi\colon T_{m}(\mu_{\mathbf{H},\mathbb{B}_{\mathbf{H}}}^{\text{ir}})^{-1}(0)
	\rightarrow T_{\mathbf{c}}\mathbb{B}_{\mathbf{H}}$
	is a surjection.
	
	Thus we conclude that
	$\pi\colon (\mu_{\mathbf{H},\mathbb{B}_{\mathbf{H}}}^{\text{ir}})^{-1}(0)
	\rightarrow \mathbb{B}_{\mathbf{H}}$
	is a submersion. 
\end{proof}

Suppose $\mathcal{M}_{\mathbf{H}}^{\text{ir}}\neq \emptyset$.
	Let $\pi_{\mathbb{B}_{\mathbf{H}}}\colon \mathcal{M}_{\mathbf{H},\mathbb{B}_{\mathbf{H}}}^{\text{ir}}
	\rightarrow \mathbb{B}_{\mathbf{H}}$ be the projection map.
	Let $\widetilde{\mathbb{B}}_{\mathbf{H}}:=\mathrm{Im\,}\pi\neq \emptyset$
	and $\pi_{\widetilde{\mathbb{B}}_{\mathbf{H}}}\colon \mathcal{M}_{\mathbf{H},\mathbb{B}_{\mathbf{H}}}^{\text{ir}}
	\rightarrow \widetilde{\mathbb{B}}_{\mathbf{H}}$
	be the projection map induced from $\pi_{\mathbb{B}_{\mathbf{H}}}$.
\begin{thm}\label{thm:main1}
	Suppose $\mathcal{M}_{\mathbf{H}}^{\mathrm{ir}}\neq \emptyset$.
	\begin{enumerate}
		\item The reduced space  $\mathcal{M}_{\mathbf{H},\mathbb{B}_{\mathbf{H}}}^{\mathrm{ir}}$ 
		is a complex orbifold.
		\item The subspace $\widetilde{\mathbb{B}}_{\mathbf{H}}$ is an open submanifold 
		of $\mathbb{C}^{k+1}$ containing $\mathbf{0}$.
		\item The projection map 
		$\pi_{\widetilde{\mathbb{B}}_{\mathbf{H}}}\colon \mathcal{M}_{\mathbf{H},\mathbb{B}_{\mathbf{H}}}^{\mathrm{ir}}
		\rightarrow \widetilde{\mathbb{B}}_{\mathbf{H}}$
		is a surjective submersion.
		\item Every stratum $\prod_{a\in |D|}\mathcal{C}(\mathcal{I}^{(a)})$
        of $\prod_{a\in |D|}\mathbb{C}^{k_{a}+1}$ has the nonempty intersection with 
        the base space $\widetilde{\mathbb{B}}_{\mathbf{H}}$.
		
		\item The orbifold $\mathcal{M}_{\mathbf{H},\mathbb{B}_{\mathbf{H}}}^{\mathrm{ir}}$
		is a deformation of $\mathcal{M}_{\mathbf{H}}^{\mathrm{ir}}$, i.e.,
		we have 
		\[
			\mathcal{M}_{\mathbf{H}}^{\mathrm{ir}}\cong \pi_{\widetilde{\mathbb{B}}_{\mathbf{H}}}^{-1}(\mathbf{0}).
		\] 
		\item For each $\mathbf{c}\in \widetilde{\mathbb{B}}_{\mathbf{H}}$
		the fiber $\pi_{\widetilde{\mathbb{B}}_{\mathbf{H}}}^{-1}(\mathbf{c})$ is isomorphic to an open dense 
		subspace of $\mathcal{M}_{\mathbf{H}(\mathbf{c})}^{\mathrm{ir}}$.
	\end{enumerate}
\end{thm}
\begin{proof}
	The assertion 1 is the first assertion of Proposition \ref{prop:main0}.
	By Proposition \ref{prop:main0}, the projection $\pi_{\mathbb{B}_{\mathbf{H}}}\colon \mathcal{M}_{\mathbf{H},\mathbb{B}_{\mathbf{H}}}^{\text{ir}}
	\rightarrow \mathbb{B}_{\mathbf{H}}$ is a submersion.
	Thus $\widetilde{\mathbb{B}}_{\mathbf{H}}=\mathrm{Im\,}\pi\neq \emptyset$
	is open since a submersion is an open map. Thus we have the assertion 2.
	The assertion 3 follows from the assertion $2$ and Proposition \ref{prop:main0}. 
	The assertion 4 is a consequence from 
	the assertion 2 and Lemma \ref{lem:zeronbd}.
	The assertion 6 and 7 directly follow from 
	the definition of $\mathcal{M}_{\mathbf{H},\mathbb{B}_{\mathbf{H}}}^{\mathrm{ir}}$
	and  Propositions \ref{prop:momentspecfib}, \ref{prop:deformtrunc}, and 
	\ref{prop:momentcompat}.
\end{proof}

Now we can give a proof of Theorem \ref{thm:addDS} as a corollary of Theorem \ref{thm:main1}.
\begin{proof}[Proof of Theorem \ref{thm:addDS}]
	The implication $2\Rightarrow 1$ is obvious. 
	
	Thus 
	we assume that there exists a solution $\nabla_{A}=A\,dz$
	to the additive Deligne-Simpson problem 
	for $\mathbf{H}$.
	Then there exists 
	the corresponding point $m_{A}\in (\mu_{\mathbf{H}}^{\mathrm{ir}})^{-1}(0)
	\subset
	 (\mu_{\mathbf{H},\mathbb{B}_{\mathbf{H}}}^{\text{ir}})^{-1}(0)$.
	Then since 
	$\pi_{\widetilde{\mathbb{B}}_{\mathbf{H}}}\colon (\mu_{\mathbf{H},\mathbb{B}_{\mathbf{H}}}^{\text{ir}})^{-1}(0)
	\rightarrow \widetilde{\mathbb{B}}_{\mathbf{H}}$
	is a submersion,
	there exists an open neighborhood $U$ of $\mathbf{0}\in \mathbb{B}_{\mathbf{H}}$
	and a local section $s\colon U\rightarrow (\mu_{\mathbf{H},\mathbb{B}_{\mathbf{H}}}^{\text{ir}})^{-1}(0)$
	of $\pi_{\widetilde{\mathbb{B}}_{\mathbf{H}}}$
	such that $s(\mathbf{0})=m_{A}$.
	Then as the composition map 
	\[
		U\overset{s}{\longrightarrow } (\mu_{\mathbf{H},\mathbb{B}_{\mathbf{H}}}^{\text{ir}})^{-1}(0)
		\overset{\iota_{\mathbf{H}}|_{(\mu_{\mathbf{H},\mathbb{B}_{\mathbf{H}}}^{\text{ir}})^{-1}(0)}}{\longrightarrow }
		\prod_{a\in |D|}\mathfrak{g}(\widehat{\mathrm{L}}(\mathbb{B}_{H^{(a)}})),
	\]
	we obtain the section $U\rightarrow \prod_{a\in |D|}\mathfrak{g}(\widehat{\mathrm{L}}(\mathbb{B}_{H^{(a)}}))$ of 
	 $\prod_{a\in |D|}\mathfrak{g}(\widehat{\mathrm{L}}(\mathbb{B}_{H^{(a)}}))\rightarrow \mathbb{B}_{\mathbf{H}}$
	and write this section as the $\prod_{a\in |D|}\mathfrak{g}(\widehat{\mathrm{L}}(\mathbb{B}_{H^{(a)}}))$-valued function
	on $U$, 
	\[
	\left(\frac{A^{(a)}_{i}(\mathbf{c})}{\prod_{j=0}^{i}(z_{a}-c^{(a)}_{j})}\right)_{a\in |D|}\in \prod_{a\in |D|}\mathfrak{g}(\widehat{\mathrm{L}}(\mathbb{B}_{H^{(a)}}))
	\quad (\mathbf{c}=((c^{(a)}_{0},c^{(a)}_{1},\ldots,c^{(a)}_{k_{a}}))_{a\in |D|}\in U).
	\]
	Then 
	the holomorphic family of $G$-connections
	\[
	\nabla_{A(\mathbf{c})}=\sum_{a\in |D|}\sum_{i=0}^{k_{a}}\frac{A^{(a)}_{i}(\mathbf{c})}{\prod_{j=0}^{i}(z_{a}-c^{(a)}_{j})}\,dz
	\]
	is the desired solution to the additive Deligne-Simpson problem for the family $(\mathbf{H}(\mathbf{c}))_{\mathbf{c}\in U}$
	with $\nabla_{A(\mathbf{0})}=\nabla_{A}$ by the assertion 6 in Theorem \ref{thm:main1}.
	As a result, we obtain the implication $1\Rightarrow 2$.
\end{proof}


\printindex[cN]


\begin{thebibliography}{48}
	\bibitem{ALR}Adem,~A., Leida,~L., Ruan,~Y.:
	Orbifolds and stringy topology.
	Cambridge Tracts in Math. {\bf 171},
	Cambridge University Press, Cambridge (2007)
	\bibitem{Ari}Arinkin,~D.:
	Rigid irregular connections on $\mathbb{P}^{1}$.
	Compos. Math. \textbf{146} (2010), no.5, 1323--1338.
	\bibitem{Ari0}Arinkin,~D.:
	Irreducible connections admit generic oper structures.
	preprint (2016) arXiv:1602.08989.
	\bibitem{BV}Babbitt,~D., Varadarajan,~V.:
	Formal reduction theory of meromorphic differential equations: a group theoretic view.
	Pacific J. Math. \textbf{109} (1983), no.1, 1--80.
	\bibitem{Boa1}Boalch,~P.:
		Symplectic manifolds and isomonodromic deformations.
		Adv. Math. \textbf{163} (2) (2001), 137--205.
	\bibitem{Boarx}Boalch,~P.:
		Irregular connections and Kac-Moody root systems.
		 2008, arXiv:0806.1050.
	\bibitem{Bo}Boalch,~P.:
        	Simply-laced isomonodromy systems.
		Publ. Math. IHES, \textbf{116}, no. 1 (2012), 1--68.
	\bibitem{BS}Bremer,~C., Sage,~D.:
	Moduli Spaces of Irregular Singular Connections.
	IMRN, \textbf{2013} Issue 8 (2013), 1800--1872.
	\bibitem{CorGre}Corwin,~L., Greenleaf,~F.:
	Representations of nilpotent Lie groups and their applications. Part I.
	Basic theory and examples.
	Cambridge Stud. Adv. Math., \textbf{18}, 
	Cambridge University Press, Cambridge, 1990. viii+269 pp.
    \bibitem{C}Crawley-Boevey,~W.:
         On matrices in prescribed conjugacy classes
         with no common invariant subspace and sum zero.
         \textrm{Duke Math. J.}  \textbf{118},  no. 2 (2003), 339--352.
	\bibitem{Del}Deligne,~P.:
	\'Equations diff\'erentielles \`a points singuliers r\'eguliers.
	Lecture Notes in Math., Vol. \textbf{163}
	Springer-Verlag, Berlin-New York, 1970. iii+133 pp.
	\bibitem{Gul}Glutsuk,~A.: 
	Stokes operators via limit monodromy of generic perturbation.
	J. Dynam. Control Systems \textbf{5} (1999), 101--135.
	\bibitem{GMR}Gaiur,~I., Mazzocco,~M. Rubtsov,~V.:
	Isomonodromic deformations: Confluence, Reduction and Quantisation.
	Comm. in Math. Phys., \textbf{400} (2023), 1385--1461.
	\bibitem{GM}Goresky,~M., MacPherson,~R.:
	\text{Intersection homology theory}. Topology \textbf{19} (1980), no. 2, 135--162.
    \bibitem{H2}Hiroe,~K.:
	     Linear differential equations on the Riemann sphere and
		 representations of quivers. Duke Math. J., \textbf{166}, 855--935,
		 (2017).
	\bibitem{H3}Hiroe,~K.:
	Unfolding of spectral types.
	Josai Math. Monographs \textbf{12} (2020), 53--67.
	\bibitem{H4}Hiroe,~K.:
	Index of rigidity of differential equations and Euler characteristic of their spectral curves.
	J. Geom. Phys. \textbf{162} (2021), Paper No. 104060, 16 pp.
	\bibitem{HKNS}Hiroe,~K., Kawakami,~H., Nakamura,~A., Sakai,~H.:
	4-dimensional Painlev\'e type equations. MSJ memoirs, \textbf{37},
	2018. 
	\bibitem{HY}Hiroe,~K., Yamakawa,~D.:
	  Moduli spaces of meromorphic connection and quiver
	  varieties. Adv. Math., \textbf{266} (2014), 120--151.
	\bibitem{HYp}Hiroe,~K., Yamakawa,~D.:
	Unfolding of wild  character varieties, arXiv:2511.01221, 2025.
	\bibitem{Huk}Hukuhara,~M.:
	Sur les points singuliers des \'equations dif\'erentielles lin\'eaires. III.
	Mem. Fac. Sci. Kyusyu Imp. Univ. A \textbf{2} (1942) 125--137.
	\bibitem{HuRo}Hurtubise,~J., Rousseau,~C.:
	Moduli space for generic unfolded differential linear systems.
	Adv. Math. \textbf{307} (2017), 1268--1323.
	\bibitem{Ina}Inaba,~M.:
	Unfolding of the unramified irregular singular
	generalized isomonodromic deformation.
	Bull. Sci. Math. \textbf{157} (2019), 102795, 121 pp.
	\bibitem{JY}Jakob,~K., Yun,~Z.:
	A Deligne-Simpson problem for irregular G-connections over $\mathbb{P}^{1}$.
	preprint (2023) arXiv:2301.10967.
 	\bibitem{Katz}Katz,~N.:
	Rigid local systems. Princeton University Press, Princeton, NJ, 1996. viii+223 pp.
	\bibitem{KNS}Kawakami,~H., Nakamura,~A., Sakai,~H.:
	Degeneration Scheme of 4-dimensional Painlev\'e-type Equations
	in \cite{HKNS}.
	\bibitem{Kem}Kempf,~G.:
	Instability in invariant theory (preprint version).
	preprint (1976), arXiv:1807.02890.
	\bibitem{Kl}Klime\v{s},~M.:
	Analytic classification of families
	of linear differential systems unfolding a resonant irregular singularity.
	SIGMA \textbf{16} (2020), 006, 46 pp.
	\bibitem{Kl2}Klime\v{s},~M.:
	Wild monodromy of the Fifth Painlev\'e equation 
	and its action on wild character variety: an approach of confluence.
	Ann. Inst. Fourier, Vol. \textbf{74} (2024) no. 1, 121-192.
	\bibitem{Kos}Kostov,~V.:
	The Deligne-Simpson problem-a survey.
	J. Algebra, \textbf{281} (2004), 83--108.
	\bibitem{KLMNS}Kulkarni,~M., Livesay,~N., Matherne,~J., Nguyen,~B., Sage,~D.: 
	The Deligne–Simpson problem for connections
	on $\mathbb{G}_{m}$ with a maximally ramified singularity. 
	Adv. Math. \textbf{408} (2022), 108596.
	\bibitem{LaRo}Lambert,~C., Rousseau,~C.:
	The Stokes phenomenon in the confluence of the hypergeometric equation using Riccati equation.
	J. Differential Equations \textbf{244} (2008), 2641--2664.
	\bibitem{Lev}Levelt,~G.:
	Jordan decomposition for a class of singular differential operators.
	Ark. Mat. \textbf{13} (1975) 1--27.
	\bibitem{LSN}Livesay,~N., Sage,~D., Nguyen,~B.:
	Explicit constructions of connections on the projective line with a maximally ramified irregular singularity,
	preprint (2023), arXiv:2303.06581.
	\bibitem{Mack}Mackenzie,~K.:
	General theory of Lie groupoids and Lie algebroids. London Mathematical Society Lecture
	Note Series, vol. 213, Cambridge University Press, Cambridge, 2005.
	\bibitem{Mal}Malgrange,~B.:
	\'Equations diff\'erentielles \`a coefficients polynomiaux.
	Progr. Math., \textbf{96}
	Birkh\"a user Boston, Inc., Boston, MA, 1991. vi+232 pp.
	\bibitem{MarsWein} Marsden,~J., Weinstein,~A.:
	Reduction of symplectic manifolds with symmetry.
	Rep. Mathematical Phys. \textbf{5} (1974), 121--130.
	\bibitem{MFK}Mumford,~D., Fogarty,~J., Kirwan,~F.:
	Geometric invariant theory. Third edition.
	Ergeb. Math. Grenzgeb. (2), \textbf{34}
	Springer-Verlag, Berlin, 1994. xiv+292 pp.
	\bibitem{Oshi1}Oshima,~T.: 
	Confluence and versal unfolding of Pfaffian equations. Josai Mathematical Monographs \textbf{12} (2020), 117--151.
	\bibitem{Oshi2}Oshima,~T.:
	Versal unfolding of irregular singurarities of a linear differential equation on the Riemann sphere.
	Publ. RIMS Kyoto Uinv. \textbf{57} (2021), 893--920.
	\bibitem{PR}Paul.,~E., Ramis.~J.-P.:
	Dynamics of the fifth Painlev\'e foliation. 
	Handbook of Geometry and Topology of Singularities VI: Foliations.
	Springer-Verlag, Berlin, 2024, 307--382.
	\bibitem{Ram}Ramis,~J.-P.:
	Confluence et r\'esurgence.
	J. Fac. Sci. Univ. Tokyo Sect. IA Math. \textbf{36} (1989), 703--716.
	\bibitem{Sak}Sakai,~H.:
	Isomonodromic Deformation and 4-dimensional Painlev\'e-type Equations
	in \cite{HKNS}.
	\bibitem{Scha}Sch\"afke,~R.:
	Confluence of several regular singular points into an irregular singular one.
	J. Dynam. Control Systems \textbf{4} (1998), 401--424.
	\bibitem{Sim}Simpson,~C.:
	Products of matrices.
	Differential Geometry, Global Analysis, and Topology (Halifax, 1990), CMS Conf. Proc. \textbf{12}, Amer. Math. Soc. Providence, 
	(1991), 157--185.
	\bibitem{SL}Sjamaar,~R., Lerman,~E.:
	Stratified symplectic spaces and reduction.
	Ann. Math. \textbf{134} (1991), 375--422.
	\bibitem{Sp}Springer,~T.:
	Linear algebraic groups. Second edition.
	Progr. Math., \textbf{9}
	Birkh\"auser Boston, Inc., Boston, MA, 1998. xiv+334 pp.
	\bibitem{Tur}Turrittin,~H.:
	Convergent solutions of ordinary linear homogeneous differential equations in the neighborhood of an 
	irregular singular point. Acta Math. \textbf{93} (1955) 27--66.
    \bibitem{Yam1}Yamakawa,~D.:
	Quiver Varieties with Multiplicities, Weyl Groups of Non-Symmetric Kac-Moody Algebras, and Painlev\'e Equations.
	SIGMA \textbf{6} (2010), 087.	 
	\bibitem{Yam}Yamakawa,~D.:
	Fundamental two-forms for isomonodromic deformations.
	J. Integrable Syst. \textbf{4} (2019), no.1, 1--35.
	\bibitem{Zhan}
	Zhang, ~C.:
	Confluence et ph\'enom\`ene de Stokes.
	J. Math. Sci. Univ. Tokyo \textbf{3} (1996), 91--107.
 	\end{thebibliography}
\end{document}